\pdfoutput=1
\documentclass[11pt]{article}
\usepackage{authblk}
\usepackage[toc,page]{appendix}
\usepackage[top=2cm, bottom=2cm, left=2cm, right=2cm]{geometry}

\usepackage{color}
\usepackage{helvet}         
\usepackage{courier}        
\usepackage{type1cm}        

\usepackage{framed}
\usepackage{tikz}
\usepackage{makeidx}         
\usepackage{graphicx}        
\usepackage{multicol}        
\usepackage[bottom]{footmisc}

\usepackage{amsmath}
\usepackage{amssymb}
\usepackage{bbold}
\usepackage{amsthm}
\usepackage{subcaption}
\usepackage{sidecap}
\usepackage{floatrow}
\usepackage{pdflscape}
\usepackage{pdflscape}
\usepackage{comment}
\usepackage[font=small]{caption}
\usepackage{enumitem}
\usepackage{esint}

\usepackage{scalerel}
\usepackage{hyperref}

\newtheorem{theorem}{Theorem}
\newtheorem{corollary}[theorem]{Corollary}
\newtheorem{lemma}[theorem]{Lemma}
\newtheorem{conjecture}[theorem]{Conjecture}
\newtheorem{proposition}[theorem]{Proposition}

\theoremstyle{remark}

\newtheorem{remark}[theorem]{\bf Remark}
\newtheorem{definition}[theorem]{\bf Definition}

\numberwithin{theorem}{section}
\numberwithin{question}{section}
\numberwithin{figure}{section}
\numberwithin{equation}{section}

\begin{document}

\title{Connection Probabilities of Multiple FK-Ising Interfaces}
\bigskip{}
\author[1]{Yu Feng\thanks{yufeng\_proba@163.com}}
\author[2]{Eveliina Peltola\thanks{eveliina.peltola@hcm.uni-bonn.de}}
\author{Hao Wu\thanks{hao.wu.proba@gmail.com.}}
\affil[1]{Tsinghua University, China}
\affil[2]{Aalto University, Finland, and University of Bonn, Germany}

\date{}

%
%

\global\long\def\CR{\mathrm{CR}}
\global\long\def\ST{\mathrm{ST}}
\global\long\def\SF{\mathrm{SF}}
\global\long\def\cov{\mathrm{cov}}
\global\long\def\dist{\mathrm{dist}}
\global\long\def\SLE{\mathrm{SLE}}
\global\long\def\hSLE{\mathrm{hSLE}}
\global\long\def\CLE{\mathrm{CLE}}
\global\long\def\GFF{\mathrm{GFF}}
\global\long\def\inte{\mathrm{int}}
\global\long\def\ext{\mathrm{ext}}
\global\long\def\inrad{\mathrm{inrad}}
\global\long\def\outrad{\mathrm{outrad}}
\global\long\def\dimH{\mathrm{dim}}
\global\long\def\capa{\mathrm{cap}}
\global\long\def\diam{\mathrm{diam}}
\global\long\def\free{\mathrm{free}}
\global\long\def\hF{{}_2\mathrm{F}_1}
\global\long\def\simple{\mathrm{simple}}
\global\long\def\st{\mathrm{ST}}
\global\long\def\usf{\mathrm{USF}}
\global\long\def\Leb{\mathrm{Leb}}
\global\long\def\LP{\mathrm{LP}}
\global\long\def\coulomb{\LH}
\global\long\def\coulombnew{\LG}
\global\long\def\kfunc{p}
\global\long\def\OO{\mathcal{O}}

\global\long\def\eps{\epsilon}
\global\long\def\ov{\overline}
\global\long\def\U{\mathbb{U}}
\global\long\def\T{\mathbb{T}}
\global\long\def\HH{\mathbb{H}}
\global\long\def\LA{\mathcal{A}}
\global\long\def\LB{\mathcal{B}}
\global\long\def\LC{\mathcal{C}}
\global\long\def\LD{\mathcal{D}}
\global\long\def\LF{\mathcal{F}}
\global\long\def\LK{\mathcal{K}}
\global\long\def\LE{\mathcal{E}}
\global\long\def\LG{\mathcal{G}}
\global\long\def\LI{\mathcal{I}}
\global\long\def\LJ{\mathcal{J}}
\global\long\def\LL{\mathcal{L}}
\global\long\def\LM{\mathcal{M}}
\global\long\def\LN{\mathcal{N}}
\global\long\def\LQ{\mathcal{Q}}
\global\long\def\LR{\mathcal{R}}
\global\long\def\LT{\mathcal{T}}
\global\long\def\LS{\mathcal{S}}
\global\long\def\LU{\mathcal{U}}
\global\long\def\LV{\mathcal{V}}
\global\long\def\LW{\mathcal{W}}
\global\long\def\LX{\mathcal{X}}
\global\long\def\LY{\mathcal{Y}}
\global\long\def\PartF{\mathcal{Z}}
\global\long\def\LH{\mathcal{H}}
\global\long\def\LJ{\mathcal{J}}
\global\long\def\R{\mathbb{R}}
\global\long\def\C{\mathbb{C}}
\global\long\def\N{\mathbb{N}}
\global\long\def\Z{\mathbb{Z}}
\global\long\def\E{\mathbb{E}}
\global\long\def\PP{\mathbb{P}}
\global\long\def\QQ{\mathbb{Q}}
\global\long\def\A{\mathbb{A}}
\global\long\def\one{\mathbb{1}}
\global\long\def\bn{\mathbf{n}}
\global\long\def\MR{MR}
\global\long\def\cond{\,|\,}
\global\long\def\la{\langle}
\global\long\def\ra{\rangle}
\global\long\def\tree{\Upsilon}
\global\long\def\prob{\mathbb{P}}
\global\long\def\hm{\mathrm{Hm}}

\global\long\def\sf{\mathrm{SF}}
\global\long\def\wr{\varrho}

\global\long\def\Im{\operatorname{Im}}
\global\long\def\Re{\operatorname{Re}}

\global\long\def\ud{\mathrm{d}}
\global\long\def\pder#1{\frac{\partial}{\partial#1}}
\global\long\def\pdder#1{\frac{\partial^{2}}{\partial#1^{2}}}
\global\long\def\der#1{\frac{\ud}{\ud#1}}

\global\long\def\bZnn{\mathbb{Z}_{\geq 0}}

\global\long\def\Vfunc{\LG}
\global\long\def\gfunc{g^{(\rr)}}
\global\long\def\hfunc{h^{(\rr)}}

\global\long\def\SimplexInt{\rho}
\global\long\def\CubeInt{\widetilde{\rho}}

\global\long\def\ii{\mathfrak{i}}
\global\long\def\rr{\mathfrak{r}}
\global\long\def\chamber{\mathfrak{X}}
\global\long\def\Wchamber{\mathfrak{W}}

\global\long\def\SimplexIntKappa8{\SimplexInt}

\global\long\def\nested{\boldsymbol{\underline{\Cap}}}
\global\long\def\unnested{\boldsymbol{\underline{\cap\cap}}}

\global\long\def\acycle{\vartheta}
\global\long\def\bcycle{\tilde{\acycle}}
\global\long\def\Gloop{\Theta}

\global\long\def\metric{\mathrm{dist}}

\global\long\def\adj#1{\mathrm{adj}(#1)}

\global\long\def\bs{\boldsymbol}

\global\long\def\edge#1#2{\langle #1,#2 \rangle}
\global\long\def\graph{G}

\newcommand{\conn}{\vartheta_{\scaleobj{0.7}{\mathrm{RCM}}}}
\newcommand{\FKconn}{\vartheta_{\scaleobj{0.7}{\mathrm{FK}}}}
\newcommand{\hatconn}{\widehat{\vartheta}_{\mathrm{RCM}}}
\newcommand{\FKhatconn}{\widehat{\vartheta}_{\mathrm{FK}}}
\newcommand{\realpt}{\smash{\mathring{x}}}
\newcommand{\corrind}{\LC}
\newcommand{\bssymb}{\pi}
\newcommand{\PRCM}{\mu}
\newcommand{\coeff}{p}
\newcommand{\MainConst}{C}

\global\long\def\removeLink{/}
\maketitle

\begin{center}
\begin{minipage}{0.95\textwidth}
\abstract{
We find the scaling limits of a general class of boundary-to-boundary connection
probabilities and multiple interfaces in the critical planar FK-Ising model, thus verifying predictions from the physics literature.
We also discuss conjectural formulas using Coulomb gas integrals 
for the corresponding quantities in general critical planar random-cluster models with cluster-weight $q \in [1,4)$.
Thus far, proofs for convergence, including ours, rely on discrete complex analysis techniques and are beyond reach for other values of $q$ than the FK-Ising model ($q=2$). 
Given the convergence of interfaces, the conjectural formulas for other values of $q$ could be verified similarly with relatively minor technical work.
The limit interfaces are variants of $\SLE_\kappa$ curves  (with $\kappa = 16/3$ for $q=2$). 
Their partition functions, that give the connection probabilities, also satisfy properties predicted for correlation functions in conformal field theory (CFT),  expected to describe scaling limits of critical random-cluster models. 
We verify these properties for all $q \in [1,4)$, thus providing further evidence of the expected CFT description of these models.
}

\bigskip{}


\end{minipage}
\end{center}

\newpage

\setcounter{tocdepth}{2}
\tableofcontents

\newpage
\allowdisplaybreaks

\section{Introduction}
\label{sec::intro}
Fortuin and Kasteleyn introduced the \emph{random-cluster model} around the 1970s as a general family of discrete percolation models that combines together Bernoulli percolation, graphical representations of spin models (Ising \& Potts models), and polymer models (as a limiting case).
Generally in such models, edges are declared to be open or closed according to a given probability measure, the simplest being the independent product measure of Bernoulli percolation.
Of particular interest in such models are percolation properties, that is, whether various points in space are connected by paths of open edges. The present article is concerned with boundary-to-boundary connections 
in the planar case.
Such \emph{connection events}, or \emph{crossing events}, have been used for a convenient description of the large-scale properties of the Bernoulli percolation model in~\cite{Schramm-Smirnov:Scaling_limits_of_planar_percolation, GPS:Pivotal_cluster_and_interface_measures_for_critical_planar_percolation}, whereas for dependent percolation models such a description would be much more complex (cf.~\cite[Question~1.22]{Schramm-Smirnov:Scaling_limits_of_planar_percolation},
see also~\cite{DCMT:Planar_random-cluster_model_fractal_properties_of_the_critical_phase}). 

Random-cluster models have been under active research in the past decades, for instance due to their important feature of \emph{criticality}: for certain parameter values the model exhibits a continuous phase transition. 
Criticality can be practically identified as follows.
Consider on a lattice with small mesh, say $\delta \Z^2$, the probability that an open path connects two opposite sides of a topological rectangle. It is not hard to prove that this probability tends to zero as $\delta \to 0$ when the model is ``subcritical'',  
while it tends to one as $\delta \to 0$ when the model is ``supercritical''.  
At the critical point,  
the connection probability has a nontrivial limit, which is a real number in $(0,1)$ that depends on the shape (i.e.,~conformal modulus) of the topological rectangle.
This latter fact follows from Russo-Seymour-Welsh type estimates that are now ubiquitous tools for percolation models~\cite{CDCH:Crossing_probabilities_in_topological_rectangles_for_critical_planar_FK_Ising_model, DCHN:Connection_probabilities_and_RSW_type_bounds, DCST:Continuity_of_phase_transition_for_planar_random-cluster_and_Potts_models}. 
Exact identification of the limit of the connection probability, though, is highly non-trivial.
Motivated by numerical experiments by Langlands, Pouliot, and Saint-Aubin~\cite{LPS:Conformal_invariance_in_2d_percolation}, 
an answer in the physics level of rigor using conformal field theory predictions was given by Cardy for the case of Bernoulli percolation in~\cite{Cardy:Critical_percolation_in_finite_geometries}.
The first proof of Cardy's formula was established 
by Smirnov~\cite{Smirnov:Critical_percolation_in_the_plane} using miraculous discrete complex analysis tricks 
\`a la Kenyon~\cite{Kenyon:Conformal_invariance_of_domino_tiling} and Smirnov).
To date, analogues and generalizations of Cardy's formula have been proven only for a number of other models, all of which rely on some kind of  
specific exact solvability (or ``magic'', quoting Smirnov\footnote{``Since it used magic, it only works in situations where there is magic, and we weren't able to find magic in other situations.'' in Quanta Magazine (July 8, 2021)
\emph{Mathematicians Prove Symmetry of Phase Transitions} by Allison Whitten.}), 
mainly due to underlying free fermion or free boson structures: 
critical spin-Ising model and FK-Ising model, Gaussian free field, loop-erased random walks, and uniform spanning trees
(see~\cite{Kenyon-Wilson:Boundary_partitions_in_trees_and_dimers, Chelkak-Smirnov:Universality_in_2D_Ising_and_conformal_invariance_of_fermionic_observables, Izyurov:Smirnovs_observable_for_free_boundary_conditions_interfaces_and_crossing_probabilities, Peltola-Wu:Global_and_local_multiple_SLEs_and_connection_probabilities_for_level_lines_of_GFF, Karrila:UST_branches_martingales_and_multiple_SLE2,  KKP:Boundary_correlations_in_planar_LERW_and_UST, 
Izyurov:On_multiple_SLE_for_the_FK_Ising_model,
LPW:UST_in_topological_polygons_partition_functions_for_SLE8_and_correlations_in_logCFT} 
and references therein). 
In the continuum, some connection probabilities for $\CLE$ loops were found in~\cite{Miller-Werner:Connection_probabilities_for_conformal_loop_ensembles}, see also~\cite{Ang-Sun:Integrability_of_CLE} for recent results relating to Liouville theory.
Analogous numerical results and predictions for connectivity events in the bulk for the random-cluster and Potts models were found in~\cite{DPSV:Connectivities_of_Potts_Fortuin-Kasteleyn_clusters_and_time-like_Liouville_correlator}.

\bigskip

The phase transition in random-cluster models has been argued 
to result in \emph{conformal invariance} and \emph{universality} for the scaling limit $\delta \to 0$ of the model (see, e.g.,~\cite{Cardy:Scaling_and_renormalization_in_statistical_physics}).
Since then, tremendous progress has been established towards verifying this prediction. 
Recently, in~\cite{DKKMO:Rotational_invariance_in_critical_planar_lattice_models} 
it was shown that correlations in the critical random-cluster model with cluster-weight $q \in [1,4]$ 
do indeed become rotationally invariant in the scaling limit. This provides very strong evidence of conformal invariance, while still not being enough to prove it. 
For the special case of the FK-Ising model ($q=2$), 
conformal invariance has been established rigorously to a large extent, thanks to special integrability properties of the model that allow the use of discrete complex analysis in a fundamental way (the ``magic'' referred to above),
cf.~\cite{Smirnov:Conformal_invariance_in_random_cluster_models1, Chelkak-Smirnov:Universality_in_2D_Ising_and_conformal_invariance_of_fermionic_observables,  CDHKS:Convergence_of_Ising_interfaces_to_SLE, Izyurov:Smirnovs_observable_for_free_boundary_conditions_interfaces_and_crossing_probabilities,  Kemppainen-Smirnov:Conformal_invariance_in_random-cluster_models-II, Kemppainen-Smirnov:Conformal_invariance_of_boundary_touching_loops_of_FK_Ising_model,  Izyurov:On_multiple_SLE_for_the_FK_Ising_model}.

Crucially, in addition to proving conformal invariance, identifying the scaling limit objects with their corresponding counterparts in \emph{conformal field theory} (CFT) 
is necessary in order to get access to the full power of the CFT formalism applicable to critical lattice models. The purpose of this article is to provide such an identification for boundary-to-boundary connection probabilities in the FK-Ising model 
with various boundary conditions (Theorems~\ref{thm::FKIsing_Loewner} and~\ref{thm::FKIsing_crossingproba}). 
Analogous results remain conjectural for other values\footnote{Bernoulli site percolation on the triangular lattice ($q=1$,~a slightly different setup) is presented in~\cite{PW:Crossing_probabilities_of_critical_percolation_interfaces}.} 
of $q \in [1,4)$.  
We also provide formulas for the quantities of interest for all $q \in [1,4)$ in terms of solutions to PDE boundary value problems and Coulomb gas integrals, earlier appearing, e.g., in~\cite{Dubedat:Euler_integrals_for_commuting_SLEs, Flores-Kleban:Solution_space_for_system_of_null-state_PDE4, FSKZ:A_formula_for_crossing_probabilities_of_critical_systems_inside_polygons}. 
We also verify CFT predictions for all these formulas (Theorem~\ref{thm::CGI_property}),
thus providing further evidence for the CFT description of these critical planar models.

\smallbreak

Our main results are summarized in Sections~\ref{subsec::Results_FKIsing_Loewner}--\ref{subsec::CGI_properties}. 
We first discuss the general setup and common terminology for the random-cluster models and the conjectural formulas for the connection probabilities (Sections~\ref{subsec::RCM}--\ref{subsec::Conjectures}). 
Section~\ref{subsec::Results_FKIsing_Loewner} then focuses on results in the special case of the FK-Ising model, and 
Section~\ref{subsec::CGI_properties} gathers important properties of the Coulomb gas integral formulas in general.

\subsection{Random-cluster models in polygons}
\label{subsec::RCM}

Here, we summarize notation and terminology to be used throughout, and define the random-cluster model.
For more background and properties of these models, we recommend~\cite{Grimmett:Random_cluster_model, Duminil-Copin:PIMS_lectures}.

\begin{figure}[ht!]
\includegraphics[width=0.45\textwidth]{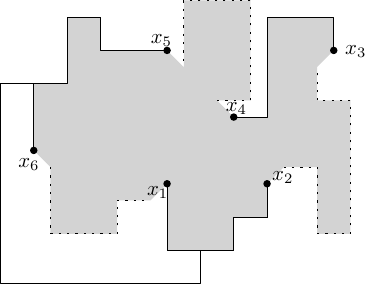}
$\quad$
\includegraphics[width=0.45\textwidth]{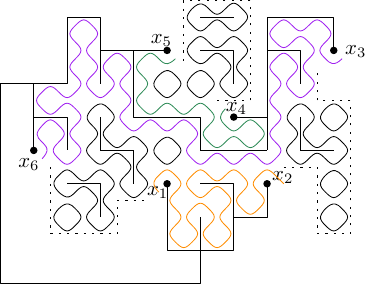}
\caption{\label{fig::loop_representation}
Consider discrete polygons (gray) with six marked boundary points. One possible boundary condition for the random-cluster model is illustrated in the left figure, where the arcs $(x_1 \, x_2), (x_3 \, x_4), (x_5 \, x_6)$ are wired, and the arcs $(x_1 \, x_2)$ and $(x_5 \, x_6)$ are further wired outside of the polygon. 
This boundary condition corresponds to the non-crossing partition $\{\{1,3\}, \{2\}\}$ of the three wired boundary arcs. 
One possible random-cluster configuration in terms of its loop representation is illustrated in the right figure.  
It comprises loops (black) and three interfaces inside the polygon: the orange curve connects $x_1^{\diamond}$ and $x_{2}^{\diamond}$; the purple curve connects $x_3^{\diamond}$ and $x_6^{\diamond}$; and the green curve connects $x_4^{\diamond}$ and $x_5^{\diamond}$. See Section~\ref{sec::FKIsing_Loewner} for details. }
\end{figure}

\paragraph*{Notation and terminology.}
For definiteness, we consider subgraphs $\graph = (V(\graph), E(\graph))$ of the square lattice $\Z^2$, which is the graph with vertex set $V(\Z^2):=\{ z = (m, n) \colon m, n\in \Z\}$ and edge set $E(\Z^2)$ given by edges between 
those vertices whose Euclidean distance equals one (called neighbors). 
This is our primal lattice. Its standard dual lattice is denoted by $(\Z^2)^{\bullet}$. 
The medial lattice $(\Z^2)^{\diamond}$ is the graph with centers of edges of $\Z^2$ as its vertex set and edges connecting  
neighbors. 
For a subgraph $\graph \subset \Z^2$ (resp.~of $(\Z^2)^{\bullet}$ or $(\Z^2)^{\diamond}$), we define its \emph{boundary} to be the following set of vertices:
\begin{align*}
\partial \graph = \{ z \in V(\graph) \, \colon \, \exists \; w \not\in V(\graph) \textnormal{ such that }\edge{z}{w}\in E(\Z^2)\} .
\end{align*}
When we add the subscript or superscript $\delta$, we mean that subgraphs of the lattices $\Z^2, (\Z^2)^{\bullet}, (\Z^2)^\diamond$ have been scaled by $\delta > 0$. 
We consider the models in the \emph{scaling limit} $\delta \to 0$. 
For a given medial graph $\Omega^{\delta, \diamond} \subset (\delta \Z^2)^{\diamond}$,
let $\Omega^{\delta}\subset\delta\Z^2$ be the graph on the primal lattice corresponding to $\Omega^{\delta, \diamond}$ (see details in Section~\ref{subsec::rcm_pre}). 
By a (discrete) \emph{polygon} we either refer to the medial graph $\Omega^{\delta, \diamond}$ endowed with given distinct boundary points $x_1^{\delta, \diamond}, \ldots, x_{2N}^{\delta, \diamond}$ in counterclockwise order, 
or to the corresponding primal graph  
$(\Omega^{\delta}; x_1^{\delta}, \ldots, x_{2N}^{\delta})$ 
with given boundary points $x_1^{\delta}, \ldots, x_{2N}^{\delta}$ in counterclockwise order.
We consider random-cluster models on such polygons, where the boundary behavior changes at the marked boundary points.

\paragraph*{Random-cluster model.}
Let $\graph = (V(\graph), E(\graph))$ be a finite subgraph of $\Z^2$.
A random-cluster \emph{configuration} 
$\omega=(\omega_e)_{e \in E(\graph)}$ is an element of $\{0,1\}^{E(\graph)}$.
An edge $e \in E(\graph)$ is said to be \emph{open} (resp.~\emph{closed}) if $\omega_e=1$ (resp.~$\omega_e=0$).
We view the configuration $\omega$ 
as a subgraph of $\graph$ with vertex set $V(\graph)$  and edge set $\{e\in E(\graph) \colon \omega_e=1\}$.
We denote by $o(\omega)$ (resp.~$c(\omega)$) the number of open (resp.~closed) edges in~$\omega$.

We are interested in the connectivity properties of the graph $\omega$ with various boundary conditions. 
The maximal connected\footnote{Two vertices $z$ and $w$ are said to be \emph{connected} by $\omega$ if there exists a sequence  $\{z_j \colon 0\le j\le l\}$  of vertices such that
$z_0 = z$ and $z_l = w$, and each edge $\edge{z_j}{z_{j+1}}$ is open in $\omega$ for $0 \le j < l$.} components of $\omega$ are called \emph{clusters}.
The boundary conditions encode how the vertices are connected outside of $\graph$.
Precisely, by a \emph{boundary condition} $\bssymb$ we refer to a partition $\bssymb_1 \sqcup \cdots \sqcup \bssymb_m$ of the boundary $\partial \graph$.
Two vertices $z,w \in \partial \graph$ are said to be \emph{wired} in $\bssymb$ if $z,w \in \bssymb_j$ for some common $j$. 
In contrast, \emph{free} boundary segments comprise vertices that are not wired with any other vertex (so the corresponding part $\pi_j$ is a singleton).  
We denote by $\omega^{\bssymb}$
the (quotient) graph obtained from the configuration $\omega$ by identifying the wired vertices in $\bssymb$.

Finally, the \emph{random-cluster model} on $\graph$ with edge-weight $p\in [0,1]$, cluster-weight $q>0$,
and boundary condition $\bssymb$, is the probability measure $\smash{\PRCM^{\bssymb}_{p,q,\graph}}$ on
the set $\{0,1\}^{E(\graph)}$ of configurations $\omega$  defined by
\begin{align*}
\PRCM^{\bssymb}_{p,q,\graph}[\omega] 
:= \; & \frac{p^{o(\omega)}(1-p)^{c(\omega)}q^{k(\omega^{\bssymb})}}{\underset{\varpi \in \{0,1\}^{E(\graph)}}{\sum} p^{o(\varpi)}(1-p)^{c(\varpi)}q^{k(\varpi^{\bssymb})} } ,
\end{align*}
where $k(\omega^{\bssymb})$ is the number of 
connected components of the graph $\omega^{\bssymb}$.
For $q=2$, this model is also known as the \emph{FK-Ising model}, while for $q=1$, it is simply the Bernoulli bond percolation (assigning independent values for each $\omega_e$). The random-cluster model combines together several important  models in the same family.
For integer values of $q$, it is very closely related to the $q$-Potts model, 
and by taking a suitable limit, the case of $q=0$ corresponds to the uniform spanning tree (see, e.g.,~\cite{Duminil-Copin:PIMS_lectures}).
It has been proven for the range $q \in [1,4]$ in~\cite{DCST:Continuity_of_phase_transition_for_planar_random-cluster_and_Potts_models} 
that when the edge-weight is chosen suitably, namely as (the critical, self-dual value)
\begin{align} \label{eq: pcrit}
p = p_c(q) := \frac{\sqrt{q}}{1+\sqrt{q}} ,
\end{align}
then the random-cluster model exhibits a \emph{continuous phase transition}
in the sense that after taking the infinite-volume (thermodynamic) limit,
for $p > p_c(q)$ there almost surely exists an infinite cluster, while for $p < p_c(q)$ there does not, 
and the limit $p \searrow p_c(q)$ is approached in a continuous way. 
(This is also expected to hold when $q \in (0,1)$, while it is known that the phase transition is discontinuous when $q > 4$ by~\cite{DCGHMT:Discontinuity_of_the_phase_transition_for_the_planar_random-cluster_and_Potts_with_q_larger_than_4}.)
Therefore, the scaling limit of the model at its critical point~\eqref{eq: pcrit}
is expected to be conformally invariant for all $q \in [0,4]$.
In the present article, we will consider multiple interfaces and boundary-to-boundary connection probabilities in the critical random-cluster model with $q \in [1,4)$.
See also~\cite{LPW:UST_in_topological_polygons_partition_functions_for_SLE8_and_correlations_in_logCFT} for the uniform spanning tree model corresponding to $q=0$.

\paragraph*{Markov property.}
At the heart of many geometric arguments concerning the random-cluster model is its (domain) \emph{Markov property}:
the restriction of the model to a smaller graph only depends on the boundary condition induced by such a restriction. 
To state this more precisely, fix any $p\in [0,1]$ and $q>0$, and 
suppose that $\graph \subset \graph'$ are two finite subgraphs of $\Z^2$ 
and that we have fixed a boundary condition $\bssymb$ for the model on the boundary $\partial \graph'$ of the larger graph.
Let $X$ be a random variable which is measurable with respect to the status of the edges in the smaller graph $\graph$.
Then, for all $\upsilon \in \{0,1\}^{E(\graph')\setminus E(\graph)}$,  we have
\begin{align*}
\PRCM^{\bssymb}_{p,q,\graph'} \big[X \; | \; \omega_e = \upsilon_e  \textnormal{ for all } e \in E(\graph')\setminus E(\graph)\big] 
= \PRCM^{\upsilon^{\bssymb}}_{p,q,\graph}[X],
\end{align*}
where $\upsilon^{\bssymb}$ is the partition on $\partial \graph$ obtained by wiring two vertices in $\partial \graph$  if they are connected in $\upsilon$.

For instance, taking $\graph$ to be a connected component of the complement of the purple curve 
in Figure~\ref{fig::loop_representation}, we obtain a random-cluster model on the smaller graph $\graph$ with modified boundary conditions.

\paragraph*{Boundary conditions.}

Consider now the random-cluster model on a polygon $(\Omega^{\delta}; x_1^{\delta}, \ldots, x_{2N}^{\delta})$ 
with the following boundary conditions: 
first, every other boundary arc is wired, 
\begin{align*}
(x_{2r-1}^{\delta} \, x_{2r}^{\delta}) \textnormal{ is wired,} \qquad \textnormal{ for all } r \in \{1,2,\ldots, N\} ,
\end{align*}
and second, these $N$ wired arcs are further wired together according to a \emph{non-crossing partition} $\bssymb$ outside of $\Omega^{\delta}$, as illustrated in Figures~\ref{fig::loop_representation} and~\ref{fig::6points}.
Note that there is a natural bijection $\beta \leftrightarrow \bssymb_\beta$ between non-crossing partitions $\bssymb_\beta$ of the $N$ wired boundary arcs and \emph{planar link patterns} $\beta$ with $N$ links,
\begin{align} \label{eqn::linkpatterns_ordering}
\begin{split}
& \beta = \{ \{a_1,b_1\},  \{a_2,b_2\},\ldots , \{a_N,b_N\}\} \\
& \textnormal{with link endpoints ordered as } \; 
a_1 < a_2 < \cdots < a_N \textnormal{ and } a_r < b_r , \textnormal{ for all } 1 \leq r \leq N ,  \\
& \textnormal{and such that there are no indices } 1 \leq r , s \leq N \textnormal{ with } a_r < a_s < b_r < b_s , 
\end{split}
\end{align} 
where $\{a_1, b_1, \ldots, a_N, b_N\}=\{1,2, \ldots, 2N\}$ and the pairs $\{a_j,b_j\}$ are called \emph{links}.
Hence, we encode the boundary condition $\bssymb_\beta$ in a label $\beta$. 
We denote by $\LP_N \ni \beta$ the set of planar link patterns of $N$ links. 

\begin{figure}[ht!]
\begin{subfigure}[b]{\textwidth}
\begin{center}
\includegraphics[width=0.14\textwidth]{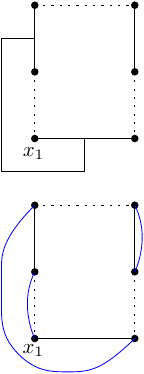}
\end{center}
\caption{This boundary condition is encoded in the planar link pattern $\beta=\{\{1,6\},\{2,5\},\{3,4\}\}$. }
\end{subfigure}\\
\vspace{0.5cm}
\begin{subfigure}[b]{\textwidth}
\begin{center}
\includegraphics[width=0.14\textwidth]{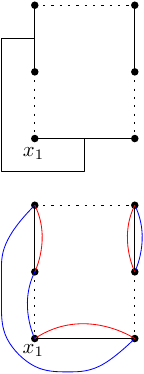}$\quad$
\includegraphics[width=0.14\textwidth]{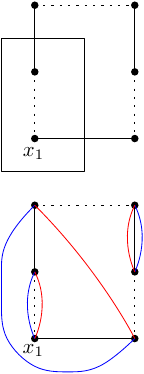}$\quad$
\includegraphics[width=0.14\textwidth]{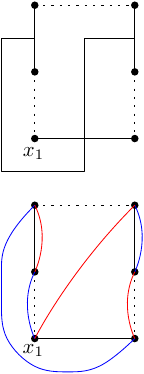}$\quad$
\includegraphics[width=0.14\textwidth]{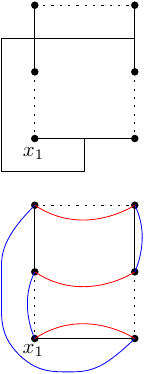}$\quad$
\includegraphics[width=0.14\textwidth]{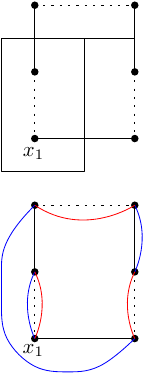}
\end{center}
\caption{For six marked boundary points, 
there are five possible planar \emph{internal} link patterns $\alpha$. From left to right, the meanders formed from $\alpha$ and $\beta$ have two loops, three loops, one loop, one loop, and two loops, respectively.}
\end{subfigure}
\caption{\label{fig::6points} 
Consider discrete polygons with six marked points on the boundary. 
One possible boundary condition  
for the random-cluster model is illustrated in (a). 
The corresponding possible planar link patterns $\alpha$ formed by the interfaces are depicted in red in (b)(bottom),
and they correspond to non-crossing partitions inside (b)(top). }
\end{figure}

Let $\omega$ be a critical random-cluster configuration on $\Omega^{\delta}$ with boundary condition $\beta$.
For notational ease, keeping $q \in [1,4)$ and $p = p_c(q)$ fixed, we denote its law by 
\begin{align*}
\PP_{\beta}^{\delta} := \PRCM^{\bssymb_{\beta}}_{p_c(q),q,\Omega^{\delta}} .
\end{align*}
We consider in particular the cluster boundaries of $\omega$ (that is, its loop representation, see Figure~\ref{fig::loop_representation} and Section~\ref{sec::FKIsing_Loewner}). 
By planarity, there exist $N$ curves, \emph{interfaces}, on the medial graph $\Omega^{\delta, \diamond}$ running along $\omega$ and connecting the marked points $\{x_1^{\delta, \diamond}, x_2^{\delta, \diamond},\ldots, x_{2N}^{\delta, \diamond}\}$ pairwise, as also illustrated in Figure~\ref{fig::loop_representation}.  
Let us denote by $\conn^{\delta}$ the random planar connectivity in $\LP_N$ formed by the $N$ discrete interfaces. 
In this article, we are particularly interested in the \emph{connection probabilities} 
$\smash{\PP_{\beta}^{\delta}}[\conn^{\delta}=\alpha]$ for $\alpha\in\LP_N$,
as functions of the marked boundary points ---
Figure~\ref{fig::6points} illustrates these crossing events.
The goal is to study conjectures for the scaling limits of the  interfaces and their 
connection probabilities, and prove these conjectures for the case of the critical FK-Ising model ($q=2$ and $p = p_c(2) = \frac{\sqrt{2}}{1+\sqrt{2}}$).

\paragraph*{Scaling limits.}
To specify in which sense the convergence as $\delta \to 0$ should take place, we need a notion of convergence of  polygons. In contrast to the commonly used Carath\'{e}odory convergence of planar sets, we need a slightly stronger notion termed \emph{close-Carath\'{e}odory convergence},
following Karrila~\cite{Karrila:Limits_of_conformal_images_and_conformal_images_of_limits_for_planar_random_curves}.
The precise definition will be given in Section~\ref{subsec::rcm_pre} (Definition~\ref{def:closeCara}). 
Roughly speaking, the usual Carath\'{e}odory convergence allows wild behavior of the boundary approximations, 
while in order to obtain \emph{tightness} of the random interfaces (i.e.,~\emph{precompactness} needed to find convergent subsequences), a slightly stronger convergence which guarantees good approximations around the marked boundary points is required.

We also need a topology for the interfaces, which we regard as (images of) continuous mappings from $[0,1]$ to $\C$ modulo reparameterization (i.e., planar oriented curves). 
For a simply connected domain $\Omega \subsetneq \C$, we will consider curves in $\overline{\Omega}$.
For definiteness, we map $\Omega$ onto the unit disc 
$\U := \{ z \in \C \colon |z| < 1 \}$: 
for this we shall fix\footnote{The metric~\eqref{eq::curve_metric} depends on the choice of the conformal map $\Phi$, but the induced topology does not.} 
any conformal map $\Phi$ from $\Omega$ onto $\U$. 
Then, we endow the curves with the metric
\begin{align} \label{eq::curve_metric} 
\metric(\eta_1, \eta_2) 
:= \inf_{\psi_1, \psi_2} \sup_{t\in [0,1]} |\Phi(\eta_1(\psi_1(t)))-\Phi(\eta_2(\psi_2(t)))| ,
\end{align}
where the infimum is taken over all increasing homeomorphisms $\psi_1, \psi_2 \colon [0,1]\to[0,1]$.
The space of continuous curves on $\overline{\Omega}$ modulo reparameterizations then becomes a complete separable metric space.

\paragraph*{Loewner chains.}
To describe scaling limits of interfaces, we recall that
planar chordal curves can be dynamically generated by Loewner evolution. In general, any continuous real-valued function, called the \emph{driving function}  $W_t \colon [0,\infty) \to \R$, gives rise to a growing family of sets via the following recipe
(see~\cite{Lawler:Conformally_invariant_processes_in_the_plane, Rohde-Schramm:Basic_properties_of_SLE} for background).
The \emph{Loewner equation} 
\begin{align} \label{eqn:LE}
\partial_{t}{g}_{t}(z) 
= \frac{2}{g_{t}(z)-W_{t}} ,
\qquad \textnormal{with initial condition} \qquad g_{0}(z)=z ,
\end{align}
is an ordinary differential equation in time $t \geq 0$, for each fixed point $z \in \HH := \{ z \in \C \colon \Im(z) > 0 \}$ in the upper half-plane. 
It has a unique solution $(g_{t}, t\ge 0)$ up to 
$T_{z} := \sup\{t\ge 0 \colon \min_{s\in[0,t]} |g_{s}(z)-W_{s}|>0\}$,
called the \emph{swallowing time} of $z$. 
The Loewner chain is a dynamical family of conformal bijections\footnote{In fact, $g_t \colon \HH\setminus K_{t}  \to \HH$ is the unique conformal map such that $|g_K(z)-z| \to 0$ as $z \to \infty$.} 
$g_{t} \colon \HH\setminus K_{t} \to \HH$, where the hull of swallowed points is $K_{t}:=\overline{\{z\in\HH \colon T_{z}\le t\}}$. 
We also say that the Loewner chain is parameterized by half-plane capacity, which refers to the property that for each time $t \geq 0$, the coefficient of $z^{-1}$ in the series expansion of $g_t$ at infinity equals $2t$ (this coefficient is, by definition, the half-plane capacity of the hull $K_{t}$, measuring its size as seen from infinity).

The family $(K_{t}, t\ge 0)$ of hulls is also often called a \emph{Loewner chain}, and it is said to be generated by a continuous curve $\eta \colon [0,T) \to \overline{\HH}$ if for each $t \in [0,T)$, the set $\HH\setminus K_{t}$ is the unbounded connected component of $\HH\setminus \eta[0,t]$. We also refer to the curve $\eta$ as a Loewner chain.
An example of a Loewner chain generated by a continuous curve is the
chordal \emph{Schramm-Loewner evolution}, $\SLE_{\kappa}$, 
that is the random Loewner chain driven by $W = \sqrt{\kappa} \, B$,
a standard one-dimensional Brownian motion $B$ of speed $\kappa > 0$. 
This family indexed by $\kappa$ is uniquely determined by the following two properties. 
\begin{itemize}[leftmargin=2em]
\item \emph{Conformal invariance:} 
The law of the $\SLE_{\kappa}$ curve $\eta$ in any simply connected domain $\Omega$ is the pushforward of the law of the $\SLE_{\kappa}$ curve in $\HH$ by a conformal map
$\varphi \colon \HH \to \Omega$ which maps the two points $0, \infty$ to the two endpoints 
of $\eta$.

\item \emph{Domain Markov property:} 
given a stopping time $\tau$ and initial segment $\eta[0,\tau]$ of
the $\SLE_\kappa$ curve in $\HH$, 
the conditional law of the remaining piece 
$\eta[\tau,\infty)$ is the law of the $\SLE_\kappa$ curve from the tip $\eta(\tau)$ to $\infty$
in the unbounded connected component of $\HH \setminus \eta[0,\tau]$. 
\end{itemize}

The standard $\SLE_\kappa$ curve in $\HH$ connects the boundary points $0 = \eta(0)$ and $\infty = \lim_{t \to \infty} |\eta(t)|$. 
One can change the target point by adding a specific drift to the driving Brownian motion 
(corresponding to the case $N=1$ in Theorem~\ref{thm::FKIsing_Loewner} when $\kappa = 16/3$). 
The parameter $\kappa > 0$ describes the behavior and the fractal dimension of the $\SLE_\kappa$ curve. 
For instance, it is almost surely a simple curve when $\kappa \leq 4$, while for $\kappa \geq 8$, the $\SLE_\kappa$ curve is almost surely space-filling.
In the intermediate parameter range $\kappa \in (4,8)$, 
including the parameter range considered in the present article, 
the $\SLE_\kappa$ curve almost surely has self-touchings, but is not space-filling.
See~\cite{Lawler:Conformally_invariant_processes_in_the_plane, Rohde-Schramm:Basic_properties_of_SLE} for background and further properties of this process. 

\subsection{Conjectures for random-cluster models}
\label{subsec::Conjectures}

Let us now fix parameters 
\begin{align*}
\kappa\in (4,8) , \qquad 
h(\kappa) := \frac{6-\kappa}{2\kappa} , \qquad \textnormal{and} \qquad 
q(\kappa) := 4\cos^2(4\pi/\kappa) .
\end{align*}
Note that when $\kappa\in (4,6]$, we have $q = q(\kappa) \in [1,4)$ corresponding to the critical random-cluster model with $p = p_c(q)$. 
(The case of $\kappa = 4$ corresponds to $q = 4$, which is still critical. We comment on this case in Remark~\ref{rem:q=4}.)
To state the expected formulas describing the scaling limits of multiple interfaces and  
connection probabilities in the critical random-cluster models, we define for each $\beta \in \LP_N$ the basis \emph{Coulomb gas integral} functions\footnote{Since $\kappa > 4$, these integrals are convergent, for their singularities at the endpoints of the contours are mild enough.} as 
\begin{align}
\nonumber
\coulombnew_{\beta} \colon \; & \chamber_{2N} \to \R , 
\qquad \textnormal{where}\qquad 
\chamber_{2N} := \big\{ \bs{x} := (x_{1},\ldots,x_{2N}) \in \R^{2N} \colon x_{1} < \cdots < x_{2N} \big\} , \\
\label{eqn::coulombgasintegral}
\coulombnew_\beta (x_1, \ldots, x_{2N}) :=  \; &
\bigg( \frac{\sqrt{q(\kappa)} \, \Gamma(2-8/\kappa)}{\Gamma(1-4/\kappa)^2} \bigg)^N
\landupint_{x_{a_1}}^{x_{b_1}} 
\cdots \landupint_{x_{a_N}}^{x_{b_N}}
f (\bs{x};u_1,\ldots,u_N) \, \ud u_1 \cdots \ud u_N ,
\end{align}
where the integration contours are pairwise non-intersecting paths in the upper half-plane connecting the marked points pairwise according to the connectivity $\beta$, and 
the integrand is 
\begin{align} \label{eq: integrand}
f (\bs{x};u_1,\ldots,u_N) := \; &
\prod_{1\leq i<j\leq 2N}(x_{j}-x_{i})^{2/\kappa} 
\prod_{1\leq r<s\leq N}(u_{s}-u_{r})^{8/\kappa} 
\prod_{\substack{1\leq i\leq 2N \\ 1\leq r\leq N}}
(u_{r}-x_{i})^{-4/\kappa} , 
\end{align} 
and the branch of this multivalued integrand is chosen to be real and positive when
\begin{align*}
x_{a_r} < \Re(u_r) < x_{a_r+1} , \qquad \textnormal{ for all } 1 \leq r \leq N .
\end{align*}
In~\eqref{eqn::coulombgasintegral}, we use the integration symbols $\smash{\landupint_{x_{a_r}}^{x_{b_r}}} \ud u_r$ to indicate that the integration  
of the variable $u_r$ is performed from $x_{a_r}$ to $x_{b_r}$ in the upper half-plane.
Formulas of type~\eqref{eqn::coulombgasintegral}, while originating from the Coulomb gas formalism of conformal field theory~\cite{Dotsenko-Fateev:Conformal_algebra_and_multipoint_correlation_functions_in_2D_statistical_models,  Kytola-Peltola:Conformally_covariant_boundary_correlation_functions_with_quantum_group}, 
have appeared in the $\SLE$ literature~\cite{Dubedat:Euler_integrals_for_commuting_SLEs, Dubedat:Commutation_relations_for_SLE, Kytola-Peltola:Pure_partition_functions_of_multiple_SLEs} as partition functions for $\SLE_\kappa$ variants, 
and have then been used in the physics literature~\cite{Flores-Kleban:Solution_space_for_system_of_null-state_PDE4, FSKZ:A_formula_for_crossing_probabilities_of_critical_systems_inside_polygons} pertaining to Conjecture~\ref{conj::rcm_crossingproba}. 
Our formulas are motivated by their properties listed in Theorem~\ref{thm::CGI_property}.
In particular, $\coulombnew_{\beta}$ are indeed partition functions of multiple $\SLE_\kappa$ curves.

\smallbreak

For fixed $N\ge 1$, by a \emph{polygon} $(\Omega; x_1, \ldots, x_{2N})$ we refer to a bounded simply connected domain $\Omega\subset\C$ with distinct marked boundary points $x_1, \ldots, x_{2N} \in \partial\Omega$ in counterclockwise order, such that $\partial\Omega$ is locally connected. 
We extend the definition of $\coulombnew_{\beta}$ to a general polygon 
$(\Omega; x_1, \ldots, x_{2N})$ whose marked boundary points 
$x_1, \ldots, x_{2N}$ lie on sufficiently regular boundary segments (e.g., $C^{1+\eps}$ for some $\eps>0$) as
\begin{align}\label{eqn::coulombgasintegral_general}
\coulombnew_{\beta}(\Omega; x_1, \ldots, x_{2N}) 
:= \prod_{j=1}^{2N} |\varphi'(x_j)|^{h(\kappa)} \times \coulombnew_{\beta}(\varphi(x_1), \ldots, \varphi(x_{2N})) ,
\end{align}
where $\varphi$ is any conformal map from $\Omega$ onto $\HH$ with $\varphi(x_1)<\cdots<\varphi(x_{2N})$. 
It follows from the M\"obius covariance~\eqref{eqn::COV} in Theorem~\ref{thm::CGI_property} that this definition is independent of the choice of the map $\varphi$.

\smallbreak

We formulate the next Conjectures~\ref{conj::rcm_Loewner} and~\ref{conj::rcm_crossingproba} in the case of square-lattice approximations, which is the setup that we use to give detailed proofs of these conjectures for the critical FK-Ising model in Theorems~\ref{thm::FKIsing_Loewner} and~\ref{thm::FKIsing_crossingproba}. By universality, we expect the same results to hold with any approximations.
In fact, one should be able to readily extend Theorems~\ref{thm::FKIsing_Loewner} and~\ref{thm::FKIsing_crossingproba} to more general discrete approximations following the lines of~\cite{Chelkak-Smirnov:Universality_in_2D_Ising_and_conformal_invariance_of_fermionic_observables, Chelkak:Ising_model_and_s-embeddings_of_planar_graphs}. 
For the sake of presentation, we content ourselves in the present work to the simplest setup.

\begin{conjecture} \label{conj::rcm_Loewner}
Fix a polygon $(\Omega; x_1, \ldots, x_{2N})$ and a link pattern $\beta \in \LP_N$. 
Suppose that a sequence $(\Omega^{\delta, \diamond}; x_1^{\delta, \diamond}, \ldots, x_{2N}^{\delta, \diamond})$ of medial polygons converges to $(\Omega; x_1, \ldots, x_{2N})$ in the close-Carath\'{e}odory sense \textnormal{(}as detailed in Definition~\ref{def:closeCara}\textnormal{)}. 
Consider the critical random-cluster model with cluster-weight $q \in [1,4)$ on the primal polygon $(\Omega^{\delta}; x_1^{\delta}, \ldots, x_{2N}^{\delta})$ with boundary condition $\beta$. 
For each $i\in\{1,2,\ldots, 2N\}$, let $\eta_i^{\delta}$ be the interface starting from the boundary point $x_i^{\delta, \diamond}$. Let $\varphi$ be any conformal map from $\Omega$ onto $\HH$ such that $\varphi(x_1)<\cdots<\varphi(x_{2N})$. 
Then, $\eta_i^{\delta}$ converges weakly to the image under $\varphi^{-1}$ of the Loewner chain with the following driving function, up to the first time when $\varphi(x_{i-1})$ or $\varphi(x_{i+1})$ is swallowed: 
\begin{align} \label{eqn::rcm_Loewner_chain}
\begin{cases}
\ud W_t = \sqrt{\kappa} \, \ud B_t + \kappa(\partial_i\log \coulombnew_{\beta})(V_t^{1}, \ldots, V_t^{i-1}, W_t, V_t^{i+1}, \ldots, V_t^{2N}) \, \ud t, \\
\ud V_t^j =\frac{2 \, \ud t}{V_t^j-W_t},\\ 
W_0=\varphi(x_i),\\
V_0^j=\varphi(x_j),\quad j\in\{1,\ldots, i-1, i+1, \ldots, 2N\}, 
\end{cases}
\end{align} 
where $\coulombnew_{\beta}$ is defined in~\eqref{eqn::coulombgasintegral}. 
\end{conjecture}

\smallbreak

We prove Conjecture~\ref{conj::rcm_Loewner}
for $q=2$ in Theorem~\ref{thm::FKIsing_Loewner}.

\smallbreak

\begin{definition} \label{def::meander}
A \emph{meander} formed from two link patterns $\alpha,\beta\in\LP_N$ is the planar diagram obtained by placing $\alpha$ and the horizontal reflection $\beta$ on top of each other. We denote by $\LL_{\alpha,\beta}$ the number of loops in the meander formed from $\alpha$ and $\beta$.  We define the \emph{meander matrix} $\{\LM_{\alpha, \beta}(q(\kappa)) \colon \alpha,\beta\in\LP_N\}$ via 
\begin{align} \label{eqn::meandermatrix_def_general}
\LM_{\alpha,\beta}(q(\kappa)) := \sqrt{q(\kappa)}^{\; \LL_{\alpha,\beta}}. 
\end{align}
An example of a meander is
\begin{align*} 
\alpha \quad = \quad \vcenter{\hbox{\includegraphics[scale=0.275]{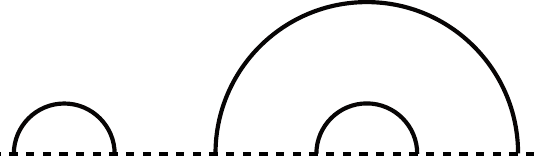}}} 
\quad  , \quad  
\beta \quad = \quad\vcenter{\hbox{\includegraphics[scale=0.275]{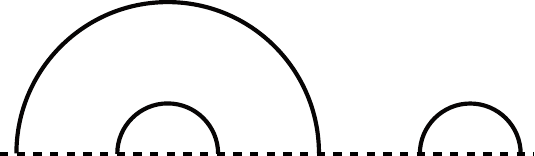}}} 
\quad\quad\quad \Longrightarrow \quad\quad\quad
\vcenter{\hbox{\includegraphics[scale=0.275]{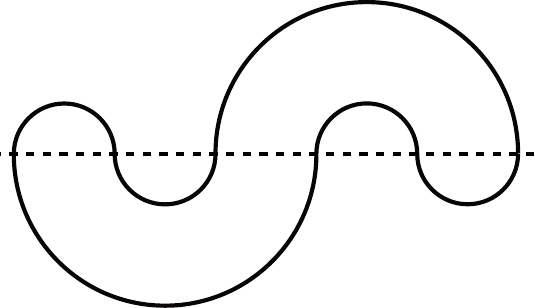}}} 
\end{align*}
\end{definition}

\begin{conjecture} \label{conj::rcm_crossingproba}
Assume the same setup as in Conjecture~\ref{conj::rcm_Loewner}. The endpoints of the $N$ interfaces give rise to a random planar link pattern $\conn^{\delta}$ in $\LP_N$. For any $\alpha\in\LP_N$, we have 
\begin{align} \label{eqn::rcm_crossingproba}
\lim_{\delta\to 0}\PP_{\beta}^{\delta}[\conn^{\delta}=\alpha] 
\quad = \quad \LM_{\alpha,\beta}(q(\kappa)) \; \frac{\PartF_{\alpha}(\Omega; x_1, \ldots, x_{2N})}{\coulombnew_{\beta}(\Omega; x_1, \ldots, x_{2N})}, 
\end{align}
where $\coulombnew_{\beta}$ and $\LM_{\alpha,\beta}$ are defined in~\textnormal{(\ref{eqn::coulombgasintegral},~\ref{eqn::coulombgasintegral_general})}  and~\eqref{eqn::meandermatrix_def_general}, respectively, and $\{\PartF_{\alpha} \colon \alpha\in\LP_N\}$ is the collection of pure partition functions for multiple $\SLE_{\kappa}$ described 
in Definition~\ref{def::PPF_general} below. 
\end{conjecture}

We prove Conjecture~\ref{conj::rcm_crossingproba}
for $q=2$ in Theorem~\ref{thm::FKIsing_crossingproba}.

\smallbreak

The content of Conjectures~\ref{conj::rcm_Loewner} and~\ref{conj::rcm_crossingproba} has been predicted in the physics literature and also numerically verified in some cases with high precision, see~\cite{SKFZ:Cluster_densities_at_2D_critical_points_in_rectangular_geometries, FSKZ:A_formula_for_crossing_probabilities_of_critical_systems_inside_polygons} and references therein.
Via a similar strategy as in the proof of Theorem~\ref{thm::totalpartition}, by using Theorem~\ref{thm::purepartition_unique} one can verify that 
our formula~\eqref{eqn::coulombgasintegral} for $\coulombnew_\beta$ is consistent with the prediction in~\cite[Eq.~(11)]{FSKZ:A_formula_for_crossing_probabilities_of_critical_systems_inside_polygons}.

\smallbreak

``Pure partition functions'' refer to a family of smooth functions defined as solutions to a system of partial differential equations (PDEs) important in both CFT and SLE theory, with certain recursive asymptotic boundary conditions.
Uniqueness results for solutions to PDEs are usually not available. However, it was proven by Flores \&~Kleban~\cite{Flores-Kleban:Solution_space_for_system_of_null-state_PDE2, Flores-Kleban:Solution_space_for_system_of_null-state_PDE3} that 
in this particular case, we do have a classification if we impose certain additional requirements (covariance (COV) and growth bound (PLB)).
The PDEs appear in the pioneering CFT articles~\cite{BPZ:Infinite_conformal_symmetry_in_2D_QFT, BPZ:Infinite_conformal_symmetry_of_critical_fluctuations_in_2D} 
of Belavin, Polyakov, and Zamolodchikov (BPZ) as a feature of the algebraic structure of conformal symmetry for certain fields, 
and in early articles in SLE theory by 
Bauer, Bernard, and Kyt{\"o}l{\"a}~\cite{BBK:Multiple_SLEs_and_statistical_mechanics_martingales},
and 
Dub{\'e}dat~\cite{Dubedat:Euler_integrals_for_commuting_SLEs, Dubedat:Commutation_relations_for_SLE},
as a manifestation of certain martingales. 
\begin{itemize}[leftmargin=3em]
\item[\textnormal{(PDE)}] 
\textnormal{\bf BPZ equations:} 
\begin{align}\label{eqn::PDE}
\hspace*{-5mm}
\bigg[ 
\frac{\kappa}{2}\pdder{x_{j}}
+ \sum_{i \neq j} \Big( \frac{2}{x_{i}-x_{j}} \pder{x_{i}} 
- \frac{2h(\kappa)}{(x_{i}-x_{j})^{2}} \Big) \bigg]
F(x_1,\ldots,x_{2N}) =  0 , \qquad \textnormal{for all } j \in \{1,\ldots,2N\} .
\end{align}
\end{itemize}

The covariance gives a version of global conformal symmetry for the functions. 
\begin{itemize}[leftmargin=3em]
\item[\textnormal{(COV)}] 
\textnormal{\bf M\"{o}bius covariance:} 
\begin{align}\label{eqn::COV}
F(x_{1},\ldots,x_{2N}) = 
\prod_{i=1}^{2N} \varphi'(x_{i})^{h(\kappa)} 
\times F(\varphi(x_{1}),\ldots,\varphi(x_{2N})) ,
\end{align}
for all M\"obius maps $\varphi$ of the upper half-plane $\HH$ 
such that $\varphi(x_{1}) < \cdots < \varphi(x_{2N})$.
\end{itemize}

\begin{definition} \label{def::PPF_general}
Fix $\kappa\in (0,6]$. 
The \emph{pure partition functions} of multiple $\SLE_{\kappa}$ are the recursive collection $\{\PartF_{\alpha} \colon \alpha \in \bigsqcup_{N\geq 0} \LP_N\}$ 
of functions $\PartF_{\alpha} \colon \chamber_{2N}\to\R_{>0}$ 
uniquely determined by the following properties. 
They satisfy the PDE system~\eqref{eqn::PDE}, 
M\"{o}bius covariance~\eqref{eqn::COV}, 
as well as~\textnormal{(ASY)} and~\textnormal{(PLB)} given below. 
\begin{itemize}[leftmargin=3em]
\item[\textnormal{(ASY)}] 
\textnormal{\bf Asymptotics:} 
With $\PartF_{\emptyset} \equiv 1$ for the empty link pattern $\emptyset \in \LP_0$, the collection $\{\PartF_{\alpha} \colon \alpha\in\LP_N\}$ satisfies the following recursive 
asymptotics property.
Fix $N \ge 1$ and $j \in \{1,2, \ldots, 2N-1 \}$. 
Then, we have
\begin{align}
\label{eqn::PPFASY_general} 
\; & \lim_{x_j,x_{j+1}\to\xi} \frac{\PartF_{\alpha}(\bs{x})}{ (x_{j+1}-x_j)^{-2h(\kappa)} }
= 
\begin{cases}
\PartF_{\alpha/\{j,j+1\}}(\bs{\ddot{x}}_j), 
& \textnormal{if }\{j, j+1\}\in\alpha , \\
0 ,
& \textnormal{if }\{j, j+1\} \not\in \alpha ,
\end{cases}
\end{align}
where
\begin{align} \label{eqn::bs_notation}
\begin{split}
\bs{x} = \; & (x_1, \ldots, x_{2N}) \in \chamber_{2N} , 
\\
\bs{\ddot{x}}_j = \; & (x_1, \ldots, x_{j-1}, x_{j+2}, \ldots, x_{2N}) \in \chamber_{2N-2} ,
\end{split}
\end{align}
and $\xi \in (x_{j-1}, x_{j+2})$
\textnormal{(}with the convention that $x_0 = -\infty$ and  $x_{2N+1} = +\infty$\textnormal{)}. 

\item[\textnormal{(PLB)}] 
The functions are positive and satisfy the power-law bound 
\begin{align} \label{eqn::PPFBounds_general}
0<\PartF_{\alpha}(x_1, \ldots, x_{2N})\le\prod_{\{a,b\}\in\alpha}|x_b-x_a|^{-2h(\kappa)}, \qquad \textnormal{for all }(x_1, \ldots, x_{2N})\in\chamber_{2N} . 
\end{align}
\end{itemize}
We extend the definition of $\PartF_{\alpha}$ to more general
polygons $(\Omega; x_1, \ldots, x_{2N})$ as in~\eqref{eqn::coulombgasintegral_general} 
\textnormal{(}replacing $\coulombnew_{\beta}$ by $\PartF_{\alpha}$\textnormal{)}.
\end{definition}

With a weaker power-law bound and relaxing the positivity requirement in~\eqref{eqn::PPFBounds_general}, the collection $\{\PartF_{\alpha} \colon \alpha\in\LP_N\}$ was first constructed in~\cite{Flores-Kleban:Solution_space_for_system_of_null-state_PDE3} indirectly by using Coulomb gas integrals for all $\kappa\in (0,8)$, and explicitly for all $\kappa \in (0,8) \setminus \QQ$ in~\cite{Kytola-Peltola:Pure_partition_functions_of_multiple_SLEs}, following the conjectures from~\cite{BBK:Multiple_SLEs_and_statistical_mechanics_martingales}.
It is believed that these functions satisfy~\eqref{eqn::PPFBounds_general} for all $\kappa\in (0,8)$. 
In general, for the range $\kappa \in (0,8]$, to our knowledge there are explicit formulas for $\PartF_{\alpha}$ only when $\kappa \notin \QQ$ (cf.~\cite{Kytola-Peltola:Pure_partition_functions_of_multiple_SLEs}) and for a few special rational cases: $\kappa = 2$~\cite{KKP:Boundary_correlations_in_planar_LERW_and_UST}; $\kappa=4$~\cite{Peltola-Wu:Global_and_local_multiple_SLEs_and_connection_probabilities_for_level_lines_of_GFF}; and $\kappa=8$~\cite{LPW:UST_in_topological_polygons_partition_functions_for_SLE8_and_correlations_in_logCFT}.
For $\kappa\in (0,6]$, an explicit probabilistic construction was given in~\cite[Theorem~1.7]{Wu:Convergence_of_the_critical_planar_ising_interfaces_to_hypergeometric_SLE}, which immediately implies~\eqref{eqn::PPFBounds_general}. 
See also Remark~\ref{rem::ppf}.

\subsection{Results: Multiple interfaces and connection probabilities for the FK-Ising model}
\label{subsec::Results_FKIsing_Loewner}

Our first main result concerns the scaling limit of the FK-Ising interfaces.
\begin{theorem} \label{thm::FKIsing_Loewner}
Conjecture~\ref{conj::rcm_Loewner} holds for $q=2$ and $\kappa=16/3$. In this case, we have
\begin{align} \label{eqn::totalpartition_def}
\begin{split} 
\coulombnew_\beta (x_1, \ldots, x_{2N}) 
= \; & \LF_{\beta} (x_1, \ldots, x_{2N}) \\
:= \; & \prod_{s=1}^N |x_{b_s}-x_{a_s}|^{-1/8}
\bigg(\sum_{\bs{\sigma} \in \{\pm 1\}^N}\prod_{1\le s < t \le N}\chi(x_{a_s}, x_{a_t}, x_{b_t}, x_{b_s})^{\sigma_s \sigma_t / 4}\bigg)^{1/2} ,
\end{split} 
\end{align}
where $\bs{\sigma} = (\sigma_1, \sigma_2, \ldots, \sigma_N)\in \{\pm 1\}^N$
and $\chi \colon \R^4\to \R$ is the cross-ratio
\begin{align}\label{eqn::crossratio}
\chi(y_1, y_2, y_3,  y_4):= \frac{|y_2-y_1| \, |y_4-y_3|}{|y_3-y_1| \, |y_4-y_2|} .
\end{align}
\end{theorem}

\begin{remark}
The square of this formula also appears in moments of the real part of an imaginary Gaussian multiplicative chaos distribution~\textnormal{\cite[Theorem~1.5]{JSW:IGMC_Moments_regularity_and_connections_to_Ising_model}}.
\end{remark}

The case $N=1$ of one curve in Theorem~\ref{thm::FKIsing_Loewner} was proven in a celebrated group effort summarized in~\cite{CDHKS:Convergence_of_Ising_interfaces_to_SLE}. 
The scaling limit curve is the chordal \emph{Schramm-Loewner evolution}. 
The proof in the case of $N=1$ involves two main steps.
The first step is to show that the sequence  $\{\eta_1^{\delta}\}_{\delta>0}$ of interfaces is \emph{tight}, which implies \emph{precompactness} by Prokhorov's theorem, and thus
enables finding convergent subsequences
$\eta_1^{\delta_n} \to \eta_1$ with some limit curve $\eta_1$.
Second, one has to show that all of these subsequences actually \emph{converge} to the same limit, identified in this case with the chordal $\SLE_{16/3}$. 
The precompactness
step is established by refined crossing estimates~\cite{DCHN:Connection_probabilities_and_RSW_type_bounds, Kemppainen-Smirnov:Random_curves_scaling_limits_and_Loewner_evolutions}, 
while the identification of the limit curve involves an ingenious usage of a discrete holomorphic spinor observable 
(devised by Smirnov~\cite{Smirnov:Conformal_invariance_in_random_cluster_models1} and further developed by Chelkak, Smirnov, and others, cf.~\cite{Chelkak-Smirnov:Universality_in_2D_Ising_and_conformal_invariance_of_fermionic_observables, CHI:Conformal_invariance_of_spin_correlations_in_planar_Ising_model, CHI:Correlations_of_primary_fields_in_the_critical_planar_Ising_model}) 
converging to its continuum counterpart, which 
gives the sought driving function $W_t = \sqrt{16/3} \, B_t$ via a suitable series expansion.

In the case $N=2$ of two curves $(\eta_1,\eta_2)$, 
Theorem~\ref{thm::FKIsing_Loewner} was proven in~\cite{Chelkak-Smirnov:Universality_in_2D_Ising_and_conformal_invariance_of_fermionic_observables, Kemppainen-Smirnov:Configurations_of_FK_Ising_interfaces}.  
Since the conformal invariance fixes three real degrees of freedom, while the polygon $(\Omega; x_1, x_2, x_3, x_4)$ has four real degrees of freedom, 
a similar strategy as in the case of one curve gives the result, and
the driving function of one curve, say $\eta_1$ (in its marginal law), is given by Brownian motion with a drift involving the hypergeometric function. 
Essentially, the only additional input compared to the case of $N=1$ is that one has to solve an ordinary differential equation for the drift term, which results in the hypergeometric equation.

The case of $N \geq 3$ is significantly more involved. Because there are several degrees of freedom,  
the identification of the scaling limit requires finding a suitable multi-point discrete holomorphic {spinor}  observable, or alternatively, some other proof strategy. 
For the special case where the boundary condition is the totally unnested link pattern
\begin{align} \label{eqn::unnested}
\beta = \unnested := \{\{1,2\} , \{3,4\} , \ldots, \{2N-1, 2N\}\} ,
\end{align}
Theorem~\ref{thm::FKIsing_Loewner} was proven recently by Izyurov~\cite{Izyurov:On_multiple_SLE_for_the_FK_Ising_model} and earlier implicitly conjectured by Flores, Simmons, Kleban, and Ziff~\cite{FSKZ:A_formula_for_crossing_probabilities_of_critical_systems_inside_polygons}. 
In Section~\ref{sec::FKIsing_Loewner}, we will prove Theorem~\ref{thm::FKIsing_Loewner} with general boundary conditions $\beta$. The main addition compared to the earlier results is the identification of the drift term for general $\beta$, given by~\eqref{eqn::totalpartition_def}, and finding a suitably general multi-point observable. 
The rough strategy is the following. 
\begin{itemize}[leftmargin=*]

\item We construct a discrete holomorphic observable 
 with general boundary conditions in Section~\ref{subsec::holo_observable} and
identify its scaling limit observable $\phi_{\beta}$ in Section~\ref{subsec::observable_cvg}. 
This is a generalization of the previous observables constructed in~\cite{Chelkak-Smirnov:Universality_in_2D_Ising_and_conformal_invariance_of_fermionic_observables, Izyurov:Smirnovs_observable_for_free_boundary_conditions_interfaces_and_crossing_probabilities, Izyurov:On_multiple_SLE_for_the_FK_Ising_model}.
Some key ideas for the proof in Section~\ref{subsec::observable_cvg} are learned from~\cite{Izyurov:Smirnovs_observable_for_free_boundary_conditions_interfaces_and_crossing_probabilities}.

\item We analyze the observable $\phi_{\beta}$, expand it to certain precision, and relate its expansion coefficients to $\LF_{\beta}$ in Section~\ref{subsec::holo_limiting}. This step is rather technical, but contains the gist of the proof of Theorem~\ref{thm::FKIsing_Loewner}: identification of the scaling limit~\eqref{eqn::rcm_Loewner_chain} with the \emph{explicit} drift given by the function $\LF_{\beta}$ in formula~\eqref{eqn::totalpartition_def}.
The form of the function $\LF_{\beta}$ is very similar to~\cite[Theorem~1.1]{Izyurov:On_multiple_SLE_for_the_FK_Ising_model}, 
but we allow a general external connectivity that gives the boundary condition $\beta$.

\item Most importantly, in Section~\ref{sec::totalpartition_analysis} (Theorem~\ref{thm::totalpartition}) 
we also show that the function $\LF_{\beta}$ coincides with the prediction $\coulombnew_\beta$ from the Coulomb gas formalism of CFT 
related to~\cite[Eq.~(C.14)]{FSKZ:A_formula_for_crossing_probabilities_of_critical_systems_inside_polygons}.

\item Finally, we derive the Loewner equation~\eqref{eqn::rcm_Loewner_chain} for $\kappa = 16/3$ from the observable $\phi_{\beta}$ in Section~\ref{subsec::FKIsing_Loewner} using its properties derived in Section~\ref{subsec::holo_limiting}. This step is relatively standard.  
\end{itemize}

\begin{remark} \label{rem::spin_correlation_function}
Note that formula~\eqref{eqn::totalpartition_def} has the form of a bulk spin correlation function in the Ising model~\textnormal{\cite[Eq.~(1.4)]{CHI:Conformal_invariance_of_spin_correlations_in_planar_Ising_model}}, 
but with the spins put on the real line instead, in such way that each pair $\{ x_{a_r}, x_{b_r} \}$ 
corresponds to a bulk point $z_r$ and its complex conjugate $\overline{z}_r$
\textnormal{(}see also~\textnormal{\cite[Eq.~(C.14)]{FSKZ:A_formula_for_crossing_probabilities_of_critical_systems_inside_polygons}~and 
\cite[Theorem~1.1]{Izyurov:On_multiple_SLE_for_the_FK_Ising_model}} for the special case where $\beta = \unnested$~\eqref{eqn::unnested}\textnormal{)}. 
This observation, or ``reflection trick'', was used by Flores, Simmons, Kleban, and Ziff~\textnormal{\cite[Figure~3]{SKFZ:Cluster_densities_at_2D_critical_points_in_rectangular_geometries}} 
and later in~\textnormal{\cite{FSKZ:A_formula_for_crossing_probabilities_of_critical_systems_inside_polygons}}
to predict formulas\footnote{Our formula~\eqref{eqn::coulombgasintegral} for $\coulombnew_\beta$ is seemingly different from~\textnormal{\cite[Eq.~(11)]{FSKZ:A_formula_for_crossing_probabilities_of_critical_systems_inside_polygons}}, but they actually coincide.} 
for $\coulombnew_\beta$ in~\textnormal{\cite[Eq.~(11)]{FSKZ:A_formula_for_crossing_probabilities_of_critical_systems_inside_polygons}}. 
The idea is, to our knowledge, originally due to Cardy~\textnormal{\cite{Cardy:Conformal_invariance_and_surface_critical_behavior}}, 
who observed that via the reflection trick, 
bulk correlations satisfying so-called BPZ differential equations~\textnormal{\cite{BPZ:Infinite_conformal_symmetry_in_2D_QFT, BPZ:Infinite_conformal_symmetry_of_critical_fluctuations_in_2D}} 
can be related to boundary correlations also satisfying similar equations\footnote{Note that the reflection trick only indicates that certain formulas satisfy certain partial differential equations, and does not give much physical interpretation of this relationship.}. 
We show in Theorem~\ref{thm::CGI_property} that $\coulombnew_\beta$ indeed satisfies these equations, along with specific asymptotic boundary conditions that heuristically give the ``fusion rules'' for the corresponding CFT primary fields.
See also~\textnormal{\cite[Theorem~8]{Flores-Kleban:Solution_space_for_system_of_null-state_PDE3}} and~\textnormal{\cite[Theorem~2]{Flores-Kleban:Solution_space_for_system_of_null-state_PDE4}}. 
\end{remark}

\begin{theorem} \label{thm::FKIsing_crossingproba}
Conjecture~\ref{conj::rcm_crossingproba} holds for $q=2$ and $\kappa=16/3$, with $\coulombnew_\beta = \LF_{\beta}$ as in~\eqref{eqn::totalpartition_def}.
\end{theorem}

Our formula~\eqref{eqn::rcm_crossingproba} with $N=2$ and $\kappa=16/3$
is consistent with~\cite[Eq.~(117)]{FSKZ:A_formula_for_crossing_probabilities_of_critical_systems_inside_polygons};
see also~\cite[Eq.~(1.1)]{Chelkak-Smirnov:Universality_in_2D_Ising_and_conformal_invariance_of_fermionic_observables} for a formula with 
different boundary conditions.
Izyurov proved the conformal invariance of some further probabilities of (unions of) 
connection events~\cite{Izyurov:Smirnovs_observable_for_free_boundary_conditions_interfaces_and_crossing_probabilities, Izyurov:On_multiple_SLE_for_the_FK_Ising_model} 
--- see in particular~\cite[Corollary~1.3]{Izyurov:On_multiple_SLE_for_the_FK_Ising_model}. 
Our result settles the general case for any $\alpha, \beta \in \LP_N$. 
We prove Theorem~\ref{thm::FKIsing_crossingproba} in Section~\ref{sec::crossingproba} via the following strategy: 
\begin{itemize}[leftmargin=*]
\item We first prove~\eqref{eqn::rcm_crossingproba} for $\kappa = 16/3$ with $\beta=\unnested$ (Section~\ref{subsec::crossingproba_unnested}) via a martingale argument using the convergence of the interfaces.
This step depends on fine analysis of the martingale observable given by the ratio $\PartF_{\alpha}/\LF_{\unnested}$ (which is a local martingale with respect to growing any of the interfaces thanks to the PDEs~\eqref{eqn::PDE}). 
There are two key ingredients: a cascade relation for the pure partition functions $\PartF_{\alpha}$ from~\cite{Wu:Convergence_of_the_critical_planar_ising_interfaces_to_hypergeometric_SLE}, 
and technical work that we defer to Appendix~\ref{appendix_technical}. 

\item We then derive~\eqref{eqn::rcm_crossingproba} for $\kappa = 16/3$ 
and for general boundary condition $\beta$ (Section~\ref{subsec::crossingproba_general}), 
by using the conclusion for $\beta = \unnested$. 
Indeed, we can relate the case of general $\beta$ to the case of $\unnested$ for any random-cluster model directly in the discrete setup --- see Proposition~\ref{prop::crossingproba_comparison} for such a useful formula. 
\end{itemize}

\subsection{Results: Properties of the Coulomb gas integrals}
\label{subsec::CGI_properties}

Lastly, we show that the functions appearing in  Conjectures~\ref{conj::rcm_Loewner} and~\ref{conj::rcm_crossingproba} do indeed satisfy important properties predicted by conformal field theory. These properties are also needed for the identification of $\coulombnew_{\beta}$ with $\LF_{\beta}$ for the case of $\kappa = 16/3$ in Theorem~\ref{thm::totalpartition}.

\begin{theorem} \label{thm::CGI_property}
Fix $\kappa\in (4,8)$. The functions $\coulombnew_{\beta}$ defined in~\eqref{eqn::coulombgasintegral} satisfy the following properties. 
\begin{itemize}[leftmargin=3em]
\item[\textnormal{(PDE)}] 
The BPZ equations~\eqref{eqn::PDE}.

\item[\textnormal{(COV)}] 
The M\"{o}bius covariance~\eqref{eqn::COV}.

\item[\textnormal{(ASY)}] 
\textnormal{\bf Asymptotics:} 
With $\coulombnew_{\emptyset} \equiv 1$ for the empty link pattern $\emptyset \in \LP_0$, the collection $\{\coulombnew_{\beta} \colon \beta\in\LP_N\}$ satisfies the following recursive 
asymptotics property.
Fix $N \ge 1$ and $j \in \{1,2, \ldots, 2N-1 \}$. 
Then, for all $\xi \in (x_{j-1}, x_{j+2})$, using the notation~\eqref{eqn::bs_notation}, we have
\begin{align}
\label{eqn::ASY} 
\; & \lim_{x_j,x_{j+1}\to\xi} \frac{\coulombnew_{\beta}(\bs{x})}{ (x_{j+1}-x_j)^{-2h(\kappa)} }
= 
\begin{cases}
\sqrt{q(\kappa)} \, \coulombnew_{\beta/\{j,j+1\}}(\bs{\ddot{x}}_j),
& \textnormal{if }\{j, j+1\}\in\beta , \\
\coulombnew_{\wp_j(\beta)/\{j,j+1\}}(\bs{\ddot{x}}_j),
& \textnormal{if }\{j, j+1\} \not\in \beta , 
\end{cases}
\end{align}
where $\beta/\{j,j+1\} \in \LP_{N-1}$ denotes the link pattern obtained from $\beta$ by removing the link $\{j,j+1\}$ and relabeling the remaining indices by $1, 2, \ldots, 2N-2$, 
and $\wp_j$ 
is the ``tying operation'' defined by 
\begin{align*}
\wp_j \colon \LP_N\to \LP_N , \qquad
\wp_j(\beta) = 
\big(\beta\setminus(\{j,k_1\}, \{j+1, k_2\})\big)\cup \{j,j+1\}\cup \{k_1, k_2\} , 
\end{align*}  
where $k_1$ \textnormal{(}resp.~$k_2$\textnormal{)} 
is the pair of $j$ \textnormal{(}resp.~$j+1$\textnormal{)} in $\beta$ \textnormal{(}and $\{j,k_1\}, \{j+1, k_2\}, \{k_1, k_2\}$ are unordered\textnormal{)}.
\begin{align*}
\vcenter{\hbox{\includegraphics[scale=0.25]{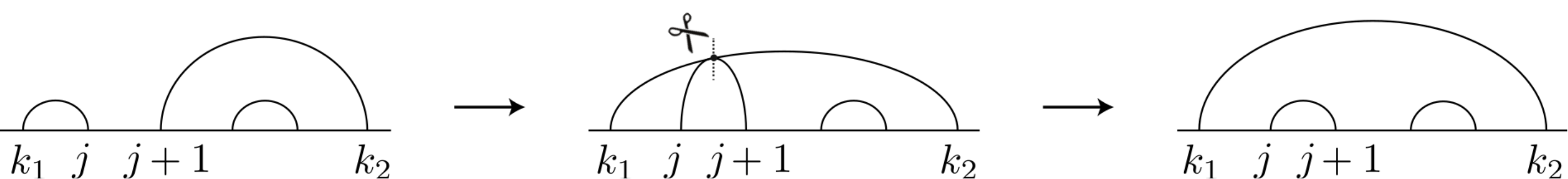}}} 
\end{align*}
\end{itemize}
\end{theorem}

One can also relate these Coulomb gas integral functions directly to the pure partition functions by using the meander matrix. Such a relation appears implicitly in~\cite[Theorem~8]{Flores-Kleban:Solution_space_for_system_of_null-state_PDE3} for all $\kappa \in (0,8)$.

\begin{proposition} \label{prop::linearcombination}
Fix $\kappa\in (4,6]$. 
For all $\bs{x} = (x_1, \ldots, x_{2N})\in\chamber_{2N}$, we have
\begin{align} \label{eqn::linearcombination}
\coulombnew_{\beta}(\bs{x}) 
= \sum_{\alpha\in\LP_N} \LM_{\alpha,\beta}(q(\kappa)) \, \PartF_{\alpha}(\bs{x}) \; > 0 , \qquad \textnormal{for all }\beta\in\LP_N, 
\end{align}
where $\coulombnew_{\beta}$ and $\LM_{\alpha, \beta}(q(\kappa))$ are defined in~\eqref{eqn::coulombgasintegral} and~\eqref{eqn::meandermatrix_def_general}, respectively, and $\{\PartF_{\alpha} \colon \alpha\in\LP_N\}$ is the collection of pure partition functions for multiple $\SLE_{\kappa}$ described 
in Definition~\ref{def::PPF_general}. 
\end{proposition}

We prove Proposition~\ref{prop::linearcombination} in Section~\ref{subsec::linearcombination}.
The idea is that both sides of Eq.~\eqref{eqn::linearcombination} satisfy the same PDE boundary value problem, 
which uniquely determines them. 

\begin{remark} \label{rem::ppf}
The relation~\eqref{eqn::linearcombination} in Proposition~\ref{prop::linearcombination} only allows to solve for $\PartF_{\alpha}$ explicitly when the meander matrix
$\LM^{(N)}(q(\kappa)) := \{\LM_{\alpha, \beta}(q(\kappa)) \colon \alpha,\beta\in\LP_N\}$ is invertible. 
By~\textnormal{\cite[Eq.~(5.6)]{DGG:Meanders_and_TL_algebra}}, we know that $\LM^{(N)}(q(\kappa))$ is invertible if and only if $\kappa$ is not one of the exceptional values 
\begin{align*}
\kappa_{r,s} := \frac{4r}{s} , \qquad r,s \in \Z_{>0} \, \textnormal{ coprime and } \, 1 \leq s < r < N+2 .
\end{align*}
We see that, for example, the value $\kappa = 16/3$ belongs to this set with $r=4$ and $s=3$, when $N \geq 3$. 
Indeed, in the case where $\kappa = 16/3$ and $N = 3$, the following element belongs to the kernel of $\LM^{(N)}(2)$:
\begin{align*}
\LG_{\vcenter{\hbox{\includegraphics[scale=0.2]{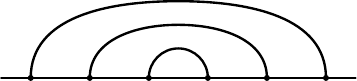}}}}
+ \LG_{\vcenter{\hbox{\includegraphics[scale=0.2]{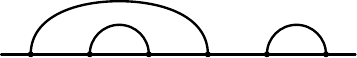}}}}
+ \LG_{\vcenter{\hbox{\includegraphics[scale=0.2]{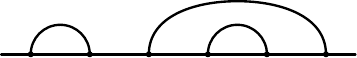}}}}
- \sqrt{2} \, \LG_{\vcenter{\hbox{\includegraphics[scale=0.2]{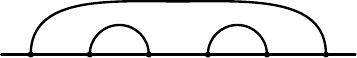}}}}
- \sqrt{2} \, \LG_{\vcenter{\hbox{\includegraphics[scale=0.2]{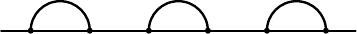}}}} .
\end{align*}
One can find the kernel explicitly also in general (cf.~\cite{Flores-Peltola:Standard_modules_radicals_and_the_valenced_TL_algebra}), but this does not immediately give means to solve for $\PartF_{\alpha}$ from~\eqref{eqn::linearcombination}.
Let us also remark that we know from~\textnormal{\cite[Theorem~8]{Flores-Kleban:Solution_space_for_system_of_null-state_PDE3}} 
that $\{\PartF_{\alpha} \colon \alpha\in\LP_N\}$ are linearly independent, but $\{\coulombnew_{\beta} \colon \beta\in\LP_N\}$ are not unless the matrix $\LM^{(N)}(q(\kappa))$ is invertible. 
\end{remark}

\begin{remark} \label{rem:q=4}
The case of $\kappa = 4$, that is, $q(\kappa)=4$, is excluded. 
Here, we believe that one can take the limit $\kappa \searrow 4$ to obtain formulas for this case, and Conjectures~\ref{conj::rcm_Loewner} and~\ref{conj::rcm_crossingproba} will still hold. 
Note that while the integrals in~\eqref{eqn::coulombgasintegral} are not convergent if $\kappa = 4$, one can get convergent integrals easily by replacing the contours in $\smash{\landupint_{x_{a_r}}^{x_{b_r}}} \ud u_r$, that we have chosen for simplicity of the presentation, by Pochhammer type contours as illustrated in
Section~\ref{subsec::Pochhammer} \textnormal{(}Eq.~\eqref{eq::Pochhammer_contour}\textnormal{)}.
Also, the multiplicative constant in~\eqref{eqn::coulombgasintegral} equals zero when $\kappa = 4$, so a slightly different normalization is needed \textnormal{(}also chosen in accordance with Appendix~\ref{appendix_asy}\textnormal{)}. 
\end{remark}

\paragraph*{Organization of this article.}
Section~\ref{sec::partitionfunctions} and Appendix~\ref{appendix_asy} concern the Coulomb gas integral functions (Theorem~\ref{thm::CGI_property}) and their relation to the function $\LF_\beta$ when $\kappa = 16/3$ (Proposition~\ref{prop::linearcombination}). 
Section~\ref{sec::FKIsing_Loewner} and Appendix~\ref{appendix_aux} together prove the convergence of the FK-Ising interfaces (Theorem~\ref{thm::FKIsing_Loewner}), and 
Section~\ref{sec::crossingproba} and Appendix~\ref{appendix_technical} contain the proof of our scaling limit result for the connection probabilities (Theorem~\ref{thm::FKIsing_crossingproba}).

\medbreak
\paragraph*{Acknowledgments.}
\begin{itemize}
\item This material is part of a project that has received funding from the  European Research Council (ERC) under the European Union's Horizon 2020 research and innovation programme (101042460): 
ERC Starting grant ``Interplay of structures in conformal and universal random geometry'' (ISCoURaGe) 
and from the Academy of Finland grant number 340461 ``Conformal invariance in planar random geometry.''
E.P.~is also supported by 
the Academy of Finland Centre of Excellence Programme grant number 346315 ``Finnish centre of excellence in Randomness and STructures (FiRST)'' 
and by the Deutsche Forschungsgemeinschaft (DFG, German Research Foundation) under Germany's Excellence Strategy EXC-2047/1-390685813, 
as well as the DFG collaborative research centre ``The mathematics of emerging effects'' CRC-1060/211504053.

\item Part of this work was carried out while E.P. participated in a program hosted by the Mathematical Sciences Research Institute (MSRI) in Berkeley, California, during Spring 2022, 
thereby being supported by the National Science Foundation under grant number DMS-1928930. 

\item H.W. is supported by Beijing Natural Science Foundation (JQ20001). H.W. is partly affiliated at Yanqi Lake Beijing Institute of Mathematical Sciences and Applications, Beijing, China.

\item We thank Konstantin Izyurov for several discussions related to this work. 
We thank the anonymous referee for careful reading and helpful suggestions to improve this article.
\end{itemize}

\newpage


\section{Properties of partition functions}
\label{sec::partitionfunctions}
Throughout, we consider link patterns 
$\beta \in \LP_N$ with link endpoints ordered as in~\eqref{eqn::linkpatterns_ordering}.

\subsection{Coulomb gas integrals and the proof of Theorem~\ref{thm::CGI_property}}
\label{subsec::Pochhammer}

In this section, we consider the functions $\coulombnew_{\beta}$, for $\beta \in \LP_N$, defined in Coulomb gas integral form via~\eqref{eqn::coulombgasintegral}. 
Coulomb gas integrals~\cite{Dotsenko-Fateev:Conformal_algebra_and_multipoint_correlation_functions_in_2D_statistical_models,  Dubedat:Euler_integrals_for_commuting_SLEs, Kytola-Peltola:Conformally_covariant_boundary_correlation_functions_with_quantum_group} 
stem from conformal field theory (CFT), where they have been used as a general ansatz to find formulas for correlation functions. 
Specifically to our case, 
we seek correlation functions 
satisfying a system of PDEs~\eqref{eqn::PDE} 
known as Belavin-Polyakov-Zamolodchikov (BPZ) differential equations~\cite{BPZ:Infinite_conformal_symmetry_in_2D_QFT}, 
and a specific M\"obius covariance property~\eqref{eqn::COV}. 
The latter is just a manifestation of the global conformal invariance, while the former is a peculiarity in our case: the integrals $\coulombnew_{\beta}$ represent correlation functions of so-called degenerate fields at level two in a CFT. It is by now well-known that such correlation functions have a close relationship with $\SLE_\kappa$ curves: they are examples of partition functions of multiple $\SLE_\kappa$ (they are, in fact, linear combinations of the pure partition functions in Definition~\ref{def::PPF_general} --- see Proposition~\ref{prop::linearcombination}).

To understand the definition of $\coulombnew_{\beta}$ in~\eqref{eqn::coulombgasintegral},
note that as a function of the integration variables 
\begin{align*}
\bs{u} = (u_1, \ldots, u_N) \in 
\Wchamber^{(N)} = \; \Wchamber_{x_{1},\ldots,x_{2N}}^{(N)}
:= \big( \C\setminus \{x_{1},\ldots,x_{2N}\} \big)^{N} ,
\end{align*}
the integrand function $f(\bs{x};\cdot)$ given in~\eqref{eq: integrand} has ramification points $u_r = x_j$ and $u_{r} = u_{s}$ for $r \neq s$.
To define a branch for it on a simply connected subset of $\Wchamber^{(N)}$, we impose $f(\bs{x};\cdot)$  to be real and positive on 
\begin{align} \label{eq:: branch choice set}
\LR_\beta :=  \big\{\bs{u} \in \Wchamber^{(N)}  \colon x_{a_r} < \Re(u_r) < x_{a_r+1} \; \textnormal{ for all } 1 \leq r \leq N \big\} ,
\end{align}
and for definiteness, we denote this branch choice as
$f_\beta(\bs{x};\cdot) \colon \LR_\beta \to \R_{>0}$.
Then, its values elsewhere in $\Wchamber^{(N)}$ are completely determined by analytic continuation.

\smallbreak

The goal of this section is to give a proof of Theorem~\ref{thm::CGI_property} via establishing a relation between 
$\coulombnew_{\beta}$
with similar integrals $\smash{\coulomb_\beta^\circ}$ involving Pochhammer contours, which are easier to analyze.
The latter only involve integrations avoiding the marked points $x_1, \ldots, x_{2N}$ and are thus convergent for all $\kappa > 0$. 
Our choice in~\eqref{eqn::coulombgasintegral} for the integration contours touching the marked points is merely a notational simplification (for $\kappa \in (4,8)$).
The proof of Theorem~\ref{thm::CGI_property} comprises several auxiliary results presented in this section.

\begin{proof}[Proof of Theorem~\ref{thm::CGI_property}]
The proof is a collection of the following results. 
\begin{itemize}[leftmargin=2em] 
\item $\coulombnew_{\beta}$ satisfies the BPZ PDEs~\eqref{eqn::PDE} due to Eq.~\eqref{eqn::rela_G_and_H}, Lemma~\ref{lem::related_two_Gs},  and Proposition~\ref{prop: PDEs Fcirc}.

\item $\coulombnew_{\beta}$ satisfies M\"obius covariance~\eqref{eqn::COV} due to Eq.~\eqref{eqn::rela_G_and_H}, Lemma~\ref{lem::related_two_Gs}, and Proposition~\ref{prop: full Mobius covariance Fcirc}.

\item $\coulombnew_{\beta}$ satisfies the asymptotics~\eqref{eqn::ASY}
due to Lemma~\ref{lem: neighbor asy} and Proposition~\ref{prop::non-neighbor asy}. 
\qedhere
\end{itemize}
\end{proof}

\smallbreak

For the auxiliary results, we define the function $\smash{\coulomb_\beta^\circ} \colon \chamber_{2N} \to \C$ on the configuration space~\eqref{eqn::coulombgasintegral} as
\begin{align}
	\label{eqn::coulombgasintegral_loops}
	\coulomb_\beta^\circ (\bs{x}) 
	:= \; & \ointclockwise_{\acycle^\beta_1}  \ud u_1 \ointclockwise_{\acycle^\beta_2}  \ud u_2 \cdots \ointclockwise_{\acycle^\beta_N} \ud u_N 
	\; f_\beta(\bs{x};\bs{u}) , \qquad \bs{x} \in \chamber_{2N} ,
\end{align}
where each $\acycle^\beta_r$ is a Pochhammer contour 
which encircles each of the points $x_{a_r}, x_{b_r}$ once in the positive direction and once in the negative direction:
\begin{align} \label{eq::Pochhammer_contour}
\acycle^\beta_r \quad = \quad \vcenter{\hbox{\includegraphics[scale=0.275]{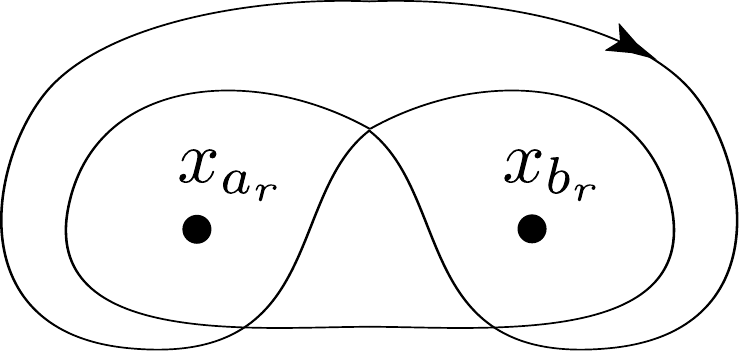}}} 
\end{align} 
and which does not encircle any other marked point among $\{x_1, \ldots, x_{2N}\}$
(cf.~illustrations in~\cite[Figure~2]{Flores-Kleban:Solution_space_for_system_of_null-state_PDE3} 
and~\cite[Figure~6]{Dubedat:Euler_integrals_for_commuting_SLEs}). 
Note that since the integration contours $\acycle^\beta_j$ avoid the marked points $x_{1},\ldots,x_{2N}$, the integral $\smash{\coulomb_\beta^\circ} (\bs{x})$ is convergent for all $\kappa > 0$. 
We also extend $\smash{\coulomb_\beta^\circ}$ to a multivalued function on the larger set
$\mathfrak{Y}_{2N} := \{ \bs{x} = (x_{1},\ldots,x_{2N}) \in \C^{2N} \colon x_{i} \neq x_{j} \textnormal{ for all } i \neq j \} $.

\begin{lemma} \label{lem::related_two_Gs}
	Fix $\kappa > 4$. 
	Writing $\bs{u}=(u_1, \ldots, u_N)$, we have 
	\begin{align} \label{eq::related_two_Gs}
		\coulomb_\beta (\bs{x}) 
		:= \landupint_{x_{a_1}}^{x_{b_1}} \ud u_1 \cdots \landupint_{x_{a_N}}^{x_{b_N}} \ud u_N 
		\; f_\beta(\bs{x};\bs{u}) 
		= \big( 4 \sin^2(4 \pi / \kappa) \big)^{-N} \,
		\coulomb_\beta^\circ(\bs{x}) .
	\end{align}
\end{lemma}
Note that the function $\coulombnew_{\beta}$ defined in Eq.~\eqref{eqn::coulombgasintegral} equals
\begin{align} \label{eqn::rela_G_and_H}
\coulombnew_{\beta}(\bs{x}) 
= \bigg( \frac{\sqrt{q(\kappa)} \, \Gamma(2-8/\kappa)}{\Gamma(1-4/\kappa)^2} \bigg)^N \, 
\coulomb_{\beta}(\bs{x}) , \qquad \bs{x} \in \chamber_{2N} .
\end{align}

\begin{proof}
Because the contours $\smash{\acycle^\beta_1}, \ldots, \smash{\acycle^\beta_N}$ in $\smash{\coulomb_\beta^\circ}$ are all disjoint, 
by Fubini's theorem, we may first evaluate the integrals over those $\smash{\acycle^\beta_s}$ for which $b_s = a_s+1$.
Suppose first that the other integration variables are frozen to some positions such that 
$x_{a_r} < \Re(u_r) < x_{a_r+1}$ for all $1 \leq r \leq N$ with $r \neq s$. Then, we have
	\begin{align} 
		\nonumber
		\ointclockwise_{\acycle^\beta_s} \ud u_s \; f_\beta(\bs{x};\bs{u}) 
		= \; & \landupint_{x_{a_s}}^{x_{b_s}} \ud u_s \; |f_\beta(\bs{x};\bs{u})|
		\; + \; e^{8 \pi \ii / \kappa} \landupint_{x_{b_s}}^{x_{a_s}} \ud u_s \; |f_\beta(\bs{x};\bs{u})| \\
		\nonumber
		\; & \; + e^{-8 \pi \ii / \kappa} e^{8 \pi \ii / \kappa} \landupint_{x_{a_s}}^{x_{b_s}} \ud u_s \; |f_\beta(\bs{x};\bs{u})|
		\; + \; e^{-8 \pi \ii / \kappa} e^{-8 \pi \ii / \kappa} e^{8 \pi \ii / \kappa} \landupint_{x_{b_s}}^{x_{a_s}} \ud u_s \; |f_\beta(\bs{x};\bs{u})| \\
		\label{eq:: loop vs interval}
		= \; & 
		4 \sin^2(4 \pi / \kappa) \landupint_{x_{a_s}}^{x_{b_s}} \ud u_s \; |f_\beta(\bs{x};\bs{u})| .
	\end{align}
From this, we also see that when the other integration variables in $\bs{\dot{u}}_{s} := (u_1, \ldots, u_{s-1}, u_{s+1}, \ldots, u_N)$ move around their respective contours in $\smash{\coulomb_\beta^\circ}$, the phase factors in both sides of~\eqref{eq:: loop vs interval} are the same. 
Therefore, we can replace each integral in $\smash{\coulomb_\beta^\circ}$ 
	of type $\smash{\ointclockwise_{\acycle^\beta_s} \ud u_s}$ for some $b_s = a_s+1$ by the integral $\smash{\landupint_{x_{a_s}}^{x_{b_s}} \ud u_s}$ times the multiplicative constant $\smash{4 \sin^2(4 \pi / \kappa)}$.
	
	Next, for any $b_s = a_s+3$, we see that the phase factors associated to the integration variable $u_s$ surrounding all of the points $\{x_{a_s+1},x_{a_s+2},u_{s+1}\}$ cancel out. Therefore, we can also replace each integral in $\smash{\coulomb_\beta^\circ}$ 
	of type $\smash{\ointclockwise_{\acycle^\beta_s} \ud u_s}$ for some $b_s = a_s+3$ by the integral $\smash{\landupint_{x_{a_s}}^{x_{b_s}} \ud u_s}$ times $\smash{4 \sin^2(4 \pi / \kappa)}$.
	
	We see iteratively that all of the integrals over the disjoint contours $\smash{\smash{\acycle^\beta_1}}, \ldots, \smash{\acycle^\beta_N}$ in $\smash{\coulomb_\beta^\circ}$ can be replaced by integrals over the corresponding intervals with the multiplicative constant as in asserted identity~\eqref{eq::related_two_Gs}.
\end{proof}

\begin{proposition} \label{prop: full Mobius covariance Fcirc}
	For each $\beta \in \LP_N$, the function $\smash{\coulomb_{\beta}^\circ}$ satisfies the 
	covariance~\eqref{eqn::COV}, 
that is, 
\begin{align}\label{eqn::COV_for_H}
\coulomb_{\beta}^\circ (x_{1},\ldots,x_{2N}) = 
\prod_{i=1}^{2N} \varphi'(x_{i})^{h(\kappa)} \; \times
\; \coulomb_{\beta}^\circ (\varphi(x_{1}),\ldots,\varphi(x_{2N})) ,
\end{align}
for all M\"obius maps $\varphi$ of the upper half-plane 
$\HH$ such that $\varphi(x_{1}) < \cdots < \varphi(x_{2N})$.
\end{proposition}

\begin{proof}
The proof is very similar to arguments appearing in~\cite[Proposition~4.15]{Kytola-Peltola:Conformally_covariant_boundary_correlation_functions_with_quantum_group} 
(for $\kappa \notin \QQ$).  
One readily checks the covariance under translations and scalings:
	\begin{align*}
		\coulomb_{\beta}^\circ (x_{1}+y,\ldots,x_{2N}+y) 
		= \coulomb_{\beta}^\circ (x_{1},\ldots,x_{2N}) 
		\qquad \textnormal{and} \qquad
		\coulomb_{\beta}^\circ (\lambda x_{1},\ldots,\lambda x_{2N}) 
		= \lambda^{-2N h(\kappa)} \coulomb_{\beta}^\circ (x_{1},\ldots,x_{2N}) ,
	\end{align*}
for all $y\in\R$ and $\lambda>0$.
Then, using this translation invariance, 
for special conformal transformations $\varphi_{c} \colon z \mapsto\frac{z}{1 + c z}$ 
satisfying
$\varphi_{c}(x_{1}) < \cdots < \varphi_{c}(x_{2N})$,
we may without loss of generality assume that
$x_1 < 0$ and $x_{2N} > 0$, so that $c \in ( -1 / x_{2N} , -1 / x_{1} )$. The covariance property~\eqref{eqn::COV_for_H} can be verified by considering the $c$-variation of the right-hand side of~\eqref{eqn::COV_for_H} with $\varphi = \varphi_{c}$:
denoting 
$\varphi_{c}(\bs{x}) = (\varphi_{c}(x_{1}),\dots,\varphi_{c}(x_{2N}))$, 
\begin{align} 	
\nonumber
		\; & \der c \bigg(
		\prod_{i=1}^{2N} \varphi_{c}'(x_{i})^{h(\kappa)} \; \times \;
		\underset{\acycle_1^\beta \times \cdots \times \acycle_N^\beta}{\int} f_\beta (\varphi_{c}(\bs{x});\bs{u}) \; \ud u_1 \cdots \ud u_N 
		\bigg) \\
\label{eq: spec conf transf}
		=\; & - \prod_{i=1}^{2N} \varphi_{c}'(x_{i})^{h(\kappa)} \times 
		\underset{\acycle_1^\beta \times \cdots \times \acycle_N^\beta}{\int}
		\sum_{j=1}^{2N} \big( x_{j}^{2}\pder{x_{j}} f_\beta - \frac{x_{j}}{4} f_\beta  \big)
		(\varphi_{c}(\bs{x});\bs{u}) \; \ud u_1 \cdots \ud u_N .
\end{align}
This can be evaluated by observing (via a long calculation combined with Liouville theorem, 
as in~\cite[Lemma~4.14]{Kytola-Peltola:Conformally_covariant_boundary_correlation_functions_with_quantum_group})
that the integrand function $f$ defined in~\eqref{eq: integrand} satisfies the partial differential equation
	\begin{align*}
		\sum_{j=1}^{2N} \Big( x_{j}^{2}\pder{x_{j}} + 2 h(\kappa) \, x_{j}\Big) f(\bs{x};\bs{u})
		=\; & \sum_{r=1}^{N} \pder{u_{r}} 
		\big(g(u_{r};\bs{x};\bs{\dot{u}}_r) \; f(\bs{x};\bs{u}) \big),
	\end{align*}
	where 
	$\bs{\dot{u}}_r = (u_{1},\ldots,u_{r-1},u_{r+1},\ldots,u_{N})$
	and 
	$g$ is a rational function which is symmetric in its last $N-1$
	variables, and whose only poles are where some of its arguments coincide. 
This gives
		\begin{align*}
\textnormal{\eqref{eq: spec conf transf}}		=\; & - \prod_{i=1}^{2N} \varphi_{c}'(x_{i})^{h(\kappa)} \times  
		\underset{\acycle_1^\beta \times \cdots \times \acycle_N^\beta}{\int}
		\sum_{r=1}^{N} \pder{u_{r}} 
		\big(g(u_{r};\varphi_{c}(\bs{x});\bs{\dot{u}}_r) \; 
		f_\beta (\varphi_{c}(\bs{x});\bs{u}) \big) 
		\; \ud u_1 \cdots \ud u_N ,
	\end{align*}
which equals zero because each term in the sum vanishes by integration by parts, as the Pochhammer contours are homologically trivial. 
Therefore, the right-hand side of the asserted formula~\eqref{eqn::COV_for_H} 
	with $\varphi = \varphi_{c}$ is constant in $c \in (-1 / x_{2N} , -1 / x_{1})$. 
	Since at $c = 0$ we have $\varphi_{0} = \mathrm{id}_{\HH}$, this constant equals $\smash{\coulomb_{\beta}^\circ} (\bs{x})$.

Since the M\"obius group is generated by these three types of transformations,~\eqref{eqn::COV_for_H} follows. 
\end{proof}

\begin{proposition} \label{prop: PDEs Fcirc}
	For each $\beta \in \LP_N$, the function $\coulomb_{\beta}^\circ$ satisfies the PDE system~\eqref{eqn::PDE}, 
that is, 
\begin{align}\label{eqn::PDEj_for_H}
\LD^{(j)} \coulomb_{\beta}^\circ (\bs{x})
:= \; & \bigg[ 
\frac{\kappa}{2} \pdder{x_{j}}
+ \sum_{i \neq j} \Big( \frac{2}{x_{i}-x_{j}} \pder{x_{i}} 
- \frac{2h(\kappa)}{(x_{i}-x_{j})^{2}} \Big) \bigg]
\coulomb_{\beta}^\circ (\bs{x}) =  0 , \qquad \textnormal{for all } j \in \{1,\ldots,2N\} .
\end{align}
\end{proposition}

\begin{proof}
Fix $j \in \{1,\ldots,2N\}$. 
The proof is very similar to arguments appearing in~\cite[Proposition~4.12]{Kytola-Peltola:Conformally_covariant_boundary_correlation_functions_with_quantum_group} 
(for $\kappa \notin \QQ$)
and~\cite[Proposition~2.8]{LPW:UST_in_topological_polygons_partition_functions_for_SLE8_and_correlations_in_logCFT}
(for $\kappa = 8$).
	By dominated convergence, we can take the differential operator $\LD^{(j)}$ inside the integral in 
	$\coulomb_{\beta}^\circ$, and thus let it act directly to the integrand $f_\beta$. 
	Explicit calculations (similar to~\cite[Lemma~4.9 and Corollary~4.11]{Kytola-Peltola:Conformally_covariant_boundary_correlation_functions_with_quantum_group}) then give
	\begin{align*}
		\LD^{(j)} \coulomb_{\beta}^\circ (\bs{x}) = 
		\sum_{r=1}^{N} \underset{\acycle_1^\beta \times \cdots \times \acycle_N^\beta}{\int}
		\pder{u_{r}} \big( g(u_{r};\bs{x};\bs{\dot{u}}_r) \; f_\beta (\bs{x};\bs{u}) \big) 
		\; \ud u_1 \cdots \ud u_N ,
	\end{align*}
and similarly as in the proof of Proposition~\ref{prop: full Mobius covariance Fcirc}, integration by parts in each term in this sum shows that each term equals zero, which gives the asserted PDE~\eqref{eqn::PDEj_for_H}.
\end{proof}

\begin{lemma} \label{lem: neighbor asy}
	Fix $\beta \in \LP_N$ with link endpoints ordered as in~\eqref{eqn::linkpatterns_ordering}.
	Fix $j \in \{1, \ldots, 2N-1 \}$ such that $\{j,j+1\} \in \beta$.	
Then, for all $\xi \in (x_{j-1}, x_{j+2})$, using the notation~\eqref{eqn::bs_notation}, we have
	\begin{align} \label{eq: neighbor asy}
		\lim_{x_j,x_{j+1}\to\xi} \frac{\coulombnew_{\beta} (\bs{x})}{(x_{j+1} - x_j)^{-2h(\kappa)}} 
		= \sqrt{q(\kappa)} \, \coulombnew_{\beta / \{j,j+1\}} (\bs{\ddot{x}}_j) .
	\end{align}
\end{lemma}

\begin{proof}
We will use the relation of $\coulombnew_{\beta}$ with $\smash{\coulomb_\beta^\circ}$
from~(\ref{eq::related_two_Gs},~\ref{eqn::rela_G_and_H}).
Let $\acycle_s^\beta \ni u_s$ be the Pochhammer loop in~\eqref{eqn::coulombgasintegral_loops} which surrounds the points $x_j$ and $x_{j+1}$. 
Note that the integration contours $\smash{\acycle_1^\beta}, \ldots, \smash{\acycle_{s-1}^\beta}, \smash{\acycle_{s+1}^\beta}, \ldots, \smash{\acycle_{N}^\beta}$ remain bounded away from each other and from $\smash{\acycle_s^\beta}$,
and their homotopy types do not change upon taking the limit~\eqref{eq: neighbor asy}. 
By the dominated convergence theorem, the integral relevant for evaluating the limit is
	\begin{align} \label{eq: limit to evaluate neighbor}
		\lim_{x_j, x_{j+1} \to \xi} 
		\int_{x_j}^{x_{j+1}} \ud u_r \; \frac{f_\beta(\bs{x};\bs{u})}{(x_{j+1} - x_j)^{-2h(\kappa)}}  
		= \; & \lim_{x_j, x_{j+1} \to \xi} 
		\landupint_{x_j}^{x_{j+1}} \ud u_r \; \frac{f_\beta(\bs{x};\bs{u})}{(x_{j+1} - x_j)^{-2h(\kappa)}}  .
	\end{align}
By making the change of variables $v = \frac{u_s - x_j}{x_{j+1} - x_j}$ in this integral and collecting all the factors,
carefully noting that no branch cuts are crossed, 
and after taking into account cancellations 
and that some terms tend to one in the limit $x_j,x_{j+1}\to\xi$,
we obtain
\begin{align*}
\textnormal{\eqref{eq: limit to evaluate neighbor}} 
= f_\beta(\bs{\ddot{x}}_j; \bs{\dot{u}}_{s})
\int_{0}^1 v^{-4/\kappa} (1-v)^{-4/\kappa} \ud v 
= \frac{\big( \Gamma(1-4/\kappa) \big)^2}{\Gamma(2-8/\kappa)} \, f_\beta(\bs{\ddot{x}}_j; \bs{\dot{u}}_{s}) ,
\end{align*}
where $\bs{\dot{u}}_{s} := (u_1, \ldots, u_{s-1}, u_{s+1}, \ldots, u_N)$ and the multiplicative factor is the Euler Beta function. 	
Thus, using Lemma~\ref{lem::related_two_Gs} together with~\eqref{eqn::rela_G_and_H}, and~\eqref{eq:: loop vs interval} from the proof of Lemma~\ref{lem::related_two_Gs}, we obtain~\eqref{eq: neighbor asy}.
\end{proof}

\begin{proposition} 
\label{prop::non-neighbor asy}
Fix $\beta \in \LP_N$ with link endpoints ordered as in~\eqref{eqn::linkpatterns_ordering}.
Fix $j \in \{1, \ldots, 2N-1 \}$ such that $\{j,j+1\} \notin \beta$.	
Then, for all $\xi \in (x_{j-1}, x_{j+2})$, using the notation~\eqref{eqn::bs_notation}, we have
\begin{align*}
\lim_{x_j , x_{j+1} \to \xi} \frac{\coulombnew_{\beta} (\bs{x})}{ (x_{j+1} - x_j)^{ -2h(\kappa) }} 
= \coulombnew_{\wp_j(\beta)/\{j,j+1\}} (\bs{\ddot{x}}_j) .
\end{align*}
\end{proposition}

\begin{proof}
We prove Proposition~\ref{prop::non-neighbor asy} in Appendix~\ref{appendix_asy}. The proof is rather long and technical.
\end{proof}

\subsection{Coulomb gas integrals as linear combinations of pure partition functions}
\label{subsec::linearcombination}
In this section, we will prove Proposition~\ref{prop::linearcombination}, 
which gives a linear relation between the Coulomb gas type partition functions $\coulombnew_{\beta}$ of Theorem~\ref{thm::CGI_property} 
and the pure partition functions 
$\PartF_{\alpha}$ of Definition~\ref{def::PPF_general}. 
To this end, we use a deep result from~\cite{Flores-Kleban:Solution_space_for_system_of_null-state_PDE2} 
concerning the uniqueness of solutions to the PDE boundary value problems associated to the BPZ equations~\eqref{eqn::PDE}.

\begin{theorem}\label{thm::purepartition_unique}
{\textnormal{\cite[Lemma~1]{Flores-Kleban:Solution_space_for_system_of_null-state_PDE2}}}
Fix $\kappa\in (0,8)$. 
Let $F \colon \chamber_{2N} \to \C$ be a function satisfying
the PDE system~\eqref{eqn::PDE} and
the covariance~\eqref{eqn::COV}.
Suppose furthermore that there exist constants $C>0$ and $p>0$ such that for all 
$N \geq 1$ and $(x_1,\ldots, x_{2N}) \in \chamber_{2N}$, we have
\begin{align}\label{eqn::powerlawbound}
|F(x_1, \ldots, x_{2N})| \le C \prod_{1 \leq i<j \leq 2N}(x_j-x_i)^{\mu_{ij}(p)}, 
\quad \textnormal{where } \quad
\mu_{ij}(p) :=
\begin{cases}
p, &\textnormal{if } |x_j-x_i| > 1, \\
-p, &\textnormal{if } |x_j-x_i| < 1.
\end{cases}
\end{align}
If $F$ also has the asymptotics property
\begin{align}\label{eqn::differenceASY}
\lim_{x_j , x_{j+1} \to \xi} 
\frac{F(x_1 , \ldots , x_{2N})}{(x_{j+1} - x_j)^{-2h(\kappa)}} = 0 ,
\quad \textnormal{for all } j \in \{ 2, 3, \ldots , 2N-1 \}  \textnormal{ and } \xi \in (x_{j-1}, x_{j+2}) 
\end{align}
\textnormal{(}with the convention that $x_0 = -\infty$ and  $x_{2N+1} = +\infty$\textnormal{)},
then $F \equiv 0$. 
\end{theorem}

Thanks to Theorem~\ref{thm::purepartition_unique}, to verify the linear relation~\eqref{eqn::linearcombination} asserted in Proposition~\ref{prop::linearcombination} 
between the two sets of functions $\{\coulombnew_{\beta} \colon \beta \in \LP_N\}$ and $\{\PartF_{\alpha} \colon \alpha \in \LP_N\}$, it suffices to show that the difference 
\begin{align*}
\coulombnew_{\beta}
- \underbrace{\sum_{\alpha\in\LP_N} \LM_{\alpha,\beta}(q(\kappa)) \, \PartF_{\alpha}}_{=: \tilde{\LG}_{\beta}} =: \coulombnew_{\beta} - \tilde{\LG}_{\beta}
\end{align*}
satisfies all of the properties in Theorem~\ref{thm::purepartition_unique}.

\begin{proof}[Proof of Proposition~\ref{prop::linearcombination}]
Fix $\kappa\in (4,6]$.  Let us consider the functions $\tilde{\LG}_{\beta}$.
As $\{\PartF_{\alpha} \colon \alpha\in\LP_N\}$ satisfy~(\ref{eqn::PDE},~\ref{eqn::COV}), the functions $\tilde{\LG}_{\beta}$ also satisfy~(\ref{eqn::PDE},~\ref{eqn::COV}) by linearity.
Also, as $\PartF_{\alpha}$ satisfy~\eqref{eqn::PPFBounds_general}, the functions $\tilde{\LG}_{\beta}$ satisfy~\eqref{eqn::powerlawbound}. 
It remains to study the asymptotics of $\tilde{\LG}_{\beta}$. 
To this end, we fix $N \ge 1$, a link pattern $\beta \in \LP_N$, index $j \in \{1,2, \ldots, 2N-1 \}$, and point $\xi \in (x_{j-1}, x_{j+2})$. Then, using the notation~\eqref{eqn::bs_notation}, we find the following asymptotics for $\tilde{\LG}_{\beta}$. 
\begin{itemize}[leftmargin=2em]
\item If $\{j,j+1\}\in\beta$, then 
for any $\alpha \in \LP_N$, we have
\begin{align} \label{eq:: LMrelation_neighbor}
\LM_{\alpha,\beta}(q(\kappa)) = \sqrt{q(\kappa)} \, \LM_{\alpha/\{j,j+1\}, \beta/\{j,j+1\}}(q(\kappa)) ,
\end{align}
since the number of loops in the meander satisfies $\LL_{\alpha,\beta} = \LL_{\alpha/\{j,j+1\}, \beta/\{j,j+1\}}+1$.
Using this, we find 
\begin{align*}
\lim_{x_j, x_{j+1}\to \xi}\frac{\tilde{\LG}_{\beta}(\bs{x})}{(x_{j+1}-x_j)^{-2h(\kappa)}} 
= & \; \sum_{\substack{\alpha\in\LP_N\\\{j,j+1\}\in\alpha}}\LM_{\alpha,\beta}(q(\kappa)) \, 
\PartF_{\alpha/\{j,j+1\}}(\bs{\ddot{x}}_j) 
&& \textnormal{[by~\eqref{eqn::PPFASY_general}]}
\\
= & \; \sum_{\gamma\in\LP_{N-1}}\sqrt{q(\kappa)}  \,  \LM_{\gamma, \beta/\{j,j+1\}}(q(\kappa)) \, 
\PartF_{\gamma}(\bs{\ddot{x}}_j) 
&& \textnormal{[by~\eqref{eq:: LMrelation_neighbor}]} 
\\
= & \; \sqrt{q(\kappa)} \, \tilde{\LG}_{\beta/\{j,j+1\}}(\bs{\ddot{x}}_j) ,
\end{align*}
by re-indexing the sum using the bijection $\alpha \leftrightarrow \alpha/\{j,j+1\} = \gamma$.

\item If $\{j,j+1\}\not\in\beta$, then
for any $\alpha \in \LP_N$, we have
\begin{align} \label{eq:: LMrelation_non_neighbor}
\LM_{\alpha,\beta}(q(\kappa)) = \LM_{\gamma, \wp_j(\beta)/\{j,j+1\}}(q(\kappa))
\end{align}
since the number of loops in the meander satisfies 
$\LL_{\alpha,\beta} = \LL_{\alpha/\{j,j+1\}, \wp_j(\beta)/\{j,j+1\}}$.
Using this, we find 
\begin{align*} 
\lim_{x_j, x_{j+1}\to \xi}\frac{\tilde{\LG}_{\beta}(\bs{x})}{(x_{j+1}-x_j)^{-2h(\kappa)}} 
= & \; \sum_{\substack{\alpha\in\LP_N\\\{j,j+1\}\in\alpha}}\LM_{\alpha,\beta}(q(\kappa)) \, 
\PartF_{\alpha/\{j,j+1\}}(\bs{\ddot{x}}_j) 
&&\textnormal{[by~\eqref{eqn::PPFASY_general}]}
\\
= & \; \sum_{\gamma\in\LP_{N-1}}
\LM_{\gamma, \wp_j(\beta)/\{j,j+1\}}(q(\kappa)) \, 
\PartF_{\gamma}(\bs{\ddot{x}}_j)
&& \textnormal{[by~\eqref{eq:: LMrelation_non_neighbor}]} 
\\
= & \; \tilde{\LG}_{\wp_j(\beta)/\{j,j+1\}}(\bs{\ddot{x}}_j) ,
\end{align*}
by re-indexing the sum using the bijection $\alpha \leftrightarrow \alpha/\{j,j+1\} = \gamma$.
\end{itemize}

With these properties of $\tilde{\LG}_{\beta}$ at hand, 
recalling that $\coulombnew_{\beta}$ satisfy the asymptotics~\eqref{eqn::ASY} analogous to the asymptotics of $\tilde{\LG}_{\beta}$,
we see recursively (by induction on $N \geq 1$) 
that the collection $\{\coulombnew_{\beta} - \tilde{\LG}_{\beta} \colon \beta\in\LP_N\}$ satisfies all of the properties in Theorem~\ref{thm::purepartition_unique}. 
Therefore, we conclude that $\coulombnew_{\beta}=\tilde{\LG}_{\beta}$, for all $\beta\in\LP_N$.

Lastly, we see that $\coulombnew_{\beta} > 0$ because 
$\PartF_{\alpha} > 0$ and $\LM_{\alpha,\beta}(q(\kappa)) > 0$, 
for all $\alpha, \beta \in \LP_N$. 
\end{proof}

\subsection{Partition functions $\LF_{\beta}$ when $\kappa = 16/3$}
\label{sec::totalpartition_analysis}
The aim of this section is to verify the alternative formula~\eqref{eqn::totalpartition_def} in Theorem~\ref{thm::FKIsing_Loewner} for $\coulombnew_\beta$ when $\kappa = 16/3$.

\begin{theorem}\label{thm::totalpartition}
The functions $\LF_{\beta}$ defined in~\eqref{eqn::totalpartition_def} satisfy the PDEs~\eqref{eqn::PDE} and the M\"{o}bius covariance~\eqref{eqn::COV} with $\kappa=16/3$, 
as well as the asymptotics 
\textnormal{(}using the notation~\eqref{eqn::bs_notation}\textnormal{)} 
\begin{align} \label{eqn::ASY_LF} 
\; & \lim_{x_j,x_{j+1}\to\xi} \frac{\LF_{\beta}(\bs{x})}{ (x_{j+1}-x_j)^{-2h(\kappa)} }
= 
\begin{cases}
\sqrt{q(\kappa)} \, \LF_{\beta/\{j,j+1\}}(\bs{\ddot{x}}_j),
& \textnormal{if }\{j, j+1\}\in\beta , \\
\LF_{\wp_j(\beta)/\{j,j+1\}}(\bs{\ddot{x}}_j),
& \textnormal{if }\{j, j+1\} \not\in \beta , 
\end{cases}
\end{align}
for all $\xi \in (x_{j-1}, x_{j+2})$, $j \in \{1,2, \ldots, 2N-1 \}$, and $N \ge 1$. Consequently, $\LF_{\beta}$ equals $\coulombnew_{\beta}$~when $\kappa=16/3$.  
\end{theorem}

To prove Theorem~\ref{thm::totalpartition}, we shall again make use of Theorem~\ref{thm::purepartition_unique}.

\begin{proof}[Proof of Theorem~\ref{thm::totalpartition}]
It suffices to verify that the difference $\LF_{\beta} - \coulombnew_{\beta}$ (with $\kappa=16/3$) satisfies all of the properties in Theorem~\ref{thm::purepartition_unique}. 
Indeed, we will prove in this section the following properties for $\LF_{\beta}$. 
\begin{itemize}[leftmargin=2em]
\item $\LF_{\beta}$ satisfies the PDE system~\eqref{eqn::PDE} with $\kappa=16/3$ due to 
Proposition~\ref{prop::totalpartition_PDE}.

\item $\LF_{\beta}$ satisfies the M\"{o}bius covariance~\eqref{eqn::COV} with $\kappa=16/3$ due to Proposition~\ref{prop::totalpartition_COV}. 

\item $\LF_{\beta}$ satisfies the asymptotics~\eqref{eqn::ASY_LF} with $\kappa=16/3$ due to Proposition~\ref{prop::totalpartition_ASY}. 
\end{itemize}
Hence, by Theorem~\ref{thm::CGI_property}, the difference
$\LF_{\beta} - \coulombnew_{\beta}$ satisfies the power law bound~\eqref{eqn::powerlawbound}, the PDE system~\eqref{eqn::PDE}, and the M\"{o}bius covariance~\eqref{eqn::COV}. Since also similar asymptotics~\eqref{eqn::ASY_LF} and~\eqref{eqn::differenceASY}
hold for $\LF_{\beta}$ and $\coulombnew_{\beta}$, 
we see recursively\footnote{That is, by induction on $N \geq 1$.}  that the collection $\{\LF_{\beta} - \coulombnew_{\beta} \colon \beta\in\LP_N\}$ satisfies all of the properties in Theorem~\ref{thm::purepartition_unique}. 
\end{proof}

\begin{corollary}\label{cor::linearcombination_FKIsing}
We have
\begin{align*} 
\LF_{\beta}(\bs{x}) 
= \sum_{\alpha\in\LP_N}\LM_{\alpha,\beta}(2) \, 
\PartF_{\alpha}(\bs{x}), \qquad \textnormal{for all }\beta\in\LP_N, 
\end{align*}
where $\LF_{\beta}$ is defined in~\eqref{eqn::totalpartition_def}, 
$\LM_{\alpha,\beta}(2)$ is defined in~\eqref{eqn::meandermatrix_def_general} with $q=2$, 
and $\{\PartF_{\alpha} \colon \alpha\in\LP_N\}$ is the collection of pure partition functions for multiple $\SLE_{\kappa}$ described in Definition~\ref{def::PPF_general} with $\kappa=16/3$. 
\end{corollary}

\begin{proof}
This is immediate from Proposition~\ref{prop::linearcombination} and Theorem~\ref{thm::totalpartition}. 
\end{proof}

In the remainder of this section, we prove the missing ingredients for Theorem~\ref{thm::totalpartition}. 

\smallbreak

\begin{proposition}\label{prop::totalpartition_PDE}
The functions $\LF_{\beta}$ defined in~\eqref{eqn::totalpartition_def} satisfy the PDE system~\eqref{eqn::PDE} with $\kappa=16/3$. 
\end{proposition}

It has already been known for a long time in the physics literature 
that the bulk spin correlation functions in the Ising model satisfy the BPZ PDEs~\eqref{eqn::PDE} (see, e.g.,~\cite[Chapter~12.2.2]{DMS:CFT}). 
This was recently verified explicitly by Izyurov in~\cite[Corollary~1.3]{Izyurov:On_multiple_SLE_for_the_FK_Ising_model}, 
and we recover the same result from  Theorem~\ref{thm::FKIsing_Loewner} 
(which will be proven in Section~\ref{sec::FKIsing_Loewner}, independently of the results of the present section).

\begin{proof}
The PDEs~\eqref{eqn::PDE} follow from Theorem~\ref{thm::FKIsing_Loewner} 
together with the commutation relations for SLEs derived by Dub\'edat~\cite[Theorem~7]{Dubedat:Commutation_relations_for_SLE}, see also~\cite[Appendix~A]{Kytola-Peltola:Pure_partition_functions_of_multiple_SLEs}, 
and~\cite[Corollary~1.3]{Izyurov:On_multiple_SLE_for_the_FK_Ising_model}.
\end{proof}

\begin{proposition}\label{prop::totalpartition_COV}
The functions $\LF_{\beta}$ defined in~\eqref{eqn::totalpartition_def} satisfy the 
covariance~\eqref{eqn::COV} with $\kappa=16/3$. 
\end{proposition}
\begin{proof}
For any M\"{o}bius map $\varphi$ of $\HH$ such that $\varphi(x_1)<\cdots<\varphi(x_{2N})$, we have 
\begin{align*}
\varphi(y)-\varphi(x)=\varphi'(x)^{1/2}\varphi'(y)^{1/2}(y-x) , \qquad \textnormal{for all } x_1\le x<y\le x_{2N} .
\end{align*}
This gives the desired the covariance by direct inspection of the formula~\eqref{eqn::totalpartition_def}. 
\end{proof}

\begin{proposition}\label{prop::totalpartition_ASY}
The functions $\LF_{\beta}$ defined in~\eqref{eqn::totalpartition_def} satisfy the asymptotics~\eqref{eqn::ASY_LF} with $\kappa=16/3$. 
\end{proposition}

\begin{proof}
We use the notation~\eqref{eqn::bs_notation}.
We first treat the case where $\{j,j+1\}\in\beta$. 
Write $a_{r}=j$ and $b_{r}=j+1$ for some $r\in\{1, \ldots, N\}$. Then, we 
easily find the desired asymptotics~\eqref{eqn::ASY_LF} from formula~\eqref{eqn::totalpartition_def}:
\begin{align*}
\lim_{x_j, x_{j+1}\to \xi}\frac{\LF_{\beta}(\bs{x})}{|x_{j+1}-x_j|^{-1/8}} 
= & \; \prod_{\substack{1\le s \le N\\ s \neq r}} |x_{b_s}-x_{a_s}|^{-1/8} 
\bigg(\sum_{\bs{\sigma} \in \{\pm 1\}^N}\prod_{\substack{1\le s < t \le N\\ s,t\neq r}} 
\chi(x_{a_s}, x_{a_t}, x_{b_t}, x_{b_s})^{\sigma_s \sigma_t /4}\bigg)^{1/2}\\
= & \; \sqrt{2} \, \LF_{\beta/\{j,j+1\}}(\bs{\ddot{x}}_j) .
\end{align*}

\smallbreak

Next, we treat the more complicated case where $\{j,j+1\}\not\in\beta$.  We consider three cases separately.
\begin{enumerate}[label=(\Alph*):, ref=\Alph*]
\item \label{asy-1}
Suppose there exist $1\le r<s\le N$ such that $a_r<b_r=j<j+1=a_s<b_s$. First, we have 
\begin{align}\label{eqn::ASY2A1}
\lim_{x_j,x_{j+1}\to \xi}\prod_{1\le t \le N} |x_{b_t}-x_{a_t}|^{-1/8} 
= |\xi-x_{a_r}|^{-1/8} \, |x_{b_s}-\xi|^{-1/8} \,
\prod_{\substack{1\le t \le N\\ t \neq r, s}} |x_{b_t}-x_{a_t}|^{-1/8}  ;
\end{align}
and second,  for fixed $\bs{\sigma} = (\sigma_1, \ldots, \sigma_N) \in \{\pm 1\}^N$, we have 
\begin{align*}
\hspace*{-5mm}
\prod_{1\le t < u \le N}\chi(x_{a_t}, x_{a_u}, x_{b_u}, x_{b_t})^{\sigma_t \sigma_u /4} 
\; = \; \bigg| \frac{(x_{j+1}-x_{a_r})(x_{b_s}-x_j)}{(x_{b_s}-x_{a_r})(x_{j+1}-x_j)}\bigg|^{\sigma_r \sigma_s/4} \, 
\prod_{\substack{1\le t < u \le N\\ \{t , u\}\neq\{r,s\}}}\chi(x_{a_t}, x_{a_u}, x_{b_u}, x_{b_t})^{\sigma_t \sigma_u/4} . 
\end{align*} 
After normalizing by $|x_{j+1}-x_j|^{-1/4}$ and letting $x_j, x_{j+1}\to \xi$, only the terms with $\sigma_r \sigma_s=1$ survive. 
Thus, for fixed $\bs{\sigma}\in\{\pm 1\}^N$ with $\sigma_r\sigma_s=1$, we have 
\begin{align}\label{eqn::ASY2A2aux1}
& \; \lim_{x_j, x_{j+1}\to \xi} \frac{1}{|x_{j+1}-x_j|^{-1/4}} 
\prod_{1\le t < u \le N}\chi(x_{a_t}, x_{a_u}, x_{b_u}, x_{b_t})^{\sigma_t \sigma_u / 4} \\
= & \; \prod_{\substack{1\le t < u \le N\\ \{t,u\}\cap\{r,s\}=\emptyset}} \chi(x_{a_t}, x_{a_u}, x_{b_u}, x_{b_t})^{\sigma_t \sigma_u / 4} 
\prod_{1\le t<r}\chi(x_{a_t}, x_{a_r}, \xi, x_{b_t})^{\sigma_t \sigma_r/4} 
\prod_{\substack{1\le t < s\\ t \neq r}}\chi(x_{a_t}, \xi, x_{b_s}, x_{b_t})^{\sigma_t \sigma_s / 4}
\nonumber \\
& \; \times \bigg| \frac{(\xi-x_{a_r})(x_{b_s}-\xi)}{(x_{b_s}-x_{a_r})}\bigg|^{1/4}  \; 
\prod_{\substack{r < u \le N\\ u \neq s}}\chi(x_{a_r}, x_{a_u}, x_{b_u}, \xi)^{\sigma_r\sigma_u/4} 
\prod_{s < u \le N}\chi(\xi, x_{a_u}, x_{b_u}, x_{b_s})^{\sigma_s\sigma_u/4} . 
\nonumber 
\end{align}
Let us consider the terms on the right-hand side of~\eqref{eqn::ASY2A2aux1}. For $1\le t<r$, we have $\sigma_t \sigma_r = \sigma_t \sigma_s$, and 
\begin{align}\label{eqn::ASY2A2aux2}
\chi(x_{a_t}, x_{a_r}, \xi, x_{b_t}) \; \chi(x_{a_t}, \xi, x_{b_s}, x_{b_t}) = \chi(x_{a_t}, x_{a_r}, x_{b_s}, x_{b_t}) ;
\end{align}
while for $s < u \le N$, we have $\sigma_r \sigma_u = \sigma_s \sigma_u$, and 
\begin{align}\label{eqn::ASY2A2aux3}
\chi(x_{a_r}, x_{a_u}, x_{b_u}, \xi) \; \chi(\xi, x_{a_u}, x_{b_u}, x_{b_s}) = \chi(x_{a_u}, x_{a_r}, x_{b_s}, x_{b_u}) ;
\end{align}
while for $r<t<s$, we have $\sigma_t \sigma_s = \sigma_t \sigma_r$, and 
\begin{align}\label{eqn::ASY2A2aux4}
\chi(x_{a_t}, \xi, x_{b_s}, x_{b_t}) \; \chi(x_{a_r}, x_{a_t}, x_{b_t}, \xi) = \chi(x_{a_t}, x_{a_r},  x_{b_s}, x_{b_t}) . 
\end{align}
Thus, after plugging all of~(\ref{eqn::ASY2A2aux2},~\ref{eqn::ASY2A2aux3},~\ref{eqn::ASY2A2aux4})
into~\eqref{eqn::ASY2A2aux1}, for each $\bs{\sigma}\in\{\pm 1\}^N$ with $\sigma_r\sigma_s=1$, we find 
\begin{align}\label{eqn::ASY2A2aux5}
& \; \lim_{x_j, x_{j+1}\to \xi}\frac{1}{|x_{j+1}-x_j|^{-1/4}} 
\prod_{1\le t < u \le N}\chi(x_{a_t}, x_{a_u}, x_{b_u}, x_{b_t})^{\sigma_t \sigma_u / 4} 
\\
= & \; \bigg| \frac{(\xi-x_{a_r})(x_{b_s}-\xi)}{(x_{b_s}-x_{a_r})}\bigg|^{1/4} 
\; \prod_{\substack{1\le t < u \le N\\ \{t,u\}\cap\{r,s\}=\emptyset}} 
\chi(x_{a_t}, x_{a_u}, x_{b_u}, x_{b_t})^{\sigma_t \sigma_u/4}
\; 
\prod_{\substack{1\le t \le N\\ t \neq r, s}}\chi(x_{a_t}, x_{a_r}, x_{b_s}, x_{b_t})^{\sigma_t \sigma_r/4} .
\nonumber 
\end{align}
Finally, by combining~\eqref{eqn::ASY2A1} and~\eqref{eqn::ASY2A2aux5}, we find the desired asymptotics~\eqref{eqn::ASY_LF}:
\begin{align*}
& \; \lim_{x_j, x_{j+1}\to \xi}\frac{\LF_{\beta}(\bs{x})}{|x_{j+1}-x_j|^{-1/8}} \\
= & \; \; |\xi-x_{a_r}|^{-1/8} |x_{b_s}-\xi|^{-1/8}
\; \bigg| \frac{(\xi-x_{a_r})(x_{b_s}-\xi)}{(x_{b_s}-x_{a_r})}\bigg|^{1/8}  \;\prod_{\substack{1\le t \le N\\ t \neq r, s}}|x_{b_t}-x_{a_t}|^{-1/8}
\\
& \; \times 
\bigg(\sum_{\substack{\bs{\sigma}\in\{\pm 1\}^N\\\sigma_r\sigma_s=1}}\prod_{\substack{1\le t < u \le N\\ \{t,u\}\cap\{r,s\}=\emptyset}} \chi(x_{a_t}, x_{a_u}, x_{b_u}, x_{b_t})^{\sigma_t \sigma_u/4}
\; \prod_{\substack{1\le t \le N\\ t \neq r, s}}\chi(x_{a_t}, x_{a_r}, x_{b_s}, x_{b_t})^{\sigma_t \sigma_r/4}\bigg)^{1/2}  \\
= & \; \; |x_{b_s}-x_{a_r}|^{-1/8}
\prod_{\substack{1\le t \le N\\ t \neq r, s}}(x_{b_t}-x_{a_t})^{-1/8} \\
& \; 
\times \bigg(\sum_{\substack{\bs{\sigma}\in\{\pm 1\}^N\\\sigma_r\sigma_s=1}}\prod_{\substack{1\le t < u \le N\\ \{t,u\}\cap\{r,s\}=\emptyset}} \chi(x_{a_t}, x_{a_u}, x_{b_u}, x_{b_t})^{\sigma_t \sigma_u/4}
\; \prod_{\substack{1\le t \le N\\ t \neq r, s}}\chi(x_{a_t}, x_{a_r}, x_{b_s}, x_{b_t})^{\sigma_t \sigma_r/4}\bigg)^{1/2}\\
= & \; \; \LF_{\wp_j(\beta)/\{j,j+1\}}(\bs{\ddot{x}}_j) . 
\end{align*}
This completes the proof of Case~\ref{asy-1}.

\item \label{asy-2}
Suppose there exist $1\le r<s\le N$ such that $a_r=j<j+1=a_s<b_s<b_r$. 
This case can be derived in a similar way as Case~\ref{asy-1}. 

\item \label{asy-3}
Suppose there exist $1\le r<s\le N$ such that $a_r<a_s<b_s=j<j+1=b_r$. 
This case can be derived in a similar way as Case~\ref{asy-1}.
\end{enumerate}
This completes the proof. 
\end{proof}

\section{Interfaces in the FK-Ising model: Proof of Theorem~\ref{thm::FKIsing_Loewner}}
\label{sec::FKIsing_Loewner}
In this section, we consider the FK-Ising model on finite subgraphs of the square lattice $\Z^2$, or rather, of the square lattice $\delta \Z^2$ scaled by $\delta > 0$. We take $\delta \to 0$, which we call the \emph{scaling limit} of the model.
In this article, we only consider the \emph{critical} model, which has the following edge-weight~\cite{Beffara-Duminil-Copin:The_self-dual_point_of_the_two-dimensional_random-cluster_model_is_critical_for_q_bigger_than_one}: 
\begin{align*}
p = p_c(2) := \frac{\sqrt{2}}{1+\sqrt{2}} .
\end{align*} 
We endow the model with various boundary conditions and prove the convergence of multiple interfaces to multiple $\SLE_{16/3}$ curves in the scaling limit (Theorem~\ref{thm::FKIsing_Loewner}, whose proof is completed in Section~\ref{subsec::FKIsing_Loewner}).
In the next Section~\ref{sec::crossingproba}, we prove
the convergence of connection probabilities of the interfaces
(Theorem~\ref{thm::FKIsing_crossingproba}).

\subsection{Preliminaries on random-cluster models}
\label{subsec::rcm_pre}
In this section, we use the notation and terminology specified in Section~\ref{subsec::RCM}.
We also recommend~\cite{Grimmett:Random_cluster_model, Duminil-Copin:PIMS_lectures} for more background and details on the discrete models, and~\cite{DCS:Conformal_invariance_of_lattice_models} for methods addressing the scaling limit.

\paragraph*{Discrete polygons.}
A \emph{discrete (topological) polygon}, 
whose precise definition is given below, 
is a finite simply connected subgraph 
of $\Z^2$, or $\delta \Z^2$, 
with $2N$ marked boundary points in counterclockwise order. 

\begin{enumerate}[leftmargin=*]
\item  
First, we define the \emph{medial polygon}. We give orientation to edges of the medial lattice $(\Z^2)^\diamond$ as follows: edges of each face containing a vertex of $\Z^2$ are oriented clockwise, and edges of each face containing a vertex of $(\Z^2)^{\bullet}$ are oriented counterclockwise. 
Let $x_1^\diamond,\ldots, x_{2N}^\diamond$ be $2N$ distinct medial vertices. Let $(x_1^\diamond \, x_2^\diamond), (x_2^\diamond \, x_3^\diamond), \ldots , (x_{2N}^\diamond  \, x_{1}^\diamond)$ be $2N$ oriented paths on $(\Z^2)^\diamond$ satisfying the following conditions\footnote{Throughout, we use the convention that $x_{2N+1}^\diamond := x_{1}^\diamond$.}: 
\begin{itemize}[leftmargin=1.0em]
\item 
each path $(x_{2r-1}^\diamond \, x_{2r}^\diamond)$ has counterclockwise oriented edges for $1\leq r \leq N$; 

\item  
each path $(x_{2r}^\diamond \, x_{2r+1}^\diamond)$ has clockwise oriented edges for $1\leq r \leq N$; 

\item  
all paths are edge-avoiding and satisfy $(x_{i-1}^\diamond \, x_i^\diamond) \cap (x_i^\diamond \, x_{i+1}^\diamond) = \{x_i^\diamond\}$ for $1\leq i \leq 2N$;

\item  
if $j\notin \{i+1,i-1\}$, then $(x_{i-1}^\diamond \, x_{i}^\diamond) \cap (x_{j-1}^\diamond \, x_j^\diamond) = \emptyset$; 

\item  
the infinite connected component of 
$(\Z^2)^\diamond\setminus \smash{\bigcup_{i=1}^{2N}} (x_i^\diamond \, x_{i+1}^\diamond)$ 
lies to the right of the oriented path~$(x_1^\diamond \, x_2^\diamond)$. 
\end{itemize}
Given $\{(x_i^\diamond \, x_{i+1}^\diamond) \colon 1\leq i\leq 2N\}$, the medial polygon $(\Omega^\diamond; x_1^\diamond,\ldots, x_{2N}^\diamond)$ is defined 
as the subgraph of $(\Z^2)^\diamond$ induced by the vertices lying on or enclosed by the non-oriented loop obtained by concatenating all of $(x_i^\diamond \, x_{i+1}^\diamond)$. 
For each $i \in \{1,2,\ldots,2N\}$, the \emph{outer corner}  $y_{i}^{\diamond}\in (\mathbb{Z}^2)^\diamond\setminus\Omega^\diamond$ is defined to be a medial vertex adjacent to $x_i^\diamond$, and the \emph{outer corner edge} $e_i^\diamond$ is defined to be the medial edge connecting them.

\item  
Second, we define the \emph{primal polygon} 
$(\Omega;x_1,\ldots,x_{2N})$ induced by $(\Omega^\diamond;x_1^\diamond,\ldots,x_{2N}^\diamond)$ as follows: 
\begin{itemize}[leftmargin=1.0em]
\item 
its edge set $E(\Omega)$ comprises edges passing through endpoints of medial edges in 
$E(\Omega^\diamond)\setminus \smash{\bigcup_{r=1}^N} (x_{2r}^\diamond \, x_{2r+1}^\diamond)$; 

\item  
its vertex set $V(\Omega)$ consists of endpoints of edges in $E(\Omega)$; 

\item  
the marked boundary vertex $x_i$ is defined to be the vertex in $\Omega$ nearest to $x_i^\diamond$ for each $1\leq i\leq 2N$; 

\item  
the arc $(x_{2r-1} \, x_{2r})$ is the set of edges whose midpoints are vertices in $(x_{2r-1}^\diamond \, x_{2r}^\diamond)\cap \partial \Omega^\diamond$ for $1\leq r \leq N$.
\end{itemize}

\item  
Third, we define the \emph{dual polygon} $(\Omega^{\bullet};x_1^{\bullet},\ldots,x_{2N}^{\bullet})$ induced by $(\Omega^\diamond; x_1^\diamond,\ldots,x_{2N}^\diamond)$ in a similar way. 
More precisely, $\Omega^{\bullet}$ is the subgraph of $(\Z^2)^{\bullet}$ with 
\begin{itemize}[leftmargin=1.0em]
\item 
edge set consisting of edges passing through endpoints of medial edges in $E(\Omega^\diamond)\setminus \smash{\bigcup_{r=1}^{N}} (x_{2r-1}^\diamond \, x_{2r}^\diamond)$; 
\item and vertex set consisting of the endpoints of these edges. 
\end{itemize}
For each $i \in \{1,2,\ldots,2N\}$, 
the marked boundary vertex $x_i^{\bullet}$ is defined to be the vertex in $\Omega^{\bullet}$ nearest to $x_i^\diamond$; and 
for each $r \in \{1,2,\ldots,N\}$, 
the boundary arc $(x_{2r}^{\bullet} \, x_{2r+1}^{\bullet})$ is defined to be 
the set of edges whose midpoints are vertices in $(x_{2r}^\diamond \, x_{2r+1}^\diamond)\cap \Omega^\diamond$. 
\end{enumerate}

\paragraph*{Boundary conditions.}
In this work, 
we shall focus on the critical FK-Ising model on the primal polygon $(\Omega;x_1,\ldots,x_{2N}) = (\Omega^\delta; x_1^\delta,\ldots,x_{2N}^\delta)$, with the following boundary conditions: 
first, every other boundary arc is wired, 
\begin{align*}
(x_{2r-1}^{\delta} \, x_{2r}^{\delta}) \textnormal{ is wired,} \qquad \textnormal{ for all } r \in\{1,2,\ldots, N\} ,
\end{align*}
and second, these $N$ wired arcs are further wired together 
outside of $\Omega^{\delta}$ according to a planar link pattern $\beta\in\LP_N$ as in~\eqref{eqn::linkpatterns_ordering} 
--- see Figure~\ref{fig::6points} in Section~\ref{sec::intro}.
In this setup, we say that the model has \emph{boundary condition} (b.c.) $\beta$. We denote by $\PP_{\beta}^{\delta}$ the law, and by $\mathbb{E}_{\beta}^{\delta}$ the expectation, of the critical model on $(\Omega^{\delta}; x_{1}^{\delta},\ldots,x_{2N}^{\delta})$ with b.c.~$\beta$, where the cluster-weight has the fixed value $q=2$ in this section.

\paragraph*{Loop representation and interfaces.}
Let $\omega \in \{0,1\}^{E(\Omega^\delta)}$ be a configuration 
with b.c.~$\beta\in \LP_N$ on the primal polygon $(\Omega^\delta; x_1^\delta,\ldots,x_{2N}^\delta)$, as defined in Section~\ref{subsec::RCM}. 
Note that $\omega$ induces a dual configuration $\omega^{\bullet}$ on $\Omega^{\bullet}$ via $\omega^{\bullet}_e = 1 - \omega_e$. An edge $e \in E(\Omega^{\bullet})$ is said to be \emph{dual-open} (resp.~\emph{dual-closed}) if $\omega^{\bullet}_e=1$ (resp.~$\omega^{\bullet}_e=0$).
Given $\omega$, we can draw self-avoiding 
paths on the medial graph $\Omega^{\delta, \diamond}$ between $\omega$ and $\omega^{\bullet}$ as follows: 
a path arriving at a vertex of $\Omega^{\delta,\diamond }$ always makes a turn of $\pm\pi/2$, so as not to cross the open or dual-open edges through this vertex. 
The \emph{loop representation} of $\omega$ contains a number of loops and $N$ pairwise-disjoint and self-avoiding \emph{interfaces} connecting the $2N$ outer corners $y_{1}^{\delta,\diamond}, \ldots,y_{2N}^{\delta,\diamond}$ of  the medial polygon $(\Omega^{\delta,\diamond};x_1^{\delta,\diamond},\ldots,x_{2N}^{\delta,\diamond})$. For each $i\in \{1,2,\ldots,2N\}$, we shall denote by $\eta_i^\delta$ the interface starting from the medial vertex $y_{i}^{\delta,\diamond}$
(and we also refer to it as the interface starting from the boundary point $x_{i}^{\delta,\diamond}$).
See~Figure~\ref{fig::loop_representation} in Section~\ref{sec::intro}. 

\paragraph*{Convergence of polygons.}
To investigate the scaling limit, we use the following notion of convergence of domains~\cite{Pommerenke:Boundary_behaviour_of_conformal_maps}.  
Abusing notation, for a discrete polygon, we will occasionally denote by $\Omega^{\delta}$ also the open simply connected subset of $\C$ defined as the interior of the set $\overline{\Omega}^{\delta}$ comprising all vertices, edges, and faces of the polygon $\Omega^{\delta}$. 

\smallbreak

Let $\{\Omega^{\delta}\}_{\delta>0}$ and $\Omega$ be simply connected open sets $\Omega^{\delta}, \Omega\subsetneq\C$, all containing a common point $u$. 
We say that $\Omega^{\delta}$ converges to $\Omega$ in the sense of \emph{kernel convergence with respect to} $u$, and 
denote $\Omega^{\delta}\to\Omega$, if 
\begin{enumerate} 
\item
every $z\in\Omega$ has some neighborhood $U_z$ such that $U_z\subset\Omega^{\delta}$, for all small enough $\delta > 0$; and 

\item 
for every boundary point $p\in\partial\Omega$, there exists a sequence $p^{\delta}\in\partial\Omega^{\delta}$ such that $p^{\delta}\to p$ as $\delta\to 0$. 
\end{enumerate}
If $\Omega^{\delta}\to\Omega$ in the sense of kernel convergence with respect to $u$, then the same convergence holds with respect to any $\tilde{u}\in\Omega$. 
We say that $\Omega^{\delta}\to\Omega$ in the Carath\'{e}odory sense as $\delta\to 0$.  
By~\cite[Theorem~1.8]{Pommerenke:Boundary_behaviour_of_conformal_maps}, $\Omega^{\delta}\to \Omega$ in the Carath\'{e}odory sense if and only if there exist conformal maps $\varphi_{\delta}$ from $\Omega^{\delta}$ onto the unit disc 
$\U := \{ z \in \C \colon |z| < 1 \}$, and a conformal map $\varphi$ from $\Omega$ onto $\U$, such that $\varphi_{\delta}^{-1}\to\varphi^{-1}$ locally uniformly on $\U$ as $\delta\to 0$, see~\cite[Theorem~1.8]{Pommerenke:Boundary_behaviour_of_conformal_maps}.

For polygons, we say that a sequence of discrete polygons $(\Omega^{\delta}; x_1^{\delta}, \ldots, x_{2N}^{\delta})$
converges as $\delta \to 0$ to a polygon $(\Omega; x_1, \ldots, x_{2N})$ in the \emph{Carath\'{e}odory sense}
if there exist conformal maps $\varphi_{\delta}$ from $\Omega^{\delta}$ onto $\U$,
and a conformal map $\varphi$ from $\Omega$ onto $\U$,
such that $\varphi_{\delta}^{-1} \to \varphi^{-1}$ locally uniformly on $\U$,
and $\varphi_{\delta}(x_j^{\delta}) \to \varphi(x_j)$ for all $1\le j\le 2N$. 
Note that Carath\'{e}odory convergence allows wild behavior of the boundaries around the marked points. 
In order to ensure precompactness of the interfaces in Theorem~\ref{thm::FKIsing_Loewner}, we need a convergence of polygons stronger than the above Carath\'{e}odory convergence. 
The following notion was introduced by Karrila, see in particular~\cite[Theorem~4.2]{Karrila:Limits_of_conformal_images_and_conformal_images_of_limits_for_planar_random_curves}. (See also~\cite{Karrila:Multiple_SLE_local_to_global} and~\cite{Chelkak-Wan:On_the_convergence_of_massive_loop-erased_random_walks_to_massive_SLE2_curves}.) 

\begin{definition} \label{def:closeCara}
We say that a sequence of discrete polygons $(\Omega^{\delta}; x_1^{\delta}, \ldots, x_{2N}^{\delta})$  
converges as $\delta \to 0$ to a polygon $(\Omega; x_1, \ldots, x_{2N})$ in the \emph{close-Carath\'{e}odory sense} if it converges in the Carath\'{e}odory sense and in addition, for all $1\le j\le 2N$, we have $x_j^{\delta}\to x_j$ as $\delta\to 0$ and the following is fulfilled. 
Given a reference point $u\in\Omega$ and 
$r>0$ small enough, let $S_r$ be the arc of $\partial B(x_j,r)\cap\Omega$ disconnecting \textnormal{(}in $\Omega$\textnormal{)} $x_j$ from $u$ and from all other arcs of this set. We require that, for each $r$ small enough and for all sufficiently small $\delta$ \textnormal{(}depending on $r$\textnormal{)}, the boundary point $x_j^{\delta}$ is connected to the midpoint of $S_r$ inside $\Omega^{\delta}\cap B(x_j,r)$. 
\end{definition}

In this setup, the FK-Ising interfaces, and more generally, the random-cluster interfaces for any parameter $q\in [1,4)$, always have a convergent subsequence in the curve space with metric~\eqref{eq::curve_metric}. 

\begin{lemma}\label{lem::FKIsing_tightness}
Assume the same setup as in 
Conjecture~\ref{conj::rcm_Loewner}.
Fix $i\in\{1,2, \ldots, 2N\}$. The family of laws of $\{\eta_i^{\delta}\}_{\delta>0}$ is  
precompact in the space of curves with metric~\eqref{eq::curve_metric}.
Furthermore, any subsequential limit $\eta_i$ does not hit any other point in $\{x_1, x_2, \ldots, x_{2N}\}$ than its two endpoints, almost surely.
\end{lemma}

\begin{proof}
The proof is standard nowadays. 
For instance, the case where $q=2$ is treated in~\cite[Lemmas~4.1 and~5.4]{Izyurov:On_multiple_SLE_for_the_FK_Ising_model}. 
The main tools are the so-called RSW bounds from~\cite{DCHN:Connection_probabilities_and_RSW_type_bounds, Kemppainen-Smirnov:Random_curves_scaling_limits_and_Loewner_evolutions} 
--- see also~\cite{Karrila:Limits_of_conformal_images_and_conformal_images_of_limits_for_planar_random_curves, Karrila:Multiple_SLE_local_to_global}. 
The case of general $q\in [1,4)$ follows from~\cite[Theorem~6]{DCST:Continuity_of_phase_transition_for_planar_random-cluster_and_Potts_models} and~\cite[Section~1.4]{DCMT:Planar_random-cluster_model_fractal_properties_of_the_critical_phase}. 
\end{proof}

In the rest of this section, we fix $q=2$ and thus focus on the critical FK-Ising model. 
\subsection{Exploration process and holomorphic spinor observable}
\label{subsec::holo_observable}
Fix $N\geq 1$
and a boundary condition $\beta\in \LP_N$ for the FK-Ising model as in~\eqref{eqn::linkpatterns_ordering}. By planarity, the pair of $1 = a_1$ in $\beta$ is some even index $2\ell = b_1$, that is, we have
$\beta=\{\{1,2\ell\}, \{a_2,b_2\},\ldots,\{a_{N},b_N\}\}$ with
\begin{align} \label{eq::pair_of_1}
\{1,2\ell\} \in \beta \qquad \textnormal{for some} \qquad \ell=\ell(\beta) \in \{1, 2, \ldots, N\} .
\end{align}

Consider a configuration $\omega$ of the critical FK-Ising model on the primal polygon $(\Omega^\delta;x_1^\delta,\ldots,x_{2N}^\delta)$ with b.c.~$\beta$. 
Its loop representation contains 
$N$ interfaces $\smash{\eta_{2r-1}^\delta}$ starting from $\smash{y_{2r-1}^{\delta,\diamond}}$, with $1 \leq r \leq N$, terminating among the medial vertices $\smash{\{y_{2r}^{\delta,\diamond} \colon 1 \leq r \leq N\}}$. 
Inspired by~\cite{LPW:UST_in_topological_polygons_partition_functions_for_SLE8_and_correlations_in_logCFT} (see also \cite[Figure~2]{Izyurov:Smirnovs_observable_for_free_boundary_conditions_interfaces_and_crossing_probabilities}), 
we define an exploration path $\smash{\xi_{\beta}^{\delta}}$ starting from the outer corner $\smash{y_{1}^{\delta,\diamond}}$ and terminating at the outer corner $\smash{y_{2\ell}^{\delta,\diamond}}$ via the following procedure (see Figure~\ref{fig::6points_curve}).  
The idea is that $\smash{\xi_{\beta}^{\delta}}$ traces a loop in the meander formed by the b.c.~$\beta$ and the random internal connectivity $\conn^{\delta}$ of the interfaces in the loop representation of $\omega$.

\begin{definition} \label{def: exploration path}
The following rules uniquely determine 
$\xi_{\beta}^{\delta}$, called the \emph{exploration path} associated to the configuration $\omega$ with b.c.~$\beta$. 
\begin{enumerate}
\item \label{item::explore-a}
$\xi_{\beta}^{\delta}$ starts from $y_{1}^{\delta,\diamond}$ and follows $\eta_1^\delta$ until it reaches some point in $ \{y_{2r}^{\delta,\diamond} \colon 1 \leq r \leq N\}$.
	
\item \label{item::explore-b}
When $\xi_{\beta}^{\delta}$ arrives at some point in $\{y_{2r}^{\delta,\diamond} \colon 1 \leq r \leq N\}$, it follows the contour given by $\beta$ outside of $\Omega^\delta$ until it reaches some point in $\{y_{2r-1}^{\delta,\diamond} \colon 1 \leq r \leq N\}$. 

\item \label{item::explore-c}
When $\xi_{\beta}^{\delta}$ arrives at some point in $\{y_{2r-1}^{\delta,\diamond} \colon 1 \leq r \leq N\}$, it follows the corresponding interface until it reaches some point in $\{y_{2r}^{\delta,\diamond} \colon 1 \leq r \leq N\}$.
	
\item \label{item::explore-d}
After repeating the steps~\ref{item::explore-b}--\ref{item::explore-c} 
sufficiently many times, $\xi_{\beta}^{\delta}$ arrives at $y_{2\ell}^{\delta,\diamond}$ and it then stops.
\end{enumerate}
\end{definition}

The path $\smash{\xi_{\beta}^{\delta}}$ also gives information about the connectivity of the interfaces, see~\eqref{eqn::const_jump} in Lemma~\ref{lem::discrete_H}. 
Note, however, that if the meander associated to $\beta$ and $\conn^{\delta}$ has more than one loop, then the exploration path $\smash{\xi_{\beta}^{\delta}}$ does not fully reveal $\conn^{\delta}$, and further exploration would be needed.

\begin{figure}[ht!]
\begin{subfigure}[b]{0.3\textwidth}
\begin{center}
\includegraphics[width=0.47\textwidth]{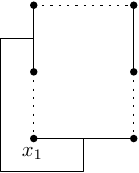}
\end{center}
\caption{A boundary condition $\beta$.}
\end{subfigure}
\begin{subfigure}[b]{0.675\textwidth}
\begin{center}
\includegraphics[width=0.21\textwidth]{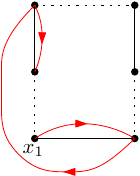}$\quad$
\includegraphics[width=0.21\textwidth]{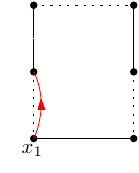}$\quad$
\includegraphics[width=0.21\textwidth]{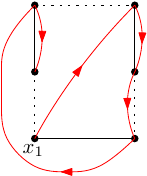}$\quad$
\includegraphics[width=0.21\textwidth]{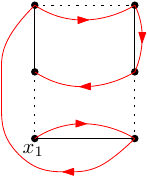}
\end{center}
\caption{There are four possibilities for the exploration path from $x_1$ to $x_6$.}
\end{subfigure}
\caption{\label{fig::6points_curve} 
Consider discrete polygons with six marked points on the boundary. 
One possible boundary condition $\beta=\{\{1,6\},\{2,5\},\{3,4\}\}$ is depicted in (a). 
The corresponding exploration path from $x_1$ to $x_6$ is depicted in (b).
Note that the second possibility in (b) does not fully reveal the internal connectivity pattern of the interfaces. 
}
\end{figure}

\smallbreak

Recall that for each medial edge, we have defined its orientation. 
For each medial edge $e^\diamond$, we also associate a direction $\nu(e^\diamond)$ as follows: 
we view the oriented edge $e^\diamond$ as a complex number and define 
\begin{align*}
\nu(e^{\diamond}) := \Big(\frac{e^\diamond}{|e^\diamond|}\Big)^{-1/2} .
\end{align*}
Note that $\nu(e^\diamond)$ is defined up to sign, which we will specify when necessary.

\begin{definition} \label{def::observables}
For the critical FK-Ising model on the primal polygon $(\Omega^\delta;x_1^\delta,\ldots,x_{2N}^\delta)$ with b.c.~$\beta$, we define the following discrete observables, inspired by~\textnormal{\cite[Section~4]{Kemppainen-Smirnov:Conformal_invariance_of_boundary_touching_loops_of_FK_Ising_model}.} 
\textnormal{(}We use the notation~\eqref{eq::pair_of_1}.\textnormal{)}
\begin{itemize}[leftmargin=2em]
	\item We define the \emph{edge observable} on edges and outer corner edges $e$ of $\smash{\Omega^{\delta,\diamond}}$ as 
\begin{align*}
F^{\delta}_{\beta}(e) :=  
\nu(e_{2\ell}^{\delta,\diamond}) 
\; \E^{\delta}_{\beta} \Big[ \one \{e\in\xi_{\beta}^{\delta}\} 
\exp \big(-\tfrac{\ii}{2} W_{\xi_{\beta}^{\delta}} \big( e_{2\ell}^{\delta,\diamond} ,e \big) \big) \Big] ,
\end{align*}
where 
\begin{itemize}
\item 
$\smash{\xi_{\beta}^{\delta}}$ is the exploration path from Definition~\ref{def: exploration path}; 

\item 
$\smash{e_{2\ell}^{\delta,\diamond}}$ is the oriented outer corner edge connecting to $\smash{y_{2\ell}^{\delta,\diamond}}$ 
 \textnormal{(}oriented to have $\smash{y_{2\ell}^{\delta,\diamond}}$ as its end vertex\textnormal{)};
 
\item 
$\smash{W_{\xi_{\beta}^\delta}(e_{2\ell}^{\delta,\diamond}, e)} \in \R$ is the winding number from $\smash{y_{2\ell}^{\delta,\diamond}}$ to $e$ 
along the reversal of $\smash{\xi_{\beta}^{\delta}}$; and 

\item the value of $\smash{\nu(e_{2\ell}^{\delta,\diamond})}$ will be specified in Proposition~\ref{prop::observable_cvg} and its proof.
\end{itemize}

Note that $\smash{F^{\delta}_{\beta}}$ is only defined up to sign \textnormal{(}hence, it is a so-called ``spinor'' observable\textnormal{)}.

	\item 
	We define the \emph{vertex observable} on interior vertices $z^{\diamond}$ of $\Omega^{\delta,\diamond}$ as
	\begin{align*}
F^{\delta}_{\beta}(z^{\diamond}) := \frac{1}{2}\sum_{e^{\diamond} \sim z^{\diamond}} F^{\delta}_{\beta}(e^{\diamond}) ,	
	\end{align*}
where the sum is over the four medial edges $e^{\diamond} \sim z^{\diamond}$ having $z^{\diamond}$ as an endpoint. 

	\item 
	We define the \emph{vertex observable} on vertices $z^{\diamond} \in \partial\Omega^{\delta,\diamond} \setminus \{x_1^{\delta,\diamond},x_2^{\delta,\diamond},\ldots,x_{2N}^{\delta,\diamond}\}$ as follows. 
Suppose that $z^{\diamond} \in (x_{i}^{\delta,\diamond} \, x_{i+1}^{\delta,\diamond})$ and let $e_-^{\diamond}, e_+^{\diamond} \in (x_{i}^{\delta,\diamond} \, x_{i+1}^{\delta,\diamond})$ be the oriented medial edges having $z^{\diamond}$ as their end vertex and beginning vertex, respectively. Set 
\begin{align} \label{eqn::boundary_verte_obser}
F_{\beta}^{\delta}(z^{\diamond}) :=
\begin{cases}
\sqrt{2} \exp(-\ii\frac{\pi}{4})F_{\beta}^{\delta}(e_+^{\diamond}) + \sqrt{2} \exp(\ii\frac{\pi}{4})F_{\beta}^{\delta}(e_-^{\diamond}), & \textnormal{if $i$ is odd} , \\[.5em]
\sqrt{2} \exp(-\ii\frac{\pi}{4})F_{\beta}^{\delta}(e_-^{\diamond}) + \sqrt{2} \exp(\ii\frac{\pi}{4})F_{\beta}^{\delta}(e_+^{\diamond}), & \textnormal{if $i$ is even} .
\end{cases}
\end{align} 
\end{itemize}
\end{definition}

A key result of this section is the convergence of the observable $\smash{F_{\beta}^{\delta}}$ as $\delta \to 0$ (Propositions~\ref{prop::observable_cvg} and~\ref{prop::holo_limiting}, which are slight generalizations of~\cite[Theorem~2.6]{Izyurov:Smirnovs_observable_for_free_boundary_conditions_interfaces_and_crossing_probabilities},
{see also~\cite[Theorem~4.3]{Chelkak-Smirnov:Universality_in_2D_Ising_and_conformal_invariance_of_fermionic_observables}}). 
We later relate the limit of $\smash{F_{\beta}^{\delta}}$ to the partition function $\LF_\beta$ in Proposition~\ref{prop::totalpartition_observable} in Section~\ref{subsec::holo_limiting} (which generalizes~\cite[Proposition~3.5]{Izyurov:On_multiple_SLE_for_the_FK_Ising_model}, cf.~\cite{CHI:Conformal_invariance_of_spin_correlations_in_planar_Ising_model}). 
Note that, as a function on $\Omega$, the scaling limit $\phi_{\beta}$ of $\smash{F_{\beta}^{\delta}}$ is a priori only determined up to a sign, while it is a holomorphic function on a double-cover $\Sigma_{x_1,\ldots,x_{2N}}$ of $(\Omega; x_1,\ldots,x_{2N})$. 
Usually, we shall not be concerned with the choice of branch (i.e., sign) for this ``spinor'' observable $\phi_{\beta}$.

\begin{proposition} \label{prop::observable_cvg}
Fix a polygon $(\Omega; x_1, \ldots, x_{2N})$. 
If a sequence $(\Omega^{\delta, \diamond}; x_1^{\delta, \diamond}, \ldots, x_{2N}^{\delta, \diamond})$ 
of medial polygons converges to $(\Omega; x_1, \ldots, x_{2N})$ in the Carath\'{e}odory sense, 
then the scaled vertex observables converge as 
\begin{align*}
2^{-1/4}\delta^{-1/2} F^{\delta}_{\beta} (\cdot) 
\qquad \overset{\delta\to 0}{\longrightarrow} \qquad
\phi_{\beta}(\cdot \, ; \Omega; x_1, \ldots, x_{2N}) 
\qquad \textnormal{locally uniformly},
\end{align*}
where both sides are determined up to a common sign, 
$\phi_{\beta}$ is a holomorphic function on the Riemann surface $\Sigma_{x_1,\ldots,x_{2N}}$ as detailed in Proposition~\ref{prop::holo_limiting} 
and Remark~\ref{rem::holo_limiting_polygon},  
and where the vertex observable $F_\beta^\delta$ is extended continuously to the planar domain corresponding to $\Omega^{\delta,\diamond}$ via linear interpolation.
\end{proposition}

For later use, we define a function (sometimes called ``spinor'' in the literature, e.g.,~\cite{Chelkak-Smirnov:Universality_in_2D_Ising_and_conformal_invariance_of_fermionic_observables, CHI:Conformal_invariance_of_spin_correlations_in_planar_Ising_model})
\begin{align} \label{eq::sqrt_branches}
z \quad \longmapsto \quad \prod_{j=1}^{2N} \frac{1}{\sqrt{z-x_j}} =: S_{x_1, \ldots, x_{2N}}(z) 
= S_{\bs{x}}(z)  ,
\end{align}
that is holomorphic and single-valued on a 
Riemann surface $\Sigma_{\bs{x}} = \Sigma_{x_1, \ldots, x_{2N}}$ 
which is a two-sheeted branched covering of the Riemann sphere $\hat{\C} = \C \cup \{\infty\}$ ramified at the points $x_1, \ldots, x_{2N}$. 
To determine the value of $S_{\bs{x}}(z) = S_{x_1, \ldots, x_{2N}}(z)$ at $z \in \hat{\C} \setminus \{x_1, \ldots, x_{2N}\}$ one has to choose a branch for it. 
We consider $S_{\bs{x}}$ as a holomorphic function on $\Sigma_{\bs{x}}$ formed by gluing two copies of the Riemann sphere together along $N$ fixed branch cuts 
that are simple non-crossing paths on the complement of $\Omega$ joining pairs of the points $x_1, \ldots, x_{2N}$ 
(for example, we could pick the branch cuts according to $\beta$). 
Locally around each ramification point $x_i$, we may consider the square root $z \mapsto \sqrt{z-x_i}$ as a holomorphic and single-valued function on the local chart of $\Sigma_{\bs{x}}$ at $x_i$ 
(with the two sheets locally identified with those of $\Sigma_{\bs{x}}$ so that $\sqrt{z-x_i}$ and $S_{\bs{x}}$ have the same sign). 
The properties~(\ref{eqn::holo_restriction1},~\ref{eqn::holo_restriction2}) stated in Proposition~\ref{prop::holo_limiting} are thus well-defined.

\begin{proposition} \label{prop::holo_limiting}
Let $\Omega = \HH$ and fix $\bs{x}=(x_1, \ldots, x_{2N})\in\chamber_{2N}$. 
There exists a unique polynomial $P_{\beta}$ of degree at most $N-1$ and  with real coefficients such that 
the holomorphic function 
\begin{align} \label{eqn::holo_limiting}
\phi_{\beta}(z) := \frac{\ii \, P_{\beta}(z)}{\prod_{j=1}^{2N} \sqrt{z-x_j}}
= \ii \, P_{\beta}(z) \, S_{\bs{x}}(z)
\end{align}
on the Riemann surface $\Sigma_{\bs{x}}$ 
satisfies the following $N$ properties: 
\begin{align}
\lim_{ z\to x_1} \sqrt{\pi} \, \sqrt{z-x_1} \, \phi_{\beta}(z) = \; & 1,\label{eqn::holo_restriction1} \\
\lim_{z\to x_{a_r}} \sqrt{z-x_{a_r}} \, \sqrt{z-x_{b_r}} \, \phi_{\beta}(z) = \; & - \lim_{z\to x_{b_r}} \sqrt{z-x_{a_r}} \, \sqrt{z-x_{b_r}} \, \phi_{\beta}(z), \qquad\textnormal{for all } \; 
r \in \{2,3,\ldots,N\} \label{eqn::holo_restriction2}
\end{align}
\end{proposition}

We first prove Proposition~\ref{prop::holo_limiting} in Section~\ref{subsec::holo_limiting} and using it, we prove Proposition~\ref{prop::observable_cvg} in Section~\ref{subsec::observable_cvg}.

\begin{remark}
The 		special case 		$\beta=\unnested$ of  Proposition~\eqref{prop::holo_limiting} was proved 
		in~\textnormal{\cite[Lemma~2.4]{Izyurov:Smirnovs_observable_for_free_boundary_conditions_interfaces_and_crossing_probabilities}} using complex analysis techniques, which fail to work for general boundary conditions $\beta\in \LP_N$. 
		One can, in fact,  
		use the computation in~\textnormal{\cite[Appendix~A]{CHI:Conformal_invariance_of_spin_correlations_in_planar_Ising_model}} to prove uniqueness and existence in Proposition~\ref{prop::holo_limiting} and to show Proposition~\ref{prop::totalpartition_observable} in Section~\ref{subsec::holo_limiting}, as Izyurov did in~\textnormal{\cite[Proof of Proposition~3.5]{Izyurov:On_multiple_SLE_for_the_FK_Ising_model}}. 
		We give an alternative computation in Section~\ref{subsec::holo_limiting}, which could be applied\footnote{To achieve this, one has to consider the ratio $Q_{\beta}(\bs{\hat{\sigma}}_1)/Q_{\beta}(\bs{\hat{\sigma}}_2)$ for $\bs{\hat{\sigma}}_1, \bs{\hat{\sigma}}_2\in \{\pm1\}^{N-1}$ when following the analysis in the proof of Lemma~\ref{lem::phase_factor}.}  in turn to bulk spin correlations in~\textnormal{\cite[Theorem~1.2]{CHI:Conformal_invariance_of_spin_correlations_in_planar_Ising_model}}.
\end{remark}

\begin{remark} \label{rem::Riemann_surface}
From the definition~\eqref{eq::sqrt_branches} of $S_{\bs{x}}$, we see that
the function $z \mapsto \phi_{\beta}(z)$ in Proposition~\ref{prop::holo_limiting} is holomorphic and single-valued on the Riemann surface $\Sigma_{\bs{x}} = \Sigma_{x_1, \ldots, x_{2N}}$. 
Note that up to a choice of sign (that is, sheet of $\Sigma_{\bs{x}}$, or branch for $\phi_{\beta}$), 
$z \mapsto \phi_{\beta}(z)$ gives a holomorphic function on the upper half-plane $\HH$. 
Moreover, $\phi_{\beta}(z)$ is purely real
when $z\in (x_{2r-1}, x_{2r})$, and purely imaginary when $z\in (x_{2r},x_{2r+1})$. 
\end{remark}

\begin{remark} \label{rem::holo_limiting_polygon}
Because $\phi_{\beta}$ depends on $\bs{x}\in\chamber_{2N}$, we also write $\phi_{\beta}(z)=\phi_{\beta}(z; \HH; \bs{x})=\phi_{\beta}(z; \bs{x})$ when necessary. 
The proof of Proposition~\ref{prop::observable_cvg} 
(in Section~\ref{subsec::observable_cvg})
implies that, for all M\"obius maps $\varphi$ of $\HH$ such that $\varphi(x_1)<\cdots<\varphi(x_{2N})$, we have
\begin{align} \label{eqn::covari_phi_beta}
\big( \phi_{\beta}(z;\HH;x_1,\ldots,x_{2N}) \big)^2 
= { \varphi'(z) } \, \big( \phi_{\beta} ( \varphi(z);\HH;\varphi(x_1),\ldots,\varphi(x_{2N}) ) \big)^2.
\end{align}
Hence, we can define $\phi_{\beta}$ for general polygons  
$(\Omega; x_1, \ldots, x_{2N})$ via its conformal covariance rule\footnote{If needed, we could use some fixed branch of the square root, which is well-defined because $\varphi' \neq 0$, 
by picking it in a simply connected neighborhood of some reference point and extending to all of $\Omega$ by analytic continuation.}: 
\begin{align*}
\phi_{\beta}(z; \Omega; x_1, \ldots, x_{2N}) := 
\sqrt{\varphi'(z)} 
\; \phi_{\beta}(\varphi(z);\HH; \varphi(x_1), \ldots, \varphi(x_{2N})) , \qquad z \in \Omega ,
\end{align*}
where $\varphi$ is any conformal map from $\Omega$ onto $\HH$ such that $\varphi(x_1)<\cdots<\varphi(x_{2N})$. 
Note that~\eqref{eqn::covari_phi_beta} ensures that $\phi_\beta$ for general domains is independent of the choice of the conformal map $\varphi$ up to a sign. 
\end{remark}

\smallbreak

Let us make some further remarks for small values of $N$.
\begin{itemize}%
[leftmargin=*]
	\item 
When $N=1$, the function in Proposition~\ref{prop::holo_limiting} is 
\begin{align} \label{eqn::holo_limiting_N=1}
\phi_{\vcenter{\hbox{\includegraphics[scale=0.2]{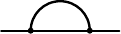}}}}(z; x_1, x_2) 
= \frac{\ii}{\sqrt{\pi}} \, \frac{\sqrt{x_2-x_1}}{\sqrt{z-x_1} \, \sqrt{z-x_2}} ,
\end{align} 
and for a polygon $(\Omega; x_1, x_2)$ with two marked points, we have (up to a sign)
\begin{align*}
\phi_{\vcenter{\hbox{\includegraphics[scale=0.2]{figures/link-0.pdf}}}}(z; \Omega; x_1, x_2) := 
{\sqrt{\varphi'(z)}}
\, \phi_{\vcenter{\hbox{\includegraphics[scale=0.2]{figures/link-0.pdf}}}}(\varphi(z);\HH; \varphi(x_1), \varphi(x_{2})) ,
\end{align*}
where $\varphi \colon \Omega \to \HH$ is {any} conformal map  such that $\varphi(x_1)<\varphi(x_2)$. 
In this case, Smirnov proved 
Proposition~\ref{prop::observable_cvg} in~\cite[Theorem~2.2]{Smirnov:Conformal_invariance_in_random_cluster_models1}: 
\begin{align*}
2^{-1/4} \, \delta^{-1/2} \, F^{\delta}_{\vcenter{\hbox{\includegraphics[scale=0.2]{figures/link-0.pdf}}}} (\cdot) 
\qquad \overset{\delta\to 0}{\longrightarrow} \qquad 
\phi_{\vcenter{\hbox{\includegraphics[scale=0.2]{figures/link-0.pdf}}}}(\cdot \, ; \Omega; x_1, x_2) 
\qquad\textnormal{locally uniformly}. 
\end{align*}

\item 
When $N=2$, we may verify Proposition~\ref{prop::holo_limiting} by a direct computation. 
In this case, there are two possible boundary conditions, 
$\vcenter{\hbox{\includegraphics[scale=0.2]{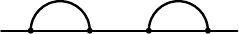}}}=\{\{1,2\}, \{3,4\}\}$ and $\vcenter{\hbox{\includegraphics[scale=0.2]{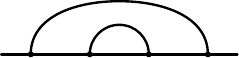}}}=\{\{1,4\}, \{2,3\}\}$, and
\begin{align}
\phi_{\vcenter{\hbox{\includegraphics[scale=0.2]{figures/link-1.pdf}}}} \; & (z; x_1, x_2, x_3, x_4) 
\label{eqn::obse_N=2_unne} \\
= \; & \frac{\ii}{\sqrt{\pi}} \,\frac{\Big(\sqrt{\frac{(x_3-x_1)(x_4-x_1)}{(x_2-x_1)}}-\sqrt{\frac{(x_3-x_2)(x_4-x_2)}{(x_2-x_1)}}\Big)(z-x_1)-\sqrt{(x_2-x_1)(x_3-x_1)(x_4-x_1)}}{\sqrt{z-x_1} \, \sqrt{z-x_2} \, \sqrt{z-x_3} \, \sqrt{z-x_4}} ,
\nonumber \\ 
\phi_{\vcenter{\hbox{\includegraphics[scale=0.2]{figures/link-2.pdf}}}} \; & (z; x_1, x_2, x_3, x_4)
\nonumber  \\
= \; & \frac{\ii}{\sqrt{\pi}} \, 
\frac{\Big(\sqrt{\frac{(x_4-x_2)(x_4-x_3)}{(x_4-x_1)}}+\sqrt{\frac{(x_2-x_1)(x_3-x_1)}{(x_4-x_1)}}\Big)(z-x_1)-\sqrt{(x_2-x_1)(x_3-x_1)(x_4-x_1)}}{\sqrt{z-x_1} \, \sqrt{z-x_2} \, \sqrt{z-x_3} \, \sqrt{z-x_4}} . 
\nonumber
\end{align}

\item For general $N$ and $\beta \in \LP_N$,
one can derive an explicit expression for $\phi_\beta$ using Cramer's rule.
\end{itemize}

\subsection{Proof of Proposition~\ref{prop::holo_limiting} and emergence of $\LF_\beta$}
\label{subsec::holo_limiting}
Our first goal is to show Proposition~\ref{prop::holo_limiting} via two auxiliary Lemmas~\ref{lem::R_beta_invertible} and~\ref{lem::phase_factor} (the latter in Appendix~\ref{appendix_aux}).
To this end, we first set some notation.
For $2\le r \le N$, we define row vectors $\smash{\bs{U}_{\beta}^{\pm}(r)}$ of size $N-1$ as
\begin{align*}
\bs{U}_{\beta}^{\pm}(r) 
:= \big( U_{\beta}^{\pm}(r,1), \, U_{\beta}^{\pm}(r,2), \, \ldots, \, U_{\beta}^{\pm}(r,N-1) \big) ,
\end{align*}
where for $2\le r \le N$ and $0 \le s \le N-1$, we denote 
\begin{align} \label{eqn::U_elements}
U_{\beta}^{+}(r,s) 
:= (x_{a_{r}} - x_1)^s \, \ddot{S}^{a_{r}, b_{r}}_{x_1, \ldots, x_{2N}}(x_{a_{r}})
\qquad \textnormal{and} \qquad 
U_{\beta}^{-}(r,s) 
:= (x_{b_{r}}-x_1)^s \, \ddot{S}^{a_{r}, b_{r}}_{x_1, \ldots, x_{2N}}(x_{b_{r}}) ,
\end{align}
and where the function 
\begin{align*}
z \; & \quad \longmapsto \quad \sqrt{z-x_{a_r}} \, \sqrt{z-x_{b_r}} \, S_{\bs{x}}(z) =: \underset{j\notin \{a_{r}, b_{r}\}}{\prod} \frac{1}{\sqrt{z - x_j}} 
=: \ddot{S}^{a_{r}, b_{r}}_{x_1, \ldots, x_{2N}}(z)
\end{align*}
is holomorphic and single-valued on a 
Riemann surface $\Sigma_{\bs{x}} = \Sigma_{x_1, \ldots, x_{2N}}$ as in Remark~\ref{rem::Riemann_surface}.
We also define an $(N-1) \times (N-1)$-matrix
\begin{align} \label{eqn::R_beta_def}
R_{\beta} :=
\begin{pmatrix} 
\bs{U}_{\beta}^{+}(2) \, + \, \bs{U}_{\beta}^{-}(2) \\
\cdot\\
\cdot\\
\cdot\\
\bs{U}_{\beta}^{+}(N) \, + \, \bs{U}_{\beta}^{-}(N)
\end{pmatrix},
\end{align}
that is, we define $R_{\beta}(r,s):=U_{\beta}^{+}(r+1,s)+U_{\beta}^{-}(r+1,s)$ for $1\leq r \leq N-1$ and $1\leq s\leq N-1$. 
Note that writing $\bs{\hat{\sigma}} = (\hat{\sigma}_2, \ldots, \hat{\sigma}_N)\in\{\pm 1\}^{N-1}$, 
and identifying $\pm 1$ with the superscript $\pm$,
we have
\begin{align} \label{eqn::determi_decomposition}
\det(R_{\beta}) = \sum_{\bs{\hat{\sigma}} \in \{\pm 1\}^{N-1}}Q_{\beta} (\bs{\hat{\sigma}}) ,
\qquad\textnormal{where}\qquad
Q_{\beta} (\bs{\hat{\sigma}}) 
:= \det
\begin{pmatrix} \bs{U}_{\beta}^{\hat{\sigma}_2}(2) \\
\cdot\\
\cdot\\
\cdot\\
\bs{U}_{\beta}^{\hat{\sigma}_{N}}(N)
\end{pmatrix}. 
\end{align}

\begin{proof}[Proof of Proposition~\ref{prop::holo_limiting}]
We write the polynomial $P_{\beta}$ as 
\begin{align*}
P_{\beta}(z) = \coeff_0 + \coeff_1(z-x_1) + \cdots + \coeff_{N-1}(z-x_1)^{N-1}, 
\end{align*}
where $\coeff_0, \coeff_1, \ldots, \coeff_{N-1} \in \R$ are some real coefficients. 
Note that $\coeff_0=P_{\beta}(x_1)$ and $\coeff_1=P_{\beta}'(x_1)$. 
Defining an $(N-1)$-component vector 
$\bs{V}_{\beta} = \big( V_{\beta}(1), \, V_{\beta}(2) , \, \ldots, \, V_{\beta}(N-1) \big)$ 
with entries
\begin{align}\label{eqn::r_beta_def}
V_{\beta}(r) 
:= R_{\beta}(r,0) 
,  \qquad 1\le r \le N-1 ,
\end{align}
we note that the restrictions~\eqref{eqn::holo_restriction1} and~\eqref{eqn::holo_restriction2} read  
\begin{align}
\label{eqn::holo_restriction1_aux} 
\sqrt{\pi} \, \ii \, \coeff_0 \, S_{x_2, x_3, \ldots, x_{2N}}(x_1) = & \; 1 , \\
\label{eqn::holo_restriction2_aux}
\sum_{n=1}^{N-1} \frac{\coeff_n}{\coeff_0} \, R_{\beta}(r,n) 
= &  \; - V_{\beta}(r), \qquad 1\le r \le N-1 ,
\end{align}
where 
$S_{x_2, x_3, \ldots, x_{2N}}(z) = \underset{j \neq 1}{\prod} \frac{1}{\sqrt{z - x_j}} := \sqrt{z-x_{1}} \, S_{\bs{x}}$
is holomorphic and single-valued on $\Sigma_{\bs{x}}$ as in Remark~\ref{rem::Riemann_surface}.
Proposition~\ref{prop::holo_limiting} now follows by showing that the matrix $R_{\beta}$ in~\eqref{eqn::R_beta_def} is invertible (Lemma~\ref{lem::R_beta_invertible}). 
\end{proof}

\begin{lemma} \label{lem::R_beta_invertible}
The matrix $R_{\beta}$ defined by~\eqref{eqn::R_beta_def} is invertible. 
\end{lemma}

\begin{proof}
We need to show that $\det(R_{\beta})$
in~\eqref{eqn::determi_decomposition} is non-zero. 
Write
\begin{align} \label{eq::def_y_k}
y_r^{+, \beta} := x_{a_r} 
\qquad \textnormal{and} \qquad 
y_r^{-, \beta} := x_{b_r} , \qquad 2\le r \le N .
\end{align}
Using the Vandermonde determinant, we have 
\begin{align} \label{eqn::Vander}
\begin{split}
Q_{\beta} (\bs{\hat{\sigma}})
= & \; Q_{\beta}(\hat{\sigma}_2, \ldots, \hat{\sigma}_N) \\ 
= & \; \prod_{2\leq r \leq N} 
\big( y_r^{\hat{\sigma}_r, \beta} - x_1 \big) \; 
\prod_{2\leq s < t\leq N} \big( y_t^{\hat{\sigma}_t, \beta} - y_s^{\hat{\sigma}_s, \beta} \big) \; 
\prod_{2\leq r \leq N} 
 \ddot{S}^{a_{r}, b_{r}}_{x_1, \ldots, x_{2N}} \big( y_r^{\hat{\sigma}_r, \beta} \big) .
\end{split}
\end{align}
From Lemma~\ref{lem::phase_factor}, 
we find a constant $\theta_{\beta}\in \{ \pm 1, \pm \ii \}$ depending only on $\beta$ such that 
\begin{align} \label{eqn::Q_beta_positive}
\frac{Q_{\beta}(\bs{\hat{\sigma}})}{\theta_{\beta}} > 0 .
\end{align}
Combining~\eqref{eqn::determi_decomposition} with~\eqref{eqn::Q_beta_positive}, we obtain
$\frac{\det R_{\beta}}{\theta_{\beta}} > 0$, 
which implies that $R_{\beta}$ is invertible.
\end{proof}

The second goal of this section is to derive the expansion of $\phi_{\beta}$ as $z\to x_1$ (Lemma~\ref{lem::holo_expansion}) and to relate its expansion coefficients to the partition function $\LF_{\beta}$ defined in~\eqref{eqn::totalpartition_def} (Proposition~\ref{prop::totalpartition_observable}). 

\begin{lemma} \label{lem::holo_expansion}
Write $\bs{x}=(x_1, \ldots, x_{2N})\in\chamber_{2N}$. The holomorphic function~\eqref{eqn::holo_limiting} on $\Sigma_{\bs{x}}$
satisfies
\begin{align*}
\phi_{\beta}(z; \bs{x}) 
= \frac{1}{\sqrt{\pi} \, \sqrt{z-x_1}} 
+ \LK_{\beta}(\bs{x}) \, \sqrt{z-x_1} 
+ o(\sqrt{z-x_1}) , \qquad\textnormal{as } z\to x_1 ,
\end{align*}
where 
\begin{align}\label{eqn::holo_limiting_expansion_coefficient}
\LK_{\beta}(\bs{x}) 
= \LK_{\beta}(x_1, \ldots, x_{2N}) 
= \frac{1}{\sqrt{\pi}} \bigg( \frac{P_{\beta}'(x_1)}{P_{\beta}(x_1)}+\frac{1}{2}\sum_{k=2}^{2N}\frac{1}{x_k-x_1} \bigg). 
\end{align}
\end{lemma}

\begin{proof}
From the expression in~\eqref{eqn::holo_limiting}, we may write 
\begin{align*}
\phi_{\beta}(z; \bs{x}) = 
\frac{\LJ_{\beta}(\bs{x})}{\sqrt{z-x_1}}
+ \LK_{\beta}(\bs{x}) \, \sqrt{z-x_1} 
+ o(\sqrt{z-x_1}) , \qquad\textnormal{as } z\to x_1 ,
\end{align*}
where 
\begin{align*}
\LJ_{\beta}(x_1, \ldots, x_{2N}) 
= & \; \ii \, P_{\beta}(x_1) \, S_{x_2, x_3, \ldots, x_{2N}}(x_1) , \\
\LK_{\beta}(x_1, \ldots, x_{2N})
= & \; \ii \, P_{\beta}(x_1) \, S_{x_2, x_3, \ldots, x_{2N}}(x_1)
\times \bigg( \frac{P_{\beta}'(x_1)}{P_{\beta}(x_1)}+\frac{1}{2}\sum_{k=2}^{2N}\frac{1}{x_k-x_1} \bigg) .
\end{align*}
From~\eqref{eqn::holo_restriction1_aux} with $\coeff_0=P_{\beta}(x_1)$, we see that $\LJ_{\beta}=1/\sqrt{\pi}$ and~\eqref{eqn::holo_limiting_expansion_coefficient} holds. This completes the proof. 
\end{proof}

\begin{proposition}\label{prop::totalpartition_observable}
Write $\bs{x}=(x_1, \ldots, x_{2N})\in\chamber_{2N}$.
We have 
\begin{align}\label{eqn::totalpartition_observable}
\partial_1\log\LF_{\beta}(\bs{x}) = \frac{\sqrt{\pi}}{4} \, \LK_{\beta}(\bs{x}), 
\end{align}
where $\LF_{\beta}$ is defined in~\eqref{eqn::totalpartition_def} and $\LK_{\beta}$ is defined in~\eqref{eqn::holo_limiting_expansion_coefficient}. 
\end{proposition}
\begin{proof}
	On the one hand, let us compute $\partial_1\log \LF_{\beta}(\bs{x})$. For 
	the cross-ratios,  
(recalling~\eqref{eq::pair_of_1})
	we have
	\begin{align*}
		\partial_1 \chi(x_1,x_{a_r},x_{b_r},x_{2\ell}) 
		= -\chi(x_1,x_{a_r},x_{b_r},x_{2\ell}) \; \frac{x_{b_r}-x_{a_r}}{(x_{a_r}-x_1)(x_{b_r}-x_1)} , \qquad 2\leq r \leq N .
	\end{align*}
Thus, writing
$\bs{\sigma}=(\sigma_1, \ldots, \sigma_N)\in\{\pm 1\}^N$ 
and $\bs{\hat{\sigma}}=(\hat{\sigma}_2, \ldots, \hat{\sigma}_N)\in\{\pm 1\}^{N-1}$, 
where the variables $\hat{\sigma}_r$ in $\bs{\hat{\sigma}}$ could be viewed as products $\sigma_1\sigma_r$ of the variables $\sigma_1$ and $\sigma_r$ in $\bs{\sigma}$ for $2\leq r\leq N$,
we obtain 
\begin{align} \label{eqn::totalpatition_aux1}
\; & 8 \, \partial_1 \log \LF_{\beta}(\bs{x})  \, - \,  \frac{1}{x_{2\ell}-x_1} \\
\nonumber 
= \; &
\frac{ \underset{\bs{\sigma}\in\{\pm 1\}^{N}}{\sum} 
{\Big(-\sigma_1 \underset{r=2}{\overset{N}{\sum}}
\, \sigma_r \frac{x_{b_r}-x_{a_r}}{(x_{a_r}-x_1)(x_{b_r}-x_1)}\Big) }
\Big(\underset{s=2}{\overset{N}{\prod}}
\, \chi(x_1,x_{a_s},x_{b_s},x_{2\ell})^{\frac{\sigma_1\sigma_s}{4}}\Big)
\Big(\underset{2\leq r< s \leq N}{\prod} 
\, \chi(x_{a_r},x_{a_s},x_{b_s},x_{b_r})^{\frac{\sigma_r \sigma_s}{4}}\Big)}{ \underset{\bs{\sigma}\in\{\pm 1\}^{N}}{\sum} 
\Big( \underset{s=2}{\overset{N}{\prod}}
\, \chi(x_1,x_{a_s},x_{b_s},x_{2\ell})^{\frac{\sigma_1\sigma_s}{4}}\Big) \Big( \underset{2\leq r< s \leq N}{\prod} 
\, \chi(x_{a_r},x_{a_s},x_{b_s},x_{b_r})^{\frac{\sigma_r \sigma_s}{4}}\Big)} \\
\nonumber 
= & \; 
\frac{ \underset{\bs{\hat{\sigma}}\in\{\pm 1\}^{N-1}}{\sum} 
{ \Big(-\underset{r=2}{\overset{N}{\sum}}
\, \hat{\sigma}_r \frac{x_{b_r}-x_{a_r}}{(x_{a_r}-x_1)(x_{b_r}-x_1)}\Big) }
\Big(\underset{s=2}{\overset{N}{\prod}}
\, \chi(x_1,x_{a_s},x_{b_s},x_{2\ell})^{\frac{\hat{\sigma}_s}{4}}\Big)
\Big(\underset{2\leq r< s \leq N}{\prod}
\, \chi(x_{a_r},x_{a_s},x_{b_s},x_{b_r})^{\frac{\hat{\sigma}_r \hat{\sigma}_s}{4}}\Big)}{ \underset{\bs{\hat{\sigma}}\in\{\pm 1\}^{N-1}}{\sum} 
\Big( \underset{s=2}{\overset{N}{\prod}}
\, \chi(x_1,x_{a_s},x_{b_s},x_{2\ell})^{\frac{\hat{\sigma}_s}{4}}\Big)
\Big(\underset{2\leq r< s \leq N}{\prod}
\, \chi(x_{a_r},x_{a_s},x_{b_s},x_{b_r})^{\frac{\hat{\sigma}_r \hat{\sigma}_s}{4}}\Big)} \\
\nonumber 
= & \; 
\frac{ \underset{\bs{\hat{\sigma}}\in\{\pm 1\}^{N-1}}{\sum} 
{ \Big(-\underset{r=2}{\overset{N}{\sum}}
\, \hat{\sigma}_r\frac{x_{b_r}-x_{a_r}}{(x_{a_r}-x_1)(x_{b_r}-x_1)}\Big) }
\Big(\underset{s=2}{\overset{N}{\prod}}
\, \chi(x_1,x_{a_s},x_{b_s},x_{2\ell})^{\frac{\hat{\sigma}_s+1}{4}}\Big)
\Big(\underset{2\leq r< s \leq N}{\prod}
\, \chi(x_{a_r},x_{a_s},x_{b_s},x_{b_r})^{\frac{\hat{\sigma}_r \hat{\sigma}_s + 1}{4}}\Big)}{\underset{\bs{\hat{\sigma}}\in\{\pm 1\}^{N-1}}{\sum}
\Big( \underset{s=2}{\overset{N}{\prod}}
\, \chi(x_1,x_{a_s},x_{b_s},x_{2\ell})^{\frac{\hat{\sigma}_s+1}{4}}\Big)
\Big(\underset{2\leq r< s \leq N}{\prod}
\, \chi(x_{a_r},x_{a_s},x_{b_s},x_{b_r})^{\frac{\hat{\sigma}_r \hat{\sigma}_s + 1}{4}}\Big)} , 
\end{align}

On the other hand, let us compute $\LK_{\beta}(\bs{x})$. We denote by $R^{\bullet}_{\beta}$ the $(N-1)\times(N-1)$-matrix obtained by replacing the first column of $R_{\beta}$ by the column vector
{$\smash{\bs{V}_{\beta} = \big( V_{\beta}(1), \, V_{\beta}(2) , \, \ldots, \, V_{\beta}(N-1) \big)^t}$}
defined in~\eqref{eqn::r_beta_def}. 
Then, combining~\eqref{eqn::holo_restriction2_aux} with Cramer's rule, we find that 
\begin{align} \label{eqn::Cramer}
	\frac{P_{\beta}'(x_1)}{P_{\beta}(x_1)}=-\frac{\det (R_{\beta}^{\bullet})}{\det (R_{\beta})}.
\end{align}
Using Lemma~\ref{lem::Cramer_decom} (from Appendix~\ref{appendix_aux}) we can find functions 
$g^{\bs{\hat{\sigma}}, \beta}(\bs{x})>0$ 
for $\bs{\hat{\sigma}}=(\hat{\sigma}_2,\ldots,\hat{\sigma}_{N})\in \{\pm 1\}^{N-1}$ such that
\begin{align}  \label{eqn::Cramer_decom}
	\frac{\det (R_{\beta}^{\bullet})}{\det (R_{\beta})}
	= \frac{\underset{\bs{\hat{\sigma}}\in\{\pm 1\}^{N-1}}{\sum} \, g^{\bs{\hat{\sigma}}, \beta}(\bs{x}) \; 
	\underset{r=2}{\overset{N}{\sum}} \; \big( y_{r}^{\hat{\sigma}_r, \beta} - x_1 \big)^{-1} }{\underset{\bs{\hat{\sigma}}\in\{\pm 1\}^{N-1}}{\sum} \; g^{\bs{\hat{\sigma}}, \beta}(\bs{x})} ,
\end{align}
where $y_{r}^{\hat{\sigma}_r, \beta}$ are defined in~\eqref{eq::def_y_k}. 
Lemma~\ref{lem::connection_chi} (from Appendix~\ref{appendix_aux}) implies that there exist functions $f_{\beta}(\bs{x})>0$ such that, for all $\bs{\hat{\sigma}}=(\hat{\sigma}_2, \ldots, \hat{\sigma}_N)\in\{\pm 1\}^{N-1}$, we have 
\begin{align} \label{eqn::connection_chi}
	g^{\bs{\hat{\sigma}}, \beta}(\bs{x}) 
	= f_{\beta}(\bs{x}) \; \prod_{2\leq r \leq N}\chi(x_1,x_{a_r},x_{b_r},x_{2\ell})^{\frac{\hat{\sigma}_r+1}{4}} 
	\prod_{2\leq s < t \leq N}\chi(x_{a_s},x_{a_t},x_{b_t},x_{b_s})^{\frac{\hat{\sigma}_s \hat{\sigma}_t+1}{4}} . 
\end{align}
Plugging all of~(\ref{eqn::Cramer},~\ref{eqn::Cramer_decom},~\ref{eqn::connection_chi})  
into~\eqref{eqn::holo_limiting_expansion_coefficient}, 
and recalling~\eqref{eq::pair_of_1}, 
we obtain 
\begin{align*} 
2\sqrt{\pi} \, \LK_{\beta}(\bs{x}) 
= \; & 2\frac{P_{\beta}'(\bs{x})}{P_{\beta}(\bs{x})} 
\, + \, \sum_{k=2}^{2N} \frac{1}{x_{k}-x_1} \\
= \; & \frac{1}{x_{2\ell}-x_1} \, + \, 2\frac{P_{\beta}'(\bs{x})}{P_{\beta}(\bs{x})}
\, + \, \sum_{k=2}^{N} \Big( \frac{1}{x_{a_k}-x_1}+\frac{1}{x_{b_k}-x_1} \Big)
= \; \frac{1}{x_{2\ell}-x_1} \, + \,  \textnormal{\eqref{eqn::totalpatition_aux1}},
\end{align*}
where we also used the identity 
\begin{align*}
	-2\sum_{r=2}^N \frac{1}{y_r^{\hat{\sigma}_r,\beta}-x_1}+\sum_{r=2}^N \left(\frac{1}{x_{a_r}-x_1}+\frac{1}{x_{b_r}-x_1}\right)=\sum_{r=2}^N \hat{\sigma}_r\left(\frac{1}{x_{b_r}-x_1}-\frac{1}{x_{a_r}-x_1}\right)
\end{align*}
for all $\bs{\hat{\sigma}}=(\hat{\sigma}_2,\ldots,\hat{\sigma}_{N})\in \{\pm1\}^{N-1}$.
This gives the asserted identity~\eqref{eqn::totalpartition_observable}.
\end{proof}

We fill in the details to finish the proof of Proposition~\ref{prop::totalpartition_observable} (Lemmas~\ref{lem::Cramer_decom} and~\ref{lem::connection_chi}) in Appendix~\ref{appendix_aux}.

\subsection{Scaling limit of the observable:  Proof of Proposition~\ref{prop::observable_cvg}}
\label{subsec::observable_cvg}
Some key ideas in the proof of Proposition~\ref{prop::observable_cvg} are learned from~\cite{Chelkak-Smirnov:Universality_in_2D_Ising_and_conformal_invariance_of_fermionic_observables, Izyurov:Smirnovs_observable_for_free_boundary_conditions_interfaces_and_crossing_probabilities} --- we adjust them to deal with the FK-Ising model in polygons in our setup. 
We first fix some terminology on discrete complex analysis --- see~\cite{Chelkak-Smirnov:Discrete_complex_analysis_on_isoradial_graphs} for more details on discrete harmonicity, holomorphicity, and s-holomorphicity.

\begin{itemize}[leftmargin=*]
\item We say that a function $u \colon \Z^2 \to \C$ is (discrete) \emph{harmonic} (resp.~sub/superharmonic) 
at a vertex $z \in \Z^2$ if 
$\Delta u(z) := \sum_{w \sim z} (u(w)-u(z))=0$  
(resp.~$\Delta u(z) \geq 0$, $\Delta u(z) \leq 0$), where the sum is taken over all neighbors of $z$.
We say that a function $u$ is harmonic (resp.~sub/superharmonic) 
on a subgraph of $\Z^2$ if $u$ is harmonic (resp.~sub/superharmonic) at all vertices of this subgraph.

\item We say that a function $\phi \colon \Z^2 \cup (\Z^2)^{\bullet}\to \C$ is (discrete) \emph{holomorphic} around a medial vertex $z^\diamond$ if the (discrete) Cauchy-Riemann equation at $z^\diamond$ holds: $\phi(n)-\phi(s)=\ii (\phi(e)-\phi(w))$, where $n, w, s, e$ are the vertices incident to $z^\diamond$ in counterclockwise order (two of them are primal vertices while the other two are dual vertices).

\item We say that a function $f \colon (\Z^2)^\diamond\to \C$ is \emph{spin-holomorphic (s-holomorphic)} around a medial edge $e^\diamond$ if 
\begin{align*}
\mathrm{Proj}_{\nu(e^\diamond) \, \R}[f(z_-^\diamond)] = \mathrm{Proj}_{\nu(e^\diamond) \, \R}[f(z_+^\diamond)] , 
\end{align*}
where $z_-^\diamond$ and $z_+^\diamond$ are endpoints of the medial edge $e^\diamond$, 
and $\mathrm{Proj}_{L}$ is the orthogonal projection onto the line $L$ on the complex plane. 
Note that, if $f$ is s-holomorphic around all medial edges of $\Omega^{\delta,\diamond}$ that are not adjacent to the marked medial vertices, 
then it is holomorphic around all interior vertices of $\Omega^{\delta}$ and around all interior dual vertices of $\Omega^{\delta,\bullet}$ (see, e.g.,~\cite[Remark~3.3]{Smirnov:Conformal_invariance_in_random_cluster_models1}). 
\end{itemize}

The next lemma shows that the observable $F_{\beta}^{\delta}$ has Riemann type boundary behavior.

\begin{lemma} \label{lem::discrete_obser}
The observable $F_{\beta}^{\delta}$ has the following properties.
\begin{enumerate}
\item \label{item::edge_obser_direc}
If $e^{\diamond}$ is a medial edge connecting two vertices on $\partial\Omega^{\delta,\diamond}\setminus \{x_1^{\delta,\diamond},x_2^{\delta,\diamond},\ldots,x_{2N}^{\delta,\diamond}\}$, then
$F_{\beta}^{\delta}(e^{\diamond})\parallel \nu(e^{\diamond})$.  

\item \label{item::vertex_obser_direc_odd}
If $x^{\diamond} \in \partial\Omega^{\delta,\diamond}$ is a medial vertex lying on some primal edge in  
$\smash{\bigcup_{r=1}^{N}} 
(x_{2r-1}^{\delta} \, x_{2r}^{\delta})$, 
then $F_{\beta}^{\delta}(x^{\diamond})\parallel \frac{1}{\sqrt{e(x^{\diamond})}}$,
where $e(x^{\diamond})$ is the primal edge having $x^{\diamond}$ as its midpoint, oriented to have the primal polygon on its left, 
and the branch choice of the square root is arbitrary. 

\item \label{item::vertex_obser_direc_even}
If $x^{\diamond} \in \partial\Omega^{\delta,\diamond}$ is a medial vertex lying on some dual edge in 
$\smash{\bigcup_{r=1}^{N}} 
(x_{2r}^{\delta,\bullet} \, x_{2r+1}^{\delta,\bullet})$, 
then $F_{\beta}^{\delta}(x^{\diamond}) \parallel \frac{\ii}{\sqrt{e(x^{\diamond})}}$,
where $e(x^{\diamond})$ is the dual edge having $x^{\diamond}$ as its midpoint, oriented to have the dual polygon on its left,  and the branch choice of the square root is arbitrary. 
\end{enumerate}
\end{lemma}

\begin{proof}
The same argument as in~\cite[Lemma~4.1]{Smirnov:Conformal_invariance_in_random_cluster_models1} proves Item~\ref{item::edge_obser_direc}. 
Covering both Items~\ref{item::vertex_obser_direc_odd} and~\ref{item::vertex_obser_direc_even}, suppose that $x^{\diamond} \in (x_{i}^{\delta,\diamond} \, x_{i+1}^{\delta,\diamond})$. 
Let $e_-^{\diamond}, e_+^{\diamond} \in  (x_{i}^{\delta,\diamond} \, x_{i+1}^{\delta,\diamond})$ be the oriented medial edges having $x^{\diamond}$ as end vertex and beginning vertex, respectively. 
It follows from Definition~\ref{def: exploration path} (recalling also~\eqref{eq::pair_of_1}) 
of the exploration path $\xi^{\delta}$ that it passes through $e_-^{\diamond}$ if and only if it passes through $e_+^{\diamond}$. 
Moreover, when $\xi^{\delta}$ passes through $e_-^{\diamond}$, the winding is
\begin{align*}
W_{\xi^{\delta}} \big( e_{2\ell}^{\delta,\diamond},e_+^{\diamond} \big) 
= 
\begin{cases}
W_{\xi^{\delta}} \big( e_{2\ell}^{\delta,\diamond},e_-^{\diamond} \big)  + \frac{\pi}{2} ,  & \textnormal{if $i$ is odd;}\\
W_{\xi^{\delta}} \big( e_{2\ell}^{\delta,\diamond},e_-^{\diamond} \big) - \frac{\pi}{2} ,  & \textnormal{if $i$ is even.}
\end{cases}
\end{align*}
Consequently, we have
\begin{align}\label{eqn::obser_one_step_change}
F_{\beta}^{\delta}(e_+^{\diamond}) =
\begin{cases}
F_{\beta}^{\delta}(e_-^{\diamond})\exp(-\ii\frac{\pi}{4}), & \textnormal{if $i$ is odd},\\
F_{\beta}^{\delta}(e_-^{\diamond})\exp(\ii\frac{\pi}{4}), & \textnormal{if $i$ is even}.
\end{cases}
\end{align}
Thus, by~\eqref{eqn::boundary_verte_obser} and~\eqref{eqn::obser_one_step_change}, we have
\begin{align} \label{eqn::obser_tangent}
\begin{cases}
F_{\beta}^{\delta}(x^{\diamond}) \parallel F_{\beta}^{\delta}(e_-^{\diamond})\exp(-\ii\frac{\pi}{8}), & \textnormal{if $i$ is odd},\\
F_{\beta}^{\delta}(x^{\diamond}) \parallel F_{\beta}^{\delta}(e_-^{\diamond})\exp(\ii\frac{\pi}{8}), &\textnormal{if $i$ is even}.
\end{cases}
\end{align}
Items~\ref{item::vertex_obser_direc_odd} and~\ref{item::vertex_obser_direc_even} 
now follow from~\eqref{eqn::obser_tangent} and Item~\ref{item::edge_obser_direc}.
\end{proof}

The key property of the observable $F_{\beta}^{\delta}$ is its discrete holomorphicity.

\begin{lemma}\label{lem::s_holomo}
If $z^{\diamond}$ and $w^{\diamond}$ are either two interior vertices of $\Omega^{\delta,\diamond}$, or
two boundary vertices such that $z^{\diamond} , w^{\diamond} \in \smash{\bigcup_{r=1}^{N}} (x_{2r}^{\delta,\diamond} \, x_{2r+1}^{\delta,\diamond})\setminus\{x_1^{\delta,\diamond},x_2^{\delta,\diamond},\ldots,x_{2N}^{\delta,\diamond}\}$, 
and $e^{\diamond}$ is the medial edge connecting them, then $F_{\beta}^{\delta}$ is s-holomorphic around $e^{\diamond}$, that is, 
\begin{align} \label{eqn::s_holomo}
\mathrm{Proj}_{\nu(e^{\diamond}) \, \R}[F_{\beta}^{\delta}(z^{\diamond})]=\mathrm{Proj}_{\nu(e^{\diamond}) \, \R}[F_{\beta}^{\delta}(w^{\diamond})]=F_{\beta}^{\delta}(e^{\diamond}).
\end{align}
In particular, the vertex observable $F_{\beta}^{\delta}$ is holomorphic around all interior vertices of $\Omega^{\delta}$ and around all interior dual vertices of $\Omega^{\delta,\bullet}$.
\end{lemma}

\begin{proof}
If $z^{\diamond} , w^{\diamond} \in \smash{\bigcup_{r=1}^{N}}  (x_{2r}^{\delta,\diamond} \, x_{2r+1}^{\delta,\diamond})\setminus\{x_1^{\delta,\diamond},x_2^{\delta,\diamond},\ldots,x_{2N}^{\delta,\diamond}\}$,
then~\eqref{eqn::s_holomo} follows immediately from the definition~\eqref{eqn::boundary_verte_obser} of $F_{\beta}^{\delta}$ 
together with Item~\ref{item::edge_obser_direc} of
Lemma~\ref{lem::discrete_obser} and the observation~\eqref{eqn::obser_one_step_change}.
For two interior medial vertices~\eqref{eqn::s_holomo} follows from~\cite[Lemma~4.5]{Smirnov:Conformal_invariance_in_random_cluster_models1}. 
The discrete holomorphicity of the vertex observable $F_{\beta}^{\delta}$ can be deduced from its s-holomorphicity (see, e.g.,~\cite[Remark~3.3]{Smirnov:Conformal_invariance_in_random_cluster_models1}).  
\end{proof}

From Lemma~\ref{lem::s_holomo} and~\cite[Lemma~3.6]{Smirnov:Conformal_invariance_in_random_cluster_models1}, 
we see that there exists a unique function 
(the imaginary part of the discrete ``primitive'' of $(F_{\beta}^{\delta})^2$) 
\begin{align*} 
H^{\delta}_{\beta} \colon \Omega^\delta\cup \Omega^{\delta,\bullet}\to \R
\qquad \textnormal{such that} \qquad 
\begin{cases} 
H^{\delta}_{\beta}(x_{1}^{\delta}) = 0 , \\
H^{\delta}_{\beta}(w^{\bullet}) - H^{\delta}_{\beta}(z) 
= \big| \mathrm{Proj}_{\nu(e^{\diamond}) \, \R} [F_{\beta}^{\delta}(e^{\diamond}) ] \big|^2 ,
\end{cases} 
\end{align*}
for each medial edge $e^{\diamond}$ bordered by a primal vertex $z \in \Omega^\delta$ and a dual vertex $w^{\bullet} \in \Omega^{\delta,\bullet}$. 
Let $\smash{H^{\delta,\bullet}_{\beta}}$ and $\smash{H^{\delta,\circ}_{\beta}}$ be the restrictions of $H_{\beta}^{\delta}$ on $\Omega^{\delta,\bullet}$ and $\Omega^{\delta}$, respectively.
Note that, if $z,w \in \Omega^{\delta}$ are two neighboring primal vertices, 
then we have (see, e.g.,~\cite[Remark~3.7]{Smirnov:Conformal_invariance_in_random_cluster_models1})
\begin{align} \label{eqn::discrete_integr}
H^{\delta,\circ}_{\beta}(z) - H^{\delta,\circ}_{\beta}(w)
= \Im \bigg( \frac{\big(F_{\beta}^{\delta}(\frac{z+w}{2})\big)^2}{\sqrt{2}\delta} (z - w) \bigg) . 
\end{align}
Notably, the function $H^{\delta}_{\beta}$ has Dirichlet type boundary conditions that are more directly related to the exploration path --- see Eq.~\eqref{eqn::const_jump} in the next lemma.

\begin{lemma} \label{lem::discrete_H}
There exist constants $(\MainConst_1^{\delta},\ldots,\MainConst_{2N}^{\delta}) \in \R^{2N}$
with $\MainConst_{1}^{\delta}=0$ such that the following  
hold.
\begin{enumerate}
\item \label{item::boundary_value_H} 
The function $\smash{H^{\delta,\bullet}_{\beta}}$ is subharmonic on the interior vertices of $\Omega^{\delta,\bullet}$. The function $\smash{H^{\delta,\circ}_{\beta}}$ is superharmonic on the interior vertices of $\Omega^{\delta}$. For each $r \in  \{1,2,\ldots,N\}$, we have the boundary values
\begin{align*}
\begin{cases}
H^{\delta,\bullet}_{\beta} = \MainConst_{2r}^{\delta} & \textnormal{ on }(x_{2r}^{\delta,\bullet} \, x_{2r+1}^{\delta,\bullet}) , \\
H^{\delta,\circ}_{\beta} = \MainConst_{2r-1}^{\delta} & \textnormal{ on }(x_{2r-1}^{\delta} \, x_{2r}^{\delta}).
\end{cases}
\end{align*}

\item \label{item::boundary midification} 
For each $r \in  \{1,2,\ldots,N\}$, set 
$H^{\delta,\bullet}_{\beta} := \MainConst_{2r-1}^{\delta}$ on dual vertices in $(\delta\Z^2)^{\bullet} \setminus \Omega^{\delta,\bullet}$ adjacent to $(x_{2r-1}^{\delta,\bullet} \, x_{2r}^{\delta,\bullet})$ 
and $H^{\delta,\circ}_{\beta} := \MainConst_{2r}^{\delta}$ on primal vertices in $\delta\Z^2 \setminus \Omega^{\delta}$ adjacent to $(x_{2r}^{\delta} \, x_{2r+1}^{\delta})$. 
Then, the function $\smash{H^{\delta,\bullet}_{\beta}}$ is also subharmonic at all $z^{\bullet} \in \smash{\bigcup_{r=1}^{N}} (x_{2r-1}^{\delta,\bullet} \, x_{2r}^{\delta,\bullet})$ with Laplacian modified on the boundary: 
\begin{align*}
\Delta H^{\delta,\bullet}_{\beta}(z^{\bullet}) := \sum_{w^{\bullet}\sim z^{\bullet}} d(z^{\bullet},w^{\bullet}) \, (H^{\delta,\bullet}_{\beta}(w^{\bullet})-H^{\delta,\bullet}_{\beta}(z^{\bullet})) \geq 0 ,
\end{align*}
 where $d(z^{\bullet},w^{\bullet}):=1$ if $w^{\bullet} \in \Omega^{\delta,\bullet}$ and $d(z^{\bullet},w^{\bullet}):=2\tan\frac{\pi}{8}=2(\sqrt{2}-1)$ if $w^{\bullet}\notin\Omega^{\delta,\bullet}$.

Besides, $\smash{H^{\delta,\circ}_{\beta}}$ is superharmonic at all $z \in \smash{\bigcup_{r=1}^{N}} (x_{2r}^{\delta} \, x_{2r+1}^{\delta})$ with Laplacian modified on the boundary: 
\begin{align*}
\Delta H^{\delta,\circ}_{\beta}(z) 
:= \sum_{w \sim z} d(z,w) \, (H^{\delta,\circ}_{\beta}(w) - H^{\delta,\circ}_{\beta}(z)) \leq 0 , 
\end{align*}
where $d(z,w) := 1$ if $w \in \Omega^{\delta}$ and $d(z,w) = 2(\sqrt{2}-1)$ if $w \notin\Omega^{\delta}$.      

\item \label{item::const_comparison} For each $r \in \{1,2,\ldots,N\}$, we have 
$\MainConst_{2r}^{\delta}\geq \MainConst_{2r-1}^{\delta}$ and $\MainConst_{2r}^{\delta}\geq \MainConst_{2r+1}^{\delta}$.

\item \label{item::const_jump} 
For each $r \in \{1,2,\ldots,N\}$, we have
\begin{align}\label{eqn::const_jump}
|\MainConst_{a_r-1}^{\delta}-\MainConst_{a_r}^{\delta}| 
= \; & |\MainConst_{b_r-1}^{\delta}-\MainConst_{b_r}^{\delta}| \\
= \; & \Big( \PP_{\beta}^{\delta} \big[\textnormal{$\xi^{\delta}$ passes through the outer corners $y_{a_r}^{\delta,\diamond}$ and $y_{b_r}^{\delta,\diamond}$} \big]\Big)^2.\notag
\end{align} 
In particular, we have $|\MainConst_{1}^{\delta}-\MainConst_{2N}^{\delta}|=1$. As a consequence, the family 
$\{ \MainConst_1^{\delta},\ldots,\MainConst_{2N}^{\delta} \}_{\delta>0}$
of constants is uniformly bounded.
\end{enumerate}
\end{lemma}

\begin{proof}
The subharmonicity of $\smash{H^{\delta,\bullet}_{\beta}}$ and superharmonicity of $\smash{H^{\delta,\circ}_{\beta}}$ on interior vertices both follow from 
\cite[Lemma~3.8]{Smirnov:Conformal_invariance_in_random_cluster_models1}. 
 By construction, $\smash{H^{\delta,\bullet}_{\beta}}$ is constant on $(x_{2r}^{\delta,\bullet} \, x_{2r+1}^{\delta,\bullet})$ 
and $\smash{H^{\delta,\circ}_{\beta}}$ is constant on $(x_{2r-1}^{\delta} \, x_{2r}^{\delta})$. This gives Item~\ref{item::boundary_value_H}. 
Item~\ref{item::boundary midification} follows from~\cite[Lemma~3.14]{Chelkak-Smirnov:Universality_in_2D_Ising_and_conformal_invariance_of_fermionic_observables}.
Item~\ref{item::const_comparison} and relation~\eqref{eqn::const_jump} hold by construction. 
The identity $|\MainConst_{1}^{\delta}-\MainConst_{2N}^{\delta}|=1$ follows from~\eqref{eqn::const_jump} since $\xi^{\delta}$ goes through $y_{1}^{\delta,\diamond}$ with probability one. 
Lastly, as $\MainConst_{1}^{\delta}=0$, we find from~\eqref{eqn::const_jump} that
$|\MainConst_{k}^{\delta}|\leq 2N-1$, for all $\delta>0$ and $1 \leq k \leq 2N$. 
\end{proof}

We see from Lemma~\ref{lem::discrete_H} that the collection $\{ \MainConst_1^{\delta},\ldots,\MainConst_{2N}^{\delta} \}_{\delta>0}$ of constants has convergent subsequences. For the convergence of the observable, we also need the following key lemma.

\begin{lemma} \label{lem::sub_limit_obser}
Assume the same setup as in Proposition~\ref{prop::observable_cvg}. 
We extend $H^{\delta}_{\beta}$ to continuous functions on the planar domains corresponding to $\Omega^{\delta,\diamond}$ via linear interpolation. 
Then, the sequence 
\begin{align*}
\big\{ (2^{-1/4}\delta^{-1/2} \, F^{\delta}_{\beta}, \, H^{\delta}_{\beta}) \big\}_{\delta>0}
\end{align*}
has 
\textnormal{(}locally uniformly\textnormal{)} convergent subsequences. 
Moreover, any subsequential limit $(F_{\beta},H_{\beta})$, 
with also $(\MainConst_1^\delta,\MainConst_2^\delta,\ldots,\MainConst_{2N}^\delta)$ converging to some $(\MainConst_1,\MainConst_2,\ldots,\MainConst_{2N})\in\R^{2N}$, 
satisfies the following properties. 
\begin{enumerate}
\item \label{item::conti_integral}
The function $F_{\beta}$ is holomorphic on $\Omega$, and $H_{\beta}(w)=\Im\int^w F_{\beta}(z)^2 \, \ud z$ on $\Omega \ni w$. 

\item \label{item::bounded_harmo_h}
The function $H_{\beta}$ is bounded and harmonic on $\Omega$. 

\item \label{item::const_boundar_h}
We have $H_{\beta}(z)\to \MainConst_k$ as $z\to (x_k \, x_{k+1})$ {in $\HH$}, for all $k \in \{1,2,\ldots,2N\}$. 

\item \label{item::const_comparison_limit}
The relations $\MainConst_{2r} \geq \MainConst_{2r-1}$ and $\MainConst_{2r} \geq \MainConst_{2r+1}$ hold for all $r \in  \{1,2,\ldots,N\}$.  

\item \label{item::const_jump_limit}
The relation $|\MainConst_{a_r-1}-\MainConst_{a_r}| = |\MainConst_{b_r-1}-\MainConst_{b_r}|$ holds for all $r \in \{1,2,\ldots,N\}$, and we have $|\MainConst_{1}-\MainConst_{2N}|=1$. 

\item \label{item::partial_n} 
The outer normal derivative $\partial_{\mathrm{n}} H_{\beta}$
of the function $H_{\beta}$ satisfies $\partial_{\mathrm{n}} H_{\beta}\geq 0$ on $\smash{\bigcup_{r=1}^{N}}  (x_{2r} \, x_{2r+1})$ and $\partial_{\mathrm{n}} H_{\beta}\leq 0$ on $\smash{\bigcup_{r=1}^{N}} (x_{2r-1} \, x_{2r})$ in the following sense: 
if $z\in (x_{2r} \, x_{2r+1})$ for some $r$, then
\begin{align*}
H_{\beta}^{-1}( {-\infty} , \MainConst_{2r}] \cap \{w \in \Omega \colon |w-z|<\epsilon \} \neq \emptyset , \qquad\textnormal{for all } \epsilon > 0,
\end{align*}
while if $z\in (x_{2r-1} \, x_{2r})$ for some $r$, then
\begin{align*}
H_{\beta}^{-1} [\MainConst_{2r-1},\infty) \cap \{w \in \Omega \colon |w-z|<\epsilon\} \neq \emptyset , \qquad \textnormal{for all } \epsilon > 0 .
\end{align*}
\end{enumerate}
\end{lemma}

\begin{proof}
The sequence $\{H_{\beta}^{\delta}\}_{\delta>0}$ is uniformly bounded by
Items~\ref{item::boundary_value_H}~\&~\ref{item::const_jump} of Lemma~\ref{lem::discrete_H}:
we have
\begin{align} \label{eq::H_uniformly_bounded}
|H_{\beta}^{\delta}| \leq M , \qquad \textnormal{for all } \, \delta > 0 ,
\end{align}
with some $M\in (0,\infty)$.
Thus, the sequence $\big\{ (2^{-1/4}\delta^{-1/2} \, F^{\delta}_{\beta}, \, H^{\delta}_{\beta}) \big\}_{\delta>0}$ has (locally uniformly) convergent subsequences by~\cite[Theorem~3.12]{Chelkak-Smirnov:Universality_in_2D_Ising_and_conformal_invariance_of_fermionic_observables}. 
Item~\ref{item::const_jump} of Lemma~\ref{lem::discrete_H} ensures that $\{(\MainConst_{1}^{\delta},\MainConst_{2}^{\delta},\ldots,\MainConst_{2N}^{\delta})\}_{\delta>0}$ has convergent subsequences.
Let $(F_{\beta},H_{\beta})$ be any subsequential limit 
along a sequence $\delta_n \to 0$ as $n \to \infty$
of $\big\{ (2^{-1/4}\delta^{-1/2} \, F^{\delta}_{\beta}, \, H^{\delta}_{\beta}) \big\}_{\delta>0}$ with $(\MainConst_{1}^{\delta_n},\ldots,\MainConst_{2N}^{\delta_n})$ also converging to some $(\MainConst_{1},\ldots,\MainConst_{2N})\in\R^{2N}$ 
(choosing a simultaneously convergent subsequence by refining the sequence if necessary).
Since $F_{\beta}^{\delta}$ is (discrete) holomorphic for each $\delta>0$ (Lemma~\ref{lem::s_holomo}) and the convergence is locally uniform, the limit $F_{\beta}$ is holomorphic due to Morera's theorem. By~\eqref{eqn::discrete_integr} and the locally uniform convergence, we obtain the relation $H_{\beta}(w)=\Im\int^w F_{\beta}(z)^2 \, \ud z$. 
Being the imaginary part of the holomorphic function $w \mapsto \int^{w} F_{\beta}(z)^2 \, \ud z$, the function $H_{\beta}(w)$ is harmonic on $\Omega$, and~\eqref{eq::H_uniformly_bounded} implies that $H_{\beta}$ is bounded on $\Omega$. 
This proves Items~\ref{item::conti_integral}~\&~\ref{item::bounded_harmo_h}.

Next, fix $r \in  \{1,2,\ldots,N\}$. We will prove that $H_{\beta}(z)\to \MainConst_{2r-1}$ as $z\to (x_{2r-1} \, x_{2r})$. Let $z \in \Omega$ be any point. On the one hand, 
let $\{z^{\delta_n}\}_{n \geq 1}$ be a sequence of interior primal vertices approximating $z$. Denote by $\hm(z^{\delta_n}; E; \Omega^{\delta_n})$ the discrete harmonic measure of $E \subset \partial\Omega^{\delta_n}$ viewed from $z^{\delta_n}$. 
Then, we have 
\begin{align*}
H_{\beta}(z) 
= \; & \lim_{n \to \infty} H_{\beta}^{\delta_n,\circ}(z^{\delta_n}) \\
\geq  \; & 
\limsup_{n \to \infty} \Big( \MainConst_{2r-1}^{\delta_n} \, \hm \big( z^{\delta_n};(x_{2r-1}^{\delta_n} \, x_{2r}^{\delta_n});\Omega^{\delta_n}\big) 
- M \, \hm \big( z^{\delta_n};(x_{2r}^{\delta_n} \, x_{2r-1}^{\delta_n});\Omega^{\delta_n}\big)\Big)\\
= \; &  \MainConst_{2r-1} \, \hm \big(z;(x_{2r-1} \, x_{2r});\Omega\big) 
- M \, \hm \big(z;(x_{2r} \, x_{2r-1});\Omega \big),
\end{align*}
where the inequality in the second line 
follows from the superharmonicity of $\smash{H_{\beta}^{\delta_n,\circ}}$ (Items~\ref{item::boundary_value_H}~\&~\ref{item::boundary midification} of Lemma~\ref{lem::discrete_H}) 
and the fact that $\smash{H_{\beta}^{\delta_n,\circ}}$ takes the constant value $\MainConst_{2r-1}^{\delta_n}$ along $(x_{2r-1}^{\delta_n} \, x_{2r}^{\delta_n})$ (Item~\ref{item::boundary_value_H} of Lemma~\ref{lem::discrete_H});
and the equality in the third line is due to the convergence of the discrete polygons in the Carath\'{e}odory sense and~\cite[Theorem~3.12]{Chelkak-Smirnov:Discrete_complex_analysis_on_isoradial_graphs}. 
Therefore, we have
\begin{align}\label{eqn::sublimit_aux_lower}
H_{\beta}(z) \geq \MainConst_{2r-1} - 2 M \, \hm \big(z;(x_{2r} \, x_{2r-1});\Omega \big) .
\end{align}
On the other hand, let $\{z^{\delta_n,\bullet}\}_{n \geq 1}$ be a sequence of interior dual vertices approximating $z$. 
Denote by $\hm(z^{\delta_n,\bullet};E;\Omega^{\delta_n,\bullet})$ the discrete harmonic measure of $E \subset \partial\Omega^{\delta_n,\bullet}$ viewed from $z^{\delta_n,\bullet}$. 
Then, we have
\begin{align*}
H_{\beta}(z) = \; & \lim_{n \to \infty} H_{\beta}^{\delta_n,\bullet}(z^{\delta_n,\bullet})\\
\leq \; & \liminf_{n \to \infty} \Big(\MainConst_{2r-1}^{\delta_n} \, \hm \big(z^{\delta_n,\bullet};(x_{2r-1}^{\delta_n,\bullet} \, x_{2r}^{\delta_n,\bullet});\Omega^{\delta_n,\bullet}\big) 
+ M \, \hm \big(z^{\delta_n,\bullet};(x_{2r}^{\delta_n,\bullet} \, x_{2r-1}^{\delta_n,\bullet});\Omega^{\delta_n,\bullet}\big) \Big) \\
= \; & \MainConst_{2r-1} \, \hm \big(z;(x_{2r-1} \, x_{2r};\Omega)\big) 
+ M \, \hm \big(z;(x_{2r} \, x_{2r-1});\Omega\big),
\end{align*}
where the inequality in the second line is due to the subharmonicity of $H_{\beta}^{\delta_n,\bullet}$  (Items~\ref{item::boundary_value_H}~\&~\ref{item::boundary midification} of Lemma~\ref{lem::discrete_H}) 
and the fact that $H_{\beta}^{\delta_n,\bullet}$ takes the constant value $\MainConst_{2r-1}^{\delta_n}$ along $(x_{2r-1}^{\delta_n,\bullet}x_{2r}^{\delta_n,\bullet})$ (Item~\ref{item::boundary midification} of Lemma~\ref{lem::discrete_H}). 
Therefore, we have
\begin{align}\label{eqn::sublimt_aux_upper}
H_{\beta}(z)\leq \MainConst_{2r-1} + 2M \, \hm \big(z;(x_{2r} \, x_{2r-1});\Omega\big).
\end{align}
Combining the bounds~(\ref{eqn::sublimit_aux_lower},~\ref{eqn::sublimt_aux_upper}), we obtain
$H_{\beta}(z) \to \MainConst_{2r-1}$ as $z\to (x_{2r-1} \, x_{2r})$. 
A similar argument shows that $H_{\beta}(z)\to \MainConst_{2r}$ as $z\to (x_{2r} \, x_{2r+1})$. This proves Item~\ref{item::const_boundar_h}.

Lastly, Items~\ref{item::const_comparison_limit}~\&~\ref{item::const_jump_limit} follow respectively from 
Items~\ref{item::const_comparison}~\&~\ref{item::const_jump} of Lemma~\ref{lem::discrete_H};
while
Item~\ref{item::partial_n} follows from~\cite[Remark~6.3]{Chelkak-Smirnov:Universality_in_2D_Ising_and_conformal_invariance_of_fermionic_observables} and 
Items~\ref{item::vertex_obser_direc_odd}~\&~\ref{item::vertex_obser_direc_even} of Lemma~\ref{lem::discrete_obser}.
This concludes the proof. 
\end{proof}

We are now ready to prove Proposition~\ref{prop::observable_cvg}.

\begin{proof}[Proof of Proposition~\ref{prop::observable_cvg}]
For definiteness, fix a sign for 
$\phi_{\beta}(\cdot \, ;\Omega;x_1,\ldots,x_{2N})$.
Lemmas~\ref{lem::discrete_H}~\&~\ref{lem::sub_limit_obser} ensure that the sequences $\{(\MainConst_{1}^{\delta},\ldots,\MainConst_{2N}^{\delta})\}_{\delta>0}$ of constants and $\big\{ (2^{-1/4}\delta^{-1/2} \, F^{\delta}_{\beta}, \, H^{\delta}_{\beta}) \big\}_{\delta>0}$ of pairs of functions have convergent subsequences. 
Let $(F_{\beta},H_{\beta})$ be any subsequential limit of the latter and 
$(\MainConst_{1},\ldots,\MainConst_{2N})\in\R^{2N}$ of the former. 
It suffices to show that $F_{\beta}(\cdot) = \phi_{\beta}(\cdot \, ;\Omega;x_{1},\ldots,x_{2N})$ 
(with appropriate choice of sign for $\nu(e_{2\ell}^{\delta,\diamond})$). 
We consider the situation in the upper half-plane. 
Fix a sign for the function 
$\phi_{\beta}(\cdot \, ;\HH;x_1,\ldots,x_{2N})$.
Let $\varphi$ be a conformal map from $\Omega$ onto $\HH$ such that $\varphi(x_{1})<\cdots<\varphi(x_{2N})$. We define
\begin{align*}
h_{\HH}(z) := H_{\beta}( \varphi^{-1}(z)) , \qquad  
f_{\HH}(z) := F_{\beta}(\varphi^{-1}(z)) \; \sqrt{(\varphi^{-1})'(z)} ,
 \qquad \textnormal{and} \qquad 
 \realpt_{i} := \varphi(x_i) , \qquad \textnormal{for } 1\leq i \leq 2N,
 \end{align*}
where we fix the branch of the square root so that 
$\sqrt{(\varphi)'(\cdot)} \, \phi_{\beta}(\varphi(\cdot) \, ; \HH;\realpt_1,\ldots,\realpt_{2N}) = \phi_{\beta}(\cdot \, ;\Omega;x_1,\ldots,x_{2N})$.

Items~\ref{item::bounded_harmo_h} \&~\ref{item::const_boundar_h} of Lemma~\ref{lem::sub_limit_obser} imply that $h_{\HH}$ can be extended to a bounded continuous function on $\overline{\HH}\setminus\{\realpt_{1},\realpt_{2},\ldots,\realpt_{2N}\}$ which is harmonic on $\HH$ with constant value $\MainConst_{i}$ on each $(\realpt_{i} \, \realpt_{i+1})$ for $i \in \{1,\ldots,2N\}$. 
Consequently, the function $h_{\HH}(z)$ is 
a (real) linear combination of $\mathrm{Hm}(z;(\realpt_{i} \, \realpt_{i+1});\HH)$,
the harmonic measures of $(\realpt_{i} \, \realpt_{i+1})$ viewed from $z\in\HH$ with $1\leq i \leq 2N$.

Item~\ref{item::conti_integral} of 
Lemma~\ref{lem::sub_limit_obser} gives the holomorphicity of $f_{\HH}$ on $\HH$ and the relation $h_{\HH}(w)=\Im\int^{w} f_{\HH}(z)^2\ud z$.
Consequently, there exists a polynomial $Q(z)$ of degree at most $2N-1$ with real coefficients such that  
\begin{align*}
f_{\HH}(z)^2=\frac{Q(z)}{\prod_{i=1}^{2N}(z-\realpt_{i})} .
\end{align*}

Item~\ref{item::partial_n} 
of Lemma~\ref{lem::sub_limit_obser} implies that the outer normal derivative\footnote{In this case, we also use $\partial_{\mathrm{n}} h_\HH$ to denote the ordinary outer normal derivative, since the boundary $\partial\HH=\R$ is smooth.}  
of the function $h_\HH$
satisfies $\partial_{\mathrm{n}} h_{\HH}\leq 0$ on $\smash{\bigcup_{r=1}^{N}}  (\realpt_{2r-1} \, \realpt_{2r})$ and $\partial_{\mathrm{n}}h_{\HH}\geq 0$ on $\smash{\bigcup_{r=1}^{N}} (\realpt_{2r} \, \realpt_{2r+1})$. Furthermore, for each $z\in \R\setminus\{\realpt_{1},\realpt_{2},\ldots,\realpt_{2N}\}$ we have $\partial_{\mathrm{n}}h_{\HH}(z)=-f_{\HH}(z)^2$, 
which implies that $Q(z)\leq 0$ whenever $z\in \R$. Since $f_{\HH}$ is holomorphic on $\HH$, the polynomial $Q(z)$ cannot have zeros of odd degree in $\HH$. Thus, we have $Q(z)=-P(z)^2$ for some polynomial $P(z)$ of degree at most $N-1$ with real coefficients. Since $|\MainConst_{1}-\MainConst_{2N}|=1$ 
(by Item~\ref{item::const_jump_limit} of Lemma~\ref{lem::sub_limit_obser}), 
by computing the residue of $f_{\HH}(z)^2$ at $\realpt_{1}$, we conclude that 
with appropriate choice of the sign of $\nu(e_{2\ell}^{\delta,\diamond})$ and hence the sign of $f_{\HH}$, we have
\begin{align}\label{eqn::jump_1}
\lim_{z\to \realpt_{1}} \sqrt{\pi} \, \sqrt{z-\realpt_1} \, f_{\HH}(z)=1.
\end{align}

For any $r \in \{2,\ldots,N\}$, since $\MainConst_{a_r-1}-\MainConst_{a_r} = -(\MainConst_{b_r-1}-\MainConst_{b_r})$ (Items~\ref{item::const_comparison_limit}~\&~\ref{item::const_jump_limit} of Lemma~\ref{lem::sub_limit_obser}), 
by computing the residues of $f_{\HH}(z)^2$ at $\realpt_{a_r}$ and $\realpt_{b_r}$, we conclude that for some sign $\varepsilon_r \in \{1,-1\}$, we have
\begin{align} \label{eqn::sign_change}
\lim_{z \to \realpt_{a_r}} \sqrt{z-\realpt_{a_r}} \, \sqrt{z-\realpt_{b_r}} \, f_{\HH}(z) 
= \varepsilon_r \lim_{z\to \realpt_{b_r}} \sqrt{z-\realpt_{a_r}} \, \sqrt{z-\realpt_{b_r}} \, f_{\HH}(z).
\end{align}
 Combining~(\ref{eqn::jump_1},~\ref{eqn::sign_change}) with Proposition~\ref{prop::holo_limiting}, it remains to show that $\varepsilon_r = -1$ for all $2 \leq r \leq N$.
Without loss of generality, we may assume that $a_r$ is odd. Consider the critical FK-Ising model on $\Omega^{\delta}$ with the boundary condition 
\begin{align}\label{eqn::bc_Dobru}
\textnormal{wired on }(x_{a_r}^{\delta} \, x_{b_r}^{\delta}) 
\qquad
\textnormal{and} 
\qquad
\textnormal{free on }(x_{b_r}^{\delta} \, x_{a_r}^{\delta}),
\end{align} 
and denote by  $\E_{\vcenter{\hbox{\includegraphics[scale=0.2]{figures/link-0.pdf}}}}^{\delta}$ the expectation of this model. 
For this model, the edge observable  $\smash{F_{\vcenter{\hbox{\includegraphics[scale=0.2]{figures/link-0.pdf}}}}^{\delta}}$ on the medial edges of $\Omega^{\delta,\diamond}$ and the outer corner edges $\{e_{a_r}^{\delta,\diamond}, e_{b_r}^{\delta,\diamond}\}$ is
\begin{align*}
F_{\vcenter{\hbox{\includegraphics[scale=0.2]{figures/link-0.pdf}}}}^{\delta}(e)
:=  \nu(e_{b_r}^{\delta,\diamond}) \, 
\E_{\vcenter{\hbox{\includegraphics[scale=0.2]{figures/link-0.pdf}}}}^{\delta}\Big[\one \{e\in \eta_{a_r}^{\delta}\} \exp \big(-\tfrac{\ii}{2}W_{\eta_{a_r}^{\delta}} \big( e_{b_r}^{\delta,\diamond},e\big) \big)\Big] ,
\end{align*}
where $\eta_{a_r}^{\delta}$ is the exploration path from $y_{a_r}^{\delta,\diamond}$ to $y_{b_r}^{\delta,\diamond}$ 
and the number $\smash{W_{\eta_{a_r}^{\delta}} \big( y_{b_r}^{\delta,\diamond},e \big)}$ 
is the winding from $y_{b_r}^{\delta,\diamond}$ to $e$ along the reversal of $\eta_{a_r}^{\delta}$.
One can prove similarly as in~\cite[Lemma~4.1]{Smirnov:Conformal_invariance_in_random_cluster_models1} that
$\smash{F_{\beta}^{\delta}(e_{b_r}^{\delta,\diamond})\parallel \nu(e_{b_r}^{\delta,\diamond})}$, which implies that
\begin{align} \label{eqn::obser_compa_same_sign} F_{\vcenter{\hbox{\includegraphics[scale=0.2]{figures/link-0.pdf}}}}^{\delta}(e_{b_r}^{\delta,\diamond}) 
= \lambda_{b_r} \, F_{\beta}^{\delta}(e_{b_r}^{\delta,\diamond})\qquad \textnormal{for some } \lambda_{b_r} > 0 .
\end{align}
The vertex observable $\smash{F_{\vcenter{\hbox{\includegraphics[scale=0.2]{figures/link-0.pdf}}}}^{\delta}}$ on interior vertices of $\Omega^{\delta,\diamond}$ is 
\begin{align*}
F_{\vcenter{\hbox{\includegraphics[scale=0.2]{figures/link-0.pdf}}}}^{\delta}(z) := \frac{1}{2}\sum_{e\sim z}F_{\vcenter{\hbox{\includegraphics[scale=0.2]{figures/link-0.pdf}}}}^{\delta}(e) ,
\end{align*}
and on boundary vertices it is
\begin{align*}
F_{\vcenter{\hbox{\includegraphics[scale=0.2]{figures/link-0.pdf}}}}^{\delta}(z) := 
\begin{cases}
\sqrt{2}\exp(-\ii\frac{\pi}{4})F_{\vcenter{\hbox{\includegraphics[scale=0.2]{figures/link-0.pdf}}}}^{\delta}(e_+^{\diamond})+\sqrt{2}\exp(\ii\frac{\pi}{4})F_{\vcenter{\hbox{\includegraphics[scale=0.2]{figures/link-0.pdf}}}}^{\delta}(e_-^{\diamond}), & \textnormal{if $v\in (x_{a_r}^{\delta,\diamond}x_{b_r}^{\delta,\diamond})$,} \\[.5em]
\sqrt{2}\exp(-\ii\frac{\pi}{4})F_{\vcenter{\hbox{\includegraphics[scale=0.2]{figures/link-0.pdf}}}}^{\delta}(e_-^{\diamond})+\sqrt{2}\exp(\ii\frac{\pi}{4})F_{\vcenter{\hbox{\includegraphics[scale=0.2]{figures/link-0.pdf}}}}^{\delta}(e_+^{\diamond}), & \textnormal{if $v\in (x_{b_r}^{\delta,\diamond}x_{a_r}^{\delta,\diamond})$,}
\end{cases}
\end{align*}
where for a medial vertex $z^{\diamond} \in \partial\Omega^{\delta,\diamond}\setminus\{x_{a_r}^{\delta,\diamond}, x_{b_r}^{\delta,\diamond}\}$, we denote by $e_-^{\diamond}, e_+^{\diamond} \in \Omega^{\delta,\diamond}$ the medial edges having $z^{\diamond}$ as end vertex and beginning vertex, respectively. 
We extend the vertex observable $\smash{F_{\vcenter{\hbox{\includegraphics[scale=0.2]{figures/link-0.pdf}}}}^{\delta}}$ to a continuous function on the planar domain corresponding to $\Omega^{\delta,\diamond}$ via linear interpolation. 
A similar argument as for $\smash{F_{\vcenter{\hbox{\includegraphics[scale=0.2]{figures/link-0.pdf}}}}^{\delta}}$ shows that the sequence $\{2^{-1/4}\delta^{-1/2} \, F_{\vcenter{\hbox{\includegraphics[scale=0.2]{figures/link-0.pdf}}}}^{\delta}\}_{\delta>0}$ of scaled vertex observables has locally uniformly convergent subsequences, and by~\cite[Theorem~2.2]{Smirnov:Conformal_invariance_in_random_cluster_models1}, any subsequential limit equals $\pm \phi_{\vcenter{\hbox{\includegraphics[scale=0.2]{figures/link-0.pdf}}}}(\cdot \, ; \Omega; \realpt_{a_r}, \realpt_{b_r})$ defined in~\eqref{eqn::holo_limiting_N=1}.  
Note also that\footnote{Note that this does not violate~\eqref{eqn::holo_restriction2}, since $r=1$ in~\eqref{eqn::sign_change_Dobru}.}  
by~\eqref{eqn::holo_limiting_N=1}, we have
\begin{align} \label{eqn::sign_change_Dobru}
\lim_{z\to \realpt_{a_r}} \sqrt{z-\realpt_{a_r}} \, \sqrt{z-\realpt_{b_r}} \, \phi_{\vcenter{\hbox{\includegraphics[scale=0.2]{figures/link-0.pdf}}}}(z;\HH;\realpt_{a_r}, \realpt_{b_r}) 
=  \lim_{z\to \realpt_{b_r}} \sqrt{z-\realpt_{a_r}} \, \sqrt{z-\realpt_{b_r}} \, \phi_{\vcenter{\hbox{\includegraphics[scale=0.2]{figures/link-0.pdf}}}}(z;\HH;\realpt_{a_r}, \realpt_{b_r}).
\end{align}
Now, let us compare $\smash{F_{\vcenter{\hbox{\includegraphics[scale=0.2]{figures/link-0.pdf}}}}^{\delta}}$ and $F_{\beta}^{\delta}$. 
To this end, a key observation is that 
\begin{align} \label{eqn::sign_change_discrete}
\textnormal{$\big|F_{\vcenter{\hbox{\includegraphics[scale=0.2]{figures/link-0.pdf}}}}^{\delta}(e_{b_r}^{\delta,\diamond}) + F_{\beta}^{\delta}(e_{b_r}^{\delta,\diamond})\big|$ and $\big|F_{\vcenter{\hbox{\includegraphics[scale=0.2]{figures/link-0.pdf}}}}^{\delta}(e_{a_r}^{\delta,\diamond}) + F_{\beta}^{\delta}(e_{a_r}^{\delta,\diamond})\big|$ \; differ by $2\min \big\{|F_{\vcenter{\hbox{\includegraphics[scale=0.2]{figures/link-0.pdf}}}}^{\delta}(e_{b_r}^{\delta,\diamond})|, |F_{\beta}^{\delta}(e_{b_r}^{\delta,\diamond})|\big\}$.}
\end{align}
Observation~\eqref{eqn::sign_change_discrete} can be derived as follows. 
\begin{itemize}
\item First, by construction, the exploration path $\xi^{\delta}$ passes through $e_{a_r}^{\delta,\diamond}$ and $e_{b_r}^{\delta,\diamond}$ if and only if it passes through the contour corresponding to $\{a_r,b_r\}$ outside of $\Omega^{\delta}$. 
In this case, we denote by $W_1$ the winding from $e_{a_r}^{\delta,\diamond}$ to $e_{b_r}^{\delta,\diamond}$ along the reversal of $\xi^{\delta}$, which is independent of the configuration. 
Then, we have (recalling~\eqref{eq::pair_of_1})
\begin{align*} 
W_{\xi^{\delta}}(e_{2\ell}^{\delta,\diamond}, e_{a_r}^{\delta,\diamond}) = W_{\xi^{\delta}} \big( e_{2\ell}^{\delta,\diamond}, e_{b_r}^{\delta,\diamond} \big) - W_{1} 
\qquad \Longrightarrow \qquad 
F_{\beta}^{\delta}(e_{a_r}^{\delta,\diamond}) 
= e^{\ii W_{1} / 2}
\, F_{\beta}^{\delta}(e_{b_r}^{\delta,\diamond}).
\end{align*}

\item Second, consider the critical FK-Ising model on $\Omega^{\delta}$ with boundary condition~\eqref{eqn::bc_Dobru}. The exploration path $\eta_{a_r}^{\delta}$ passes through $e_{a_r}^{\delta,\diamond}$ and $e_{b_r}^{\delta,\diamond}$ with probability one. 
Denote by $W_2$ the winding from $e_{b_r}^{\delta,\diamond}$ to $e_{a_r}^{\delta,\diamond}$ along the reversal of $\eta_{a_r}^{\delta}$, which is also independent of the configuration. Then, we have
\begin{align*} 
W_{\eta_{a_r}^{\delta}}(e_{b_r}^{\delta,\diamond}, e_{a_r}^{\delta,\diamond}) = W_{\eta_{a_r}^{\delta}}\big( e_{b_r}^{\delta,\diamond}, e_{b_r}^{\delta,\diamond}\big) + W_{2} 
\qquad \Longrightarrow \qquad 
F_{\vcenter{\hbox{\includegraphics[scale=0.2]{figures/link-0.pdf}}}}^{\delta}(e_{a_r}^{\delta,\diamond}) 
= e^{- \ii W_{2} / 2}
\, F_{\vcenter{\hbox{\includegraphics[scale=0.2]{figures/link-0.pdf}}}}^{\delta}(e_{b_r}^{\delta,\diamond}). 
\end{align*}

\item Third, the exploration path $\eta_{a_r}^{\delta}$ inside of $\Omega^{\delta}$ and the contour corresponding to $\{a_r,b_r\}$ outside of $\Omega^{\delta}$ always form a loop, which implies that $W_{1}+W_{2}=2\pi$.
\end{itemize}
Combining the above observations for the windings $W_{1}$ and $W_{2}$ 
with~\eqref{eqn::obser_compa_same_sign}, we obtain
\begin{align} \label{eqn::obser_compa_opp_sign}
F_{\vcenter{\hbox{\includegraphics[scale=0.2]{figures/link-0.pdf}}}}^{\delta}(e_{a_r}^{\delta,\diamond}) = \lambda_{a_r} \, F_{\beta}^{\delta}(e_{a_r}^{\delta,\diamond}) , \qquad \textnormal{for some } \lambda_{a_r} < 0 .
\end{align}
The relations~\eqref{eqn::obser_compa_same_sign} and~\eqref{eqn::obser_compa_opp_sign} now together imply~\eqref{eqn::sign_change_discrete}.

Now, we are ready to show that~\eqref{eqn::sign_change} holds with signs $\varepsilon_r = -1$ for all $2 \leq r \leq N$. 
First of all, if $\MainConst_{a_r-1}=\MainConst_{a_r}$, then the left-hand side of~\eqref{eqn::sign_change} equals zero, so we can take $\epsilon_r=-1$. In contrast, if $\MainConst_{a_r-1} \neq \MainConst_{a_r}$, then~\eqref{eqn::sign_change_Dobru} shows that  the function 
\begin{align*}
w \quad \longmapsto \quad
\Im\int^{w}(f_{\HH}(\cdot)+\phi_{\vcenter{\hbox{\includegraphics[scale=0.2]{figures/link-0.pdf}}}}(\cdot \, ;\HH;\realpt_{a_r},\realpt_{b_r}))^2
\end{align*}
has jumps of the same size at $\realpt_{a_r}$ and $\realpt_{b_r}$,
while by~\eqref{eqn::sign_change_discrete}, the function 
defined via a subsequential limit along some $\delta_n \to 0$ as $n\to \infty$, 
\begin{align*}
w \quad \longmapsto \quad
\lim_{n\to \infty} \Im\int^{\varphi^{-1}(w)}\frac{(F^{\delta_n}_{\beta}(\cdot)+F_{\vcenter{\hbox{\includegraphics[scale=0.2]{figures/link-0.pdf}}}}^{\delta_n}(\cdot))^2}{\sqrt{2}{\delta_n}} 
\; = \; \Im\int^{w} \big( f_{\HH}(\cdot) + \phi_{\vcenter{\hbox{\includegraphics[scale=0.2]{figures/link-0.pdf}}}}(\cdot \, ;\HH;\realpt_{a_r},\realpt_{b_r}) \big)^2 ,
\end{align*}
has jumps of different sizes $(1-|\MainConst_{a_r-1}-\MainConst_{a_r}|)^2$ and $(1+|\MainConst_{a_r-1}-\MainConst_{a_r}|)^2$ at $\realpt_{a_r}$ and $\realpt_{b_r}$, respectively.  
This is a contradiction. 
Hence, we conclude that $\varepsilon_r = -1$ for all $2 \leq r \leq N$. The proof is now complete.
\end{proof}

 \begin{corollary} \label{cor::cvg_c_i}
The limit $\smash{\underset{\delta\to 0}{\lim} (\MainConst_1^{\delta},\ldots,\MainConst_{2N}^{\delta}) := (\MainConst_1,\ldots,\MainConst_{2N})}$ exists and satisfies 
\begin{align} \label{eqn::difference_c_i}
\lim_{\delta\to 0}|\MainConst_{k-1}^{\delta} - \MainConst_{k}^{\delta}| 
= \lim_{z\to \varphi(x_k)} \pi \,  |z-\varphi(x_k)| \, 
\big| \phi_{\beta}( z;\varphi(x_1),\ldots,\varphi(x_{2N}) ) \big|^2 , \qquad \textnormal{for  $1\leq k \leq 2N$},
\end{align}
where $\varphi$ is any conformal map from $\Omega$ onto $\HH$ such that $\varphi(x_1)<\cdots<\varphi(x_{2N})$.
\end{corollary}

\begin{proof}
Proposition~\ref{prop::observable_cvg} implies that $\MainConst_{k-1}^{\delta}-\MainConst_{k}^{\delta}$ converges as $\delta\to 0$ for all $1\leq k \leq 2N$. Combining this with the fact that $\MainConst_{1}^{\delta}=0$ (Lemma~\ref{lem::discrete_H}), we obtain the convergence of the sequence $\{(\MainConst_{1}^{\delta},\ldots,\MainConst_{2N}^{\delta})\}_{\delta>0}$ as $\delta\to 0$. 
Identity~\eqref{eqn::difference_c_i} then follows from Lemma~\ref{lem::sub_limit_obser}, Proposition~\ref{prop::observable_cvg}, after computing the residues of $| \phi_{\beta}(z;\varphi(x_1),\ldots,\varphi(x_{2N})) |^2$ at $\varphi(x_k)$ for $1\leq k \leq 2N$.
\end{proof}
\subsection{Scaling limit of the interfaces: Proof of Theorem~\ref{thm::FKIsing_Loewner}}
\label{subsec::FKIsing_Loewner}
We are now ready to prove the convergence of the interfaces in Conjecture~\ref{conj::rcm_Loewner} for the FK-Ising model (random-cluster model with $q=2$), that is, the assertion in Theorem~\ref{thm::FKIsing_Loewner}. 
With precompactness
from Lemma~\ref{lem::FKIsing_tightness} 
and the convergence of the observable from
Propositions~\ref{prop::observable_cvg},~\ref{prop::holo_limiting},~\&~\ref{prop::totalpartition_observable} at hand, the proof is 
a standard martingale argument. 
We summarize its steps below. 

\begin{proof}[Proof of Theorem~\ref{thm::FKIsing_Loewner}]
By rotation symmetry of the partition function~\eqref{eqn::totalpartition_def} on one hand and of the discrete model on the other hand, 
we may without loss of generality consider the interface
$\smash{\eta_1^{\delta}}$ starting from $\smash{x_1^{\delta, \diamond}}$, i.e., assume that $i=1$. 
By assumption, the medial polygons $(\Omega^{\delta, \diamond}; x_1^{\delta, \diamond}, \ldots, x_{2N}^{\delta, \diamond})$ converge to 
$(\Omega; x_1, \ldots, x_{2N})$ in the Carath\'{e}odory sense, so there are
conformal maps $\varphi_{\delta} \colon \Omega^{\delta, \diamond} \to \HH$ and $\varphi \colon \Omega \to \HH$
such that $\varphi(x_1)<\cdots<\varphi(x_{2N})$ and,
as $\delta \to 0$, the maps $\varphi_{\delta}^{-1}$ converge to $\varphi^{-1}$ locally uniformly, and $\varphi_{\delta}(x_j^{\delta, \diamond}) \to \varphi(x_j)$ for all $j$.
Denote by $\tilde{\eta}_1^{\delta} := \varphi_{\delta}(\eta_1^{\delta})$ the conformal image of the interface $\eta_1^{\delta}$ parameterized by half-plane capacity. 
By Lemma~\ref{lem::FKIsing_tightness}, we may choose a subsequence $\delta_{n}\to 0$ such that $\eta^{\delta_{n}}_1$ converges weakly in the metric~\eqref{eq::curve_metric} as $n\to\infty$. 
We denote the limit by $\eta_1$, define $\tilde{\eta}_1:=\varphi(\eta_1)$, 
and parameterize it also by half-plane capacity. 
It follows from the proof of Lemma~\ref{lem::FKIsing_tightness} together 
with~\cite[Corollary~1.7]{Kemppainen-Smirnov:Random_curves_scaling_limits_and_Loewner_evolutions} that the family $\{\tilde{\eta}^{\delta_{n}}_1|_{[0,t]} \colon [0,t] \to \overline{\HH} \}_{n \geq 1}$  is 
precompact in the uniform topology of curves parameterized by half-plane capacity.
Thus (also by coupling them into the same probability space), 
we can choose a further subsequence, 
still denoted $\delta_n$, such that $\tilde{\eta}^{\delta_{n}}_1$ converges to $\tilde{\eta}_1$ locally uniformly as $n\to\infty$, almost surely.  
Next, define $\tau^{\delta_n}$ to be the first time when $\eta_1^{\delta_n}$ hits the arc $(x_2^{\delta_n} \, x_{2N}^{\delta_n})$ 
and $\tau$ to be the first time when $\eta_1$ hits $(x_2 \, x_{2N})$. 
By properly adjusting the coupling (see, e.g.,~\cite[Section~4]{Garban-Wu:On_the_convergence_of_FK-Ising_percolation_to_SLE} or~\cite[Lemma~4.3]{Izyurov:On_multiple_SLE_for_the_FK_Ising_model})
we may furthermore assume that 
$\underset{n \to \infty}{\lim} \tau^{\delta_n} = \tau$ almost surely.

Now, denote by $(W_t, t\ge 0)$ the Loewner driving function of $\tilde{\eta}_1$ and by $(g_t, t\ge 0)$ the corresponding conformal maps. 
Write $V_t^j := g_t(\varphi(x_j))$ for $j\in\{2,3, \ldots, 2N\}$. 
Via a standard argument (see, e.g.,~\cite[Lemmas~3.3 and~4.3]{Izyurov:On_multiple_SLE_for_the_FK_Ising_model}), 
we derive from the spinor observable $\phi_{\beta}$ of  Proposition~\ref{prop::holo_limiting} the local martingale 
\begin{align} \label{eqn::ust_polygon_mart}
M_t(z) := (g_t'(z))^{1/2} \times 
\phi_{\beta}(g_t(z); W_t, V_t^2, \ldots, V_t^{2N}) , \qquad t < \tau,
\end{align}
where throughout the proof, $(\cdot)^{1/2}$ uses the principal branch of the square root.

It remains to argue that $(W_t, t\ge 0)$ is a semimartingale and to find the SDE for it. 
This step is also standard by now.  
For any $w<y_2<\cdots<y_{2N}$, the function $\partial_w \phi_\beta(\cdot;w,y_2,\ldots,y_{2N})$ 
is holomorphic and not identically zero, so its zeros are isolated. 
Pick $z\in\HH$ with $|z|$ large enough such that $\partial_w\phi_\beta(z;w,y_2,\ldots,y_{2N})\neq 0$.
By the implicit function theorem, $w$ is locally a smooth function of $(\phi_\beta,z,y_2,\ldots,y_{2N})$. 
Thus, by continuity, each time $t < \tau$ has a neighborhood $I_t$ for which
we can choose a deterministic $z$ such that $W_s$ is locally a smooth function of 
$(M_s(z),g_s(z),g_s(y_2),\ldots,g_s(y_{2N}))$ for all $s \in I_t$.
This implies that $(W_t, t\ge 0)$ is a semimartingale. 
To find the SDE for $W_t$, let $D_t$ denote the drift term of $W_t$. 
By a computation using It\^o's formula, we find from~\eqref{eqn::ust_polygon_mart} and using the Loewner equation~\eqref{eqn:LE} the identities
\begin{align*}
\frac{\ud M_t(z) }{(g_t'(z))^{1/2}}
= & \; \frac{-\phi_{\beta} \, \ud t}{(g_t(z)-W_t)^2} 
\, + \, \frac{2(\partial_z\phi_{\beta}) \, \ud t}{g_t(z)-W_t}
\, +\, (\partial_1\phi_{\beta}) \, \ud W_t+\sum_{j=2}^{2N}\frac{2(\partial_j\phi_{\beta}) \, \ud t}{V_t^j-W_t}+\frac{1}{2}(\partial_1^2\phi_{\beta}) \, \ud\langle W\rangle_t . 
\end{align*}
Combining this with Lemma~\ref{lem::holo_expansion}, 
we find the expansion 
\begin{align*}
\frac{\ud M_t(z)}{(g_t'(z))^{1/2}} 
= & \; (g_t(z)-W_t)^{-5/2} \Big(-\frac{2}{\sqrt{\pi}} \, \ud t 
 \, +  \, \frac{3}{8\sqrt{\pi}} \, \ud\langle W\rangle_t\Big) \\
\; & +  \, (g_t(z)-W_t)^{-3/2}\Big(\frac{1}{2\sqrt{\pi}} \, \ud W_t  \, 
-  \, \frac{1}{8}\LK_{\beta} \, \ud\langle W\rangle_t\Big)
\, + \, o(g_t(z)-W_t)^{-3/2} .
\end{align*}
As the drift term of $M_t(z)$ has to vanish, we conclude that 
\begin{align*}
\nonumber
\ud \langle W\rangle_t = \; & \frac{16}{3}  \, \ud t 
\qquad  \textnormal{and} \qquad
\frac{1}{2\sqrt{\pi}} \, \ud D_t  \, - \, \frac{1}{8}\LK_{\beta} \, \ud\langle W\rangle_t = 0 \\
\qquad \Longrightarrow \qquad
\ud \langle W\rangle_t = \; & \frac{16}{3} \, \ud t
\qquad  \textnormal{and} \qquad
\ud D_t = \frac{4\sqrt{\pi}}{3}\LK_{\beta} \, \ud t . 
\end{align*}
Now, recalling that the goal is to derive an SDE for the driving function $W$, 
we conclude from Proposition~\ref{prop::totalpartition_observable} that 
\begin{align*}
\ud W_t = \sqrt{\frac{16}{3}} \, \ud B_t  \, +  \, \frac{16}{3} (\partial_1\log\LF_{\beta})(W_t, V_t^2, \ldots, V_t^{2N}) \, \ud t , \qquad t < \tau. 
\end{align*}
This proves the convergence of the interface, 
and the identity $\LF_{\beta} = \coulombnew_{\beta}$ from the proof of Theorem~\ref{thm::totalpartition} 
completes the proof of Theorem~\ref{thm::FKIsing_Loewner}.
\end{proof}

\newpage

\section{FK-Ising model connection probabilities: Proof of Theorem~\ref{thm::FKIsing_crossingproba}}
\label{sec::crossingproba}
The goal of this section is to derive the scaling limit of the connection probabilities (Theorem~\ref{thm::FKIsing_crossingproba}). 

\smallbreak

The convergence of the boundary values 
$\{(\MainConst_1^{\delta},\ldots,\MainConst_{2N}^{\delta})\}_{\delta>0}$ of the discrete primitive 
in Corollary~\ref{cor::cvg_c_i} is related to the convergence of the connection probabilities: 
indeed, when $N=2$, the former implies the latter via~\eqref{eqn::const_jump}, see Lemma~\ref{lem::cvg_proba_N=2}.
However, for general $N$ and general boundary conditions $\beta \in \LP_N$, this is not the case 
since \emph{the exploration path may not fully determine the internal connectivity pattern of the interfaces}.
To find the scaling limit for general~$\beta$, we first 
derive it with $\beta=\unnested$ in Section~\ref{subsec::crossingproba_unnested} 
(via a martingale argument using the convergence of the interfaces from Theorem~\ref{thm::FKIsing_Loewner}, or~\cite[Theorem~1.1]{Izyurov:On_multiple_SLE_for_the_FK_Ising_model}),
and then address a general $\beta$ in Section~\ref{subsec::crossingproba_general} by comparing it to the case of $\unnested$.  
The comparison relies on combinatorial properties of the meander matrix (Definition~\ref{def::meander}) together with those of the random-cluster model, also of independent interest (Proposition~\ref{prop::crossingproba_comparison}).

Actually, we only really need from Theorem~\ref{thm::FKIsing_Loewner} the case of $\beta = \unnested$ to show Theorem~\ref{thm::FKIsing_crossingproba} for general $\beta$ (using the combinatorial observation from Proposition~\ref{prop::crossingproba_comparison}).
Indeed, the main inputs for proving Theorem~\ref{thm::FKIsing_crossingproba} in the case of $\beta = \unnested$ are
Theorem~\ref{thm::FKIsing_Loewner} in the case of $\beta = \unnested$, 
Corollary~\ref{cor::linearcombination_FKIsing}, and a priori estimates from Section~\ref{subsec::crossingproba_unnested} and Appendix~\ref{appendix_technical}.
The additional non-trivial inputs to derive Theorem~\ref{thm::FKIsing_crossingproba} for general $\beta \in \LP_N$ 
are the aforementioned Proposition~\ref{prop::crossingproba_comparison} 
and the cascade relation in Lemma~\ref{lem::ppf_cascade_mart}.

\begin{lemma} \label{lem::cvg_proba_N=2}
Theorem~\ref{thm::FKIsing_crossingproba} holds with $N=2$.
\end{lemma}

Note that this is consistent with~\cite[Eq.~(117)]{FSKZ:A_formula_for_crossing_probabilities_of_critical_systems_inside_polygons}
(see also~\textnormal{\cite[Corollary~2.7]{Izyurov:Smirnovs_observable_for_free_boundary_conditions_interfaces_and_crossing_probabilities}}).

\begin{proof}
We have two possible boundary conditions, $\vcenter{\hbox{\includegraphics[scale=0.2]{figures/link-1.pdf}}}=\{\{1,2\}, \{3,4\}\}$ and $\vcenter{\hbox{\includegraphics[scale=0.2]{figures/link-2.pdf}}}=\{\{1,4\}, \{2,3\}\}$. 
We will show the convergence  
		\begin{align} \label{eqn::cro_proba_N=2}
		\lim_{\delta\to 0} \PP_{\vcenter{\hbox{\includegraphics[scale=0.2]{figures/link-1.pdf}}}}^{\delta} [ \FKconn^{\delta} 
		= \vcenter{\hbox{\includegraphics[scale=0.2]{figures/link-2.pdf}}} ] 
	\,	= \, \frac{\sqrt{2} \, \PartF_{\vcenter{\hbox{\includegraphics[scale=0.2]{figures/link-2.pdf}}}}(\Omega ; x_1,x_2,x_3,x_4)}{\LF_{\vcenter{\hbox{\includegraphics[scale=0.2]{figures/link-1.pdf}}}}(\Omega ; x_1,x_2,x_3,x_4)} ,
	\end{align}
	which also implies the assertion for 
	$\FKconn^{\delta} = \vcenter{\hbox{\includegraphics[scale=0.2]{figures/link-1.pdf}}}$, since by 
	combining~\eqref{eqn::cro_proba_N=2} with Corollary~\ref{cor::linearcombination_FKIsing}, we have 
	\begin{align*}
\lim_{\delta\to 0}\PP_{\vcenter{\hbox{\includegraphics[scale=0.2]{figures/link-1.pdf}}}}^{\delta} [ \FKconn^{\delta} 
\, = \, \vcenter{\hbox{\includegraphics[scale=0.2]{figures/link-1.pdf}}} ]
\, = \, 1 \, - \,  \lim_{\delta\to 0}\PP_{\vcenter{\hbox{\includegraphics[scale=0.2]{figures/link-1.pdf}}}}^{\delta} [ \FKconn^{\delta}=\vcenter{\hbox{\includegraphics[scale=0.2]{figures/link-2.pdf}}} ]
\, = \, \frac{2\,\PartF_{\vcenter{\hbox{\includegraphics[scale=0.2]{figures/link-1.pdf}}}}(\Omega ; x_1,x_2,x_3,x_4)}{\LF_{\vcenter{\hbox{\includegraphics[scale=0.2]{figures/link-1.pdf}}}}(\Omega ; x_1,x_2,x_3,x_4)} .
	\end{align*}
The probabilities with boundary condition $\vcenter{\hbox{\includegraphics[scale=0.2]{figures/link-2.pdf}}}$ can be derived using rotation symmetry.

Thus, it remains to show~\eqref{eqn::cro_proba_N=2}. 
Note that the right-hand side of~\eqref{eqn::cro_proba_N=2} is conformally invariant by the covariance property~\eqref{eqn::COV} shared by both the numerator and the denominator. 
Let $\varphi$ be a conformal map from $\Omega$ onto $\HH$ such that $\varphi(x_{1})<\cdots<\varphi(x_{4})$, and denote 
\begin{align*}
\chi=\frac{(\realpt_4-\realpt_3)(\realpt_2-\realpt_1)}{(\realpt_3-\realpt_1)(\realpt_4-\realpt_2)} 
\qquad \textnormal{and} \qquad
\realpt_{i} := \varphi(x_i) \, \in \, \R , \qquad \textnormal{for } \, 1\leq i \leq 4 .
\end{align*}
On the one hand, Eq.~\eqref{eqn::totalpartition_def} and~\cite[Section~2]{Peltola-Wu:Global_and_local_multiple_SLEs_and_connection_probabilities_for_level_lines_of_GFF} give
\begin{align*}
\LF_{\vcenter{\hbox{\includegraphics[scale=0.2]{figures/link-1.pdf}}}}(\realpt_1,\realpt_2,\realpt_3,\realpt_4)
= \; & \sqrt{2} \, (\realpt_2-\realpt_1)^{-1/8} (\realpt_4-\realpt_3)^{-1/8} \big( (1-\chi)^{1/4} + (1-\chi)^{-1/4} \big)^{1/2} , \\
\PartF_{\vcenter{\hbox{\includegraphics[scale=0.2]{figures/link-2.pdf}}}}(\realpt_1,\realpt_2,\realpt_3,\realpt_4)
= \; & (\realpt_4-\realpt_1)^{-1/8} (\realpt_3-\realpt_2)^{-1/8} \chi^{3/8} (1+\sqrt{1-\chi})^{-1/2} . 
\end{align*} 
Thus, since the ratio of $\LF_{\vcenter{\hbox{\includegraphics[scale=0.2]{figures/link-1.pdf}}}}$ and $\PartF_{\vcenter{\hbox{\includegraphics[scale=0.2]{figures/link-2.pdf}}}}$ is conformally invariant by~\eqref{eqn::COV},
we find that
\begin{align} \label{eqn::cro_proba_N=2_aux2}
\frac{\sqrt{2}\PartF_{\vcenter{\hbox{\includegraphics[scale=0.2]{figures/link-2.pdf}}}}( \Omega ; x_1,x_2,x_3,x_4)}{\LF_{\vcenter{\hbox{\includegraphics[scale=0.2]{figures/link-1.pdf}}}}(\Omega ; x_1,x_2,x_3,x_4)} 
= \frac{\sqrt{\chi} \, (1+\sqrt{1-\chi})^{-1/2}}{(1-\chi)^{1/8}\big((1-\chi)^{1/4}+(1-\chi)^{-1/4}\big)^{1/2}} 
= \frac{\sqrt{\chi}}{1+\sqrt{1-\chi}} .
\end{align}
On the other hand, using the exploration path $\xi_{\vcenter{\hbox{\includegraphics[scale=0.2]{figures/link-1.pdf}}}}^{\delta}$ from Definition~\ref{def: exploration path} and the scaling limit of the observable from Section~\ref{subsec::holo_observable},
we find
\begin{align*}	\PP_{\vcenter{\hbox{\includegraphics[scale=0.2]{figures/link-1.pdf}}}}^{\delta}[\FKconn^{\delta}=\vcenter{\hbox{\includegraphics[scale=0.2]{figures/link-2.pdf}}}]
	= \; & \lim_{z\to \realpt_4} \sqrt{\pi} \, \big| (z-\realpt_{4})^{1/2} \, \phi_{\vcenter{\hbox{\includegraphics[scale=0.2]{figures/link-1.pdf}}}}(z;\realpt_1,\realpt_2,\realpt_3,\realpt_4) \big| 
	&& \textnormal{[by~\eqref{eqn::const_jump} and Cor.~\ref{cor::cvg_c_i}]}\\
	= \; &  \frac{\sqrt{(\realpt_4-\realpt_2)(\realpt_3-\realpt_1)}}{\sqrt{(\realpt_2-\realpt_1)(\realpt_4-\realpt_3)}}-\frac{\sqrt{(\realpt_4-\realpt_1)(\realpt_3-\realpt_2)}}{\sqrt{(\realpt_2-\realpt_1)(\realpt_4-\realpt_3)}}&&\textnormal{[by~\eqref{eqn::obse_N=2_unne}]}\\
	= \; &  \frac{1-\sqrt{1-\chi}}{\sqrt{\chi}}. 
\end{align*}
Comparing this with~\eqref{eqn::cro_proba_N=2_aux2}, we obtain~\eqref{eqn::cro_proba_N=2}. This completes the proof. 
\end{proof}
\subsection{Proof of Theorem~\ref{thm::FKIsing_crossingproba}: The completely unnested case}
\label{subsec::crossingproba_unnested}
The goal of this section is to prove Theorem~\ref{thm::FKIsing_crossingproba} when $\beta = \unnested$  as in~\eqref{eqn::unnested}. 
We use a standard martingale argument and the convergence of the interfaces,
which also relies on the domain Markov property of SLE curves and the Markov property of the discrete model.
The main difficulty in the proof is to establish a priori estimates for the behavior of the martingale upon swallowing marked points.

\smallbreak

For a polygon 
$(\Omega; x_1, \ldots, x_{2N})$ whose marked boundary points 
$x_1, \ldots, x_{2N}$ lie on sufficiently regular boundary segments (e.g.,~$C^{1+\eps}$ for some $\eps>0$), 
we denote 
\begin{align*} 
\LF_{\unnested}^{(N)}(\Omega; x_1, \ldots, x_{2N}) 
:= \; & \prod_{j=1}^{2N} |\varphi'(x_j)|^{1/16} \times \LF_{\unnested}^{(N)}(\varphi(x_1), \ldots, \varphi(x_{2N})) ,
\end{align*}
where $\varphi \colon \Omega \to \HH$ is any conformal map such that $\varphi(x_1) < \cdots < \varphi(x_{2N})$. 
It follows from the M\"obius covariance~\eqref{eqn::COV} in Theorem~\ref{thm::CGI_property} that this definition is independent of the choice of the map $\varphi$.
Fixing a choice and denoting throughout this section
$\realpt_{i} := \varphi(x_i)$ 
for notational simplicity, we have
\begin{align*} 
\LF_{\unnested}^{(N)}(\realpt_1, \ldots, \realpt_{2N}) 
= \; & \prod_{r=1}^N |\realpt_{2r}-\realpt_{2r-1}|^{-1/8} \times \bigg( \sum_{\bs{\sigma} \in \{\pm 1\}^N} \prod_{1\le s < t \le N}\chi(\realpt_{2s-1}, \realpt_{2t-1}, \realpt_{2t}, \realpt_{2s})^{\sigma_s \sigma_t / 4}\bigg)^{1/2} 
\end{align*}
as in~\eqref{eqn::totalpartition_def}. 
Since $\PartF_{\alpha}(\Omega; x_1, \ldots, x_{2N})$ 
is also given in terms of the conformal map $\varphi$ and Definition~\ref{def::PPF_general}, 
we see that when considering ratios $\smash{\PartF_{\alpha}/\LF_{\unnested}^{(N)}}$, 
we may relax the assumption on the regularity of $\partial \Omega$.

\begin{proposition} \label{prop::crossingproba_unnested}
Assume the same setup as in Theorem~\ref{thm::FKIsing_Loewner} 
with $\beta = \unnested$ as in~\eqref{eqn::unnested}.
The endpoints of the $N$ interfaces give rise to a random planar link pattern $\FKconn^{\delta}$ in $\LP_N$. 
We have
\begin{align} \label{eqn::crossingproba_unnested}
\lim_{\delta\to 0}\PP_{\unnested}^{\delta} [ \FKconn^{\delta}=\alpha ] 
\; = \; \LM_{\alpha,\unnested}(2) \, \frac{\PartF_{\alpha}(\Omega; x_1, \ldots, x_{2N})}{\LF_{\unnested}^{(N)}(\Omega; x_1, \ldots, x_{2N})} , 
\qquad \textnormal{for any } \, \alpha \in \LP_N .
\end{align}
\end{proposition}

\begin{proof}
We derive the probability~\eqref{eqn::crossingproba_unnested} by induction on $N \geq 1$. 
The initial case of $N=1$ is trivial, and the case of $N=2$ holds by Lemma~\ref{lem::cvg_proba_N=2}. 
Thus, we fix $N\ge 3$ and assume that~\eqref{eqn::crossingproba_unnested} holds up to $N-1$. 
For definiteness, we consider the case where $\{1,2\} \in \alpha \in \LP_N$.
The probabilities $\big\{\PP^{\delta}_{\unnested}[\FKconn^{\delta}=\alpha]\big\}_{\delta>0}$ form a sequence of numbers in $[0,1]$, so there is always subsequential limit. 
It suffices to show that any subsequential limit along a sequence $\delta_n\to 0$ 
satisfies 
\begin{align} \label{eqn::crossingproba_unnested_goal}
\mathsf{P}_{\alpha} := \underset{n\to\infty}{\lim} \PP^{\delta_n}_{\unnested} [ \FKconn^{\delta_n}=\alpha ]
\; = \; \LM_{\alpha,\unnested}(2) \, \frac{\PartF_{\alpha}(\realpt_1, \ldots, \realpt_{2N})}{\LF_{\unnested}^{(N)}(\realpt_1, \ldots, \realpt_{2N})} ,
\end{align}
since the right-hand side is conformally invariant by the covariance property~\eqref{eqn::COV} shared by both the numerator and the denominator. 
From Theorem~\ref{thm::FKIsing_Loewner}, we know that 
(up to the first time $T$ when $\realpt_1$ or $\realpt_3$ is swallowed)
the interface $\eta^\delta$ starting from $x_2^{\delta,\diamond}$ 
converges weakly to the image under $\smash{\varphi^{-1}}$ of the Loewner chain $\eta$ with driving function $W$ 
started from $W_0 = \realpt_2$ and
satisfying the SDE~\eqref{eqn::rcm_Loewner_chain} with partition function $\smash{\coulombnew_{\unnested} = \LF_{\unnested}^{(N)}}$,
where
$(V_t^{1}, W_t, V_t^{3},\ldots, V_t^{2N}) = (g_t(\realpt_1), W_t, g_t(\realpt_3), \ldots, g_t(\realpt_{2N}))$.
For convenience, we couple them 
(by the Skorohod representation theorem) 
in the same probability space so that the convergence occurs almost surely. 
Now, the process 
\begin{align*}
M_t := \frac{\PartF_{\alpha}(g_t(\realpt_1), W_t, g_t(\realpt_3), \ldots, g_t(\realpt_{2N}))}{\LF_{\unnested}^{(N)}(g_t(\realpt_1), W_t, g_t(\realpt_3), \ldots, g_t(\realpt_{2N}))} , \qquad t<T ,
\end{align*}
is a bounded martingale due to Corollary~\ref{cor::linearcombination_FKIsing} and the PDEs~\eqref{eqn::PDE} by It\^o's formula. 
Note that~\eqref{eqn::crossingproba_unnested_goal} involves its starting value $M_0$.
The key to the proof is to analyze the limiting behavior of $M_t$ as $t \nearrow T$. 

\smallbreak

We have either 
$\eta(T)\in (\realpt_j, \realpt_{j+1})$ for $j\in\{3,4,\ldots, 2N\}$,  
or $\eta(T)\in (\realpt_{2N}, \realpt_1) = (\realpt_{2N}, \infty)\cup(-\infty, \realpt_1)$.  
When considering the limit of $M_t$, we classify the possibilities $\eta(T)\in (\realpt_j, \realpt_{j+1})$ with ``correct" $j$ and ``wrong" $j$. 
For this, we define $\corrind_{\alpha}$ to be the set of indices $j\in\{4,5,\ldots, 2N\}$ such that $\{3,4,\ldots, j\}$ forms a sub-link pattern of $\alpha$ (these indices are ``correct").
After relabeling the indices by $1, 2, \ldots, j-2$, we denote this sub-link pattern by $\alpha_j$, and we denote by $\alpha/\alpha_j$ the sub-link pattern obtained from $\alpha$ by removing the links in $\alpha_j$ and relabeling the remaining indices by $1, 2, \ldots, 2N-j+2$. 
\begin{enumerate}
\item[($\corrind$):] 
On the event $\eta(T)\in (\realpt_j, \realpt_{j+1})$ with $j\in\corrind_{\alpha}$, Lemma~\ref{lem::mart_cascade} (proven below) 
gives the following cascade relation: almost surely, we have
\begin{align}\label{eqn::mart_cascade}
M_T = \lim_{t\to T}M_t 
= \frac{\PartF_{\alpha_j}(D_T^R; \realpt_3, \realpt_4, \ldots, \realpt_j)}{\LF_{\unnested}^{(j/2-1)}(D_T^R; \realpt_3, \realpt_4, \ldots, \realpt_j)} 
\; \frac{\PartF_{\alpha/\alpha_j}(D_T^L; \realpt_1, \eta(T), \realpt_{j+1}, \realpt_{j+2}, \ldots, \realpt_{2N})}{\LF_{\unnested}^{(N-j/2+1)}(D_T^L; \realpt_1, \eta(T), \realpt_{j+1}, \realpt_{j+2}, \ldots, \realpt_{2N})} ,
\end{align}
where $D_T^R$ (resp.~$D_T^L$) denotes the  
component of $\HH\setminus\eta[0,T]$ with $\realpt_3$ (resp.~$\realpt_1$) on its boundary. 

\item[($\corrind^c$):] 
On the event $\eta(T)\in (\realpt_j, \realpt_{j+1})$ with $j\in\{3,4,\ldots, 2N\}\setminus\corrind_{\alpha}$, 
from Proposition~\ref{prop::mart_vanish_combined} 
(presented in Appendix~\ref{appendix_technical}) we see that $M_T$ vanishes: almost surely, we have 
\begin{align}\label{eqn::mart_vanish}
M_T = \lim_{t\to T}M_t = 0 .
\end{align}
\end{enumerate}
Combining~(\ref{eqn::mart_cascade},~\ref{eqn::mart_vanish}) with the identity $M_0 = \E[M_T]$ from the optional stopping theorem, we obtain
\begin{align} \label{eqn::crossingproba_unnested_cascade}
\frac{\PartF_{\alpha}(\realpt_1, \ldots, \realpt_{2N})}{\LF_{\unnested}^{(N)}(\realpt_1, \ldots, \realpt_{2N})} 
\; = \; M_0 \; = \; \E[M_T] \; = \; 
\sum_{j\in\corrind_{\alpha}} 
\E \big[ \one \{\eta(T)\in(\realpt_j, \realpt_{j+1})\} \, M_T \big] .
\end{align}
To simplify notation, we replace the superscripts ``$\delta_n$'' by ``$n$'', and we drop the superscript ``$\diamond$''. 
Let us now consider the FK-Ising interface $\eta^{n}$ starting from $x_2^{n}$,
and denote by $T^{n}$ the first time when $\eta^{n}$ intersects $(x_3^{n} \, x_1^{n})$. 
Denote also by $D^{n,R}$ (resp.~$D^{n,L}$)
the connected component of $\Omega^n\setminus\eta^n[0,T^n]$ with $x_3^n$ (resp.~$x_1^n$) on its boundary.
Then for each $j\in\{3,4, \ldots, 2N\}$, on the event $\{\eta^n(T^n)\in(x_j^n \, x_{j+1}^n)\}$, almost surely
the polygon $(D^{n,R}; x_3^n,x_4^n, \ldots, x_j^n)$ converges to the polygon $(\varphi^{-1}(D_T^R); x_3,x_4, \ldots, x_j)$, and the polygon 
$(D^{n,L}; x_{1}^n, \eta^n(T^n), x_{j+1}^n, x_{j+2}^n, \ldots, x_{2N}^n)$ 
to the polygon $(\varphi^{-1}(D_T^L); x_1, \varphi^{-1}(\eta(T)), x_{j+1}, x_{j+2}, \ldots, x_{2N})$
in the close-Carath\'{e}odory sense 
(this can be seen via a standard argument, see, e.g.,~\cite[Section~4]{Garban-Wu:On_the_convergence_of_FK-Ising_percolation_to_SLE} and~\cite[Lemma~5.6]{Izyurov:On_multiple_SLE_for_the_FK_Ising_model}). 
Hence, using the domain Markov property of the FK-Ising model and the induction hypothesis, we find that on the event $\{\eta^{n}(T^{n}) \in(x_j^n \, x_{j+1}^n)\}$, 
the following almost sure convergence\footnote{By the Skorohod representation theorem, 
we can couple all of the random variables on the same probability space so that the convergence takes place almost surely.} 
holds:
\begin{align} 
\nonumber
\; & 
\E^n_{\unnested}\big[\one \{\FKconn^n=\alpha \} \cond \eta^n[0,T^n]\big] \\[.5em]
\nonumber
= \;\; & \E^n_{\unnested} \big[
\one{\{\smash{\FKhatconn^{n,R}} = \alpha_j\}}  \, 
\one{\{\smash{\FKhatconn^{n,L}} = \alpha/\alpha_j\}}
\cond \eta^n[0,T^n] \big] \\[.5em]
\nonumber
= \;\; & 
\hat{\PP}^{n,R}_{\unnested} [\smash{\FKhatconn^{n,R}} = \alpha_j ] \;
\hat{\PP}^{n,L}_{\unnested} [\smash{\FKhatconn^{n,L}} = \alpha/\alpha_j ] \\[.5em]
\begin{split} \label{eqn::induction_hypo}
\qquad \overset{n \to \infty}{\longrightarrow} \qquad \; & 
\frac{\LM_{\alpha_j, \unnested}(2) \, \PartF_{\alpha_j}(\varphi^{-1}(D_T^R); x_3, x_4, \ldots, x_j)}{\LF_{\unnested}^{(j/2-1)}(\varphi^{-1}(D_T^R); x_3, x_4, \ldots, x_j)} \\
\; & \;\times
\frac{\LM_{\alpha/\alpha_j, \unnested}(2) \, \PartF_{\alpha/\alpha_j}(\varphi^{-1}(D_T^L); x_1, \varphi^{-1}(\eta(T)), x_{j+1}, \ldots, x_{2N})}{\LF_{\unnested}^{(N-j/2+1)}(\varphi^{-1}(D_T^L); x_1, \varphi^{-1}(\eta(T)), x_{j+1}, \ldots, x_{2N})} 
\end{split}
\end{align}
where $\smash{\hat{\PP}^{n,R}_{\unnested}}$ and $\smash{\hat{\PP}^{n,L}_{\unnested}}$
are respectively the FK-Ising measures
on the random polygons $(D^{n,R}; x_3^n,x_4^n, \ldots, x_j^n)$
and $(D^{n,L}; x_{1}^n, \eta^n(T^n), x_{j+1}^n, \ldots, x_{2N}^n)$, 
both measurable with respect to $\eta^{n}$,
and $\smash{\FKhatconn^{n,R}}$ and $\smash{\FKhatconn^{n,L}}$ denote respectively the 
random connectivity patterns in $\LP_{j/2-1}$ and $\LP_{N-j/2+1}$. 
Now, we note that for all $j\in\corrind_{\alpha}$, the meander matrix~\eqref{eqn::meandermatrix_def_general} satisfies the simple factorization identity
\begin{align} \label{eqn::meander_cascade}
\LM_{\alpha_j, \unnested}(2) \; \LM_{\alpha/\alpha_j, \unnested}(2) = \LM_{\alpha,\unnested}(2) . 
\end{align}
Therefore, 
using the conformal invariance (CI) of the $\SLE_{16/3}$ type curve $\eta$ and of the martingale $M$,
we conclude that 
\begin{align*}
\mathsf{P}_{\alpha} 
:= \; & \lim_{n\to\infty}\PP^{n}_{\unnested}[\FKconn^{n}=\alpha] && \\
= \; &  \lim_{n\to\infty} \sum_{j\in\corrind_{\alpha}}
\E^n_{\unnested} \Big[ \one \{\eta^n(T^n)\in(x_j^n \, x_{j+1}^n)\} \, 
\E^n_{\unnested}\big[\one \{\FKconn^n=\alpha \} \cond \eta^n[0,T^n]\big]\Big] 
&& \textnormal{[by tower property]}  \\
= & \; \sum_{j\in\corrind_{\alpha}} 
\E\bigg[ \one \{ \varphi^{-1}(\eta(T)) \in (x_j, x_{j+1})\} \, 
\frac{\LM_{\alpha_j, \unnested}(2) \, \PartF_{\alpha_j}(\varphi^{-1}(D_T^R); x_3, x_4, \ldots, x_j)}{\LF_{\unnested}^{(j/2-1)}(\varphi^{-1}(D_T^R); x_3, x_4, \ldots, x_j)} \\
\; & \qquad\qquad \times
\frac{\LM_{\alpha/\alpha_j, \unnested}(2) \, \PartF_{\alpha/\alpha_j}(\varphi^{-1}(D_T^L); x_1, \varphi^{-1}(\eta(T)), x_{j+1}, \ldots, x_{2N})}{\LF_{\unnested}^{(N-j/2+1)}(\varphi^{-1}(D_T^L); x_1, \varphi^{-1}(\eta(T)), x_{j+1}, \ldots, x_{2N})} \bigg]
&& \textnormal{[by~\eqref{eqn::induction_hypo}]}  \\
= & \; \LM_{\alpha,\unnested}(2) \, \sum_{j\in\corrind_{\alpha}} \, 
\E\big[\one \{ \eta(T) \in (\realpt_j, \realpt_{j+1})\} \, M_T \big]
&& \textnormal{[by~(\ref{eqn::mart_cascade},~\ref{eqn::meander_cascade})~\&~CI]} \\
= & \; \LM_{\alpha,\unnested}(2) \, \frac{\PartF_{\alpha}(\realpt_1, \ldots, \realpt_{2N})}{\LF_{\unnested}^{(N)}(\realpt_1, \ldots, \realpt_{2N})} . 
&& \textnormal{[by~\eqref{eqn::crossingproba_unnested_cascade}]} 
\end{align*}
This gives the sought identification~\eqref{eqn::crossingproba_unnested_goal} and finishes the induction step.
\end{proof}

To complete the proof of Proposition~\ref{prop::crossingproba_unnested}, it remains to verify the properties~($\corrind$) and~($\corrind^c$) 
of the martingale $M$ in the limit as $t \nearrow T$. 
The latter is the topic of Appendix~\ref{appendix_technical}, while the former we prove below in Lemma~\ref{lem::mart_cascade} after two preparatory results (Lemmas~\ref{lem::ppf_cascade_mart} and~\ref{lem::totalpartition_cascade_mart}).

\begin{lemma} \label{lem::ppf_cascade_mart}
Fix $\kappa\in (4,6]$ and $(\realpt_1,\ldots,\realpt_{2N})\in \chamber_{2N}$, suppose that $\{1,2\}\in\alpha \in \LP_N$, 
and fix an index $j\in\corrind_{\alpha}$.
Let $\smash{\hat{\eta}}$ be the $\SLE_{\kappa}$ curve in $\HH$ from $\realpt_2$ to $\realpt_1$, and let $\smash{\hat{T}}$ be the first time when it swallows $\realpt_1$ or $\realpt_3$.
Let $\smash{(\hat{W}_t \colon 0\le t\le \hat{T})}$ be the Loewner driving function of $\smash{\hat{\eta}}$, 
and $\smash{(\hat{g}_t \colon 0\le t\le \hat{T})}$ the corresponding conformal maps. 
Finally, denote by $\smash{\hat{D}^R_{\hat{T}}}$ 
\textnormal{(}resp.~$\smash{\hat{D}^L_{\hat{T}}}$\textnormal{)}
the connected component of $\smash{\HH\setminus\hat{\eta}[0,\hat{T}]}$ with $\realpt_3$ \textnormal{(}resp.~$\realpt_1$\textnormal{)} on its boundary. 
Then, almost surely on the event $\smash{\{\hat{\eta}(\hat{T})\in(\realpt_j, \realpt_{j+1})\}}$, we have 
\begin{align} \label{eqn::ppf_cascade_mart}
\begin{split}
\; & \lim_{t\to \hat{T}} \Big(\prod_{i=3}^{2N}\hat{g}_t'(\realpt_i)^{h(\kappa)}\Big) \; \frac{\PartF_{\alpha}(\hat{g}_t(\realpt_1), \hat{W}_t, \hat{g}_t(\realpt_3), \hat{g}_t(\realpt_4), \ldots, \hat{g}_t(\realpt_{2N}))}{\PartF_{\vcenter{\hbox{\includegraphics[scale=0.2]{figures/link-0.pdf}}}}(\hat{g}_t(\realpt_1), \hat{W}_t)} \\
=\; & \PartF_{\alpha_j}(\hat{D}^R_{\hat{T}}; \realpt_3, \realpt_4, \ldots, \realpt_j) \; \frac{\PartF_{\alpha/\alpha_j}(\hat{D}^L_{\hat{T}};\realpt_1, \hat{\eta}(\hat{T}), \realpt_{j+1}, \realpt_{j+2}, \ldots, \realpt_{2N})}{\PartF_{\vcenter{\hbox{\includegraphics[scale=0.2]{figures/link-0.pdf}}}}(\hat{D}^L_{\hat{T}}; \realpt_1, \hat{\eta}(\hat{T}))}. 
\end{split}
\end{align}
\end{lemma}

\begin{proof}
We use the so-called ``cascade relation" for pure partition functions, see~\cite[Section~6]{Wu:Convergence_of_the_critical_planar_ising_interfaces_to_hypergeometric_SLE}. 
With $\{1,2\}\in\alpha$, this relation 
holds for the $\SLE_{\kappa}$ curve $\smash{\hat{\eta}}$ in any polygon $(\Omega; x_1, \ldots, x_{2N})$ from $x_2$ to $x_1$: 
\begin{align} \label{eqn::ppf_cascade}
\frac{\PartF_{\alpha}(\Omega; x_1, \ldots, x_{2N})}{\PartF_{\vcenter{\hbox{\includegraphics[scale=0.2]{figures/link-0.pdf}}}}(\Omega; x_1,x_2)} 
= \; & \E\Big[\one \{\LE_{\alpha}(\hat{\eta})\} \, 
\PartF_{\alpha^{R, 1}}(\hat{D}^{R,1}; \ldots)\times\cdots\times\PartF_{\alpha^{R,r}}(\hat{D}^{R,r}; \ldots)\Big] ,
\end{align}
where 
\begin{itemize}[leftmargin=2em]
\item $\smash{\LE_{\alpha}(\hat{\eta})}$ is the event that $\smash{\hat{\eta}}$ is \emph{allowed} by $\alpha$, that is, for all $\{a,b\}\in\alpha$ such that $\{a,b\}\neq\{1,2\}$, the points $x_a$ and $x_b$ lie on the boundary of the same connected component of $\smash{\Omega\setminus\hat{\eta}}$; 

\item on the event $\smash{\LE_{\alpha}(\hat{\eta})}$, 
from left to right $\smash{\hat{D}^{R, 1}}, \ldots, \smash{\hat{D}^{R, r}}$
are those the connected components of $\smash{\Omega\setminus\hat{\eta}}$ that have some of the points $x_3, \ldots, x_{2N}$ on the boundary; and

\item the link pattern $\alpha$ is divided into sub-link patterns corresponding to the marked points on the boundaries of the components $\smash{\hat{D}^{R,1}}, \ldots, \smash{\hat{D}^{R, r}}$, 
which after relabeling the indices we denote by $\alpha^{R, 1}, \ldots, \alpha^{R, r}$.
\end{itemize}
Using the cascade relation~\eqref{eqn::ppf_cascade} conditioned on 
the initial segment $\smash{\hat{\eta}[0,t]}$ together with the domain Markov property of the SLE curve $\smash{\hat{\eta}}$ and the conformal covariance~\eqref{eqn::COV}, we find that
\begin{align*}
\; & \E \Big[\one \{\LE_{\alpha}(\hat{\eta})\} \, \PartF_{\alpha^{R, 1}}(\hat{D}^{R,1}; \ldots)\times\cdots\times\PartF_{\alpha^{R,r}}(\hat{D}^{R,r}; \ldots) \cond \hat{\eta}[0,t]\Big] \\
= \; & \Big(\prod_{i=3}^{2N}\hat{g}_t'(\realpt_i)^{h(\kappa)}\Big) \; \frac{\PartF_{\alpha}(\hat{g}_t(\realpt_1), \hat{W}_t, \hat{g}_t(\realpt_3), \ldots, \hat{g}_t(\realpt_{2N}))}{\PartF_{\vcenter{\hbox{\includegraphics[scale=0.2]{figures/link-0.pdf}}}}(\hat{g}_t(\realpt_1), \hat{W}_t)} , \qquad t<\hat{T} . 
\end{align*}
On the event $\smash{\{\hat{\eta}(\hat{T})\in (\realpt_j, \realpt_{j+1})\}}$, we have $\smash{\hat{D}^{R,1}=\hat{D}^R_{\hat{T}}}$ and $\alpha^{R,1}=\alpha_j$. 
Hence, we obtain
\begin{align} \label{eqn::ppf_cascade_aux1}
\begin{split}
 & \;  \lim_{t\to \hat{T}} \Big(\prod_{i=3}^{2N}\hat{g}_t'(\realpt_i)^{h(\kappa)}\Big) \; \frac{\PartF_{\alpha}(\hat{g}_t(\realpt_1), \hat{W}_t, \hat{g}_t(\realpt_3), \ldots, \hat{g}_t(\realpt_{2N}))}{\PartF_{\vcenter{\hbox{\includegraphics[scale=0.2]{figures/link-0.pdf}}}}(\hat{g}_t(\realpt_1), \hat{W}_t)} \\
= & \; \PartF_{\alpha_j}(\hat{D}^R_{\hat{T}}; \realpt_3, \realpt_4, \ldots, \realpt_j) \; 
\E \Big[\one \{\LE_{\alpha}(\hat{\eta})\} \, 
\PartF_{\alpha^{R, 2}}(\hat{D}^{R,2}; \ldots)\times\cdots\times\PartF_{\alpha^{R,r}}(\hat{D}^{R,r}; \ldots)\cond \hat{\eta}[0,\hat{T}]\Big].  
\end{split}
\end{align}
Now, $\smash{(\hat{\eta}(t) \colon t\ge \hat{T})}$ given $\smash{\hat{\eta}[0,\hat{T}]}$ has the law of the $\SLE_{\kappa}$ curve in $\smash{\hat{D}^L_{\hat{T}}}$ from $\smash{\hat{\eta}(\hat{T})}$ to $\realpt_1$. 
Applying the cascade relation~\eqref{eqn::ppf_cascade} to the curve $\smash{(\hat{\eta}(t) \colon t\ge \hat{T})}$ in $\smash{\hat{D}^L_{\hat{T}}}$, 
together with the Markov property of the $\SLE_\kappa$ curve $\smash{\hat{\eta}}$ and the conformal covariance~\eqref{eqn::COV}, we have
\begin{align*}
\E \Big[\one \{\LE_{\alpha}(\hat{\eta})\} \, 
\PartF_{\alpha^{R, 2}}(\hat{D}^{R,2}; \ldots)\times\cdots\times\PartF_{\alpha^{R,r}}(\hat{D}^{R,r}; \ldots)\cond \hat{\eta}[0,\hat{T}]\Big] 
=& \;  \frac{\PartF_{\alpha/\alpha_j}(\hat{D}^L_{\hat{T}};\realpt_1, \hat{\eta}(\hat{T}),  \realpt_{j+1}, \ldots, \realpt_{2N})}{\PartF_{\vcenter{\hbox{\includegraphics[scale=0.2]{figures/link-0.pdf}}}}(\hat{D}^L_{\hat{T}}; \realpt_1, \hat{\eta}(\hat{T}))}. 
\end{align*} 
Plugging this into~\eqref{eqn::ppf_cascade_aux1}, we obtain the asserted identity~\eqref{eqn::ppf_cascade_mart}. 
\end{proof}

\begin{lemma} \label{lem::totalpartition_cascade_mart}
Assume the same setup as in Lemma~\ref{lem::ppf_cascade_mart} and fix $\kappa=16/3$.  
Suppose that the index $j \in \{4,6,\ldots, 2N\}$ is even. 
Then, almost surely on the event $\smash{\{\hat{\eta}(\hat{T})\in(\realpt_j, \realpt_{j+1})\}}$, we have 
\begin{align} \label{eqn::totalpartition_cascade_mart}
\begin{split}
& \; \lim_{t\to \hat{T}}\Big(\prod_{i=3}^{2N}\hat{g}_t'(\realpt_i)^{1/16}\Big) \; \frac{\LF_{\unnested}^{(N)}(\hat{g}_t(\realpt_1), \hat{W}_t, \hat{g}_t(\realpt_3), \ldots, \hat{g}_t(\realpt_{2N}))}{\PartF_{\vcenter{\hbox{\includegraphics[scale=0.2]{figures/link-0.pdf}}}}(\hat{g}_t(\realpt_1), \hat{W}_t)} \\
=& \;  \LF_{\unnested}^{(j/2-1)}(\hat{D}^R_{\hat{T}}; \realpt_3, \realpt_4, \ldots, \realpt_j) \; \frac{\LF_{\unnested}^{(N-j/2+1)}(\hat{D}^L_{\hat{T}};\realpt_1, \hat{\eta}(\hat{T}), \realpt_{j+1}, \realpt_{j+2}, \ldots, \realpt_{2N})}{\PartF_{\vcenter{\hbox{\includegraphics[scale=0.2]{figures/link-0.pdf}}}}(\hat{D}^L_{\hat{T}}; \realpt_1, \hat{\eta}(\hat{T}))}. 
\end{split}
\end{align}
\end{lemma}
\begin{proof}
From Corollary~\ref{cor::linearcombination_FKIsing}, we have
\begin{align*}
\LF_{\unnested}^{(N)}=\sum_{\gamma\in\LP_N}\LM_{\gamma, \unnested}(2) \, \PartF_{\gamma}. 
\end{align*}
We will divide $\gamma\in\LP_N$ into three groups. First of all, set 
\begin{align*}
\LJ_1:=\{\gamma\in\LP_N \colon \{1,2\} \in \gamma , \,  j \in \corrind_{\gamma}\}.
\end{align*}
Next, we consider $\gamma\in\LP_N$ such that $\{2,b\}\in\gamma$ for some $b \neq 1$. With such $\gamma$, we define $\corrind_{\gamma}$ to be the set of indices $i \in \{4,5,\ldots, b-1\}$ such that $\{3,4,\ldots, i\}$ forms a sub-link pattern of $\gamma$, and we define $\gamma_i$ and $\gamma/\gamma_i$ similarly as before. 
We set
\begin{align*}
\LJ_2(b) := \big\{\gamma\in\LP_N \colon \{2,b\} \in \gamma , \, j \in\corrind_{\gamma} \big\} , \qquad \textnormal{ for } b \in \{3,5,\ldots, 2N-1\} , 
\qquad 
\LJ_2 := \bigsqcup_{b \in\{3,5,\ldots, 2N-1\}} \LJ_2(b) .
\end{align*}
Lastly, we define $\LJ_3 := \{\gamma\in\LP_N \colon j\not\in\corrind_{\gamma}\}$. 
We will treat the cases of $\LJ_1$, $\LJ_2$, and $\LJ_3$ one by one. 
\begin{enumerate}[leftmargin=*]
\item \label{item:1}
For $\gamma\in\LJ_1$, we find almost surely on the event $\smash{\{\hat{\eta}(\hat{T})\in(\realpt_j, \realpt_{j+1})\}}$ the identity 
\begin{align*} 
& \; \lim_{t\to \hat{T}}  \Big(\prod_{i=3}^{2N}\hat{g}_t'(\realpt_i)^{1/16}\Big) \; \frac{\PartF_{\gamma}(\hat{g}_t(\realpt_1), \hat{W}_t, \hat{g}_t(\realpt_3), \ldots, \hat{g}_t(\realpt_{2N}))}{\PartF_{\vcenter{\hbox{\includegraphics[scale=0.2]{figures/link-0.pdf}}}}(\hat{g}_t(\realpt_1), \hat{W}_t)} \\
= & \; \PartF_{\gamma_j}(\hat{D}^R_{\hat{T}}; \realpt_3, \ldots, \realpt_j) \;
\frac{\PartF_{\gamma/\gamma_j}(\hat{D}^L_{\hat{T}};\realpt_1, \hat{\eta}(\hat{T}),  \realpt_{j+1}, \ldots, \realpt_{2N})}{\PartF_{\vcenter{\hbox{\includegraphics[scale=0.2]{figures/link-0.pdf}}}}(\hat{D}^L_{\hat{T}};\realpt_1, \hat{\eta}(\hat{T}))} .
&& \textnormal{[by Lemma~\ref{lem::ppf_cascade_mart}]} 
\end{align*}

\item \label{item:2}
For $\gamma\in\LJ_2$, fix some $b\in\{3,5,\ldots 2N-1\}$ such that $\gamma\in\LJ_2(b)$. 
Let $\smash{\tilde{\eta}}$ be the $\SLE_{16/3}$ curve in $\HH$ from $\realpt_2$ to $\realpt_b$, and let $\smash{\tilde{T}}$ be the first time when it swallows $\realpt_1$ or $\realpt_3$. 
Let $\smash{(\tilde{W}_t \colon 0\le t\le \smash{\tilde{T}})}$ be the Loewner driving function of $\smash{\tilde{\eta}}$ and $\smash{(\tilde{g}_t \colon 0\le t\le \smash{\tilde{T}})}$ the corresponding conformal maps. 
Denote by $\smash{\tilde{D}^R_{\tilde{T}}}$ 
(resp.~$\smash{\tilde{D}^L_{\tilde{T}}}$)
the connected component of $\smash{\HH\setminus\tilde{\eta}[0,\tilde{T}]}$ with $\realpt_3$
(resp.~$\realpt_1$) on its boundary. 
Using a similar analysis as in Lemma~\ref{lem::ppf_cascade_mart}, almost surely on the event $\smash{\{\tilde{\eta}(\tilde{T})\in(\realpt_j, \realpt_{j+1})\}}$, we have 
\begin{align*}
\; & \lim_{t\to \tilde{T}}\Big(\prod_{i \notin\{2,b\}}\tilde{g}_t'(\realpt_i)^{1/16}\Big) \; \frac{\PartF_{\gamma}(\tilde{g}_t(\realpt_1), \tilde{W}_t, \tilde{g}_t(\realpt_3), \ldots, \tilde{g}_t(\realpt_{2N}))}{\PartF_{\vcenter{\hbox{\includegraphics[scale=0.2]{figures/link-0.pdf}}}}(\tilde{W}_t, \tilde{g}_t(\realpt_b))} \\
=\; &  \PartF_{\gamma_j}(\tilde{D}^R_{\tilde{T}}; \realpt_3, \ldots, \realpt_j)\;\frac{\PartF_{\gamma/\gamma_j}(\tilde{D}^L_{\tilde{T}};\realpt_1, \tilde{\eta}(\tilde{T}),  \realpt_{j+1}, \ldots, \realpt_{2N})}{\PartF_{\vcenter{\hbox{\includegraphics[scale=0.2]{figures/link-0.pdf}}}}(\tilde{D}^L_{\tilde{T}}; \tilde{\eta}(\tilde{T}), \realpt_b)}. 
\end{align*}
Note that, on the event $\smash{\{\tilde{\eta}(\tilde{T})\in(\realpt_j, \realpt_{j+1})\}}$, we also have
\begin{align} \label{eqn::ppf_cascadeJ2_aux2}
\begin{cases}
\underset{t\to \tilde{T}}{\lim} \;
\tilde{g}'_t(\realpt_1)=\tilde{g}'_{\tilde{T}}(\realpt_1) , \\[.5em]
\underset{t\to \tilde{T}}{\lim} \;
\tilde{g}'_t(\realpt_b)=\tilde{g}'_{\tilde{T}}(\realpt_b) , \\[.5em] 
\underset{t\to \tilde{T}}{\lim} \;
\PartF_{\vcenter{\hbox{\includegraphics[scale=0.2]{figures/link-0.pdf}}}}(\tilde{W}_t, \tilde{g}_t(\realpt_b))=\PartF_{\vcenter{\hbox{\includegraphics[scale=0.2]{figures/link-0.pdf}}}}(\tilde{W}_{\tilde{T}}, \tilde{g}_{\tilde{T}}(\realpt_b)).
\end{cases}
\end{align}
Therefore, we obtain 
\begin{align*}
& \; \lim_{t\to \tilde{T}}\Big(\prod_{i=3}^{2N}\tilde{g}_t'(\realpt_i)^{1/16}\Big) \; \frac{\PartF_{\gamma}(\tilde{g}_t(\realpt_1), \tilde{W}_t, \tilde{g}_t(\realpt_3), \ldots, \tilde{g}_t(\realpt_{2N}))}{\PartF_{\vcenter{\hbox{\includegraphics[scale=0.2]{figures/link-0.pdf}}}}(\tilde{g}_t(\realpt_1), \tilde{W}_t)} \\
= & \; \lim_{t\to \tilde{T}}\Big(\prod_{i \notin\{2,b\}}\tilde{g}_t'(\realpt_i)^{1/16}\Big) \; \frac{\PartF_{\gamma}(\tilde{g}_t(\realpt_1), \tilde{W}_t, \tilde{g}_t(\realpt_3), \ldots, \tilde{g}_t(\realpt_{2N}))}{\PartF_{\vcenter{\hbox{\includegraphics[scale=0.2]{figures/link-0.pdf}}}}(\tilde{W}_t, \tilde{g}_t(\realpt_b))} 
\; \frac{\tilde{g}'_t(\realpt_b)^{1/16} \; \PartF_{\vcenter{\hbox{\includegraphics[scale=0.2]{figures/link-0.pdf}}}}(\tilde{W}_t, \tilde{g}_t(\realpt_b))}{\tilde{g}'_t(\realpt_1)^{1/16} \; \PartF_{\vcenter{\hbox{\includegraphics[scale=0.2]{figures/link-0.pdf}}}}(\tilde{g}_t(\realpt_1), \tilde{W}_t)}
\\
= & \; \PartF_{\gamma_j}(\tilde{D}^R_{\tilde{T}}; \realpt_3, \ldots, \realpt_j) \; 
\frac{\PartF_{\gamma/\gamma_j}(\tilde{D}^L_{\tilde{T}};\realpt_1, \tilde{\eta}(\tilde{T}),  \realpt_{j+1}, \ldots, \realpt_{2N})}{\PartF_{\vcenter{\hbox{\includegraphics[scale=0.2]{figures/link-0.pdf}}}}(\tilde{D}^L_{\tilde{T}}; \tilde{\eta}(\tilde{T}), \realpt_b)} \;  \frac{\tilde{g}'_{\tilde{T}}(\realpt_b)^{1/16} \; \PartF_{\vcenter{\hbox{\includegraphics[scale=0.2]{figures/link-0.pdf}}}}(\tilde{W}_{\tilde{T}}, \tilde{g}_{\tilde{T}}(\realpt_b))}{\tilde{g}'_{\tilde{T}}(\realpt_1)^{1/16} \; \PartF_{\vcenter{\hbox{\includegraphics[scale=0.2]{figures/link-0.pdf}}}}(\tilde{g}_{\tilde{T}}(\realpt_1), \tilde{W}_{\tilde{T}})} .
&& \textnormal{[by~\eqref{eqn::ppf_cascadeJ2_aux2}]} \\
= & \; \PartF_{\gamma_j}(\tilde{D}^R_{\tilde{T}}; \realpt_3, \ldots, \realpt_j)\;\frac{\PartF_{\gamma/\gamma_j}(\tilde{D}^L_{\tilde{T}};\realpt_1, \tilde{\eta}(\tilde{T}),  \realpt_{j+1}, \ldots, \realpt_{2N})}{\PartF_{\vcenter{\hbox{\includegraphics[scale=0.2]{figures/link-0.pdf}}}}(\tilde{D}^L_{\tilde{T}}; \realpt_1, \tilde{\eta}(\tilde{T}))} , 
&& \textnormal{[by~\eqref{eq::equals_one}]} 
\end{align*}
using also the observation 
\begin{align} \label{eq::equals_one}
\frac{\PartF_{\vcenter{\hbox{\includegraphics[scale=0.2]{figures/link-0.pdf}}}}(\tilde{D}^L_{\tilde{T}}; \realpt_1, \tilde{\eta}(\tilde{T}))}{\PartF_{\vcenter{\hbox{\includegraphics[scale=0.2]{figures/link-0.pdf}}}}(\tilde{D}^L_{\tilde{T}};\tilde{\eta}(\tilde{T}),  \realpt_b)}\; \frac{\tilde{g}'_{\tilde{T}}(\realpt_b)^{1/16} \; \PartF_{\vcenter{\hbox{\includegraphics[scale=0.2]{figures/link-0.pdf}}}}(\tilde{W}_{\tilde{T}}, \tilde{g}_{\tilde{T}}(\realpt_b))}{\tilde{g}'_{\tilde{T}}(\realpt_1)^{1/16} \; \PartF_{\vcenter{\hbox{\includegraphics[scale=0.2]{figures/link-0.pdf}}}}(\tilde{g}_{\tilde{T}}(\realpt_1), \tilde{W}_{\tilde{T}})}=1 .
\end{align}
As the law of $\smash{( \hat{\eta}(t) \colon t\le \hat{T})}$ conditional on $\smash{\{\hat{\eta}(\hat{T})\in(\realpt_j, \realpt_{j+1})\}}$ is absolutely continuous 
to that of $\smash{( \tilde{\eta}(t) \colon t\le \tilde{T})}$ conditional on $\smash{\{\tilde{\eta}(\tilde{T})\in(\realpt_j, \realpt_{j+1})\}}$, 
the above relation also holds for $\smash{\hat{\eta}}$
--- see, e.g.,~\cite{Schramm-Wilson:SLE_coordinate_changes}.

\item \label{item:3}
For $\gamma\in\LJ_3$, 
Item~\ref{item::mart_vanish2} of Proposition~\ref{prop::mart_vanish_combined} 
gives that almost surely on the event $\smash{\{\hat{\eta}(\hat{T})\in(\realpt_j, \realpt_{j+1})\}}$, 
\begin{align*}
\lim_{t\to\hat{T}}\frac{\PartF_{\gamma}(\hat{g}_t(\realpt_1), \hat{W}_t, \hat{g}_t(\realpt_3), \ldots, \hat{g}_t(\realpt_{2N}))}{\LF_{\unnested}^{(N)}(\hat{g}_t(\realpt_1), \hat{W}_t, \hat{g}_t(\realpt_3), \ldots, \hat{g}_t(\realpt_{2N}))}=0. 
\end{align*}
\end{enumerate}
Combining Cases~\ref{item:1}--\ref{item:3}, 
we see that
almost surely on the event $\smash{\{\hat{\eta}(\hat{T})\in(\realpt_j, \realpt_{j+1})\}}$, we have
\begin{align*}
1 = & \; \lim_{t\to\hat{T}}\frac{ \underset{\gamma\in\LJ_1\cup\LJ_2}{\sum} \LM_{\gamma, \unnested}(2) \, \PartF_{\gamma}(\hat{g}_t(\realpt_1), \hat{W}_t, \hat{g}_t(\realpt_3), \ldots, \hat{g}_t(\realpt_{2N}))}{\LF_{\unnested}^{(N)}(\hat{g}_t(\realpt_1), \hat{W}_t, \hat{g}_t(\realpt_3), \ldots, \hat{g}_t(\realpt_{2N}))} \\
& \; + \lim_{t\to\hat{T}}\frac{\underset{\gamma\in\LJ_3}{\sum} \LM_{\gamma, \unnested}(2) \, \PartF_{\gamma}(\hat{g}_t(\realpt_1), \hat{W}_t, \hat{g}_t(\realpt_3), \ldots, \hat{g}_t(\realpt_{2N}))}{\LF_{\unnested}^{(N)}(\hat{g}_t(\realpt_1), \hat{W}_t, \hat{g}_t(\realpt_3), \ldots, \hat{g}_t(\realpt_{2N}))} \\
= & \; \frac{\underset{\gamma\in\LJ_1\cup\LJ_2}{\sum} \LM_{\gamma_j, \unnested}(2) \, \PartF_{\gamma_j}(\hat{D}^R_{\hat{T}}; \realpt_3, \ldots, \realpt_j)\;\LM_{\gamma/\gamma_j, \unnested}(2) \, \frac{\PartF_{\gamma/\gamma_j}(\hat{D}^L_{\hat{T}};\realpt_1, \hat{\eta}(\hat{T}),  \realpt_{j+1}, \ldots, \realpt_{2N})}{\PartF_{\vcenter{\hbox{\includegraphics[scale=0.2]{figures/link-0.pdf}}}}(\hat{D}^L_{\hat{T}}; \realpt_1, \hat{\eta}(\hat{T}))}}{\underset{t\to\hat{T}}{\lim} \Big(\prod_{i=3}^{2N}\hat{g}'_t(\realpt_i)^{1/16}\Big)\;\frac{\LF_{\unnested}^{(N)}(\hat{g}_t(\realpt_1), \hat{W}_t, \hat{g}_t(\realpt_3), \ldots, \hat{g}_t(\realpt_{2N}))}{\PartF_{\vcenter{\hbox{\includegraphics[scale=0.2]{figures/link-0.pdf}}}}(\hat{g}_t(\realpt_1), \hat{W}_t)}} 
&& \textnormal{[by~\eqref{eqn::meander_cascade}~\&~\ref{item:1}--\ref{item:3}]} \\
= & \; \frac{\LF_{\unnested}^{(j/2-1)}(\hat{D}^R_{\hat{T}}; \realpt_3, \ldots, \realpt_j)\;\frac{\LF_{\unnested}^{(N-j/2+1)}(\hat{D}^L_{\hat{T}};\realpt_1, \hat{\eta}(\hat{T}),  \realpt_{j+1}, \ldots, \realpt_{2N})}{\PartF_{\vcenter{\hbox{\includegraphics[scale=0.2]{figures/link-0.pdf}}}}(\hat{D}^L_{\hat{T}}; \realpt_1, \hat{\eta}(\hat{T}))}}{\underset{t\to\hat{T}}{\lim} \Big(\prod_{i=3}^{2N}\hat{g}'_t(\realpt_i)^{1/16}\Big)\;\frac{\LF_{\unnested}^{(N)}(\hat{g}_t(\realpt_1), \hat{W}_t, \hat{g}_t(\realpt_3), \ldots, \hat{g}_t(\realpt_{2N}))}{\PartF_{\vcenter{\hbox{\includegraphics[scale=0.2]{figures/link-0.pdf}}}}(\hat{g}_t(\realpt_1), \hat{W}_t)}} .
&&\textnormal{[by~Cor.~\ref{cor::linearcombination_FKIsing}]}
\end{align*}
This gives the asserted identity~\eqref{eqn::totalpartition_cascade_mart} and completes the proof. 
\end{proof}

\begin{lemma} \label{lem::mart_cascade}
Assume the same setup as in the proof of Proposition~\ref{prop::crossingproba_unnested}.
Suppose that $j \in \corrind_{\alpha}$. 
Then, on the event $\{\eta(T) \in (\realpt_j, \realpt_{j+1}) \}$, the relation~\eqref{eqn::mart_cascade} holds almost surely. 
\end{lemma}

\begin{proof}
In the notation of Lemma~\ref{lem::ppf_cascade_mart}, 
on the event $\smash{\{\hat{\eta}(\hat{T}) \in (\realpt_j, \realpt_{j+1})\}}$, Eq.~(\ref{eqn::ppf_cascade_mart},~\ref{eqn::totalpartition_cascade_mart}) give  
almost surely 
\begin{align*}
\; & \lim_{t\to \hat{T}}\frac{\PartF_{\alpha}(\hat{g}_t(\realpt_1), \hat{W}_t, \hat{g}_t(\realpt_3), \ldots, \hat{g}_t(\realpt_{2N}))}{\LF_{\unnested}^{(N)}(\hat{g}_t(\realpt_1), \hat{W}_t, \hat{g}_t(\realpt_3), \ldots, \hat{g}_t(\realpt_{2N}))}\\
=\; &  \frac{\PartF_{\alpha_j}(\hat{D}^R_{\hat{T}}; \realpt_3, \realpt_4, \ldots, \realpt_j)}{\LF_{\unnested}^{(j/2-1)}(\hat{D}^R_{\hat{T}}; \realpt_3, \realpt_4, \ldots, \realpt_j)}\;\frac{\PartF_{\alpha/\alpha_j}(\hat{D}^L_{\hat{T}};\realpt_1, \hat{\eta}(\hat{T}), \realpt_{j+1}, \realpt_{j+2}, \ldots, \realpt_{2N})}{\LF_{\unnested}^{(N-j/2+1)}(\hat{D}^L_{\hat{T}};\realpt_1, \hat{\eta}(\hat{T}), \realpt_{j+1}, \realpt_{j+2}, \ldots, \realpt_{2N})}. 
\end{align*}
Since the law of $( \eta(t) \colon t\le T)$ conditional on $\{\eta(T)\in (\realpt_j, \realpt_{j+1})\}$ is absolutely continuous with respect to the law of $\smash{( \hat{\eta}(t) \colon t\le \hat{T})}$ conditional on $\smash{\{\hat{\eta}(\hat{T})\in (\realpt_j, \realpt_{j+1})\}}$, 
this gives~\eqref{eqn::mart_cascade}
--- see, e.g.,~\cite{Schramm-Wilson:SLE_coordinate_changes}.
\end{proof}

\subsection{Proof of Theorem~\ref{thm::FKIsing_crossingproba}: The general case}
\label{subsec::crossingproba_general}
The goal of this section is to prove Theorem~\ref{thm::FKIsing_crossingproba} with a general boundary condition $\beta\in\LP_N$, using Proposition~\ref{prop::crossingproba_unnested}. 
The key is the following observation for the discrete models --- 
which holds, in fact, for all random-cluster models 
with cluster-weight $q>0$ and edge-weight being the self-dual value~\eqref{eq: pselfdual}.

\begin{proposition}\label{prop::crossingproba_comparison}
	Consider the 
	random-cluster model on the primal polygon $(\Omega; x_1, \ldots, x_{2N})$ with cluster-weight $q>0$ 
and edge-weight 
\begin{align} \label{eq: pselfdual}
p=\frac{\sqrt{q}}{1+\sqrt{q}}.
\end{align}
The random connectivity $\conn$ in this model satisfies the identity 
	\begin{align}\label{eqn::crossingproba_comparison}
		\PP_{\beta}[\conn=\alpha] 
		= \frac{\frac{\LM_{\alpha,\beta}(q)}{\LM_{\alpha,\unnested}(q)}\, \PP_{\unnested}[\conn=\alpha]}{\underset{\gamma\in\LP_N} {\sum}\frac{\LM_{\gamma,\beta}(q)}{\LM_{\gamma,\unnested}(q)}\, \PP_{\unnested}[\conn=\gamma]} , 
		\qquad \textnormal{for all } \, \alpha,\beta\in\LP_N .
	\end{align}
\end{proposition}

\begin{proof}
	We denote by $\LW$ the set of random-cluster configurations that are wired on the boundary arcs $(x_{2r-1} \, x_{2r})$ for $1\le r \le N$, namely, 
	\begin{align*}
		\LW := \Big\{\omega=(\omega_e)_{e\in E(\Omega)}\in \{0,1\}^{E(\Omega)} \; \colon \; \omega_e=1 \textnormal{ for all } e \in \bigcup_{r=1}^N (x_{2r-1} \, x_{2r}) \Big\} .
	\end{align*}
	Also, we denote by $\LN(\omega)$ the number of loops in the loop representation of $\omega$ (recall Figure~\ref{fig::loop_representation}). 
Thanks to the hypothesis~\eqref{eq: pselfdual}, 
a standard argument (see, e.g.,~\cite[Proposition~3.17]{DCS:Conformal_invariance_of_lattice_models}) shows that 
	\begin{align} \label{eqn::proba_loop_repre}
		\PP_{\beta}[\omega] = \frac{\sqrt{q}^{\LN(\omega)}\LM_{\conn(\omega),\beta}(q)}{\underset{\varpi\in \LW}{\sum} \sqrt{q}^{\LN(\varpi)}\LM_{\conn(\varpi),\beta}(q)}, \qquad \textnormal{for all }\omega\in\LW.
	\end{align} 
On the one hand, identity~\eqref{eqn::proba_loop_repre} gives
\begin{align} \label{eqn::cro_pro_discrete_loop_repre}
\PP_{\beta}[\conn = \alpha] 
= \frac{\LM_{\alpha,\beta}(q) \underset{\omega\in \LW({\alpha})}{\sum} \sqrt{q}^{\LN(\omega)}}{\underset{\gamma\in\LP_{N}}{\sum} \LM_{\gamma,\beta}(q) \underset{\varpi \in \LW({\gamma})}{\sum} \sqrt{q}^{\LN(\varpi)}} , 
\qquad \textnormal{for all } \, \alpha,\beta\in\LP_N ,
\end{align}
where $\LW({\alpha}) := \{\omega\in \LW \colon \conn(\omega)=\alpha\}$.
On the other hand, applying~\eqref{eqn::cro_pro_discrete_loop_repre} to the right-hand side (RHS) of~\eqref{eqn::crossingproba_comparison}, we find that
\begin{align*}
\textnormal{RHS of~\eqref{eqn::crossingproba_comparison}} 
= \; & \left( \frac{\LM_{\alpha,\beta}(q) \underset{\omega\in\LW({\alpha})}{\sum} \sqrt{q}^{\LN(\omega)}}{\underset{\delta \in \LP_N}{\sum} \LM_{\delta,\unnested}(q) \underset{\upsilon\in\LW({\delta})}{\sum} \sqrt{q}^{\LN(\upsilon)}} \right)
\left( \underset{\gamma\in\LP_N}{\sum} \frac{\LM_{\gamma,\beta}(q) \underset{\varpi \in \LW({\gamma})}{\sum} \sqrt{q}^{\LN(\varpi)}}{\underset{\delta \in \LP_N}{\sum} \LM_{\delta,\unnested}(q) \underset{\upsilon\in\LW({\delta})}{\sum} \sqrt{q}^{\LN(\upsilon)}} \right)^{-1} \\
= \; & \frac{\LM_{\alpha,\beta}(q) \underset{\omega\in \LW({\alpha})}{\sum} \sqrt{q}^{\LN(\omega)}}{\underset{\gamma\in\LP_N}{\sum} \LM_{\gamma,\beta}(q)\underset{\varpi \in  \LW({\gamma})}{\sum} \sqrt{q}^{\LN(\varpi)}}  
\; = \; \PP_{\beta}[\conn = \alpha] ,
 &&\textnormal{[by~\eqref{eqn::cro_pro_discrete_loop_repre}]}
\end{align*}
which is the left-hand side of the asserted identity~\eqref{eqn::crossingproba_comparison}. 
\end{proof}

The general case in Theorem~\ref{thm::FKIsing_crossingproba} follows now with little effort.

\begin{proof}[Proof of Theorem~\ref{thm::FKIsing_crossingproba}]
For any $\alpha,\beta\in\LP_N$, we have 
\begin{align*} 
\lim_{\delta\to 0}\PP_{\beta}^{\delta}[\FKconn^{\delta}=\alpha]
= & \; \frac{\frac{\LM_{\alpha,\beta}(2)}{\LM_{\alpha,\unnested}(2)} \; \underset{\delta\to 0}{\lim} \; \PP_{\unnested}^{\delta}[\FKconn^{\delta}=\alpha]}{\underset{\gamma\in\LP_N}{\sum} \frac{\LM_{\gamma,\beta}(2)}{\LM_{\gamma,\unnested}(2)}\; \underset{\delta\to 0}{\lim} \; \PP_{\unnested}^{\delta}[\FKconn^{\delta}=\gamma]} 
&&\textnormal{[by~Prop.~\ref{prop::crossingproba_comparison} with $q=2$]} \\
=& \; \frac{\LM_{\alpha,\beta}(2) \, \frac{\PartF_{\alpha}(\Omega; x_1, \ldots, x_{2N})}{\LF_{\unnested}^{(N)}(\Omega; x_1, \ldots, x_{2N})}}{\underset{\gamma\in\LP_N}{\sum} \LM_{\gamma,\beta}(2) \, \frac{\PartF_{\gamma}(\Omega; x_1, \ldots, x_{2N})}{\LF_{\unnested}^{(N)}(\Omega; x_1, \ldots, x_{2N})}} 
&&\textnormal{[by~Prop.~\ref{prop::crossingproba_unnested}]} \\
=& \; \LM_{\alpha,\beta}(2) \, \frac{\PartF_{\alpha}(\Omega; x_1, \ldots, x_{2N})}{\LF_{\beta}(\Omega; x_1, \ldots, x_{2N})}.
&&\textnormal{[by~Cor.~\ref{cor::linearcombination_FKIsing}]}
\end{align*}
This completes the proof. 
\end{proof}

\begin{remark}
It follows from Theorems~\ref{thm::FKIsing_Loewner}~\&~\ref{thm::FKIsing_crossingproba} that 
the so-called ``global'' multiple $\SLE_{16/3}$ associated to $\alpha$, as defined in~\textnormal{\cite[Proposition~1.4]{BPW:On_the_uniqueness_of_global_multiple_SLEs}},  
is the same as the so-called ``local'' multiple $\SLE_{16/3}$ associated to $\alpha$.
We leave the details to a dedicated reader. 
\end{remark}


\appendix

\section{Combinatorial lemmas for Section~\ref{sec::FKIsing_Loewner} --- details for Proposition~\ref{prop::totalpartition_observable}}
\label{appendix_aux}
Here, we fill in the details to finish the proof of Proposition~\ref{prop::totalpartition_observable}.
We use the notation from Section~\ref{subsec::holo_limiting}.

\begin{lemma} \label{lem::phase_factor}
For $Q_{\beta}(\bs{\hat{\sigma}})$ appearing in~\textnormal{(\ref{eqn::determi_decomposition}, \ref{eqn::Vander})}, 
there exists a constant $\theta_{\beta} \in \{ \pm 1, \pm \ii\}$ depending only on $\beta$ such that~\eqref{eqn::Q_beta_positive} holds for all $\bs{\hat{\sigma}} = (\hat{\sigma}_2,\ldots,\hat{\sigma}_N)\in \{\pm 1\}^{N-1}$\textnormal{:}
\begin{align} \label{eqn::Q_beta_positive_again}
\frac{Q_{\beta}(\bs{\hat{\sigma}})}{\theta_{\beta}} > 0 .
\end{align}
\end{lemma}

\begin{proof}
We prove the claim~\eqref{eqn::Q_beta_positive_again} by induction on $N \geq 2$. 
For the initial case where $N=2$, we have the two boundary conditions $\vcenter{\hbox{\includegraphics[scale=0.2]{figures/link-1.pdf}}}=\{\{1,2\}, \{3,4\}\}$ and $\vcenter{\hbox{\includegraphics[scale=0.2]{figures/link-2.pdf}}}=\{\{1,4\}, \{2,3\}\}$, and\footnote{We use $\sqrt{\cdot}$ to denote the principal branch of the square root.}
\begin{align*}
Q_{\vcenter{\hbox{\includegraphics[scale=0.2]{figures/link-1.pdf}}}}(-) = \; & \frac{x_4-x_1}{\sqrt{x_4-x_1} \, \sqrt{x_{4}-x_2}} , \qquad 
\; && Q_{\vcenter{\hbox{\includegraphics[scale=0.2]{figures/link-2.pdf}}}}(-) = -\ii \, \frac{x_3-x_1}{\sqrt{x_3-x_1} \, \sqrt{x_4-x_3}}  , \\
Q_{\vcenter{\hbox{\includegraphics[scale=0.2]{figures/link-1.pdf}}}}(+) = \; & \frac{x_3-x_1}{\sqrt{x_3-x_1} \, \sqrt{x_3-x_2}} , \qquad  
\; && Q_{\vcenter{\hbox{\includegraphics[scale=0.2]{figures/link-2.pdf}}}}(+) = -\ii \, \frac{x_2-x_1}{\sqrt{x_2-x_1} \, \sqrt{x_4-x_2}} . 
\end{align*}
Thus, the claim~\eqref{eqn::Q_beta_positive_again} holds for $N=2$ with $\theta_{\vcenter{\hbox{\includegraphics[scale=0.2]{figures/link-1.pdf}}}} = 1$
and $\theta_{\vcenter{\hbox{\includegraphics[scale=0.2]{figures/link-2.pdf}}}} = -\ii$.

Next, fix $N\geq 3$ and assume that the claim~\eqref{eqn::Q_beta_positive_again} holds up to $N-1$. 
Fix $\beta\in\LP_N$. Choose an index $r \in \{2,\ldots,N\}$ such that $b_{r}=a_{r}+1$. With this choice of $r$, we have
\begin{align} \label{eqn::partial_order_locations}
s < a_{r} , \qquad \textnormal{ for all } \, s \notin \{a_{r},b_{r}\} 
\qquad \Longleftrightarrow \qquad 
s < b_{r} , \qquad \textnormal{ for all } \, s \notin \{a_{r},b_{r}\} .
\end{align}
For any $\bs{\hat{\sigma}} = (\hat{\sigma}_2,\ldots,\hat{\sigma}_N) \in \{\pm 1\}^{N-1}$, note that~\eqref{eqn::Vander} implies that
\begin{align*}
Q_{\beta} (\bs{\hat{\sigma}}) 
= \; & \underbrace{\bigg(\prod_{2\leq s \leq N} \big( y_s^{\hat{\sigma}_s, \beta} - x_1 \big) \bigg)}_{=: T_1}
\; 
\underbrace{\bigg(\frac{1}{\prod_{j \notin \{a_{r},b_{r}\}} \sqrt{y_{r}^{\hat{\sigma}_r, \beta} - x_{j} }}\bigg)}_{=: T_2} \\[.5em]
\; & \times \,
\underbrace{\bigg(\prod_{2\leq s < r}\frac{y_{r}^{\hat{\sigma}_r, \beta}-y_{s}^{\hat{\sigma}_s, \beta}}{\sqrt{ y_{s}^{\hat{\sigma}_s, \beta}-x_{a_{r}}} \, \sqrt{ y_{s}^{\hat{\sigma}_s, \beta}-x_{b_{r}} }}\bigg)}_{=: T_3}
\;
\underbrace{\bigg(\prod_{r< s \leq N}\frac{y_{s}^{\hat{\sigma}_s, \beta}-y_{r}^{\hat{\sigma}_r, \beta}}{ \sqrt{ y_{s}^{\hat{\sigma}_s, \beta}-x_{a_{r}} } \, \sqrt{ y_{s}^{\hat{\sigma}_s, \beta}-x_{b_{r}} }}\bigg)}_{=: T_4} \\[.5em]
\; & \times \,
\underbrace{\bigg(\prod_{\substack{2\leq s < t \leq N \\ s,t\neq r}}\big( y_{t}^{\hat{\sigma}_t, \beta}-y_{s}^{\hat{\sigma}_s, \beta} \big) \bigg)
\;
\bigg(\prod_{\substack{2\leq s \leq N \\ s \neq r}}  \ddot{S}_{x_1,\ldots,x_{a_r-1},x_{b_r+1},\ldots,x_{2N}}^{a_s,b_s}\big(y_s^{\hat{\sigma}_s,\beta}\big) \bigg)}_{=: T_5} ,
\end{align*}
where $y_{r}^{\hat{\sigma}_r, \beta}$ are defined in~\eqref{eq::def_y_k}. 
Let us analyze the phase factors of the terms $T_k$ for $1\leq k \leq 5$:
\begin{enumerate}[leftmargin=2em]
	\item We always have $T_1>0$.
	
	\item The phase factor of $T_2$ is independent of the choice of $\bs{\hat{\sigma}}$, due to the observation~\eqref{eqn::partial_order_locations}.
	
	\item According to the explicit formula of $T_3$, its phase factor depends on $\bs{\hat{\sigma}}$ 
	only through $(\hat{\sigma}_2,\ldots,\hat{\sigma}_{r})$. 
	The observation~\eqref{eqn::partial_order_locations} readily implies that the phase factor of $T_3$ is independent of the choice of $\hat{\sigma}_{r}$. Moreover, it is also independent of the choice of $(\hat{\sigma}_2,\ldots,\hat{\sigma}_{r-1})$, since for each $s\leq r-1$, we have
\begin{itemize}
\item if $y_{s}^{\hat{\sigma}_s, \beta} < x_{a_{r}}$, 
then 
\begin{align*}
\frac{y_{r}^{\hat{\sigma}_r, \beta}-y_{s}^{\hat{\sigma}_s, \beta}}{\sqrt{ y_{s}^{\hat{\sigma}_s, \beta} - x_{a_{r}} } \, 
\sqrt{ y_{s}^{\hat{\sigma}_s, \beta}-x_{b_{r}} }} 
= - \frac{y_{r}^{\hat{\sigma}_r, \beta}-y_{s}^{\hat{\sigma}_s, \beta}}{\sqrt{ x_{a_{r}}-y_{s}^{\hat{\sigma}_s, \beta} } \, \sqrt{x_{b_{r}}-y_{s}^{\hat{\sigma}_s, \beta} }}
\; < \; 0 ;
\end{align*}
\item if $y_{s}^{\hat{\sigma}_s, \beta} > x_{b_{r}}$, then 
\begin{align*}
\frac{y_{r}^{\hat{\sigma}_r, \beta}-y_{s}^{\hat{\sigma}_s, \beta}}{\sqrt{ y_{s}^{\hat{\sigma}_s, \beta}-x_{a_{r}} } \, \sqrt{ y_{s}^{\hat{\sigma}_s, \beta}-x_{b_{r}} }} 
= - \frac{y_{s}^{\hat{\sigma}_s, \beta}-y_{r}^{\hat{\sigma}_r, \beta}}{\sqrt{ y_{s}^{\hat{\sigma}_s, \beta}-x_{a_{r}} } \, \sqrt{ y_{s}^{\hat{\sigma}_s, \beta}-x_{b_{r}} }}
\; < \; 0 .
\end{align*}
\end{itemize}
Thus, in both cases the phase factor of $T_3$ is independent of the choice of $\bs{\hat{\sigma}}$.

\item The phase factor of $T_4$ is similarly independent of the choice of $\bs{\hat{\sigma}}$.

\item By the induction hypothesis, the phase factor of $T_5$ equals $\theta_{\beta/\{a_{r},b_{r}\}} \in \{ \pm 1, \pm \ii \}$.
\end{enumerate}
As the phase factor of $Q_{\beta}(\bs{\hat{\sigma}})$  
equals the product of the phase factors of $T_k$ for $1\leq k \leq 5$, we find a constant $\theta_\beta \in \{ \pm 1, \pm \ii \}$ depending only on $\beta$ such that~\eqref{eqn::Q_beta_positive} holds. 
This completes the induction step. 
\end{proof}

\begin{lemma} \label{lem::Cramer_decom}
There exist functions $g^{\bs{\hat{\sigma}}, \beta}(\bs{x})>0$ for $\bs{\hat{\sigma}}=(\hat{\sigma}_2, \ldots, \hat{\sigma}_N)\in \{\pm 1\}^{N-1}$ such that~\eqref{eqn::Cramer_decom} holds:
\begin{align}  \label{eqn::Cramer_decom_app}
	\frac{\det (R_{\beta}^{\bullet})}{\det (R_{\beta})}
	= \frac{\underset{\bs{\hat{\sigma}}\in\{\pm 1\}^{N-1}}{\sum} \, g^{\bs{\hat{\sigma}}, \beta}(\bs{x}) \; 
	\underset{r=2}{\overset{N}{\sum}} \; \big( y_{r}^{\hat{\sigma}_r, \beta} - x_1 \big)^{-1} }{\underset{\bs{\hat{\sigma}}\in\{\pm 1\}^{N-1}}{\sum} \; g^{\bs{\hat{\sigma}}, \beta}(\bs{x})} .
\end{align}
\end{lemma}

\begin{proof}
By Lemma~\ref{lem::phase_factor}, 
with $Q_{\beta}(\bs{\hat{\sigma}})$ defined in~\eqref{eqn::determi_decomposition}, 
we have 
\begin{align*}
g^{\bs{\hat{\sigma}}, \beta}(\bs{x}) := 
\frac{Q_{\beta}(\bs{\hat{\sigma}})}{\theta_{\beta}} > 0 . 
\end{align*}
It remains to verify~\eqref{eqn::Cramer_decom_app}.
On the one hand, the identity~\eqref{eqn::determi_decomposition} implies that
\begin{align} \label{eqn::Cramer_decom_aux1}
	\det(R_{\beta}) 
	= \theta_{\beta} \, \sum_{\bs{\hat{\sigma}}\in\{\pm 1\}^{N-1}}g^{\bs{\hat{\sigma}}, \beta}(\bs{x}).
\end{align}
On the other hand, let us compute $\det R_{\beta}^{\bullet}$. For $2\le r \le N$, we define row vectors $\smash{\bs{U}_{\beta}^{\pm,\bullet}(r)}$ of size $N$ as
\begin{align*}
\bs{U}_{\beta}^{\pm,\bullet}(r) 
:= \big( U_{\beta}^{\pm}(r,0), \, U_{\beta}^{\pm}(r,1), \, U_{\beta}^{\pm}(r,2), \, \ldots, \, U_{\beta}^{\pm}(r,N-1) \big) ,
\end{align*} 
where $U_{\beta}^{\pm}(r,n)$ are defined in~\eqref{eqn::U_elements}. We then define another row vector of size $N$ for a variable $z$ as 
\begin{align}
\bs{Z} := (1, z, z^2, \ldots, z^{N-1}) ,
\end{align}
and consider two polynomials 
$Q(z)$ and
$\smash{Q_{\beta}^\bullet (\bs{\hat{\sigma}}; z) }$,
for $\bs{\hat{\sigma}} = (\hat{\sigma}_2,\ldots,\hat{\sigma}_{N})\in\{\pm 1\}^{N-1}$, defined as 
\begin{align*}
		Q(z) := & \; \det
		\begin{pmatrix}
		{ \bs{Z} } \\
		\bs{U}_{\beta}^{+,\bullet}(2) \, + \, \bs{U}_{\beta}^{-,\bullet}(2)\\
		\cdot\\
		\cdot\\
		\cdot\\
		\bs{U}_{\beta}^{+,\bullet}(N) \, + \, \bs{U}_{\beta}^{-,\bullet}(N)
	\end{pmatrix} 
\qquad \textnormal{and} \qquad
Q_{\beta}^\bullet (\bs{\hat{\sigma}}; z) 
:= \det 
\begin{pmatrix} 
{ \bs{Z} } \\
\bs{U}_{\beta}^{\hat{\sigma}_2, \bullet}(2) \\
\cdot\\
\cdot\\
\cdot\\
\bs{U}_{\beta}^{\hat{\sigma}_{N}, \bullet}(N)
\end{pmatrix} .
\end{align*} 
Then, using the Vandermonde determinant, we find that
\begin{align} \label{eqn::Vander_Qx}
\begin{split} 
Q(z) = & \; 
\sum_{\bs{\hat{\sigma}}\in \{\pm 1\}^{N-1}} 
{ Q_{\beta}^\bullet (\bs{\hat{\sigma}}; z) }
\\
= & \; \sum_{\bs{\hat{\sigma}}\in \{\pm 1\}^{N-1}}
\prod_{2\leq r \leq N} \big( y_{r}^{\hat{\sigma}_r, \beta} - x_1 - z \big) 
\prod_{2\leq s < t \leq N} \big( y_{t}^{\hat{\sigma}_t, \beta} - y_{s}^{\hat{\sigma}_s, \beta} \big)
\prod_{2\leq r\leq N} \ddot{S}^{a_{r}, b_{r}}_{x_1, \ldots, x_{2N}} \big( y_r^{\hat{\sigma}_r, \beta} \big) .
\end{split} 
\end{align}
Combining~\eqref{eqn::Vander} and~\eqref{eqn::Vander_Qx} with the fact that $-\det(R_{\beta}^{\bullet})$ equals the coefficient of $z$ in the polynomial $Q(z)$, we finally obtain
\begin{align} \label{eqn::Cramer_decom_aux2}
	\det(R_{\beta}^{\bullet}) 
	= \; & \sum_{\bs{\hat{\sigma}}\in \{\pm 1\}^{N-1}} 
	Q_{\beta}(\bs{\hat{\sigma}}) \;
	 \sum_{r=2}^{N} \; \frac{1}{y_{r}^{\hat{\sigma}_r, \beta}-x_1} 
\; = \; \theta_{\beta} \, \sum_{\bs{\hat{\sigma}}\in \{\pm 1\}^{N-1}} g^{\bs{\hat{\sigma}}, \beta}(\bs{x}) \; 
\sum_{r=2}^{N} \; \frac{1}{y_{r}^{\hat{\sigma}_r, \beta}-x_1}  .
\end{align} 
Combining~\eqref{eqn::Cramer_decom_aux1} with~\eqref{eqn::Cramer_decom_aux2}, we obtain the sought identity~\eqref{eqn::Cramer_decom_app}. 
\end{proof}

\begin{lemma} \label{lem::connection_chi}
For the functions $g^{\bs{\hat{\sigma}}, \beta}(\bs{\hat{\sigma}})$ in Lemma~\ref{lem::Cramer_decom}, there exist functions $f_{\beta}(\bs{\hat{\sigma}})$ such that~\eqref{eqn::connection_chi} holds for all $\bs{\hat{\sigma}}=(\hat{\sigma}_2, \ldots, \hat{\sigma}_N)\in\{\pm 1\}^{N-1}$.
\end{lemma}

\begin{proof}
	Fix $\bs{\hat{\sigma}}$. 
Recall from~\eqref{eq::pair_of_1} that in the boundary condition $\beta=\{\{a_1,b_1\},\ldots,\{a_N,b_N\}\}$, we have $a_1=1$ and $b_1=2\ell$. 	
	Combining the facts that $|\theta_{\beta}|=1$ and $g^{\bs{\hat{\sigma}}, \beta}(\bs{x})>0$ with~\eqref{eqn::Vander}, we obtain
	\begin{align*} 
		g^{\bs{\hat{\sigma}}, \beta}(\bs{x}) 
		= & \; \bigg( \prod_{2\leq r \leq N} 
		\underbrace{\frac{\big| y_{r}^{\hat{\sigma}_r, \beta} - x_1 \big|}{\sqrt{\big| y_{r}^{\hat{\sigma}_r, \beta} - x_1 \big|} \, \sqrt{\big|y_r^{\hat{\sigma}_r, \beta} - x_{2\ell}\big|}}}_{=: A(\bs{\hat{\sigma}}; r)}\bigg) 
		\\
		& \; \times \bigg( 
		\prod_{2\leq s < t \leq N} \underbrace{\frac{ \big| y_{t}^{\hat{\sigma}_t, \beta} - y_{s}^{\hat{\sigma}_s, \beta} \big|}{\sqrt{\big|y_{t}^{\hat{\sigma}_t, \beta}-x_{a_s}\big|} \, \sqrt{\big| y_{t}^{\hat{\sigma}_t, \beta} - x_{b_s} \big|} \, \sqrt{\big| y_{s}^{\hat{\sigma}_s, \beta} - x_{a_{t}} \big|} \, \sqrt{\big| y_{s}^{\hat{\sigma}_s, \beta} - x_{b_{t}} \big|}}}_{=: B(\bs{\hat{\sigma}}; s,t)} \bigg) ,
	\end{align*}
where
\begin{align*} 
A(\bs{\hat{\sigma}}; r) = \; &
\chi(x_1,x_{a_r},x_{b_r},x_{2\ell})^{\frac{\hat{\sigma}_r+1}{4}}
\, 
\frac{ \sqrt{|x_{b_r}-x_1|} }{ \sqrt{|x_{b_r}-x_{2\ell}|} } , \qquad
2\leq r \leq N , \\
B(\bs{\hat{\sigma}}; s,t) = \; & \chi(x_{a_s},x_{a_t},x_{b_t},x_{b_s})^{\frac{\hat{\sigma}_s \hat{\sigma}_t + 1}{4}}
 \, \frac{1}{ \sqrt{|x_{b_t}-x_{b_s}|} \sqrt{|x_{a_t}-x_{a_s}|} } , \qquad 2\leq s < t \leq N .
\end{align*}
Therefore, we can choose
\begin{align*}
	f_{\beta}(\bs{x}) := 
	\prod_{2\leq r \leq N} \frac{\sqrt{|x_{b_r}-x_1|}}{\sqrt{|x_{b_r}-x_{2\ell}|}} \times
	\prod_{2\leq s < t \leq N}\frac{1}{\sqrt{|x_{b_t}-x_{b_s}|} \, \sqrt{|x_{a_t}-x_{a_s}|}} .
\end{align*}
This proves the lemma.
\end{proof}

\section{Technical lemmas for Section~\ref{sec::crossingproba}}
\label{appendix_technical}
In this appendix,
we gather technical results for deterministic curves. The setup is the following.
\begin{itemize}[leftmargin=2em]
\item Fix $N \geq 1$ and marked points $\bs{x} = (x_1, \ldots, x_{2N}) \in \chamber_{2N}$. 
Suppose $\eta$ is a continuous curve in $\HH$ starting from $x_2$ with continuous Loewner driving function $W$. 
Let $T$ be the first time when $x_1$ or $x_3$ is swallowed by $\eta$. 
Assume that $\eta[0,T]$ does not hit any marked points except for the starting point $x_2$. 
Let $(g_t \colon  0\le t\le T)$ be the conformal maps corresponding to this Loewner chain. 

\item For $\alpha\in\LP_N$ such that $\{2,b\}\in\alpha$ for $b \in \{1,3,5,\ldots, 2N-1\}$, 
define $\corrind_{\alpha}$ to be the set of indices $j\in\{4,5,\ldots, b-1\}$ 
such that $\{3,4,\ldots, j\}$ forms a sub-link pattern of $\alpha$. 

\item Define the bound functions 
\begin{align*} 
\LB_{\alpha}(\bs{x}):=\prod_{\{a,b\}\in\alpha}|x_b-x_a|^{-1/8} ,
\end{align*}
and recall the formula~\eqref{eqn::totalpartition_def}:
with $\bs{\sigma} = (\sigma_1, \sigma_2, \ldots, \sigma_N)$, we have
\begin{align*}
\LF_{\unnested}^{(N)}(\bs{x}) 
= \prod_{r=1}^N |x_{2r}-x_{2r-1}|^{-1/8} \bigg( \sum_{\bs{\sigma} \in \{\pm 1\}^N} \prod_{1\le s < t \le N}\chi(x_{2s-1}, x_{2t-1}, x_{2t}, x_{2s})^{\sigma_s \sigma_t / 4}\bigg)^{1/2} .
\end{align*}

\item For notational convenience, we also define 
\begin{align*}
\LB_{\unnested}^{(N)}(\bs{x}) 
:= & \; \prod_{r=1}^N |x_{2r}-x_{2r-1}|^{-1/8} , \\
\LY_{\unnested}^{(N)}(\bs{x}) 
:= & \; \frac{\LF_{\unnested}^{(N)}(\bs{x})}{\LB_{\unnested}^{(N)}(\bs{x})} 
= \bigg( \sum_{\bs{\sigma} \in \{\pm 1\}^N} \prod_{1\le s < t \le N}\chi(x_{2s-1}, x_{2t-1}, x_{2t}, x_{2s})^{\sigma_s \sigma_t / 4}\bigg)^{1/2} .
\end{align*}
\end{itemize}

The goal of this appendix is to prove the following 
technical result (Proposition~\ref{prop::mart_vanish_combined}). 
To this end, we first collect basic facts in Lemma~\ref{lem::distance_conformal_image}. 
Then, we give estimates for $\smash{\LB_{\alpha}/\LB_{\unnested}^{(N)}}$ and $\smash{\LY_{\unnested}^{(N)}}$ in 
Lemmas~\ref{lem::technical_oddj}--\ref{lem::bounds_remaining_term}. 
With these at hand, we complete the proof of Proposition~\ref{prop::mart_vanish_combined} in the end.

\begin{proposition} \label{prop::mart_vanish_combined}
Fix a link pattern $\alpha \in \LP_N$. Consider the 
continuous curve $\eta$ in $\HH$ in the above setup. 
\begin{enumerate}
\item \label{item::mart_vanish1}
Suppose $\{1,2\}\in\alpha$. For odd $j\in\{3,5,\ldots, 2N-1\}$, if $\eta(T)\in (x_j, x_{j+1})$, then we have
\begin{align*} 
\lim_{t\to T}\frac{\LB_{\alpha}(g_t(x_1), W_t, g_t(x_3), \ldots, g_t(x_{2N}))}{\LF_{\unnested}^{(N)}(g_t(x_1), W_t, g_t(x_3), \ldots, g_t(x_{2N}))}=0. 
\end{align*}

\item \label{item::mart_vanish2}
For even $j\in\{4,6,\ldots, 2N\}$ such that $j\not\in\corrind_{\alpha}$, if $\eta(T)\in (x_j, x_{j+1})$, then we have
\begin{align*} 
\lim_{t\to T}\frac{\LB_{\alpha}(g_t(x_1), W_t, g_t(x_3), \ldots, g_t(x_{2N}))}{\LF_{\unnested}^{(N)}(g_t(x_1), W_t, g_t(x_3), \ldots, g_t(x_{2N}))}=0. 
\end{align*}
\end{enumerate}
\end{proposition}

To simplify notation, we denote $f\lesssim g$ if $f/g$ is bounded by a finite constant from above, by $f\gtrsim g$ if $g\lesssim f$, and by $f\asymp g$ if $f\lesssim g$ and $f\gtrsim g$.

\begin{lemma}\label{lem::distance_conformal_image}
Fix marked points $x_1 < x_2 < y_1, \, y_2, \, y_3, \, y_4 < x_3 < x_4$. If $\eta(T)\in (x_3, x_4)$, then we have
\begin{align}\label{eqn::basicfact1}
\bigg| \frac{g_t(y_1)-g_t(y_2)}{g_t(y_3)-g_t(y_4)} \bigg| \asymp 1, 
\end{align}
where the constants in $\asymp$ depend on $\eta[0,T]$ and the marked points and are independent of $t \geq 0$, and 
\begin{align}\label{eqn::basicfact2}
\lim_{t\to T} \bigg|\frac{g_t(y_2)-g_t(y_1)}{W_t-g_t(y_3)}\bigg|=0. 
\end{align}
\end{lemma}
\begin{proof}
See, for instance,~\cite[Eq.~(A.1)~\&~(A.2)]{Liu-Wu:Scaling_limits_of_crossing_probabilities_in_metric_graph_GFF}. 
\end{proof}

\begin{lemma}\label{lem::technical_oddj}
Suppose $\{1,2\}\in\alpha$. For odd $j\in \{3,5,\ldots, 2N-1\}$, if $\eta(T)\in (x_j, x_{j+1})$, then we have 
\begin{align*} 
\frac{\LB_{\alpha}(g_t(x_1),W_t,g_t(x_{3}),\ldots,g_t(x_{2N}))}{\LB_{\unnested}^{(N)}(g_t(x_1),W_t,g_t(x_3),\ldots, g_t(x_{2N}))}\lesssim 1,
\end{align*}
where the constant in $\lesssim$ depends on $\eta[0,T]$ and $\bs{x} \in \chamber_{2N}$ and is independent of $t \geq 0$. 
\end{lemma}

\begin{proof}
Write the link pattern $\alpha=\{\{a_1,b_1\},\ldots, \{a_N,b_N\}\}$ as in~\eqref{eqn::linkpatterns_ordering}, so that $\{a_1,b_1\}=\{1,2\}$. 
Assuming that $\eta(T)\in (x_j, x_{j+1})$, 
we have $g_t(x_l)-g_t(x_k)\asymp 1$ for all indices  $2<k<j<l$ or $j\le k<l$. Thus, we see that 
\begin{align}\label{eqn::technical_oddj_aux1}
			\frac{\LB_{\alpha}(g_t(x_1),W_t,g_t(x_{3}),\ldots,g_t(x_{2N}))}{\LB_{\unnested}^{(N)}(g_t(x_1),W_t,g_t(x_3),\ldots, g_t(x_{2N}))}\asymp 
			\frac{\underset{r \in \LI_{\alpha}^{j}}{\prod}
			 |g_t(x_{b_r}) - g_t(x_{a_r})|^{-1/8}}{\underset{s \in \LI_{\unnested}^{j}}{\prod} |g_t(x_{2s}) - g_t(x_{2s-1})|^{-1/8}}, 
		\end{align}
where 
\begin{align} 
			\LI_{\alpha}^{j} := \; & \{r \in \{1,2,\dots, N\} \colon a_r ,b_r \in\{3,4,\ldots, j\}\}, \label{eqn::set_1} \\
			\LI_{\unnested}^{j} := \; & \{s \in \{1,2,\dots,N\} \colon 2s-1,2s \in \{3,4,\dots,j\}\}\label{eqn::set_2}.
		\end{align} 
Since $j$ is odd, we have
		\begin{align*}
			\#\LI_{\unnested}^{j}=\frac{j-3}{2}\qquad \textnormal{and} \qquad m = m(j,\alpha) := \#\LI_{\alpha}^{j}\leq \frac{j-3}{2},
		\end{align*}
		which implies that $\{2 ,3, \ldots, m+1\}\subset \LI_{\unnested}^{j}$. 
Now, for $s \in \LI_{\unnested}^{j}$, we have
$\smash{\underset{t\to T}{\lim} \; |g_t(x_{2s})-g_t(x_{2s-1})| = 0}$.
 Thus, we see that the right-hand side (RHS) of~\eqref{eqn::technical_oddj_aux1} can be estimated as
\begin{align}\label{eqn::technical_oddj_aux2}
\textnormal{RHS of}~\eqref{eqn::technical_oddj_aux1} \, \lesssim \, \frac{\underset{r \in \LI_{\alpha}^{j}}{\prod} |g_t(x_{b_r}) - g_t(x_{a_r})|^{-1/8}}{\underset{s \in \{2,3,\ldots, m+1\}}{\prod} |g_t(x_{2s}) - g_t(x_{2s-1})|^{-1/8}}. 
\end{align}
There are equally many (namely,~$m$) factors in the denominator and in the numerator of RHS of~\eqref{eqn::technical_oddj_aux2}. 
From~\eqref{eqn::basicfact1}, we then find that 
$\textnormal{RHS of}~\eqref{eqn::technical_oddj_aux2} \asymp 1$, which completes the proof. 
\end{proof}

\begin{lemma} \label{lem::technical_evenj}
For even $j\in\{4,6,\ldots, 2N\}$ and $j\not\in\corrind_{\alpha}$, if $\eta(T)\in (x_j, x_{j+1})$, then we have
 \begin{align} \label{eqn::technical_evenj}
 	\lim_{t\to T}\frac{\LB_{\alpha}(g_t(x_1),W_t,g_t(x_{3}),\ldots,g_t(x_{2N}))}{\LB_{\unnested}^{(N)}(g_t(x_1),W_t,g_t(x_3),\ldots, g_t(x_{2N}))}=0.
 \end{align}
\end{lemma}

 \begin{proof}
			Write $\alpha=\{\{a_1,b_1\},\ldots, \{a_N,b_N\}\}$ as in~\eqref{eqn::linkpatterns_ordering}, and write also $\{a_2,b_2\}=\{2,b\}$. 
\begin{itemize}[leftmargin=2em]
\item Assume that $j\in \{4,6,\ldots, 2N-2\}$.
Define the sets $\LI_{\alpha}^{j}$ and $\LI_{\unnested}^{j}$ as in~\eqref{eqn::set_1}~\&~\eqref{eqn::set_2}. Combining the facts that $j$ is even and $j\notin \corrind_{\alpha}$, we obtain
		\begin{align*} 
		\#\LI_{\unnested}^{j}= \frac{j-2}{2} \qquad \textnormal{and} \qquad m = m(j,\alpha) := \#\LI_{\alpha}^{j}\leq \frac{j-2}{2}-1, 
		\end{align*}
		which implies that
		\begin{align*} 
			\{2,3,\ldots,m+1\}\subset \LI_{\unnested}^{j} \qquad \textnormal{and}\qquad \frac{j}{2}\in \LI_{\unnested}^{j}\setminus \{2,3,\ldots,m+1\}.
		\end{align*}
		Thus, we can write
		\begin{align*}
		 & \;	\frac{\LB_{\alpha}(g_t(x_1),W_t,g_t(x_{3}),\ldots,g_t(x_{2N}))}{\LB_{\unnested}^{(N)}(g_t(x_1),W_t,g_t(x_3),\ldots, g_t(x_{2N}))} \\
		= & \;  \underbrace{\bigg( \frac{\underset{r \in \LI_{\alpha}^{j}}{\prod} |g_t(x_{b_r}) - g_t(x_{a_r})|^{-1/8}}{\underset{s \in \{2,3,\ldots,m+1\}}{\prod} |g_t(x_{2s}) - g_t(x_{2s-1})|^{-1/8}} \bigg)}_{=: A_1}
			\; \underbrace{ \bigg(\frac{|g_{t}(x_b)-W_t|^{-1/8}}{|g_t(x_j)-g_t(x_{j-1})|^{-1/8}} \bigg)}_{=: A_2} \\[.5em]
			& \; \times \underbrace{\bigg( \frac{\underset{r \notin \LI_{\alpha}^{j}\cup\{2\} }{\prod} |g_t(x_{b_r})-g_t(x_{a_r})|^{-1/8}}{|W_t-g_t(x_1)|^{-1/8} \underset{s \notin \{1,2,\ldots, m+1\}\cup\{j/2\} }{\prod} |g_t(x_{2s})-g_t(x_{2s-1})|^{-1/8}} \bigg)}_{=: A_3}.
		\end{align*}

\begin{enumerate}[leftmargin=*]
\item In $A_1$, there are equally many (namely,~$m$) factors in the denominator and in the numerator. 
Hence, we see from~\eqref{eqn::basicfact1} that $A_1\asymp 1$ in the limit $t\to T$.

\item  In $A_2$, we have $|g_t(x_j)-g_t(x_{j-1})| \to 0$ as $t\to T$. 
It remains to analyze $|g_t(x_b)-W_t|$ as $t\to T$. 
If $b=1$ or $b \ge j+1$, we have $|g_t(x_b)-W_t|\asymp 1$. If $3\le b \le j$, we have $|W_t-g_t(x_b)| \to 0$, but $A_2 \to 0$ 
due to~\eqref{eqn::basicfact2}. 
Thus, in both cases, we have $\smash{\underset{t\to T}{\lim}\; A_2 = 0}$ in the limit $t\to T$. 

\item Lastly, for $A_3$ the definition of the set $\LI_{\alpha}^{j}$ implies that
\begin{align*} 
\textnormal{either }	a_r=1\textnormal{  or }j+1\leq b_r\leq 2N,\qquad \textnormal{for all } r \notin \LI_{\alpha}^{j}\cup\{2\} .
\end{align*}
Thus, for all $r \notin \LI_{\alpha}^{j}\cup\{2\}$, we have $|g_t(x_{b_r})-g_t(x_{a_r})| \asymp 1$. 
Hence, $A_3\lesssim 1$ in the limit $t\to T$. 
\end{enumerate}
Combining the above three estimates, we obtain~\eqref{eqn::technical_evenj}. 

\item 
The case where $j=2N$ can be analyzed similarly. 
\qedhere
\end{itemize}
\end{proof}

\begin{lemma} \label{lem::bounds_remaining_term}
We have
\begin{align} \label{eqn::remaning_term_lb1}
	\LY_{\unnested}^{(N)}(x_1,\ldots,x_{2N})\geq \chi(x_{2r-1}, x_{2s-1}, x_{2s}, x_{2r})^{1/8}, \qquad \textnormal{for all } 1\leq r <  s \leq N. 
\end{align}
In particular, we have 
\begin{align} \label{eqn::remaning_term_lb2}
	\LY_{\unnested}^{(N)}(x_1,\ldots,x_{2N})\geq 1. 
\end{align}
\end{lemma}

\begin{proof}
	Note that~\eqref{eqn::remaning_term_lb2} follows from~\eqref{eqn::remaning_term_lb1} because 
	$\chi(x_{2r-1}, x_{2s-1}, x_{2s}, x_{2r}) \ge 1$ holds for all $1\leq r <  s\leq N$.
It suffices to show~\eqref{eqn::remaning_term_lb1}. We proceed by induction on $N\geq 2$.	When $N=2$, we have 
	\begin{align*}
		\LY_{\unnested}^{(2)} = & \; \big(2 \, \chi(x_{1}, x_{3}, x_{4}, x_{2})^{1/4} + 2 \, \chi(x_{1}, x_{3}, x_{4}, x_{2})^{-1/4}\big)^{1/2} \\
		= & \; \big( 2 \, \chi(x_{1}, x_{3}, x_{4}, x_{2})^{-1/2} + 2\big)^{1/2} \; 
		\chi(x_{1}, x_{3}, x_{4}, x_{2})^{1/8} \;
		\geq  \; \chi(x_{1}, x_{3}, x_{4}, x_{2})^{1/8} .
	\end{align*}
This proves~\eqref{eqn::remaning_term_lb1} in the initial case $N=2$.
Now, assume that $N\geq 3$ and~\eqref{eqn::remaning_term_lb1} holds up to $N-1$. For any $1\leq  r <  s \leq N$, fix some $t \in \{1,2,\ldots, N\}\setminus \{r,s\}$. 
Defining the function $\zeta:(0,+\infty)\to (0,+\infty)$ as $\zeta(x):=x+1/x$,  we have 
\begin{align*}
\big(\LY_{\unnested}^{(N)}(x_1,\ldots,x_{2N})\big)^2 
= & \; \sum_{\bs{\sigma}\in\{\pm1\}^N} 
\prod_{1\leq u < v \leq N}\chi(x_{2u-1},x_{2v-1},x_{2v},x_{2u})^{\sigma_r\sigma_s/4} \\
= & \; \sum_{\hat{\bs{\sigma}} \in \{\pm1\}^{N-1}}\bigg(\prod_{\substack{1\leq u < v \leq N \\ u,v \neq t}}\chi(x_{2u-1},x_{2v-1},x_{2v},x_{2u})^{\hat{\sigma}_u \hat{\sigma}_v / 4}\bigg) \\
& \; \times \zeta\bigg(\bigg(\prod_{l < t}\chi(x_{2l-1},x_{2t-1},x_{2t},x_{2l})^{\hat{\sigma}_l/4}\bigg)\bigg(\prod_{l > t}\chi(x_{2t-1},x_{2l-1},x_{2l},x_{2t})^{\hat{\sigma}_l/4}\bigg)\bigg) \\
\geq & \;  2 \big(\LY_{\unnested}^{(N-1)}(x_1,\ldots,x_{2t-2},x_{2t+1},\ldots,x_{2N})\big)^2 \\
\geq & \; \chi(x_{2r-1},x_{2s-1},x_{2s},x_{2r})^{1/4} ,
\end{align*}
where we used the induction hypothesis on the last line,
and wrote $\bs{\sigma} = (\sigma_1,\sigma_2,\ldots,\sigma_N)\in\{\pm1\}^N$ and $\hat{\bs{\sigma}} = (\hat{\sigma}_1,\ldots,\hat{\sigma}_{t-1},\hat{\sigma}_{t+1},\ldots,\hat{\sigma}_{N}) \in \{\pm1\}^{N-1}$. 
This yields~\eqref{eqn::remaning_term_lb1} and completes the proof.
\end{proof}

\begin{proof}[Proof of Proposition~\ref{prop::mart_vanish_combined}] \
\begin{enumerate}
\item If $\eta(T)\in (x_j, x_{j+1})$, then the following estimate holds:
\begin{align*}
\frac{\LB_{\alpha}(g_t(x_1), W_t, g_t(x_3), \ldots, g_t(x_{2N}))}{\LF_{\unnested}^{(N)}(g_t(x_1), W_t, g_t(x_3), \ldots, g_t(x_{2N}))}
\; \lesssim & \; \; \frac{1}{\LY_{\unnested}^{(N)}(g_t(x_1), W_t, g_t(x_3), \ldots, g_t(x_{2N}))} 
&& \textnormal{[by Lem.~\ref{lem::technical_oddj}]} \\
\le & \; \; \frac{1}{\chi(g_t(x_1),g_t(x_{j}),g_t(x_{j+1}),W_t)^{1/8}}. && \textnormal{[by~\eqref{eqn::remaning_term_lb1}]}
\end{align*}
By assumption, $j$ is odd and $x_1<x_2<x_j<x_{j+1}$. 
Thus, if $\eta(T)\in (x_j, x_{j+1})$, then we have
\begin{align*}
\chi(g_t(x_1),g_t(x_{j}),g_t(x_{j+1}),W_t) 
\; = \; \frac{(g_t(x_j)-g_t(x_1))(g_t(x_{j+1})-W_t)}{(g_t(x_{j+1})-g_t(x_1))(g_t(x_j)-W_t)} 
\; \asymp \; \frac{1}{g_t(x_j)-W_t}
\quad \overset{t\to T}{\longrightarrow} \quad \infty . 
\end{align*}
This proves Item~\ref{item::mart_vanish1}. 

\item From Lemmas~\ref{lem::technical_evenj} and~\ref{lem::bounds_remaining_term}, 
we find that if $\eta(T)\in (x_j, x_{j+1})$, then
\begin{align*}
\frac{\LB_{\alpha}(g_t(x_1), W_t, g_t(x_3), \ldots, g_t(x_{2N}))}{\LF_{\unnested}^{(N)}(g_t(x_1), W_t, g_t(x_3), \ldots, g_t(x_{2N}))} \; \le \;  
\frac{\LB_{\alpha}(g_t(x_1), W_t, g_t(x_3), \ldots, g_t(x_{2N}))}{\LB_{\unnested}^{(N)}(g_t(x_1), W_t, g_t(x_3), \ldots, g_t(x_{2N}))}
\quad \overset{t\to T}{\longrightarrow} \quad 0 , 
\end{align*}
This proves Item~\ref{item::mart_vanish2}. 
\qedhere 
\end{enumerate}
\end{proof}

\section{Asymptotic properties of the Coulomb gas integrals $\coulombnew_{\beta}$}
\label{appendix_asy}
In this appendix, we assume that $\kappa \in (4,8)$.
Recall from~\eqref{eqn::coulombgasintegral} 
the function $\coulombnew_\beta \colon \chamber_{2N} \to \R$, 
\begin{align*}
\coulombnew_\beta (\bs{x}) :=  \; &
\bigg( \frac{\sqrt{q(\kappa)} \, \Gamma(2-8/\kappa)}{\Gamma(1-4/\kappa)^2} \bigg)^N
\landupint_{x_{a_1}}^{x_{b_1}} 
\cdots \landupint_{x_{a_N}}^{x_{b_N}}
f_\beta (\bs{x};u_1,\ldots,u_N) \, \ud u_1 \cdots \ud u_N ,
\end{align*}
where the integrand is given by~\eqref{eq: integrand},  
\begin{align*} 
f_\beta (\bs{x};u_1,\ldots,u_N) := \; &
\prod_{1\leq i<j\leq 2N}(x_{j}-x_{i})^{2/\kappa} 
\prod_{1\leq r<s\leq N}(u_{s}-u_{r})^{8/\kappa} 
\prod_{\substack{1\leq i\leq 2N \\ 1\leq r\leq N}}
(u_{r}-x_{i})^{-4/\kappa} , 
\end{align*} 
with its branch chosen real and positive on the set~\eqref{eq:: branch choice set}. 
The goal of this appendix is to derive the asymptotic property~\eqref{eqn::ASY} of $\coulombnew_\beta$
for the case where $\{j,j+1\}\notin \beta$ 
(Proposition~\ref{prop::non-neighbor asy})
via a direct calculation. To this end, it suffices to derive the following asymptotics (Proposition~\ref{pro::Vbeta_ASY2}) for 
\begin{align*}
\coulomb_{\beta} (\bs{x}) := 
\bigg( \frac{\sqrt{q(\kappa)} \, \Gamma(2-8/\kappa)}{\Gamma(1-4/\kappa)^2} \bigg)^{-N} \,
\coulombnew_\beta (\bs{x}) .
\end{align*}

\begin{proposition}\label{pro::Vbeta_ASY2}
Fix $\beta \in \LP_N$ with link endpoints ordered as in~\eqref{eqn::linkpatterns_ordering}.
Fix $j \in \{1,2, \ldots, 2N-1 \}$ such that $\{j,j+1\} \notin \beta$.	
Then, for all $\xi \in (x_{j-1}, x_{j+2})$, using the notation~\eqref{eqn::bs_notation}, we have
	\begin{align} \label{eqn::ASY2}
\lim_{x_j , x_{j+1} \to \xi} \frac{\coulomb_{\beta} (\bs{x})}{ (x_{j+1} - x_j)^{ -2h(\kappa) }} 
= \frac{\Gamma(1-4/\kappa)^2}{\sqrt{q(\kappa)} \, \Gamma(2-8/\kappa)}
\coulomb_{\wp_j(\beta)/\{j,j+1\}} (\bs{\ddot{x}}_j) .
\end{align}
\end{proposition}

Proposition~\ref{pro::Vbeta_ASY2} can be proved via direct analysis. 
We consider three cases separately, according to the pairs of $j$ and of $j+1$ in $\beta$:
	\begin{enumerate}[label=(\Alph*):, ref=\Alph*]
		\item  \label{item::caseA} $\{a_{r},j\}\in\beta$ and $\{j+1,b_{s}\}\in\beta$ with $a_{r}<j<j+1<b_{s}$, 
		\item  \label{item::caseB} $\{a_{s},j\}\in\beta$ and $\{a_{r},j+1\}\in\beta$ with $a_{r}<a_{s}<j<j+1$, 
		\item  \label{item::caseC} $\{j,b_{r}\}\in\beta$ and $\{j+1,b_{s}\}\in\beta$ with $j<j+1<b_{s}<b_{r}$. 
	\end{enumerate}
In all three cases, by the ordering~\eqref{eqn::linkpatterns_ordering}, we have $r(j) = r < s = s(j)$ and $a_{r}<a_{s}$.  
	Supplementing the notation in~\eqref{eqn::bs_notation}, we write 
\begin{align*}
	\bs{u} =  \; & (u_1, \ldots, u_N)  \qquad \textnormal{and} \qquad
	\bs{\ddot{u}}_{r,s} =(u_1, \ldots, u_{r-1}, u_{r+1}, \ldots, u_{s-1}, u_{s+1}, \ldots, u_N) .
\end{align*}

As $j$, $r$, and $s$ will be fixed throughout, we omit the dependence on them in the notation for $\bs{\ddot{x}}$ and $\bs{\ddot{u}}$. 
Even though the points $x_1, \ldots, x_{2N}$ are allowed to move in this appendix, we always assume that they are ordered as $x_1 \leq \cdots \leq x_{2N}$ and only collide upon taking the limit $x_j, x_{j+1}\to \xi$.

\begin{proof}[Proof of Proposition~\ref{pro::Vbeta_ASY2}, Case~\ref{item::caseA}]
	Define $\beta_A := \beta \setminus (\{a_{r}, j\}\cup \{j+1, b_{s}\})$ (we do not relabel the indices here), and denote by 
	$\Gamma_{\beta_A}$ the integration contours in $\coulomb_{\beta}$ other than $(x_{a_r}, x_j)$, $(x_{j+1}, x_{b_s})$.
	Then, we have
	\begin{align} \label{eqn::IA_def}
		\coulomb_{\beta}(\bs{x}) 
		= 
		\int_{\Gamma_{\beta_A}} \landupint_{x_{a_{r}}}^{x_{j}} \landupint_{x_{j+1}}^{x_{b_{s}}} 
		\ud \bs{u} \, f_\beta(\bs{x}; \bs{u}) 
		= \; & \int_{\Gamma_{\beta_A}} \ud \bs{\ddot{u}} \, f_\beta(\bs{x}; \bs{\ddot{u}}) \;
		I_A(x_{a_r}, x_j, x_{j+1}, x_{b_s}) ,
	\end{align}
	where 
	\begin{align*}
		f_\beta(\bs{x}; \bs{\ddot{u}})
		= \; & 
		\prod_{1\leq i<j\leq 2N}(x_{j}-x_{i})^{2/\kappa} \, 
		\prod_{\substack{1\leq t<l\leq N \\ t,l \neq r,s}}(u_{l}-u_{t})^{8/\kappa} 
		\prod_{\substack{1\leq i\leq 2N \\ 1\leq t\leq N \\ t \neq r,s}}
		(u_{t}-x_{i})^{-4/\kappa} 
	\end{align*} 
	is a part of the integrand function~{\eqref{eq: integrand}} chosen to be real and positive on 
	\begin{align} \label{eq:: branch choice set smaller}
		\{ x_1 < \cdots < x_{2N} 
		\textnormal{ and } x_{a_t} < \Re(u_t) < x_{a_t+1} \quad \textnormal{ for all } t \neq r,s \} ,
	\end{align}
	and where $I_A(x_{a_r}, x_j, x_{j+1}, x_{b_s}) =: I_A$ is the integral
	\begin{align} \label{eqn::def_I_A}
		I_A := \landupint_{x_{a_{r}}}^{x_j} \ud u_{r} \, \frac{f_\beta^{(r)}(u_{r})}{|u_{r}-x_j|^{4/\kappa} |u_{r}-x_{j+1}|^{4/\kappa}}
		\landupint_{x_{j+1}}^{x_{b_{s}}} \ud u_{s} \,
		\frac{(u_{s}-u_{r})^{8/\kappa} \, f_\beta^{(s)}(u_{s})}{|u_{s}-x_j|^{4/\kappa} |u_{s}-x_{j+1}|^{4/\kappa}},
	\end{align}
	with $x_{a_{r}} < \Re(u_{r}) < x_j < x_{j+1} < \Re(u_{s}) < x_{b_{s}}$, 
	where the branch of $(u_{s}-u_{r})^{8/\kappa}$ is chosen to be positive when $\Re(u_{r}) < \Re(u_{s})$, 
	and $f_\beta^{(r)}$ is the multivalued function
	\begin{align*}
		f_\beta^{(r)}(y) = f_\beta^{(r)}(y;\bs{\ddot{x}};\bs{\ddot{u}})  
		:= \prod_{t \neq r,s} (y - u_t)^{8/\kappa} 
		\prod_{l\neq j,j+1}( y - x_l )^{-4/\kappa} ,
	\end{align*}
	whose branch is chosen to be positive when $ x_{a_r} < \Re(y) < x_{a_r+1}$, or more precisely, on
	\begin{align} \label{eqn::branchchoice_natural_r}
		\{ x_1 < \cdots < x_{2N} 
		\textnormal{ and } x_{a_r} < \Re(y) < x_{a_r+1}
		\textnormal{ and } x_{a_t} < \Re(u_t) < x_{a_t+1} \quad \textnormal{ for all } t \neq r,s \} ,
	\end{align}
	and $f_\beta^{(s)}$ is the multivalued function
	\begin{align*} 
		f_\beta^{(s)}(y) = f_\beta^{(s)}(y;\bs{\ddot{x}};\bs{\ddot{u}}) 
		:= \prod_{t \neq r,s} (y - u_t)^{8/\kappa} 
		\prod_{l\neq j,j+1}( y - x_l )^{-4/\kappa} ,
	\end{align*}
	whose branch is chosen to be positive when $ x_{a_s} < \Re(y) < x_{a_s+1}$, or more precisely, on
	\begin{align} \label{eqn::branchchoice_natural_s}
		\{ x_1 < \cdots < x_{2N} 
		\textnormal{ and } x_{a_s} < \Re(y) < x_{a_s+1}
		\textnormal{ and } x_{a_t} < \Re(u_t) < x_{a_t+1} \quad \textnormal{ for all } t \neq r,s \}  .
	\end{align}
Lemma~\ref{lem::Vbeta_ASY2_aux-caseA_aux} (proven below) implies that 
	\begin{align} \label{eqn::convergent_I_A}
	\; \lim_{x_j, x_{j+1}\to \xi} 
	\frac{I_A(x_{a_r}, x_j, x_{j+1}, x_{b_s})
	}{(x_{j+1}-x_j)^{1 - 8/\kappa}}
	= \frac{\Gamma(1-4/\kappa)^2}{\sqrt{q(\kappa)} \, \Gamma(2-8/\kappa)} \,  f_\beta^{(s)}(\xi) \, \landupint_{x_{a_{r}}}^{x_{b_{s}}} \ud y \, f_\beta^{(r)}(y) .
\end{align}
We thus obtain the asserted formula~\eqref{eqn::ASY2} by combining~\eqref{eqn::IA_def} with~\eqref{eqn::convergent_I_A}:
\begin{align*}
	\; & \lim_{x_j, x_{j+1}\to \xi} 
	\frac{\coulomb_{\beta}(\bs{x})}{(x_{j+1}-x_j)^{-2h(\kappa)}} \\
	= \; & \lim_{x_j,x_{j+1}\to\xi} 
	(x_{j+1}-x_j)^{6/\kappa-1}
	\int_{\Gamma_{\beta_A}} \landupint_{x_{a_{r}}}^{x_{j}} \landupint_{x_{j+1}}^{x_{b_{s}}}
	\ud \bs{u} \, f_\beta(\bs{x}; \bs{u})
	 \\
	= \; & 
	\lim_{x_j,x_{j+1}\to\xi} 
	(x_{j+1}-x_j)^{6/\kappa-1}
	\int_{\Gamma_{\beta_A}} \ud \bs{\ddot{u}} \, f_\beta(\bs{x}; \bs{\ddot{u}}) \;
	I_A(x_{a_r}, x_j, x_{j+1}, x_{b_s})
	&& \textnormal{[by~\eqref{eqn::IA_def}]}
	\\
	= \; &  \frac{\Gamma(1-4/\kappa)^2}{\sqrt{q(\kappa)} \, \Gamma(2-8/\kappa)} \,
	\coulomb_{\wp_j(\beta)/\{j,j+1\}}(\bs{\ddot{x}}_j) ,
		&& \textnormal{[by~\eqref{eqn::convergent_I_A}]}
\end{align*}
after carefully collecting the phase factors (and recalling that $\xi \in (x_{j-1}, x_{j+2})$ and that $f_\beta(\bs{x}; \bs{\ddot{u}})$ is real and positive on~\eqref{eq:: branch choice set smaller},
$\smash{f_\beta^{(r)}}$ is real and positive on~\eqref{eqn::branchchoice_natural_r}, and
$\smash{f_\beta^{(s)}}$ is real and positive on~\eqref{eqn::branchchoice_natural_s}). 
\end{proof}

In order to show the remaining identity~\eqref{eqn::convergent_I_A}, we first record an auxiliary lemma. 
Let $\hF(a,b,c;z)$ be the hypergeometric function~\cite[Eq.~(15.3.1)]{Abramowitz-Stegun:Handbook} defined as
\begin{align*}
	\hF(a,b,c;z) 
	:= \; & \frac{\Gamma(c)}{\Gamma(b) \Gamma(c-b)} \, \int_0^1 t^{b-1} (1 - t)^{c-b-1} (1-zt)^{-a} \, \ud t \\
	= \; & \frac{\Gamma(c)}{\Gamma(b) \Gamma(c-b)} \,  z^{1-c}  \, \int_0^z
	t^{b-1} (z - t)^{c-b-1} (1-t)^{-a} \, \ud t , 
	\qquad \Re(c) > \Re(b) > 0 , \; z \in \C \setminus [1,\infty) .
\end{align*}
Recall the asymptotics (cf.~\cite[Eq.~(15.3.7)]{Abramowitz-Stegun:Handbook} and note that $\hF(a,b,c;0) = 1$)
\begin{align} \label{eq: HGF ASY}
	\hF(a,b,c;z)
	\; \sim \;
	\frac{\Gamma(c) \Gamma(b-a)}{\Gamma(b) \Gamma(c-a)} \, (-z)^{-a} +
	\frac{\Gamma(c) \Gamma(a-b)}{\Gamma(a) \Gamma(c-b)} \, (-z)^{-b}  , \qquad z \to - \infty .
\end{align}

\begin{lemma} \label{lem: HGE integral}
	Let $\kappa > 4$, $\lambda > 0$, $\nu < 1$, and $\mu < \frac{1}{\lambda}$. Then, we have
	\begin{align*}
		\int_{\mu \lambda}^{\nu} \frac{\ud u}{u^{4/\kappa} (u+\lambda)^{4/\kappa}} 
		= \; & \frac{\kappa \, \lambda^{-4/\kappa}}{\kappa-4}  \, 
		\bigg( \nu^{1-4/\kappa} \, 
		\, \hF \Big(\frac{4}{\kappa}, 1-\frac{4}{\kappa}, 2-\frac{4}{\kappa}; -\frac{\nu}{\lambda} \Big) 
		- (\mu \lambda)^{1-4/\kappa} \, 
		\, \hF \Big(\frac{4}{\kappa}, 1-\frac{4}{\kappa}, 2-\frac{4}{\kappa}; -\mu \Big) \bigg) .
	\end{align*}
\end{lemma}

\begin{proof}
	This follows by considering the hypergeometric function with $b=1-4/\kappa$, $a=4/\kappa$, $c=2-4/\kappa > 0$:
	with the change of variables $u = -t \lambda$, we have
	\begin{align*}
		\int_0^z \frac{\ud u}{u^{4/\kappa} (u+\lambda)^{4/\kappa}} 
		= \; & \frac{\kappa}{\kappa-4}  \, \lambda^{-4/\kappa} \, 
		\, z^{1-4/\kappa} \, 
		\, \hF \Big(\frac{4}{\kappa}, 1-\frac{4}{\kappa}, 2-\frac{4}{\kappa}; -\frac{z}{\lambda} \Big) ,
	\end{align*}
	using also the functional equation $\frac{\Gamma(\nu+1)}{\Gamma(\nu)} = \nu$ to simplify the Gamma functions in the prefactor: 
	\begin{align*}
		\frac{\kappa}{\kappa-4} = \frac{\Gamma(1-\frac{4}{\kappa})}{\Gamma(2-\frac{4}{\kappa})} .
	\end{align*}
This implies the asserted identity. 
\end{proof}

\begin{lemma} \label{lem::Vbeta_ASY2_aux-caseA_aux}
	For $I_A=I_A(x_{a_r},x_j,x_{j+1},x_{b_r})$ defined in~\eqref{eqn::def_I_A}, we have the convergence result~\eqref{eqn::convergent_I_A}.
\end{lemma}

\begin{proof}
	Let us make some preparations before evaluating the limit. 
	\begin{itemize}[leftmargin=*]
		\item  First, note that for any fixed $\bs{\ddot{x}} \in \chamber_{2N-2}$ and $\bs{\ddot{u}} \in \Gamma_{\beta_A}$, we have
		\begin{align} \label{eq: relationhrhs}
			f_\beta^{(s)}(x) \, f_\beta^{(r)}(y) = f_\beta^{(s)}(y) \, f_\beta^{(r)}(x) ,
		\end{align}
		for all $x, y \notin \{ x_1, \ldots, x_{j-1}, x_{j+2}, \ldots, x_{2N}, u_1, \ldots, u_{r-1}, u_{r+1}, \ldots, u_{s-1}, u_{s+1}, \ldots, u_N \}$ such that $x \neq y$, since the phase factors from the exchange of $x$ and $y$ in the product cancel out.  

		\item Second, after making the changes of variables  $u=\frac{x_j-u_{r}}{x_j-x_{a_{r}}}$ and $v=\frac{u_{s}-x_{j+1}}{x_{b_{s}}-x_{j+1}}$ in the integral $I_A$, we obtain 
		\begin{align*} 
			I_A = & \; 
			\landupint_0^1 \ud u \, \frac{f_\beta^{(r)}(x_j-(x_j-x_{a_{r}})u)}{\big|u \big(u+\frac{x_{j+1}-x_j}{x_{j}-x_{a_{r}}}\big)\big|^{4/\kappa}} \;
			\landupint_0^1 \ud v  \,
			\frac{f_\beta^{(s)}((x_{b_{s}}-x_{j+1})v+x_{j+1})}{\big|v\big(v+\frac{x_{j+1}-x_j}{x_{b_{s}}-x_{j+1}}\big)\big|^{4/\kappa}} \; 
			\kfunc(u, v, x_{a_{r}}, x_j, x_{j+1}, x_{b_{s}}) ,
		\end{align*}
		where
		\begin{align*} 
			\kfunc(u, v, x_{a_{r}}, x_j, x_{j+1}, x_{b_{s}})
			:= \; & \frac{\big( x_{j+1} - x_j + u \, (x_j-x_{a_{r}}) + v \, (x_{b_{s}}-x_{j+1})\big)^{8/\kappa}}{| x_{b_{s}}-x_{j+1} |^{-1+8/\kappa} \, | x_j-x_{a_{r}} |^{-1+8/\kappa}} \\
			= \; & \frac{\big( u \, (x_j-x_{a_{r}}) + v \, (x_{b_{s}}-x_{j+1}) \big)^{8/\kappa}}{| x_{b_{s}}-x_{j+1} |^{-1+8/\kappa} \, | x_j-x_{a_{r}} |^{-1+8/\kappa}} + \OO( |x_{j+1} - x_j| ) ,
			\qquad |x_{j+1} - x_j| \to 0 .
		\end{align*}
	\item 	Third, we note that 
	\begin{align*}
		\; & \big| x_{j+1}-x_j\big |^{8/\kappa} \bigg| 
		\landupint_0^1 \ud u \, \frac{f_\beta^{(r)}(x_j-(x_j-x_{a_{r}})u)}{\big|u \big(u+\frac{x_{j+1}-x_j}{x_{j}-x_{a_{r}}}\big)\big|^{4/\kappa}} \;
		\landupint_0^1 \ud v  \,
		\frac{f_\beta^{(s)}((x_{b_{s}}-x_{j+1})v+x_{j+1})}{\big|v\big(v+\frac{x_{j+1}-x_j}{x_{b_{s}}-x_{j+1}}\big)\big|^{4/\kappa}} \bigg| \\
		\leq \; & 
		\int_0^1 \ud u \, \big|f_\beta^{(r)}(x_j-(x_j-x_{a_{r}})u)\big| \, 
		\Big| \frac{x_j-x_{a_{r}}}{u} \Big|^{4/\kappa} 
		\Big| \frac{ x_{j+1}-x_j}{(x_j-x_{a_{r}}) u + x_{j+1} - x_j} \Big|^{4/\kappa} \\
		\; & \times 
		\int_0^1  \ud v  \, \big|f_\beta^{(s)}((x_{b_{s}}-x_{j+1})v+x_{j+1}) \big| \, 
		\Big| \frac{x_{b_{s}}-x_{j+1}}{v} \Big|^{4/\kappa} 
		\Big| \frac{x_{j+1}-x_j}{(x_{b_{s}}-x_{j+1}) v + x_{j+1} - x_j} \Big|^{4/\kappa} ,
	\end{align*}
	which remains bounded as $|x_{j+1} - x_j| \to 0$ 
	(the singularities of order $4/\kappa$ are integrable since $\kappa > 4$).
\end{itemize}
	Hence, we see that
	\begin{align} 
		\label{eq: limit of IA}
		\; & \lim_{x_j, x_{j+1}\to \xi} 
		\frac{I_A(x_{a_r}, x_j, x_{j+1}, x_{b_s})
		}{|x_{j+1}-x_j|^{1 - 8/\kappa}} \\
		\nonumber
		= \; & \lim_{x_j, x_{j+1}\to \xi}  
		\landupint_0^1 \ud u \, \frac{f_\beta^{(r)}(x_j-(x_j-x_{a_{r}})u)}{\big|u \big(u+\frac{x_{j+1}-x_j}{x_{j}-x_{a_{r}}}\big)\big|^{4/\kappa}} \;
		\landupint_0^1 \ud v  \,
		\frac{f_\beta^{(s)}((x_{b_{s}}-x_{j+1})v+x_{j+1})}{\big|v\big(v+\frac{x_{j+1}-x_j}{x_{b_{s}}-x_{j+1}}\big)\big|^{4/\kappa}} 
		\; \tilde{\kfunc}(u, v, x_{a_{r}}, x_j, x_{j+1}, x_{b_{s}}) ,
	\end{align}
	where
	\begin{align*}
		\tilde{\kfunc}(u, v, x_{a_{r}}, x_j, x_{j+1}, x_{b_{s}})
		= \; & 
		(x_{j+1}-x_j)^{8/\kappa - 1} \,
		\frac{\big( u \, (x_j-x_{a_{r}}) + v \, (x_{b_{s}}-x_{j+1}) \big)^{8/\kappa}}{| x_{b_{s}}-x_{j+1} |^{-1+8/\kappa} \, | x_j-x_{a_{r}} |^{-1+8/\kappa}} .
	\end{align*}

The evaluation of~\eqref{eq: limit of IA} involves several estimates. 
	To this end, for each $\epsilon>0$ and $c_1 > 0$, we choose $c_2 \in (0,1)$ small enough such that there exist constants $M_1, M_2 \in (0,\infty)$ such that 
	\begin{align*}
		&\begin{cases}
			|f_\beta^{(r)}(x)|\le M_1 , \\[.2em]
			|f_\beta^{(r)}(x) - f_\beta^{(r)}(\xi)|\le\epsilon ,
		\end{cases}
		 &&\textnormal{for } x \in \big[ \xi-c_2(\xi-x_{a_{r}}),\xi+3c_2(\xi-x_{a_{r}}) \big] , \\
		&\begin{cases}
			|f_\beta^{(s)}(x)|\le M_2 , \\[.2em]
			|f_\beta^{(s)}(x) - f_\beta^{(s)}(\xi)|\le\epsilon ,
		\end{cases}
		 &&\textnormal{for } x \in \big[ \xi-c_2(x_{b_{s}}-\xi),\xi+3c_2(x_{b_{s}}-\xi) \big] .
	\end{align*}
	Since $x_j,x_{j+1}\to\xi$, without loss of generality we may
	suppose furthermore that 
	\begin{align*}
		x_j , x_{j+1} \in ( \xi - \delta , \xi + \delta ) ,
		\qquad \textnormal{where} \qquad 
		\delta \leq \min \bigg\{ \frac{c_2 \, (\xi - x_{a_{r}})}{1 + 2 c_1} , \frac{c_2 \, (x_{b_{s}}-\xi)}{1 + 2 c_1} \bigg\}
	\end{align*}
	Then, we have
	\begin{align*}
		c_1 \frac{x_{j+1}-x_j}{x_{j}-x_{a_{r}}}\le c_2 
		\qquad \textnormal{and} \qquad
		c_1 \frac{x_{j+1}-x_j}{x_{b_{s}}-x_{j+1}}\le c_2 .
	\end{align*}
	We divide the integration over $(u,v)\in[0,1]\times [0,1]$ into the following regions:
	\begin{align*}
		R_{1,1} := & \; \Big\{ (u,v) \textnormal{ such that } 
		u \in \Big[0,c_1\frac{x_{j+1}-x_j}{x_{j}-x_{a_{r}}} \Big] 
		\textnormal{ and } 
		v \in \Big[ 0,c_1 \frac{x_{j+1}-x_j}{x_{b_{s}}-x_{j+1}} \Big] 
		\Big\} , \\
		R_{1,2} := & \; \Big\{ (u,v) \textnormal{ such that } 
		u \in \Big[0,c_1\frac{x_{j+1}-x_j}{x_{j}-x_{a_{r}}} \Big] 
		\textnormal{ and } 
		v \in \Big[c_1 \frac{x_{j+1}-x_j}{x_{b_{s}}-x_{j+1}},c_2 \Big] 
		\Big\} , \\
		R_{1,3} := & \; \Big\{ (u,v) \textnormal{ such that } 
		u \in \Big[0,c_1\frac{x_{j+1}-x_j}{x_{j}-x_{a_{r}}} \Big] 
		\textnormal{ and } 
		v \in [c_2,1]
		\Big\} , \\
		R_{2,1} := & \; \Big\{ (u,v) \textnormal{ such that } 
		u \in \Big[c_1\frac{x_{j+1}-x_j}{x_{j}-x_{a_{r}}},c_2 \Big] 
		\textnormal{ and } 
		v \in \Big[0,c_1 \frac{x_{j+1}-x_j}{x_{b_{s}}-x_{j+1}} \Big] 
		\Big\} , \\
		R_{2,2} := & \; \Big\{ (u,v) \textnormal{ such that } 
		u \in \Big[c_1\frac{x_{j+1}-x_j}{x_{j}-x_{a_{r}}},c_2\Big] 
		\textnormal{ and } 
		v \in \Big[c_1 \frac{x_{j+1}-x_j}{x_{b_{s}}-x_{j+1}},c_2 \Big]
		\Big\} , \\
		R_{2,3} := & \; \Big\{ (u,v) \textnormal{ such that } 
		u \in \Big[c_1\frac{x_{j+1}-x_j}{x_{j}-x_{a_{r}}},c_2\Big] 
		\textnormal{ and } 
		v \in [c_2,1]
		\Big\} , \\
		R_{3,1} := & \; \Big\{ (u,v) \textnormal{ such that } 
		u \in [c_2,1]
		\textnormal{ and } 
		v \in \Big[ 0,c_1 \frac{x_{j+1}-x_j}{x_{b_{s}}-x_{j+1}} \Big]
		\Big\} , \\
		R_{3,2} := & \; \Big\{ (u,v) \textnormal{ such that } 
		u \in [c_2,1]
		\textnormal{ and } 
		v \in \Big[c_1 \frac{x_{j+1}-x_j}{x_{b_{s}}-x_{j+1}},c_2 \Big]
		\Big\} , \\
		R_{3,3} := & \; \Big\{ (u,v) \textnormal{ such that } 
		u \in [c_2,1]
		\textnormal{ and } 
		v \in [c_2,1]
		\Big\} .
	\end{align*}
	We evaluate the contribution of these integrals by first taking the limit $x_j,x_{j+1}\to\xi$, then taking the limit $c_2 \to 0$, and finally taking the limit $c_1\to 0$:
	\begin{enumerate}[leftmargin=*]
		\item In the limit $x_j,x_{j+1}\to\xi$, the negligible regions are $R_{1,1}$ and $R_{3,3}$:
		\begin{itemize}[leftmargin=*]
			\item The integral over $R_{1,1}$ can be bounded as
			\begin{align*}
				\; & \bigg| \int_{R_{1,1}} \ud u \, \ud v \, 
				\frac{f_\beta^{(r)}(x_j-(x_j-x_{a_{r}})u) }{\big|u \big(u+\frac{x_{j+1}-x_j}{x_{j}-x_{a_{r}}}\big)\big|^{4/\kappa}} \;
				\frac{f_\beta^{(s)}((x_{b_{s}}-x_{j+1})v+x_{j+1}) }{\big|v\big(v+\frac{x_{j+1}-x_j}{x_{b_{s}}-x_{j+1}}\big)\big|^{4/\kappa}}
				\; \tilde{\kfunc}(u, v, x_{a_{r}}, x_j, x_{j+1}, x_{b_{s}}) 
				\bigg| \\
				\leq \; & 2^{8/\kappa} \, c_1^{8/\kappa} \, M_1 \, M_2  \, 
				\frac{|x_{j+1}-x_j|^{16/\kappa - 1}}{| x_{b_{s}}-x_{j+1} |^{-1+8/\kappa} \, | x_j-x_{a_{r}} |^{-1+8/\kappa}} 
				\; \bigg| \frac{x_{j+1}-x_j}{x_{j}-x_{a_{r}}} \bigg|^{1-8/\kappa}
				\; \bigg| \frac{x_{j+1}-x_j}{x_{b_{s}}-x_{j+1}} \bigg|^{1-8/\kappa} \\
				\; & \times 
				\int_0^{c_1} \frac{\ud u}{ |u|^{4/\kappa} |u+1|^{4/\kappa}} 
				\int_0^{c_1} \frac{\ud v}{ |v|^{4/\kappa} |v+1|^{4/\kappa}} \\
				\leq \; & 2^{8/\kappa} \, c_1^{8/\kappa} \, M_1 \, M_2  \, 
				|x_{j+1}-x_j| \;
				\int_0^{c_1} \frac{\ud u}{ |u|^{4/\kappa} |u+1|^{4/\kappa}} 
				\int_0^{c_1} \frac{\ud v}{ |v|^{4/\kappa} |v+1|^{4/\kappa}} 
				\qquad \overset{x_j, x_{j+1}\to \xi}{\longrightarrow} \qquad 0 .
			\end{align*}
			
			\item The integral over $R_{3,3}$ can be bounded as
			\begin{align*}
				\; & \bigg| \int_{R_{3,3}} \ud u \, \ud v \, 
				\frac{f_\beta^{(r)}(x_j-(x_j-x_{a_{r}})u) }{\big|u \big(u+\frac{x_{j+1}-x_j}{x_{j}-x_{a_{r}}}\big)\big|^{4/\kappa}} \;
				\frac{ f_\beta^{(s)}((x_{b_{s}}-x_{j+1})v+x_{j+1}) }{\big|v\big(v+\frac{x_{j+1}-x_j}{x_{b_{s}}-x_{j+1}}\big)\big|^{4/\kappa}}
				\; \tilde{\kfunc}(u, v, x_{a_{r}}, x_j, x_{j+1}, x_{b_{s}}) 
				\bigg| \\
				\leq \; & c_2^{-16/\kappa}
				|x_{j+1}-x_j|^{8/\kappa - 1} \,
				\frac{|  (x_j-x_{a_{r}}) + (x_{b_{s}}-x_{j+1}) |^{8/\kappa}}{| x_{b_{s}}-x_{j+1} |^{-1+8/\kappa} \, | x_j-x_{a_{r}} |^{-1+8/\kappa}}   \\
				\; & \times 
				\int_{0}^{1} \ud u \, |f_\beta^{(r)}(x_j-(x_j-x_{a_{r}})u)| 
				\int_{0}^{1} \ud v \, |f_\beta^{(s)}((x_{b_{s}}-x_{j+1})v+x_{j+1})|  
				\qquad \overset{x_j, x_{j+1}\to \xi}{\longrightarrow} \qquad 0 .
			\end{align*}
		\end{itemize}
	\item Furthermore, the integrals over the regions $R_{1,3}$, $R_{3,1}$, $R_{1,2}$, and $R_{2,1}$ tend to zero after first taking the limit $x_j,x_{j+1}\to\xi$ and then taking the limit $c_1 \to 0$:
	\begin{itemize}[leftmargin=*]
		\item The integral over $R_{1,3} \cup R_{1,2}$ can be bounded as
		\begin{align*}
			\; & \bigg| 
			\underset{R_{1,3} \cup R_{1,2}}{\int} \ud u \, \ud v \, 
			\frac{f_\beta^{(r)}(x_j-(x_j-x_{a_{r}})u) }{\big|u \big(u+\frac{x_{j+1}-x_j}{x_{j}-x_{a_{r}}}\big)\big|^{4/\kappa}} \;
			\frac{ f_\beta^{(s)}((x_{b_{s}}-x_{j+1})v+x_{j+1}) }{\big|v\big(v+\frac{x_{j+1}-x_j}{x_{b_{s}}-x_{j+1}}\big)\big|^{4/\kappa}}
			\; \tilde{\kfunc}(u, v, x_{a_{r}}, x_j, x_{j+1}, x_{b_{s}}) 
			\bigg| \\
			\leq \; & M_1 \,
			\frac{|x_{j+1}-x_j|^{8/\kappa - 1}}{| x_{b_{s}}-x_{j+1} |^{-1+8/\kappa} \, | x_j-x_{a_{r}} |^{-1+8/\kappa}} 
			\; \bigg| \frac{x_{j+1}-x_j}{x_{j}-x_{a_{r}}} \bigg|^{1-8/\kappa}
			\; | (x_{j+1}-x_j) + (x_{b_{s}}-x_{j+1}) |^{8/\kappa} \\
			\; & \times 
			\int_0^{c_1} \frac{\ud u}{ |u|^{4/\kappa} |u+1|^{4/\kappa}} 
			\int_{0}^{1} \ud v \, |f_\beta^{(s)}((x_{b_{s}}-x_{j+1})v+x_{j+1})| 
			\; \frac{| v |^{8/\kappa}}{\big|v\big(v+\frac{x_{j+1}-x_j}{x_{b_{s}}-x_{j+1}}\big)\big|^{4/\kappa}} \\
			\leq \; & M_1 \, 
			\frac{| (x_{j+1}-x_j) + \, (x_{b_{s}}-x_{j+1}) |^{8/\kappa}}{| x_{b_{s}}-x_{j+1} |^{-1+8/\kappa}} \\
			\; & \times 
			\int_0^{c_1} \frac{\ud u}{ |u|^{4/\kappa} |u+1|^{4/\kappa}} 
			\int_{0}^{1} \ud v \, |f_\beta^{(s)}((x_{b_{s}}-x_{j+1})v+x_{j+1})| \\
			\overset{x_j, x_{j+1}\to \xi}{\longrightarrow}  \; &  \qquad 
			M_1 \, |x_{b_{s}}-\xi| \;
			\int_0^{c_1} \frac{\ud u}{ |u|^{4/\kappa} |u+1|^{4/\kappa}} 
			\int_{0}^{1} \ud v \, |f_\beta^{(s)}((x_{b_{s}}-\xi)v+\xi)| 
\qquad \overset{c_1 \to 0}{\longrightarrow} \qquad 0 ,			
		\end{align*}
		because the integrals converge for each $\kappa > 4$.
		
		\item Very similarly, the integral over the region $R_{2,1} \cup R_{3,1} $ also tends to zero after first taking the limit $x_j,x_{j+1}\to \xi$ and then taking the limit $c_1\to 0$. 
	\end{itemize}
\item 	In contrast, the regions $R_{3,2}$, $R_{2,3}$, and $R_{2,2}$  do contribute to the limit $x_j,x_{j+1}\to\xi$.
To evaluate their contribution, it is useful to further split $R_{2,2}$ into the two regions
\begin{align*}
	R_{2,2} = R_{2,2}^+ \cup R_{2,2}^-
	:= \big\{ (u,v) \in R_{2,2} \colon |u| \leq |v| \big\}
	\cup  \big\{ (u,v) \in R_{2,2} \colon |v| \leq |u| \big\} ,
\end{align*}
and to evaluate the integrals over the two regions 
$R_{2,2}^+ \cup R_{2,3}$ and $R_{2,2}^- \cup R_{3,2}$ separately.
By symmetry, it suffices to consider the integral over $R_{2,2}^+ \cup R_{2,3}$. 
\begin{itemize}[leftmargin=*]
	\item 	 First, we show that $f_\beta^{(r)}(x_j-(x_j-x_{a_{r}})u)$ can be replaced by $f_\beta^{(r)}(\xi)$ when evaluating the limit of the integral over $R_{2,2}^+ \cup R_{2,3}$:
	\begin{align*}
		\bigg| \underset{R_{2,2}^+ \cup R_{2,3}}{\int} \; &  \ud u \, \ud v \, 
		\frac{ \big( f_\beta^{(r)}(x_j-(x_j-x_{a_{r}})u) - f_\beta^{(r)}(\xi) \big)  }{\big|u \big(u+\frac{x_{j+1}-x_j}{x_{j}-x_{a_{r}}}\big)\big|^{4/\kappa}}
		\; \frac{ f_\beta^{(s)}((x_{b_{s}}-x_{j+1})v+x_{j+1}) }{\big|v\big(v+\frac{x_{j+1}-x_j}{x_{b_{s}}-x_{j+1}}\big)\big|^{4/\kappa}}
		\; \tilde{\kfunc}(u, v, x_{a_{r}}, x_j, x_{j+1}, x_{b_{s}}) 
		\bigg| \\
		\leq \; & \epsilon \, 
		|x_{j+1}-x_j|^{8/\kappa - 1} \, 
		\frac{| (x_j-x_{a_{r}})+ (x_{b_{s}}-x_{j+1}) |^{8/\kappa}}{| x_{b_{s}}-x_{j+1} |^{-1+8/\kappa} \, | x_j-x_{a_{r}} |^{-1+8/\kappa}}  \\
		\; & \times 
		\underset{R_{2,2}^+ \cup R_{2,3}}{\int} \ud u \, \ud v \,
		\frac{ |v|^{8/\kappa} \, | f_\beta^{(s)}((x_{b_{s}}-x_{j+1})v+x_{j+1}) | }{\big|v\big(v+\frac{x_{j+1}-x_j}{x_{b_{s}}-x_{j+1}}\big)\big|^{4/\kappa}} \;
		\frac{ 1 }{\big|u \big(u+\frac{x_{j+1}-x_j}{x_{j}-x_{a_{r}}}\big)\big|^{4/\kappa}} \\
		\leq \; & \epsilon \, 
		|x_{j+1}-x_j|^{8/\kappa - 1} \, 
		\frac{| (x_j-x_{a_{r}})+ (x_{b_{s}}-x_{j+1}) |^{8/\kappa}}{| x_{b_{s}}-x_{j+1} |^{-1+8/\kappa} \, | x_j-x_{a_{r}} |^{-1+8/\kappa}}  \\
		\; & \times 
		\int_0^1 \ud v \, | f_\beta^{(s)}((x_{b_{s}}-x_{j+1})v+x_{j+1}) |
		\int_{c_1\frac{x_{j+1}-x_j}{x_{j}-x_{a_{r}}}}^{c_2}
		\frac{ \ud u }{|u|^{8/\kappa}} \\
		\leq \; & \epsilon \, 
		|x_{j+1}-x_j|^{8/\kappa - 1} \, 
		\frac{| (x_j-x_{a_{r}})+ (x_{b_{s}}-x_{j+1}) |^{8/\kappa}}{| x_{b_{s}}-x_{j+1} |^{-1+8/\kappa} \, | x_j-x_{a_{r}} |^{-1+8/\kappa}}  \\
		\; & \times 
		\frac{\kappa}{\kappa-8} \;
		\bigg( c_2^{1-8/\kappa}
		- \Big( c_1\frac{x_{j+1}-x_j}{x_{j}-x_{a_{r}}} \Big)^{1-8/\kappa}
		\bigg) 
		\int_0^1 \ud v \, | f_\beta^{(s)}((x_{b_{s}}-x_{j+1})v+x_{j+1}) |  \\
		\overset{x_j, x_{j+1}\to \xi}{\longrightarrow} \; & \qquad
		\epsilon \, c_1^{1-8/\kappa} \, \frac{\kappa}{8-\kappa} \; 
		\frac{| x_{b_{s}} - x_{a_{r}} |^{8/\kappa}}{| x_{b_{s}}-\xi |^{-1+8/\kappa}} 
		\int_0^1 \ud v \, | f_\beta^{(s)}((x_{b_{s}}-\xi)v+\xi) | 
		\qquad \overset{c_2 \to 0}{\longrightarrow} 
		\qquad 0,
	\end{align*}
since we can let $\epsilon\to 0$ as $c_2\to 0$. 
	\item 	Next, we show that $\tilde{\kfunc}(u, v, x_{a_{r}}, x_j, x_{j+1}, x_{b_{s}})$ can be replaced by $\smash{|x_{j+1}-x_j|^{8/\kappa - 1} \,
		\frac{| x_{b_{s}}-x_{j+1} |}{| x_j-x_{a_{r}} |^{-1+8/\kappa}} \, |v|^{8/\kappa}}$ when evaluating the limit of the integral over $R_{2,2}^+ \cup R_{2,3}$. To verify this, 
	we write
	\begin{align*}
		R_{2,2}^+ \cup R_{2,3}
		= \; & (R_{2,2}^+ \cup R_{2,3})^- \cup (R_{2,2}^+ \cup R_{2,3})^+ , \\
		(R_{2,2}^+ \cup R_{2,3})^- := \; & \{ (u,v) \in R_{2,2} \cup R_{2,3} \colon |u| \leq |v| < c_3 \} , \\
		(R_{2,2}^+ \cup R_{2,3})^+ := \; & \{ (u,v) \in R_{2,2} \cup R_{2,3} \colon |u| \leq |v| \textnormal{ and } |v| \geq c_3 \}  ,
	\end{align*}
	where $c_3 := \frac{2c_2}{1 + \frac{c_2}{1 + 2c_1} }$. Note that, since $c_1 \frac{x_{j+1}-x_j}{x_{b_{s}}-x_{j+1}} \leq c_2 \leq 1$, we have
	\begin{align} \label{eq: choice of c3}
		c_1 \frac{x_{j+1}-x_j}{x_{b_{s}}-x_{j+1}} \leq \frac{2c_2}{1 + \frac{c_2}{1 + 2c_1} } = c_3 ,
	\end{align}
	and since $|f_\beta^{(s)}(x)|\le M_2$ for $x\in[\xi-c_2(x_{b_{s}}-\xi),\xi+3c_2(x_{b_{s}}-\xi)]$, we have
	\begin{align} \label{eq: c3 gives bound for h}
		|f_\beta^{(s)}((x_{b_{s}}-x_{j+1})v+x_{j+1})| \leq M_2
		\qquad \textnormal{for } |v| \in \Big[c_1 \frac{x_{j+1}-x_j}{x_{b_{s}}-x_{j+1}}, c_3 \Big] . 
	\end{align}
	On the one hand, for the integral over $(R_{2,2}^+ \cup R_{2,3})^+$, we find
	\begin{align}
		\nonumber 
		\hspace*{-15mm}
		\bigg| \underset{(R_{2,2}^+ \cup R_{2,3})^+}{\int} \; & \ud u \, \ud v \, 
		\frac{ f_\beta^{(s)}((x_{b_{s}}-x_{j+1})v+x_{j+1}) }{\big|v\big(v+\frac{x_{j+1}-x_j}{x_{b_{s}}-x_{j+1}}\big)\big|^{4/\kappa}} \\
		\nonumber
		\; &\qquad\qquad\qquad\qquad\qquad \times 
		\frac{ \big( \tilde{\kfunc}(u, v, x_{a_{r}}, x_j, x_{j+1}, x_{b_{s}}) - |x_{j+1}-x_j|^{8/\kappa - 1} \,
			\frac{| x_{b_{s}}-x_{j+1} |}{| x_j-x_{a_{r}} |^{-1+8/\kappa}} \, |v|^{8/\kappa} \big) }{\big|u \big(u+\frac{x_{j+1}-x_j}{x_{j}-x_{a_{r}}}\big)\big|^{4/\kappa}} 
		\bigg| \\
		\nonumber
		\leq \; & |x_{j+1}-x_j|^{8/\kappa - 1} \, 
		\int_{c_3}^1 \ud v \, \frac{ |v|^{8/\kappa} \, |f_\beta^{(s)}((x_{b_{s}}-x_{j+1})v+x_{j+1})| }{\big|v\big(v+\frac{x_{j+1}-x_j}{x_{b_{s}}-x_{j+1}}\big)\big|^{4/\kappa}} \\
		\nonumber
		\; & \times 
		\frac{\big| | (c_2/v) \, (x_j-x_{a_{r}}) +  (x_{b_{s}}-x_{j+1}) |^{8/\kappa}
			- | x_{b_{s}}-x_{j+1} |^{8/\kappa} \big|}{| x_{b_{s}}-x_{j+1} |^{-1+8/\kappa} \, | x_j-x_{a_{r}} |^{-1+8/\kappa}} \,
		\int_{c_1\frac{x_{j+1}-x_j}{x_{j}-x_{a_{r}}}}^{c_2}  
		\frac{ \ud u }{\big|u \big(u+\frac{x_{j+1}-x_j}{x_{j}-x_{a_{r}}}\big)\big|^{4/\kappa}} \\
		\nonumber
		\leq \; & |x_{j+1}-x_j|^{8/\kappa - 1} \, 
		\int_{c_3}^1 \ud v \, |f_\beta^{(s)}((x_{b_{s}}-x_{j+1})v+x_{j+1})| \\
		\nonumber
		\; & \times 
		\frac{\big| | (c_2/v) \, (x_j-x_{a_{r}}) +  (x_{b_{s}}-x_{j+1}) |^{8/\kappa}
			- | x_{b_{s}}-x_{j+1} |^{8/\kappa} \big|}{| x_{b_{s}}-x_{j+1} |^{-1+8/\kappa} \, | x_j-x_{a_{r}} |^{-1+8/\kappa}} \,
		\int_{c_1\frac{x_{j+1}-x_j}{x_{j}-x_{a_{r}}}}^{c_2}  
		\frac{ \ud u }{|u|^{8/\kappa}} \\
		\nonumber
		\leq \; & \frac{\kappa}{\kappa-8} \, |x_{j+1}-x_j|^{8/\kappa - 1} \, 
		\bigg( c_2^{1-8/\kappa}
		- \Big( c_1\frac{x_{j+1}-x_j}{x_{j}-x_{a_{r}}} \Big)^{1-8/\kappa}
		\bigg) \\
		\nonumber
		\; & \times 
		\int_{c_3}^1 \ud v \, |f_\beta^{(s)}((x_{b_{s}}-x_{j+1})v+x_{j+1})| \,
		\frac{\big| | (c_2/v) \, (x_j-x_{a_{r}}) +  (x_{b_{s}}-x_{j+1}) |^{8/\kappa}
			- | x_{b_{s}}-x_{j+1} |^{8/\kappa} \big|}{| x_{b_{s}}-x_{j+1} |^{-1+8/\kappa} \, | x_j-x_{a_{r}} |^{-1+8/\kappa}} \\
		\nonumber
		\overset{x_j, x_{j+1}\to \xi}{\longrightarrow} \; & \qquad
		\frac{\kappa}{8-\kappa} \, c_1^{1-8/\kappa}
		| x_{b_{s}}-\xi |^{1-8/\kappa} \\
		\nonumber
		\; & \times 
		\int_{c_3}^1 \ud v \, |f_\beta^{(s)}((x_{b_{s}}-\xi)v+\xi)| \,
		\big| | (c_2/v) \, (\xi-x_{a_{r}}) +  (x_{b_{s}}-\xi) |^{8/\kappa}
		- | x_{b_{s}}-\xi |^{8/\kappa} \big| \\
		\label{eq: ktilde upper bound 1}
		\overset{c_2 \to 0}{\longrightarrow} \quad \; & \qquad 0 ,
	\end{align} 
	after applying the reverse Fatou lemma as $c_2 \to 0$ 
	(note also that $c_3 \to 0$ along with $c_2 \to 0$ by our choice~\eqref{eq: choice of c3} of $c_3$) 
	to the functions
	\begin{align*}
		\; & |f_\beta^{(s)}((x_{b_{s}}-\xi)v+\xi)| \,
		\big| | (c_2/v) \, (\xi-x_{a_{r}}) +  (x_{b_{s}}-\xi) |^{8/\kappa}
		- | x_{b_{s}}-\xi |^{8/\kappa} \big| \\
		\leq \; & |f_\beta^{(s)}((x_{b_{s}}-\xi)v+\xi)| \,
		\big( | (c_2/c_3) \, (\xi-x_{a_{r}}) +  (x_{b_{s}}-\xi) |^{8/\kappa}
		+ | x_{b_{s}}-\xi |^{8/\kappa} \big) \\
		\leq \; & |f_\beta^{(s)}((x_{b_{s}}-\xi)v+\xi)| \,
		\Big( \big| \frac{1}{2} \, \Big( 1 + \frac{c_2}{1 + 2c_1} \big) \, (\xi-x_{a_{r}}) +  (x_{b_{s}}-\xi) \big|^{8/\kappa}
		+ | x_{b_{s}}-\xi |^{8/\kappa} \Big) ,
	\end{align*}
	bounded by the non-negative integrable function on the last line. On the other hand, for the integral over $(R_{2,2}^+ \cup R_{2,3})^-$, we find using~\eqref{eq: c3 gives bound for h} that
	\begin{align}
		\nonumber 
		\hspace*{-15mm}
		\bigg| \underset{(R_{2,2}^+ \cup R_{2,3})^-}{\int} \; & \ud u \, \ud v \, 
		\frac{ f_\beta^{(s)}((x_{b_{s}}-x_{j+1})v+x_{j+1}) }{\big|v\big(v+\frac{x_{j+1}-x_j}{x_{b_{s}}-x_{j+1}}\big)\big|^{4/\kappa}} \\
		\nonumber
		\; &\qquad\qquad\qquad\qquad\qquad \times 
		\frac{ \big( \tilde{\kfunc}(u, v, x_{a_{r}}, x_j, x_{j+1}, x_{b_{s}}) - |x_{j+1}-x_j|^{8/\kappa - 1} \,
			\frac{| x_{b_{s}}-x_{j+1} |}{| x_j-x_{a_{r}} |^{-1+8/\kappa}} \, |v|^{8/\kappa} \big) }{\big|u \big(u+\frac{x_{j+1}-x_j}{x_{j}-x_{a_{r}}}\big)\big|^{4/\kappa}} 
		\bigg| \\
		\nonumber
		\leq \; &  |x_{j+1}-x_j|^{8/\kappa - 1} \, 
		\frac{\big| | (x_j-x_{a_{r}}) +  (x_{b_{s}}-x_{j+1}) |^{8/\kappa}
			- | x_{b_{s}}-x_{j+1} |^{8/\kappa} \big|}{| x_{b_{s}}-x_{j+1} |^{-1+8/\kappa} \, | x_j-x_{a_{r}} |^{-1+8/\kappa}} \\
		\nonumber
		\; & \times 
		\int_{c_1 \frac{x_{j+1}-x_j}{x_{b_{s}}-x_{j+1}}}^{c_3} \ud v \, \frac{ |v|^{8/\kappa} \, |f_\beta^{(s)}((x_{b_{s}}-x_{j+1})v+x_{j+1})| }{\big|v\big(v+\frac{x_{j+1}-x_j}{x_{b_{s}}-x_{j+1}}\big)\big|^{4/\kappa}}
		\int_{c_1\frac{x_{j+1}-x_j}{x_{j}-x_{a_{r}}}}^{v}  
		\frac{ \ud u }{\big|u \big(u+\frac{x_{j+1}-x_j}{x_{j}-x_{a_{r}}}\big)\big|^{4/\kappa}} \\
		\nonumber
		\leq \; &  |x_{j+1}-x_j|^{8/\kappa - 1} \, 
		\frac{\big| | (x_j-x_{a_{r}}) +  (x_{b_{s}}-x_{j+1}) |^{8/\kappa}
			- | x_{b_{s}}-x_{j+1} |^{8/\kappa} \big|}{| x_{b_{s}}-x_{j+1} |^{-1+8/\kappa} \, | x_j-x_{a_{r}} |^{-1+8/\kappa}}  \\
		\nonumber
		\; & \times 
		\int_{c_1 \frac{x_{j+1}-x_j}{x_{b_{s}}-x_{j+1}}}^{c_3} \ud v \, \frac{ |v|^{8/\kappa} \, |((x_{b_{s}}-x_{j+1})v+x_{j+1})| }{\big|v\big(v+\frac{x_{j+1}-x_j}{x_{b_{s}}-x_{j+1}}\big)\big|^{4/\kappa}}
		\int_{c_1\frac{x_{j+1}-x_j}{x_{j}-x_{a_{r}}}}^{v}  
		\frac{ \ud u }{|u|^{8/\kappa}} \\
		\nonumber
		\leq \; & \frac{\kappa}{\kappa-8} \, 
		|x_{j+1}-x_j|^{8/\kappa - 1} \, 
		\frac{\big| | (x_j-x_{a_{r}}) +  (x_{b_{s}}-x_{j+1}) |^{8/\kappa}
			- | x_{b_{s}}-x_{j+1} |^{8/\kappa} \big|}{| x_{b_{s}}-x_{j+1} |^{-1+8/\kappa} \, | x_j-x_{a_{r}} |^{-1+8/\kappa}}  \\
		\nonumber
		\; & \times 
		\int_{c_1 \frac{x_{j+1}-x_j}{x_{b_{s}}-x_{j+1}}}^{c_3} \ud v \, \frac{ |v|^{8/\kappa} \, |f_\beta^{(s)}((x_{b_{s}}-x_{j+1})v+x_{j+1})| }{\big|v\big(v+\frac{x_{j+1}-x_j}{x_{b_{s}}-x_{j+1}}\big)\big|^{4/\kappa}}
		\bigg( v^{1-8/\kappa}
		- \Big( c_1\frac{x_{j+1}-x_j}{x_{j}-x_{a_{r}}} \Big)^{1-8/\kappa}
		\bigg) \\
		\nonumber
		\leq \; & \frac{\kappa}{8-\kappa}  \, M_2 \,
		|x_{j+1}-x_j|^{8/\kappa - 1} \, 
		\frac{\big| | (x_j-x_{a_{r}}) +  (x_{b_{s}}-x_{j+1}) |^{8/\kappa}
			- | x_{b_{s}}-x_{j+1} |^{8/\kappa} \big|}{| x_{b_{s}}-x_{j+1} |^{-1+8/\kappa} \, | x_j-x_{a_{r}} |^{-1+8/\kappa}}  \\
		\nonumber
		\; & \times 
		\bigg(
		\Big( c_1\frac{x_{j+1}-x_j}{x_{j}-x_{a_{r}}} \Big)^{1-8/\kappa}
		\int_{c_1 \frac{x_{j+1}-x_j}{x_{b_{s}}-x_{j+1}}}^{c_3} \ud v
		\; - \; \int_{c_1 \frac{x_{j+1}-x_j}{x_{b_{s}}-x_{j+1}}}^{c_3}  \frac{ |v|\, \ud v}{\big|v\big(v+\frac{x_{j+1}-x_j}{x_{b_{s}}-x_{j+1}}\big)\big|^{4/\kappa}} \bigg) \\
		\nonumber
		= \; & \frac{\kappa}{8-\kappa}  \, M_2 \,
		|x_{j+1}-x_j|^{8/\kappa - 1} \, 
		\frac{\big| | (x_j-x_{a_{r}}) +  (x_{b_{s}}-x_{j+1}) |^{8/\kappa}
			- | x_{b_{s}}-x_{j+1} |^{8/\kappa} \big|}{| x_{b_{s}}-x_{j+1} |^{-1+8/\kappa} \, | x_j-x_{a_{r}} |^{-1+8/\kappa}}  \\
		\nonumber
		\; & \times 
		\bigg(
		\Big( c_1\frac{x_{j+1}-x_j}{x_{j}-x_{a_{r}}} \Big)^{1-8/\kappa}
		\Big( c_3 - c_1 \frac{x_{j+1}-x_j}{x_{b_{s}}-x_{j+1}} \Big)
		\; - \; \int_{c_1 \frac{x_{j+1}-x_j}{x_{b_{s}}-x_{j+1}}}^{c_3}  \frac{ |v|\, \ud v}{\big|v\big(v+\frac{x_{j+1}-x_j}{x_{b_{s}}-x_{j+1}}\big)\big|^{4/\kappa}} \bigg) \\
		\nonumber
		\overset{x_j, x_{j+1}\to \xi}{\longrightarrow} \; & \qquad 
		\frac{\kappa}{8-\kappa}  \, M_2 \, c_1^{1-8/\kappa} \, c_3 \,
		\frac{\big| | x_{b_{s}} - x_{a_{r}}  |^{8/\kappa}
			- | x_{b_{s}}-\xi |^{8/\kappa} \big|}{| x_{b_{s}}-\xi |^{-1+8/\kappa}}  \\
		\label{eq: ktilde upper bound 2}
		\overset{c_2 \to 0}{\longrightarrow} \quad \; & \qquad 0 ,
	\end{align}
	where we also used~\eqref{eq: c3 gives bound for h} to bound $|f_\beta^{(s)}|$ 
	(note again that $c_3 \to 0$ along with $c_2 \to 0$ by~\eqref{eq: choice of c3}).

	In conclusion, by combining~(\ref{eq: ktilde upper bound 1},~\ref{eq: ktilde upper bound 2}), we see that $\tilde{\kfunc}(u, v, x_{a_{r}}, x_j, x_{j+1}, x_{b_{s}})$ can be replaced by $\smash{|x_{j+1}-x_j|^{8/\kappa - 1} \,
		\frac{| x_{b_{s}}-x_{j+1} |}{| x_j-x_{a_{r}} |^{-1+8/\kappa}} \, |v|^{8/\kappa}}$ when evaluating the limit of the integral over $R_{2,2}^+ \cup R_{2,3}$. 
	\item 	Third, by using Lemma~\ref{lem: HGE integral} with $0 < \lambda := \frac{x_{j+1}-x_j}{x_{j}-x_{a_{r}}}$, and $0 < \mu :=  c_1 < \frac{1}{\lambda}$, and $\nu := |v| \wedge c_2 < 1$ to evaluate the integral over $u$ in terms of the hypergeometric function $\hF(a,b,c;z)$,
	and then using the asymptotics~\eqref{eq: HGF ASY} of $\hF$ to take the limit $x_j, x_{j+1}\to \xi$, thereafter the limit $c_2 \to 0$, and finally the limit $c_1\to 0$, we find that
	\begin{align*}
		\; & \lim_{c_1 \to 0} \; \lim_{c_2\to 0} \; \lim_{x_j, x_{j+1}\to \xi} \; 
		\underset{R_{2,2}^+ \cup R_{2,3}}{\int} \ud u \, \ud v \, 
		\frac{ f_\beta^{(s)}((x_{b_{s}}-x_{j+1})v+x_{j+1}) }{\big|v\big(v+\frac{x_{j+1}-x_j}{x_{b_{s}}-x_{j+1}}\big)\big|^{4/\kappa}} \;
		\frac{ f_\beta^{(r)}(x_j-(x_j-x_{a_{r}})u) }{\big|u \big(u+\frac{x_{j+1}-x_j}{x_{j}-x_{a_{r}}}\big)\big|^{4/\kappa}} \, \\
&\quad\quad\quad\quad\quad\quad\quad\quad\quad\quad\quad\quad\quad\quad\quad\quad\quad\quad\quad\quad\times\tilde{\kfunc}(u, v, x_{a_{r}}, x_j, x_{j+1}, x_{b_{s}}) \\		 
		= \; & f_\beta^{(r)}(\xi) \, \lim_{c_1 \to 0} \; \lim_{c_2\to 0} \; \lim_{x_j, x_{j+1}\to \xi} \,
		|x_{j+1}-x_j|^{8/\kappa - 1} \,
		\frac{| x_{b_{s}}-x_{j+1} |}{| x_j-x_{a_{r}} |^{-1+8/\kappa}} \\
		\; & \times 
		\landupint_{c_1 \frac{x_{j+1}-x_j}{x_{b_{s}}-x_{j+1}}}^1 \ud v \,
		\frac{ |v|^{8/\kappa} \, f_\beta^{(s)}((x_{b_{s}}-x_{j+1})v+x_{j+1}) }{\big|v\big(v+\frac{x_{j+1}-x_j}{x_{b_{s}}-x_{j+1}}\big)\big|^{4/\kappa}} \;
		\int_{c_1\frac{x_{j+1}-x_j}{x_{j}-x_{a_{r}}}}^{|v| \wedge c_2}  
		\frac{ \ud u }{\big|u \big(u+\frac{x_{j+1}-x_j}{x_{j}-x_{a_{r}}}\big)\big|^{4/\kappa}} \\
		= \; & f_\beta^{(r)}(\xi) \, \lim_{c_1 \to 0} \;  \lim_{c_2\to 0} \; \lim_{x_j, x_{j+1}\to \xi} \,
		|x_{j+1}-x_j|^{8/\kappa - 1} \,
		\frac{| x_{b_{s}}-x_{j+1} |}{| x_j-x_{a_{r}} |^{-1+8/\kappa}} \\
		\; & \times 
		\landupint_{c_1 \frac{x_{j+1}-x_j}{x_{b_{s}}-x_{j+1}}}^1 \ud v \,
		\frac{ |v|^{8/\kappa} \, f_\beta^{(s)}((x_{b_{s}}-x_{j+1})v+x_{j+1}) }{\big|v\big(v+\frac{x_{j+1}-x_j}{x_{b_{s}}-x_{j+1}}\big)\big|^{4/\kappa}} \\ 
		\; & \times 
		\frac{\kappa}{\kappa-4} \, 
		\Big( \frac{x_{j+1}-x_j}{x_{j}-x_{a_{r}}}  \Big)^{-4/\kappa} \, 
		\bigg( (|v| \wedge c_2)^{1-4/\kappa} \, 
		\hF \Big(\frac{4}{\kappa}, 1-\frac{4}{\kappa}, 2-\frac{4}{\kappa}; -\frac{(|v| \wedge c_2) \, (x_{j}-x_{a_{r}})}{x_{j+1}-x_j} \Big) \\
		\; & \qquad\qquad\qquad\qquad\qquad\qquad\qquad\qquad\qquad
		- \Big( c_1 \,\frac{x_{j+1}-x_j}{x_{j}-x_{a_{r}}} \Big)^{1-4/\kappa} \, \hF \Big(\frac{4}{\kappa}, 1-\frac{4}{\kappa}, 2-\frac{4}{\kappa}; -c_1 \Big) \bigg) \\ 
		= \; & \frac{\kappa}{\kappa-4}  \,
		\frac{\Gamma(2-\frac{4}{\kappa}) \Gamma(\frac{8}{\kappa}-1)}{\Gamma(\frac{4}{\kappa}) \Gamma(1)} \,
		f_\beta^{(r)}(\xi) \, (x_{b_{s}}-\xi) \,
		\landupint_{0}^1 \ud v \, f_\beta^{(s)}((x_{b_{s}}-\xi)v+\xi) \\
		= \; & \frac{\kappa}{\kappa-4}  \,
		\frac{\Gamma(2-\frac{4}{\kappa}) \Gamma(\frac{8}{\kappa}-1)}{\Gamma(\frac{4}{\kappa}) \Gamma(1)} \,
		f_\beta^{(r)}(\xi) \, \landupint_{\xi}^{x_{b_{s}}} \ud y \, f_\beta^{(s)}(y) ,
	\end{align*}
	where we also made the change of variables $y = (x_{b_{s}}-\xi)v+\xi$ to obtain the last line.
\end{itemize}

	The contribution of the integral over $R_{2,2}^- \cup R_{3,2}$ can be evaluated similarly by exchanging the roles of $u$ and $v$, and the result is
\begin{align}
\nonumber 
	\; & \lim_{c_1 \to 0} \;  \lim_{c_2\to 0} \; \lim_{x_j, x_{j+1}\to \xi} \; 
	\underset{R_{2,2}^- \cup R_{3,2}}{\int} \ud u \, \ud v \, 
	\frac{ f_\beta^{(s)}((x_{b_{s}}-x_{j+1})v+x_{j+1}) }{\big|v\big(v+\frac{x_{j+1}-x_j}{x_{b_{s}}-x_{j+1}}\big)\big|^{4/\kappa}} \;
	\frac{ f_\beta^{(r)}(x_j-(x_j-x_{a_{r}})u) }{\big|u \big(u+\frac{x_{j+1}-x_j}{x_{j}-x_{a_{r}}}\big)\big|^{4/\kappa}} \, 
	\\
	\nonumber
	&\quad\quad\quad\quad\quad\quad\quad\quad\quad\quad\quad\quad\quad\quad\quad\quad\quad\quad\quad\quad\quad\times\tilde{\kfunc}(u, v, x_{a_{r}}, x_j, x_{j+1}, x_{b_{s}}) \\
			\label{eqn::R22minus} 
	= \; & \frac{\kappa}{\kappa-4}  \,
	\frac{\Gamma(2-\frac{4}{\kappa}) \Gamma(\frac{8}{\kappa}-1)}{\Gamma(\frac{4}{\kappa}) \Gamma(1)} \,
	f_\beta^{(s)}(\xi) \, \landupint_{x_{a_{r}}}^{\xi} \ud y \, f_\beta^{(r)}(y) .
\end{align}
	\end{enumerate}

Collecting all contributions, we finally obtain 
\begin{align*}
	\; & \lim_{x_j, x_{j+1}\to \xi} 
	\frac{I_A(x_{a_r}, x_j, x_{j+1}, x_{b_s})
	}{|x_{j+1}-x_j|^{1 - 8/\kappa}} \\
	= \; & \lim_{x_j, x_{j+1}\to \xi}  
	\landupint_0^1 \ud u \, \frac{f_\beta^{(r)}(x_j-(x_j-x_{a_{r}})u)}{\big|u \big(u+\frac{x_{j+1}-x_j}{x_{j}-x_{a_{r}}}\big)\big|^{4/\kappa}} \;
	\landupint_0^1 \ud v  \,
	\frac{f_\beta^{(s)}((x_{b_{s}}-x_{j+1})v+x_{j+1})}{\big|v\big(v+\frac{x_{j+1}-x_j}{x_{b_{s}}-x_{j+1}}\big)\big|^{4/\kappa}} 
	\; \tilde{\kfunc}(u, v, x_{a_{r}}, x_j, x_{j+1}, x_{b_{s}}) \\
	= \; & \frac{\kappa}{\kappa-4} \,
	\frac{\Gamma(2-\frac{4}{\kappa}) \Gamma(\frac{8}{\kappa}-1)}{\Gamma(\frac{4}{\kappa})} 
	\bigg( f_\beta^{(r)}(\xi) \, \landupint_{\xi}^{x_{b_{s}}} \ud y \, f_\beta^{(s)}(y) 
	+ f_\beta^{(s)}(\xi) \, \landupint_{x_{a_{r}}}^{\xi} \ud y \, f_\beta^{(r)}(y) 
	\bigg) 
	\qquad\qquad\;\;
\textnormal{[by~(\ref{eqn::R22minus})]}	
	\\
	= \; & \frac{\kappa}{\kappa-4} \,
	\frac{\Gamma(2-\frac{4}{\kappa}) \Gamma(\frac{8}{\kappa}-1)}{\Gamma(\frac{4}{\kappa})} \, f_\beta^{(s)}(\xi) \, \landupint_{x_{a_{r}}}^{x_{b_{s}}} \ud y \, f_\beta^{(r)}(y) .
	\qquad\qquad\qquad\qquad\qquad\qquad\qquad\qquad\, \textnormal{[by~\eqref{eq: relationhrhs}]}
\end{align*} 
Using also the functional equation $\Gamma(1-\nu)\Gamma(\nu)=\frac{\pi}{\sin(\pi \nu)}$, we find the multiplicative constant 
\begin{align*}
\frac{\kappa}{\kappa-4} \,
\frac{\Gamma(2-\frac{4}{\kappa}) \Gamma(\frac{8}{\kappa}-1)}{\Gamma(\frac{4}{\kappa})} \,=\frac{\Gamma(1-4/\kappa)^2}{\sqrt{q(\kappa)} \, \Gamma(2-8/\kappa)} .
\end{align*}
This completes the proof.
\end{proof}

\begin{proof}[Proof of Proposition~\ref{pro::Vbeta_ASY2}, Case~\ref{item::caseB}]
	Define $\beta_B := \beta \setminus (\{a_{s}, j\}\cup \{a_{r},j+1\})$ (we do not relabel the indices here), and denote by 
	$\Gamma_{\beta_B}$ the integration contours in $\coulomb_{\beta}$ other than $(x_{a_s}, x_j)$, $(x_{a_r},x_{j+1})$.
	Then, we have
	\begin{align} \label{eqn::IB_def}
		\coulomb_{\beta}(x_1, \ldots, x_{2N}) 
		= 
		\int_{\Gamma_{\beta_B}} \landupint_{x_{a_{s}}}^{x_{j}} \landupint_{x_{a_{r}}}^{x_{j+1}} 
		\ud \bs{u} \, f_\beta(\bs{x}; \bs{u}) 
		= \; & \int_{\Gamma_{\beta_B}} \ud \bs{\ddot{u}} \, f_\beta(\bs{x}; \bs{\ddot{u}}) \;
		I_B(x_{a_r}, x_{a_{s}}, x_j, x_{j+1}) ,
	\end{align}
	where, as in the proof of Case~\ref{item::caseA}, 
	$f_\beta(\bs{x}; \bs{\ddot{u}})$ 
	is a part of the integrand function~{\eqref{eq: integrand}} chosen to be real and positive on~\eqref{eq:: branch choice set smaller}, 
	and where $I_B(x_{a_r}, x_{a_{s}}, x_j, x_{j+1}) =: I_B$ is the integral 
	\begin{align*}
		I_B := \landupint_{x_{a_{s}}}^{x_j} \ud u_{s} \,
		\frac{f_\beta^{(s)}(u_{s})}{|u_{s}-x_j|^{4/\kappa} |u_{s}-x_{j+1}|^{4/\kappa}} 
		\landupint_{x_{a_{r}}}^{x_{j+1}} \ud u_{r} \,
		\frac{(u_{s}-u_{r})^{8/\kappa} \, f_\beta^{(r)}(u_{r})}{|u_{r}-x_j|^{4/\kappa} |u_{r}-x_{j+1}|^{4/\kappa}},
	\end{align*}
	with $x_{a_{s}} < \Re(u_{s}) < x_j  < x_{a_{r}} < \Re(u_{r}) < x_{j+1}$, 
	where the branch of $(u_{s}-u_{r})^{8/\kappa}$ is chosen to be positive when $\Re(u_{r}) < \Re(u_{s})$, 
	and, as before, 
	$f_\beta^{(r)}$ and $f_\beta^{(s)}$ are the multivalued functions
	with branch choices~\eqref{eqn::branchchoice_natural_r} and~\eqref{eqn::branchchoice_natural_s}, respectively.
Note that for any fixed $\bs{\ddot{x}} \in \chamber_{2N-2}$ and $\bs{\ddot{u}} \in \Gamma_{\beta_B}$, we have
	\begin{align*} 
		f_\beta^{(s)}(x) \, f_\beta^{(r)}(y) = f_\beta^{(s)}(y) \, f_\beta^{(r)}(x) ,
	\end{align*}
	for all $x, y \notin \{ x_1, \ldots, x_{j-1}, x_{j+2}, \ldots, x_{2N}, u_1, \ldots, u_{r-1}, u_{r+1}, \ldots, u_{s-1}, u_{s+1}, \ldots, u_N \}$ such that $x \neq y$, since the phase factors from the exchange of $x$ and $y$ in the product cancel out.

We proceed similarly as in the proof of Case~\ref{item::caseA}. 	
After making the changes of variables  
	$w=-\frac{x_{j+1}-u_{r}}{x_{j+1}-x_{a_{r}}}$ in the first integral
	and $u=\frac{x_j-u_{s}}{x_j-x_{a_{s}}}$ in the second integral, we obtain 
	\begin{align*}
		I_B = & \; 
		\landupint_0^1 \ud u \, \frac{f_\beta^{(s)}(x_j-(x_j-x_{a_{s}})u)}{\big|u \big(u + \frac{x_{j+1}-x_j}{x_j-x_{a_{s}}}\big)\big|^{4/\kappa}} \;
		\landupint_{-1}^0 \ud w  \,
		\frac{f_\beta^{(r)}(x_{j+1} + (x_{j+1}-x_{a_{r}})w)}{\big|w\big(w + \frac{x_{j+1} - x_j}{x_{j+1}-x_{a_{r}}}\big)\big|^{4/\kappa}} \; 
		\kfunc(u, w, x_{a_r}, x_{a_{s}}, x_j, x_{j+1}) ,
	\end{align*}
	where
	\begin{align*} 
		\kfunc(u, v, x_{a_r}, x_{a_{s}}, x_j, x_{j+1})
		:= \; & \frac{\big( x_{j+1} - x_j + u \, (x_j-x_{a_{s}}) + w \, (x_{j+1} - x_{a_{r}})\big)^{8/\kappa}}{| x_{a_{r}}-x_{j+1} |^{-1+8/\kappa} \, | x_j-x_{a_{s}} |^{-1+8/\kappa}} \\
		\nonumber
		= \; & \frac{\big( u \, (x_j-x_{a_{s}}) + w \, (x_{j+1} - x_{a_{r}}) \big)^{8/\kappa}}{| x_{a_{r}}-x_{j+1} |^{-1+8/\kappa} \, | x_j-x_{a_{s}} |^{-1+8/\kappa}} + \OO( |x_{j+1} - x_j| ) ,
		\qquad |x_{j+1} - x_j| \to 0 .
	\end{align*}
	This integral has a similar form as for $I_A$ defined in~\eqref{eqn::def_I_A}, except for the following changes:
	\begin{itemize}
		\item $x_{a_{r}}$ in $I_B$ plays the role of $x_{b_{s}}$ in $I_A$;
		
		\item $x_{a_{s}}$ in $I_B$ plays the role of $x_{a_{r}}$ in $I_A$; 
		
		\item in $I_B$, we have $x_{j+1}-x_{a_{r}} > 0$, while in $I_A$, we have $x_{b_{s}}-x_{j+1} > 0$;
		
		\item we integrate in $I_B$ the variable $w \in (-1,0)$, while in $I_A$ the corresponding variable is $v \in (0,1)$.
	\end{itemize}
	Nevertheless, this only affects the estimates slightly, so with similar estimates as in the proof of Case~\ref{item::caseA},
	one can show that
	\begin{align} \label{eqn::convergence_I_B}
		\lim_{x_j, x_{j+1}\to \xi} 
		\frac{I_B(x_{a_r}, x_{a_{s}}, x_j, x_{j+1})
		}{|x_{j+1}-x_j|^{1 - 8/\kappa}} 
		= \; & \frac{\Gamma(1-4/\kappa)^2}{\sqrt{q(\kappa)} \, \Gamma(2-8/\kappa)} \, f_\beta^{(s)}(\xi) \, \landupint_{x_{a_{r}}}^{x_{a_{s}}} \ud y \, f_\beta^{(r)}(y).
	\end{align}
	We then conclude from~\eqref{eqn::IB_def} and~\eqref{eqn::convergence_I_B} that~\eqref{eqn::ASY2} holds:
	\begin{align*}
		\; & \lim_{x_j, x_{j+1}\to \xi} 
		\frac{\coulomb_{\beta}(\bs{x})}{(x_{j+1}-x_j)^{ -2h(\kappa)}} \\
		= \; & \lim_{x_j,x_{j+1}\to\xi} 
		(x_{j+1}-x_j)^{6/\kappa-1}
		\int_{\Gamma_{\beta_B}} \landupint_{x_{a_{r}}}^{x_{j}} \landupint_{x_{a_{s}}}^{x_{j+1}}
		\ud \bs{u} \, f_\beta(\bs{x}; \bs{u})
		 \\
		= \; & 
		\lim_{x_j,x_{j+1}\to\xi} 
		(x_{j+1}-x_j)^{6/\kappa-1}
		\int_{\Gamma_{\beta_B}} \ud \bs{\ddot{u}} \, f_\beta(\bs{x}; \bs{\ddot{u}}) \;
		I_B(x_{a_r}, x_{a_{s}}, x_j, x_{j+1})
		&& \textnormal{[by~\eqref{eqn::IB_def}]}
		\\
		= \; & \frac{\Gamma(1-4/\kappa)^2}{\sqrt{q(\kappa)} \, \Gamma(2-8/\kappa)} \, 
		\coulomb_{\wp_j(\beta)/\{j,j+1\}}(\bs{\ddot{x}}_j) ,
				&& \textnormal{[by~\eqref{eqn::convergence_I_B}]}
	\end{align*}
	after carefully collecting the phase factors (and recalling that $\xi \in (x_{j-1}, x_{j+2})$ and that $f_\beta(\bs{x}; \bs{\ddot{u}})$ is real and positive on~\eqref{eq:: branch choice set smaller},
	$f_\beta^{(r)}$ is real and positive on~\eqref{eqn::branchchoice_natural_r}, and
	$f_\beta^{(s)}$ is real and positive on~\eqref{eqn::branchchoice_natural_s}).
\end{proof}

\begin{proof}[Proof of Proposition~\ref{pro::Vbeta_ASY2}, Case~\ref{item::caseC}]
This symmetric to Case~\ref{item::caseB} and can be proven very similarly. 
\end{proof}

\newpage


\newcommand{\etalchar}[1]{$^{#1}$}


\begin{thebibliography}{DCKK{\etalchar{+}}20}

\bibitem[AS64]{Abramowitz-Stegun:Handbook}
Milton Abramowitz and Irene~A. Stegun.
\newblock {\em Handbook of mathematical functions}.
\newblock Dover Publications Inc., 1964.

\bibitem[AS21]{Ang-Sun:Integrability_of_CLE}
Morris Ang and Xin Sun.
\newblock Integrability of the conformal loop ensemble.
\newblock Preprint in \url{arXiv:2107.01788}, 2021.

\bibitem[BBK05]{BBK:Multiple_SLEs_and_statistical_mechanics_martingales}
Michel Bauer, Denis Bernard, and Kalle Kyt{\"o}l{\"a}.
\newblock Multiple {S}chramm-{L}oewner evolutions and statistical mechanics
  martingales.
\newblock {\em J. Stat. Phys.}, 120(5-6):1125--1163, 2005.

\bibitem[BDC12]{Beffara-Duminil-Copin:The_self-dual_point_of_the_two-dimensional_random-cluster_model_is_critical_for_q_bigger_than_one}
Vincent Beffara and Hugo Duminil-Copin.
\newblock The self-dual point of the two-dimensional random-cluster model is
  critical for {$q \geq 1$}.
\newblock {\em Probab. Th. Rel. Fields}, 153(3-4):511--542, 2012.

\bibitem[BPW21]{BPW:On_the_uniqueness_of_global_multiple_SLEs}
Vincent Beffara, Eveliina Peltola, and Hao Wu.
\newblock On the uniqueness of global multiple $\mathrm{SLE}$s.
\newblock {\em Ann. Probab.}, 49(1):400--434, 2021.

\bibitem[BPZ84a]{BPZ:Infinite_conformal_symmetry_in_2D_QFT}
Alexander~A. Belavin, Alexander~M. Polyakov, and Alexander~B. Zamolodchikov.
\newblock Infinite conformal symmetry in two-dimensional quantum field theory.
\newblock {\em Nucl. Phys. B}, 241(2):333--380, 1984.

\bibitem[BPZ84b]{BPZ:Infinite_conformal_symmetry_of_critical_fluctuations_in_2D}
Alexander~A. Belavin, Alexander~M. Polyakov, and Alexander~B. Zamolodchikov.
\newblock Infinite conformal symmetry of critical fluctuations in two
  dimensions.
\newblock {\em J. Stat. Phys.}, 34(5-6):763--774, 1984.

\bibitem[Car84]{Cardy:Conformal_invariance_and_surface_critical_behavior}
John~L. Cardy.
\newblock Conformal invariance and surface critical behavior.
\newblock {\em Nucl. Phys. B}, 240(4):514--532, 1984.

\bibitem[Car92]{Cardy:Critical_percolation_in_finite_geometries}
John~L. Cardy.
\newblock Critical percolation in finite geometries.
\newblock {\em J. Phys. A}, 25(4):L201--206, 1992.

\bibitem[Car96]{Cardy:Scaling_and_renormalization_in_statistical_physics}
John~L. Cardy.
\newblock {\em Scaling and renormalization in statistical physics}, volume~5 of
  {\em Cambridge Lecture Notes in Physics}.
\newblock Cambridge University Press, 1996.

\bibitem[CDCH{\etalchar{+}}14]{CDHKS:Convergence_of_Ising_interfaces_to_SLE}
Dmitry Chelkak, Hugo Duminil-Copin, Cl{\'e}ment Hongler, Antti Kemppainen, and
  Stanislav Smirnov.
\newblock Convergence of {I}sing interfaces to {S}chramm's {$\mathrm{SLE}$}
  curves.
\newblock {\em C. R. Acad. Sci. Paris S{\'e}r. I Math.}, 352(2):157--161, 2014.

\bibitem[CDCH16]{CDCH:Crossing_probabilities_in_topological_rectangles_for_critical_planar_FK_Ising_model}
Dmitry Chelkak, Hugo Duminil-Copin, and Cl{\'e}ment Hongler.
\newblock Crossing probabilities in topological rectangles for the critical
  planar {FK}-{I}sing model.
\newblock {\em Electron. J. Probab.}, 21:1--28, 2016.

\bibitem[Che20]{Chelkak:Ising_model_and_s-embeddings_of_planar_graphs}
Dmitry Chelkak.
\newblock Ising model and s-embeddings of planar graphs.
\newblock Preprint in \url{arXiv:2006.14559}, 2020.

\bibitem[CHI15]{CHI:Conformal_invariance_of_spin_correlations_in_planar_Ising_model}
Dmitry Chelkak, Cl{\'e}ment Hongler, and Konstantin Izyurov.
\newblock Conformal invariance of spin correlations in the planar {I}sing
  model.
\newblock {\em Ann. Math.}, 181(3):1087--1138, 2015.

\bibitem[CHI21]{CHI:Correlations_of_primary_fields_in_the_critical_planar_Ising_model}
Dmitry Chelkak, Cl{\'e}ment Hongler, and Konstantin Izyurov.
\newblock Correlations of primary fields in the critical planar {I}sing model.
\newblock Preprint in arXiv:2103.10263, 2021.

\bibitem[CS11]{Chelkak-Smirnov:Discrete_complex_analysis_on_isoradial_graphs}
Dmitry Chelkak and Stanislav Smirnov.
\newblock Discrete complex analysis on isoradial graphs.
\newblock {\em Adv. Math.}, 228(3):1590--1630, 2011.

\bibitem[CS12]{Chelkak-Smirnov:Universality_in_2D_Ising_and_conformal_invariance_of_fermionic_observables}
Dmitry Chelkak and Stanislav Smirnov.
\newblock Universality in the {$2D$} {I}sing model and conformal invariance of
  fermionic observables.
\newblock {\em Invent. Math.}, 189(3):515--580, 2012.

\bibitem[CW21]{Chelkak-Wan:On_the_convergence_of_massive_loop-erased_random_walks_to_massive_SLE2_curves}
Dmitry Chelkak and Yijun Wan.
\newblock On the convergence of massive loop-erased random walks to massive
  {SLE}(2) curves.
\newblock {\em Electron. J. Probab.}, 26(54):1--35, 2021.

\bibitem[DC17]{Duminil-Copin:PIMS_lectures}
Hugo Duminil-Copin.
\newblock Lectures on the {I}sing and {P}otts models on the hypercubic lattice.
\newblock In {\em PIMS-CRM Summer School in Probability}, pages 35--161.
  Springer, Berlin, 2017.

\bibitem[DCGH{\etalchar{+}}21]{DCGHMT:Discontinuity_of_the_phase_transition_for_the_planar_random-cluster_and_Potts_with_q_larger_than_4}
Hugo Duminil-Copin, Maxim Gagnebin, Matan Harel, Ioan Manolescu, and Vincent
  Tassion.
\newblock Discontinuity of the phase transition for the planar random-cluster
  and {P}otts models with {$q>4$}.
\newblock {\em Ann. Sci. {\'E}c. Norm. Sup{\'e}r.}, 6(54):1363--1413, 2021.

\bibitem[DCHN11]{DCHN:Connection_probabilities_and_RSW_type_bounds}
Hugo Duminil-Copin, Cl{\'e}ment Hongler, and Pierre Nolin.
\newblock Connection probabilities and {RSW}-type bounds for the {FK} {I}sing
  model.
\newblock {\em Comm. Pure Appl. Math.}, 64(9):1165--1198, 2011.

\bibitem[DCKK{\etalchar{+}}20]{DKKMO:Rotational_invariance_in_critical_planar_lattice_models}
Hugo Duminil-Copin, Karol~Kajetan Kozlowski, Dmitry Krachun, Ioan Manolescu,
  and Mendes Oulamara.
\newblock Rotational invariance in critical planar lattice models.
\newblock Preprint in \url{arXiv:2012.11672}, 2020.

\bibitem[DCMT21]{DCMT:Planar_random-cluster_model_fractal_properties_of_the_critical_phase}
Hugo Duminil-Copin, Ioan Manolescu, and Vincent Tassion.
\newblock Planar random-cluster model: fractal properties of the critical
  phase.
\newblock {\em Probab. Theory Related Fields}, 181(1-3):401--449, 2021.

\bibitem[DCS12]{DCS:Conformal_invariance_of_lattice_models}
Hugo Duminil-Copin and Stanislav Smirnov.
\newblock Conformal invariance of lattice models.
\newblock In {\em Probability and {S}tatistical {P}hysics in two and more
  dimensions}, volume~15, pages 213--276. Clay Mathematics Proceedings,
  American Mathematical Society, Providence, RI, 2012.

\bibitem[DCST17]{DCST:Continuity_of_phase_transition_for_planar_random-cluster_and_Potts_models}
Hugo Duminil-Copin, Vladas Sidoravicius, and Vincent Tassion.
\newblock Continuity of the phase transition for planar random-cluster and
  {P}otts models with {$1 \leq q \leq 4$}.
\newblock {\em Comm. Math. Phys.}, 349(1):47--107, 2017.

\bibitem[DF84]{Dotsenko-Fateev:Conformal_algebra_and_multipoint_correlation_functions_in_2D_statistical_models}
Vladimir~S. Dotsenko and Vladimir~A. Fateev.
\newblock Conformal algebra and multipoint correlation functions in ${2D}$
  statistical models.
\newblock {\em Nucl. Phys. B}, 240(3):312--348, 1984.

\bibitem[DFGG97]{DGG:Meanders_and_TL_algebra}
P.~Di~Francesco, O.~Golinelli, and E.~Guitter.
\newblock Meanders and the {T}emperley-{L}ieb algebra.
\newblock {\em Comm. Math. Phys.}, 186(1):1--59, 1997.

\bibitem[DFMS97]{DMS:CFT}
Philippe Di~Francesco, Pierre Mathieu, and David S{\'e}n{\'e}chal.
\newblock {\em Conformal field theory}.
\newblock Graduate Texts in Contemporary Physics. Springer-Verlag, New York,
  1997.

\bibitem[Dub06]{Dubedat:Euler_integrals_for_commuting_SLEs}
Julien Dub{\'e}dat.
\newblock Euler integrals for commuting {$\mathrm{SLE}$}s.
\newblock {\em J. Stat. Phys.}, 123(6):1183--1218, 2006.

\bibitem[Dub07]{Dubedat:Commutation_relations_for_SLE}
Julien Dub{\'e}dat.
\newblock Commutation relations for {$\mathrm{SLE}$}.
\newblock {\em Comm. Pure Appl. Math.}, 60(12):1792--1847, 2007.

\bibitem[FK15a]{Flores-Kleban:Solution_space_for_system_of_null-state_PDE2}
Steven~M. Flores and Peter Kleban.
\newblock A solution space for a system of null-state partial differential
  equations 2.
\newblock {\em Comm. Math. Phys.}, 333(1):435--481, 2015.

\bibitem[FK15b]{Flores-Kleban:Solution_space_for_system_of_null-state_PDE3}
Steven~M. Flores and Peter Kleban.
\newblock A solution space for a system of null-state partial differential
  equations 3.
\newblock {\em Comm. Math. Phys.}, 333(2):597--667, 2015.

\bibitem[FK15c]{Flores-Kleban:Solution_space_for_system_of_null-state_PDE4}
Steven~M. Flores and Peter Kleban.
\newblock A solution space for a system of null-state partial differential
  equations 4.
\newblock {\em Comm. Math. Phys.}, 333(2):669--715, 2015.

\bibitem[FP18]{Flores-Peltola:Standard_modules_radicals_and_the_valenced_TL_algebra}
Steven~M. Flores and Eveliina Peltola.
\newblock Standard modules, radicals, and the valenced {T}emperley-{L}ieb
  algebra.
\newblock Preprint in arXiv:1801.10003, 2018.

\bibitem[FSKZ17]{FSKZ:A_formula_for_crossing_probabilities_of_critical_systems_inside_polygons}
Steven~M. Flores, Jacob J.~H. Simmons, Peter Kleban, and Robert~M. Ziff.
\newblock A formula for crossing probabilities of critical systems inside
  polygons.
\newblock {\em J. Phys. A}, 50(6):064005, 2017.

\bibitem[GPS13]{GPS:Pivotal_cluster_and_interface_measures_for_critical_planar_percolation}
Christophe Garban, Gabor Pete, and Oded Schramm.
\newblock Pivotal, cluster, and interface measures for critical planar
  percolation.
\newblock {\em J. Amer. Math. Soc.}, 26(4):939--1024, 2013.

\bibitem[Gri09]{Grimmett:Random_cluster_model}
Geoffrey Grimmett.
\newblock {\em The random-cluster model}, volume 333 of {\em A series of
  Comprehensive Studies in Mathematics}.
\newblock Springer-Verlag, corrected second printing edition, 2009.

\bibitem[GW20]{Garban-Wu:On_the_convergence_of_FK-Ising_percolation_to_SLE}
Christophe Garban and Hao Wu.
\newblock On the convergence of {FK}-{I}sing percolation to {SLE}$(16/3,
  16/3-6)$.
\newblock {\em J. Theor. Probab.}, 33(2):828--865, 2020.

\bibitem[Izy15]{Izyurov:Smirnovs_observable_for_free_boundary_conditions_interfaces_and_crossing_probabilities}
Konstantin Izyurov.
\newblock Smirnov's observable for free boundary conditions, interfaces and
  crossing probabilities.
\newblock {\em Comm. Math. Phys.}, 337(1):225--252, 2015.

\bibitem[Izy22]{Izyurov:On_multiple_SLE_for_the_FK_Ising_model}
Konstantin Izyurov.
\newblock On multiple {SLE} for the {FK}-{I}sing model.
\newblock {\em Ann. Probab.}, 50(2):771--790, 2022.

\bibitem[JSW20]{JSW:IGMC_Moments_regularity_and_connections_to_Ising_model}
Janne Junnila, Eero Saksman, and Christian Webb.
\newblock Imaginary multiplicative chaos: {M}oments, regularity, and
  connections to the {I}sing model.
\newblock {\em Ann. Appl. Probab.}, 30(5):2099--2164, 2020.

\bibitem[Kar18]{Karrila:Limits_of_conformal_images_and_conformal_images_of_limits_for_planar_random_curves}
Alex Karrila.
\newblock Limits of conformal images and conformal images of limits for planar
  random curves.
\newblock Preprint in \url{arXiv:1810.05608}, 2018.

\bibitem[Kar19]{Karrila:Multiple_SLE_local_to_global}
Alex Karrila.
\newblock Multiple {$\mathrm{SLE}$} type scaling limits: from local to global.
\newblock Preprint in \url{arXiv:1903.10354}, 2019.

\bibitem[Kar20]{Karrila:UST_branches_martingales_and_multiple_SLE2}
Alex Karrila.
\newblock U{ST} branches, martingales, and multiple {$\mathrm{SLE}(2)$}.
\newblock {\em Electron. J. Probab.}, 25:1--37, 2020.

\bibitem[Ken00]{Kenyon:Conformal_invariance_of_domino_tiling}
Richard~W. Kenyon.
\newblock Conformal invariance of domino tiling.
\newblock {\em Ann. Probab.}, 28(2):759--795, 2000.

\bibitem[KKP20]{KKP:Boundary_correlations_in_planar_LERW_and_UST}
Alex Karrila, Kalle Kyt{\"o}l{\"a}, and Eveliina Peltola.
\newblock Boundary correlations in planar {LERW} and {UST}.
\newblock {\em Comm. Math. Phys.}, 376(3):2065--2145, 2020.

\bibitem[KP16]{Kytola-Peltola:Pure_partition_functions_of_multiple_SLEs}
Kalle Kyt{\"o}l{\"a} and Eveliina Peltola.
\newblock Pure partition functions of multiple $\mathrm{SLE}$s.
\newblock {\em Comm. Math. Phys.}, 346(1):237--292, 2016.

\bibitem[KP20]{Kytola-Peltola:Conformally_covariant_boundary_correlation_functions_with_quantum_group}
Kalle Kyt{\"o}l{\"a} and Eveliina Peltola.
\newblock Conformally covariant boundary correlation functions with a quantum
  group.
\newblock {\em J. Eur. Math. Soc.}, 22(1):55--118, 2020.

\bibitem[KS16]{Kemppainen-Smirnov:Conformal_invariance_in_random-cluster_models-II}
Antti Kemppainen and Stanislav Smirnov.
\newblock Conformal invariance in random-cluster models. {II}. {F}ull scaling
  limit as a branching {SLE}.
\newblock Preprint in \url{arXiv:1609.08527}, 2016.

\bibitem[KS17]{Kemppainen-Smirnov:Random_curves_scaling_limits_and_Loewner_evolutions}
Antti Kemppainen and Stanislav Smirnov.
\newblock Random curves, scaling limits and {L}oewner evolutions.
\newblock {\em Ann. Probab.}, 45(2):698--779, 2017.

\bibitem[KS18]{Kemppainen-Smirnov:Configurations_of_FK_Ising_interfaces}
Antti Kemppainen and Stanislav Smirnov.
\newblock Configurations of {FK} {I}sing interfaces and hypergeometric
  {$\mathrm{SLE}$}.
\newblock {\em Math. Res. Lett.}, 25(3):875--88, 2018.

\bibitem[KS19]{Kemppainen-Smirnov:Conformal_invariance_of_boundary_touching_loops_of_FK_Ising_model}
Antti Kemppainen and Stanislav Smirnov.
\newblock Conformal invariance of boundary touching loops of {FK} {I}sing
  model.
\newblock {\em Comm. Math. Phys.}, 369(1):49--98, 2019.

\bibitem[KW11]{Kenyon-Wilson:Boundary_partitions_in_trees_and_dimers}
Richard~W. Kenyon and David~B. Wilson.
\newblock Boundary partitions in trees and dimers.
\newblock {\em Trans. Amer. Math. Soc.}, 363(3):1325--1364, 2011.

\bibitem[Law05]{Lawler:Conformally_invariant_processes_in_the_plane}
Gregory~F. Lawler.
\newblock {\em Conformally invariant processes in the plane}, volume 114 of
  {\em Mathematical Surveys and Monographs}.
\newblock American Mathematical Society, 2005.

\bibitem[LPSA94]{LPS:Conformal_invariance_in_2d_percolation}
Robert~P. Langlands, Philippe Pouliot, and Yvan Saint-Aubin.
\newblock Conformal invariance in {$2D$} percolation.
\newblock {\em Bull. Amer. Math. Soc.}, 30:1--61, 1994.

\bibitem[LPW24]{LPW:UST_in_topological_polygons_partition_functions_for_SLE8_and_correlations_in_logCFT}
Mingchang Liu, Eveliina Peltola, and Hao Wu.
\newblock Uniform spanning tree in topological polygons, partition functions
  for {$\mathrm{SLE}(8)$}, and correlations in {$c=-2$} logarithmic {CFT}.
\newblock {\em Ann. Probab.}, to appear, 2024.
\newblock Preprint in \url{arXiv:2108.04421}.

\bibitem[LW21]{Liu-Wu:Scaling_limits_of_crossing_probabilities_in_metric_graph_GFF}
Mingchang Liu and Hao Wu.
\newblock Scaling limits of crossing probabilities in metric graph {GFF}.
\newblock {\em Electron. J. Probab.}, 26(37):1--46, 2021.

\bibitem[MW18]{Miller-Werner:Connection_probabilities_for_conformal_loop_ensembles}
Jason Miller and Wendelin Werner.
\newblock Connection probabilities for conformal loop ensembles.
\newblock {\em Comm. Math. Phys.}, 362(2):415--453, 2018.

\bibitem[Pom92]{Pommerenke:Boundary_behaviour_of_conformal_maps}
Christian Pommerenke.
\newblock {\em Boundary behaviour of conformal maps}, volume 299 of {\em
  Grundlehren der mathematischen Wissenschaften}.
\newblock Springer-Verlag, Berlin Heidelberg, 1992.

\bibitem[PSVD13]{DPSV:Connectivities_of_Potts_Fortuin-Kasteleyn_clusters_and_time-like_Liouville_correlator}
Marco Picco, Raoul Santachiara, Jacopo Viti, and Gesualdo Delfino.
\newblock Connectivities of {P}otts {F}ortuin-{K}asteleyn clusters and
  time-like {L}iouville correlator.
\newblock {\em Nucl. Phys. B}, 875(3):719--737, 2013.

\bibitem[PW19]{Peltola-Wu:Global_and_local_multiple_SLEs_and_connection_probabilities_for_level_lines_of_GFF}
Eveliina Peltola and Hao Wu.
\newblock Global and local multiple $\mathrm{SLE}$s for $\kappa \leq 4$ and
  connection probabilities for level lines of {GFF}.
\newblock {\em Comm. Math. Phys.}, 366(2):469--536, 2019.

\bibitem[PW24]{PW:Crossing_probabilities_of_critical_percolation_interfaces}
Eveliina Peltola and Hao Wu.
\newblock Crossing probabilities of multiple percolation interfaces:
  generalizations of {C}ardy's formula and {W}att's formula.
\newblock In preparation, 2024.

\bibitem[RS05]{Rohde-Schramm:Basic_properties_of_SLE}
Steffen Rohde and Oded Schramm.
\newblock Basic properties of $\mathrm{SLE}$.
\newblock {\em Ann. of Math.}, 161(2):883--924, 2005.

\bibitem[SKFZ11]{SKFZ:Cluster_densities_at_2D_critical_points_in_rectangular_geometries}
Jacob J.~H. Simmons, Peter Kleban, Steven~M. Flores, and Robert~M. Ziff.
\newblock Cluster densities at $2$-{D} critical points in rectangular
  geometries.
\newblock {\em J. Phys. A}, 44(38):385002, 2011.

\bibitem[Smi01]{Smirnov:Critical_percolation_in_the_plane}
Stanislav Smirnov.
\newblock Critical percolation in the plane: conformal invariance, {C}ardy's
  formula, scaling limits.
\newblock {\em C. R. Acad. Sci.}, 333(3):239--244, 2001.
\newblock (Updated 2009, see arXiv:0909.4499).

\bibitem[Smi10]{Smirnov:Conformal_invariance_in_random_cluster_models1}
Stanislav Smirnov.
\newblock Conformal invariance in random cluster models {I}. {H}olomorphic
  fermions in the {I}sing model.
\newblock {\em Ann. Math.}, 172(2):1435--1467, 2010.

\bibitem[SS11]{Schramm-Smirnov:Scaling_limits_of_planar_percolation}
Oded Schramm and Stanislav Smirnov.
\newblock On the scaling limits of planar percolation.
\newblock {\em Ann. Probab.}, 39(5):1768--1814, 2011.
\newblock With an appendix by {C}hristophe {G}arban.

\bibitem[SW05]{Schramm-Wilson:SLE_coordinate_changes}
Oded Schramm and David~B. Wilson.
\newblock $\mathrm{SLE}$ coordinate changes.
\newblock {\em New York J. Math.}, 11:659--669, 2005.

\bibitem[Wu20]{Wu:Convergence_of_the_critical_planar_ising_interfaces_to_hypergeometric_SLE}
Hao Wu.
\newblock Hypergeometric {$\mathrm{SLE}$}: conformal {M}arkov characterization
  and applications.
\newblock {\em Comm. Math. Phys.}, 374(1):433--484, 2020.

\end{thebibliography}
\end{document}